%% file: Thesis.tex
\documentclass[a4paper]{book}

\usepackage{afterpage}
\usepackage{amsmath}
\usepackage{amssymb}
\usepackage{amsthm}
\usepackage{dutchcal}
\usepackage{color}
\usepackage{enumitem}
\usepackage{epigraph}
\usepackage{geometry}
\usepackage{graphicx}
\usepackage{hyperref}
\usepackage[utf8]{inputenc}
\usepackage{mathrsfs}
\usepackage{mathtools}
\usepackage{makeidx}
\usepackage{pict2e}
\usepackage{tikz}
\usetikzlibrary{arrows}
\usetikzlibrary{decorations.pathmorphing}
\usetikzlibrary{decorations.pathreplacing}
\usepackage{xcolor}

\hypersetup{
    colorlinks,
    linkcolor={red!50!black},
    citecolor={blue!50!black},
    urlcolor={blue!80!black}
}

\usepackage[noabbrev,capitalize]{cleveref}

\usepackage{palatino}

\allowdisplaybreaks

\include{Macro}

\include{ThmDefs}

\newcounter{counter}

\makeindex

\begin{document}

	\pagenumbering{gobble}

	\input{Titlepage.tex}
	
	\afterpage{
    	\null
    	\thispagestyle{empty}
    	\addtocounter{page}{-1}
    	\newpage}
	
	\pagenumbering{roman}
	\setcounter{page}{1}

	\setcounter{tocdepth}{1}
	\tableofcontents
	
	\cleardoublepage
	
	\newpage
	\pagestyle{empty}
	\setlength\epigraphwidth{0.55\linewidth}
	\renewcommand{\epigraphsize}{\footnotesize}
	\setlength{\beforeepigraphskip}{40ex}
	\leavevmode
	\epigraph{The math is dark and full of functors.}{\em Damien Calaque}
	\clearpage
	
	\cleardoublepage
	
	\pagestyle{headings}
	\pagenumbering{arabic}
	\setcounter{page}{1}
	
	\input{Introduction.tex}
	
	\part{Prerequisites}
	
	\input{AlgebraicOperads.tex}
	
	\input{ModelCategories.tex}
	
	\input{SimplicialHomotopy.tex}
	
	\input{RationalHomotopyTheory.tex}
	
	\input{LieAlgebras.tex}
	
	\input{DeformationTheory.tex}
	
	\part{Results}
	
	\input{RelativeInfinityMorphisms.tex}
	
	\input{NaturalHomotopyAlgebras.tex}
	
	\input{RepresentationDeformationGroupoid.tex}
	
	\input{RationalModels.tex}
	
	\input{ModelStrGMThm.tex}
	
	\cleardoublepage
	\addcontentsline{toc}{part}{Appendices}
	\part*{Appendices}
	
	\appendix
	
	\input{Trees.tex}
	
	\input{Filtered.tex}
	
	\input{FormulaVS.tex}
	
	
	\bibliographystyle{alpha}
	\bibliography{Bibliography}
	
	
	\printindex

\end{document}

%% file: Macro.tex
\renewcommand{\k}{\ensuremath{\mathbb{K}}}
\newcommand{\Z}{\ensuremath{\mathbb{Z}}}

\newcommand{\A}{\ensuremath{\mathcal{A}}}
\newcommand{\antishriek}{\text{\raisebox{\depth}{\textexclamdown}}}
\newcommand{\as}{\ensuremath{\mathrm{As}}}
\newcommand{\ass}{\ensuremath{\mathrm{Ass}}}
\renewcommand{\Bar}{\mathrm{B}}
\newcommand{\Bara}{\mathrm{B}_\alpha}
\newcommand{\Bari}{\mathrm{B}_\iota}
\newcommand{\bmanin}{\,\mbox{\tikz\draw[black,fill=black] (0,0) circle (.7ex);}\,}
\newcommand{\C}{\ensuremath{\mathscr{C}}}
\newcommand{\cCobar}{\widehat{\Omega}}
\newcommand{\Cobar}{\Omega}
\newcommand{\Cobara}{\Omega_\alpha}

\newcommand{\colie}{\ensuremath{\mathrm{coLie}}}
\newcommand{\com}{\ensuremath{\mathrm{Com}}}
\newcommand{\D}{\ensuremath{\mathscr{D}}}
\newcommand{\End}{\ensuremath{\mathrm{End}}}
\newcommand{\Ind}{\ensuremath{\mathrm{Ind}}}
\newcommand{\infcomp}{\ensuremath{\circ_{(1)}}}
\renewcommand{\L}{\ensuremath{\mathcal{L}}}
\newcommand{\lie}{\ensuremath{\mathrm{Lie}}}
\renewcommand{\P}{\ensuremath{\mathscr{P}}}
\newcommand{\Pbar}{\ensuremath{\overline{\mathscr{P}}}}
\newcommand{\Phat}{\ensuremath{\widehat{\mathscr{P}}}}
\newcommand{\Q}{\ensuremath{\mathscr{Q}}}
\newcommand{\R}{\ensuremath{\mathscr{R}}}
\renewcommand{\S}{\ensuremath{\mathbb{S}}}
\newcommand{\Sh}{\ensuremath{\mathrm{Sh}}}
\newcommand{\shriek}{\text{\raisebox{\depth}{!}}}
\newcommand{\SAoo}{\ensuremath{\mathcal{s}\mathscr{A}_\infty}}
\newcommand{\SLoo}{\ensuremath{\mathcal{s}\mathscr{L}_\infty}}
\newcommand{\susp}{\ensuremath{\mathscr{S}}}
\newcommand{\T}{\ensuremath{\mathcal{T}}}
\newcommand{\Tw}{\ensuremath{\mathrm{Tw}}}
\newcommand{\uass}{\ensuremath{\mathrm{uAss}}}
\newcommand{\ucom}{\ensuremath{\mathrm{uCom}}}
\newcommand{\vdl}{\ensuremath{\mathrm{VdL}}}
\newcommand{\wmanin}{\,\mbox{\tikz\draw[black] (0,0) circle (.7ex);}\,}

\newcommand{\augOp}{\ensuremath{\mathsf{aug.\ Op}}}
\newcommand{\cat}{\ensuremath{\mathsf{C}}}
\newcommand{\cCcog}{\ensuremath{\mathsf{conil.}\ \mathscr{C}\text{-}\mathsf{cog}}}
\newcommand{\Ccog}{\ensuremath{\mathscr{C}\text{-}\mathsf{cog}}}
\newcommand{\cchain}{\ensuremath{\mathsf{cCh}}}
\newcommand{\cdgl}{\ensuremath{\mathsf{cdgLie}}}
\newcommand{\cghaus}{\ensuremath{\mathsf{CGHaus}}}
\newcommand{\chain}{\ensuremath{\mathsf{Ch}}}
\newcommand{\cochain}{\ensuremath{\mathsf{coCh}}}
\newcommand{\coaugCoop}{\ensuremath{\mathsf{coaug.\ coOp}}}
\newcommand{\cocom}{\mathsf{coCom}}
\newcommand{\comalg}{\ensuremath{\mathsf{Com}}}
\newcommand{\coop}{\ensuremath{\mathsf{coOp}}}
\newcommand{\cctr}{\ensuremath{\mathsf{cCtr}}}
\newcommand{\dgl}{\ensuremath{\mathsf{dgLie}}}
\newcommand{\fchain}{\ensuremath{\mathsf{fCh}}}
\newcommand{\fdvect}{\ensuremath{\mathsf{fdVect}}}
\newcommand{\fsets}{\ensuremath{\mathsf{fSets}}}
\newcommand{\ooPooAlg}{\ensuremath{\infty\text{-}\mathscr{P}_\infty\text{-}\mathsf{alg}}}

\newcommand{\op}{\ensuremath{\mathsf{Op}}}
\newcommand{\Palg}{\ensuremath{\mathscr{P}\text{-}\mathsf{alg}}}
\newcommand{\Phatalg}{\ensuremath{\widehat{\mathscr{P}}\text{-}\mathsf{alg}}}
\newcommand{\PooAlg}{\ensuremath{\mathscr{P}_\infty\text{-}\mathsf{alg}}}
\newcommand{\Qalg}{\ensuremath{\mathscr{Q}\text{-}\mathsf{alg}}}
\newcommand{\rhcat}{\ensuremath{\mathsf{RH}}}
\newcommand{\sets}{\ensuremath{\mathsf{Sets}}}
\newcommand{\Smod}{{\ensuremath{\mathbb{S}\text{-}\mathsf{mod}}}}
\newcommand{\SLooalg}{\ensuremath{\mathcal{s}\mathscr{L}_\infty\text{-}\mathsf{alg}}}
\newcommand{\ssets}{\ensuremath{\mathsf{sSets}}}
\newcommand{\Top}{\mathsf{Top}}
\newcommand{\vect}{\ensuremath{\mathsf{Vect}}}

\newcommand{\FCA}{\ensuremath{\mathrm{FCA}}}
\newcommand{\rRT}{\ensuremath{\mathrm{rRT}}}
\newcommand{\RT}{\ensuremath{\mathrm{RT}}}
\newcommand{\wPT}{\ensuremath{\mathrm{wPT}}}
\newcommand{\PT}{\ensuremath{\mathrm{PT}}}
\newcommand{\PTt}{\ensuremath{\widetilde{\mathrm{PT}}}}

\DeclareMathOperator{\ad}{ad}
\DeclareMathOperator{\Cyl}{Cyl}
\DeclareMathOperator*{\colim}{colim}
\DeclareMathOperator{\id}{id}
\DeclareMathOperator{\Path}{Path}
\DeclareMathOperator{\pr}{pr}
\DeclareMathOperator{\BCH}{BCH}
\DeclareMathOperator{\Def}{Def}
\DeclareMathOperator{\Del}{Del}
\DeclareMathOperator{\proj}{proj}
\newcommand{\APL}{\ensuremath{\mathrm{A}_\mathrm{PL}}}
\newcommand{\CPL}{\ensuremath{\mathrm{C}_\mathrm{PL}}}
\newcommand{\CE}{\ensuremath{\mathrm{CE}}}
\newcommand{\F}{\ensuremath{\mathcal{F}}}
\newcommand{\g}{\ensuremath{\mathfrak{g}}}
\newcommand{\h}{\ensuremath{\mathfrak{h}}}
\newcommand{\ho}[1]{\ensuremath{\mathrm{Ho}(#1)}}
\newcommand{\Hoch}{\ensuremath{\mathrm{Hoch}}}
\newcommand{\homa}{\ensuremath{\hom^\alpha}}
\newcommand{\homi}{\ensuremath{\hom^\iota}}
\newcommand{\ind}[1]{\ensuremath{\mathsf{ind}(#1)}}
\newcommand{\forget}[1]{\ensuremath{#1^\#}}
\newcommand{\manin}{\rotatebox[origin=c]{180}{$\mathsf{M}$}}
\newcommand{\manintensor}{\ensuremath{\mathsf{M}}}
\newcommand{\map}{\ensuremath{\mathrm{Map}}}
\newcommand{\mc}{\ensuremath{\mathfrak{mc}}}
\newcommand{\MC}{\ensuremath{\mathrm{MC}}}
\newcommand{\MCbar}{\ensuremath{\overline{\mathrm{MC}}}}
\newcommand{\mcoo}{\ensuremath{\mathfrak{mc}^\infty}}
\DeclareFontFamily{U}{shuffle}{} 
\DeclareFontShape{U}{shuffle}{m}{n}{<5-8>shuffle7  <8->shuffle10}{}
\DeclareSymbolFont{Shuffle}{U}{shuffle}{m}{n}
\DeclareMathSymbol\sh{\mathbin}{Shuffle}{"001}
\newcommand{\fqi}{\mathrm{Fqi}}
\newcommand{\gaugeeq}{\ensuremath{\sim_\mathrm{g}}}
\newcommand{\cyc}{\ensuremath{\mathcal{Z}}}
\newcommand{\moore}{\ensuremath{\mathcal{M}}}
\newcommand{\pro}[1]{\ensuremath{\mathsf{pro}(#1)}}
\newcommand{\pt}{\ensuremath{\mathrm{pt}}}

\newcommand{\rect}{\ensuremath{\mathrm{Rect}}}
\newcommand{\smallmanin}{\ensuremath{\mathsf{m}}}

\makeatletter
\newcommand{\adjunction}{\@ifstar\named@adjunction\normal@adjunction}
\newcommand{\normal@adjunction}[4]{%
  #1\colon #2%
  \mathrel{\vcenter{%
    \offinterlineskip\m@th
    \ialign{%
      \hfil$##$\hfil\cr
      \longrightharpoonup\cr
      \noalign{\kern-.3ex}
      \smallbot\cr
      \longleftharpoondown\cr
    }%
  }}%
  #3 \noloc #4%
}
\newcommand{\named@adjunction}[4]{%
  #2%
  \mathrel{\vcenter{%
    \offinterlineskip\m@th
    \ialign{%
      \hfil$##$\hfil\cr
      \scriptstyle#1\cr
      \noalign{\kern.1ex}
      \longrightharpoonup\cr
      \noalign{\kern-.3ex}
      \smallbot\cr
      \longleftharpoondown\cr
      \scriptstyle#4\cr
    }%
  }}%
  #3%
}
\newcommand{\longrightharpoonup}{\relbar\joinrel\rightharpoonup}
\newcommand{\longleftharpoondown}{\leftharpoondown\joinrel\relbar}
\newcommand\noloc{%
  \nobreak
  \mspace{6mu plus 1mu}
  {:}
  \nonscript\mkern-\thinmuskip
  \mathpunct{}
  \mspace{2mu}
}
\newcommand{\smallbot}{%
  \begingroup\setlength\unitlength{.15em}%
  \begin{picture}(1,1)
  \roundcap
  \polyline(0,0)(1,0)
  \polyline(0.5,0)(0.5,1)
  \end{picture}%
  \endgroup
}
\makeatother

\makeatletter
\newcommand{\whiteleaf}{\mathord{\mathpalette\nicoud@YESNO\relax}}
\newcommand{\blackleaf}{\mathord{\mathpalette\nicoud@YESNO{\nicoud@path{\fillpath}}}}
\newcommand{\nicoud@YESNO}[2]{%
	\begingroup
	\settoheight{\unitlength}{$#1X$}%
	\begin{picture}(0.7,1)
	\linethickness{\variable@rule{#1}}%
	\roundcap\roundjoin
	\nicoud@path{\strokepath}
	#2
	\Line(0.35,0)(0.35,0.5)
	\end{picture}%
	\endgroup
}
\newcommand{\nicoud@path}[1]{%
	\moveto(0.1,0.5)
	\lineto(0.1,1)\lineto(0.6,1)\lineto(0.6,0.5)
	\closepath
	#1
}
\newcommand{\variable@rule}[1]{%
	\fontdimen8  
	\ifx#1\displaystyle\textfont3\else
	\ifx#1\textstyle\textfont3\else
	\ifx#1\scriptstyle\scriptfont3\else
	\scriptscriptfont3\relax
	\fi\fi\fi
}
\makeatletter

%% file: ThmDefs.tex
\newtheorem{lemma}{Lemma}[section]
\newtheorem{remark}[lemma]{Remark}
\newtheorem{example}[lemma]{Example}
\newtheorem{theorem}[lemma]{Theorem}
\newtheorem{corollary}[lemma]{Corollary}
\newtheorem{definition}[lemma]{Definition}
\newtheorem{proposition}[lemma]{Proposition}
\newtheorem{warning}[lemma]{Warning}
\newtheorem{notation}[lemma]{Notation}

%% file: Titlepage.tex
\newgeometry{margin=1cm}
\begin{titlepage}
	\thispagestyle{empty}
	\begin{center}
		\includegraphics[width=4.5cm]{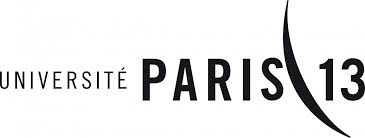} \\
		\vspace{0.5cm}
		\textsc{\large École doctorale Galilée}\\ 
		\vspace{0.5cm}
		\textsc{\large Laboratoire Analyse, Géométrie et Applications, UMR 7539}\\

		\vspace{2cm}

		\textsc{\LARGE Thèse}

		\vspace{0.5 cm}

		{\large presentée et soutenue publiquement par}

		\vspace{0.5 cm}

		{\Large\textbf{Daniel ROBERT-NICOUD}}

		\vspace{0.5cm}

		{le 22 juin 2018}

		\vspace{0.5cm}

		{pour obtenir le grade de}

		\vspace{0.5 cm}

		\textbf{Docteur de l'Université Paris 13}

		\vspace{0.5cm}

		{Discipline : Mathématiques}

		\vspace{1.5cm}

		{\LARGE\textbf{Opérades et espaces de Maurer--Cartan}\par\medskip\LARGE\textbf{Operads and Maurer--Cartan spaces}}\\[0.3cm]

		\vspace{1.5cm}

		{\large Directeur de thèse : \textbf{Bruno VALLETTE}}

		\vspace{1.4cm}

		{\large Membres du Jury}

		\vspace{0.8cm}
		
		\begin{tabular}[t]{@{}l@{\hspace{40pt}}p{.32\textwidth}@{}}
			Christian AUSONI \\
			Alexander BERGLUND \\
			Joana CIRICI \\
			Eric HOFFBECK \\
			Giovanni FELDER \\
			Bruno VALLETTE \\
			Thomas WILLWACHER
		\end{tabular}%
		\begin{tabular}[t]{@{}l@{\hspace{40pt}}p{.32\textwidth}@{}}
			PR, Université Paris 13 \\
			PR, Stockholm University \\
			PR, Universitat de Barcelona \\
			MCF, Université Paris 13 \\
			PR, ETH Zürich \\
			PR, Université Paris 13 \\
			PR, ETH Zürich
		\end{tabular}%
		\begin{tabular}[t]{@{}l@{\hspace{40pt}}p{.32\textwidth}@{}}
			Examinateur \\
			Rapporteur \\
			Examinateur \\
			Examinateur \\
			Examinateur \\
			Directeur de thèse \\
			Rapporteur
		\end{tabular}

	\end{center}
\end{titlepage}
\restoregeometry

%% file: Introduction.tex
\chapter{Introduction}

\section{Préambule (version française)}

Pour paraphraser Adrien Douady \cite{Douady}, le but de ce travail est de munir son auteur du titre de docteur en mathématiques et plusieurs différents espaces de structures d'algèbres de Lie à homotopie près. Nous en étudierons les propriétés et nous les appliquerons à la théorie de la déformation et à l'homotopie rationnelle.

\subsection{Opérades}

Notre outil de travail principal sont les opérades algébriques. Une opérade est une structure algébrique qui code un ``type d'algèbres''. Par exemple, il y a une opérade qui code les algèbres associatives, une autre les algèbres commutatives, une troisième les algèbres de Lie, etc. Un point de vue intuitif sur ces objets est le suivant. La plupart des types d'algèbres qu'on rencontre sont définis comme un espace vectoriel muni d'opérations avec un certain nombre d'entrées --- leur \emph{arité} --- et une sortie, qui satisfont certaines relations entre elles. Par exemple, une algèbre associative est un espace vectoriel $A$ avec une opération binaire
\[
m:A\otimes A\longrightarrow A
\]
qui est associative :
\[
m(m(a,b),c) = m(a,m(b,c))
\]
pour tout $a,b,c\in A$. La relation d'associtivité peut être exprimée sans faire référence aux éléments de l'espace vectoriel sous-jacent comme
\[
m(m\otimes1_A) = m(1_A\otimes m)\ .
\]
L'idée des opérades est d'oublier l'espace vectoriel sous-jacent et de se concentrer sur les op\-é\-rations. Une opérade $\P$ est une collection d'espaces vectoriels $\P(n)$, pour $n\ge0$, qui codent les opérations d'arité $n$, munie d'une règle pour composer de telles opérations. Dans l'exemple précédent on a $m\in\P(2)$, et une règle de composition qui nous dit comment on peut composer deux copies de $m$ de deux façons différentes pour obtenir au final la même opération dans $\P(3)$. Les opérades nous permettent de donner des énoncés généraux sur les types possibles d'algèbres ainsi que sur les relations entre algèbres de types différents de façon propre et catégorique. Dualement, il y a une notion de coopérade, qui code les cogèbres.

\subsection{Algèbre homotopique}

La théorie des opérades est profondément liée à l'algèbre homotopique, c'est-à-dire l'étude des structures algébriques où l'on n'identifie pas deux algèbres seulement si elles sont isomorphes, mais aussi quand elles sont reliées que par un quasi-isomorphisme, c'est-à-dire un morphisme qui induit un isomorphisme en homologie. Cette construction peut être formalisée par une structure de modèles sur la catégorie des algèbres. Faire de l'algèbre homotopique nous permet d'utiliser certains outils puissants. Ceux que l'on va manipuler le plus souvent sont les $\infty$-morphismes d'algèbres et le théorème de transfert homotopique. Les premiers sont une version relachée de la notion de morphisme d'algèbres qui peut être interprétée comme des applications qui sont des morphismes d'algèbres ``à homotopie près''. Le deuxième est un théorème qui dit que si l'on a deux complexes de chaînes homotopiquement équivalents et une structure algébrique sur un des deux, alors on peut mettre une structure algébrique ``à homotopie près'' sur le second complexe de chaînes de façon telle que les deux algèbres soient homotopiquement équivalentes.

\subsection{Algèbres de convolution}

Donnons-nous un certain type d'algèbres, codé par une opérade $\P$, un type de cogèbres, codé par une coopérade $\C$ et une relation entre $\C$ et $\P$, donnée par ce qui est appelé un morphisme tordant $\alpha:\C\to\P$. Dans cette situation, on peut mettre une structure naturelle d'algèbre de Lie à homotopie près sur l'espace des applications linéaires entre $\C$-cogèbres et $\P$-algèbres. Plus précisement, si $C$ est une $\C$-cogèbre et $A$ est une $\P$-algèbre alors on a un complexe de chaînes d'applications linéaires $\hom(C,A)$. En ayant fixé un morphisme tordant $\alpha$, on a une structure naturelle d'algèbre de Lie à homotopie près sur $\hom(C,A)$. On note  $\homa(C,A)$ l'algèbre que l'on obtient de cette façon  et on l'appelle l'algèbre de convolution de $C$ et $A$ (relative à $\alpha$). Le premier objectif de cette thèse est de développer la théorie des algèbres de convolution et de comprendre comment elles se comportent par rapport aux outils de l'algèbre homotopique. Au cours de cette thèse, nous allons montrer les résultats suivants.
\begin{enumerate}
	\item Les algèbres de convolution sont partiellement compatibles avec les $\infty$-morphismes. Plus précisement, la notion d'algèbre de convolution définit un bifoncteur des $\C$-cogèbres et des $\P$-algèbres vers les algèbres de Lie à homotopie près. Il est possible d'étendre ce bifoncteur soit en un bifoncteur qui accepte les $\infty$-morphismes de $\C$-cogèbres dans sa première entrée ou en un bifoncteur qui accepte les $\infty$-morphismes de $\P$-algèbres dans sa seconde entrée. Par contre il n'est pas possible de l'étendre ultérieurement en un bifoncteur qui accepte des $\infty$-morphismes dans ses deux entrées en même temps.
	\item Si l'on fait des suppositions ultérieures sur le morphisme tordant $\alpha$ --- notamment, en demandant que $\alpha$ soit de Koszul --- alors on peut étendre le bifoncteur des algèbres de convolution de façon à ce qu'il accepte des $\infty$-morphismes dans ses deux entrées en même temps, mais le bifoncteur qu'on obtient de cette façon n'est défini qu'à homotopie près.
	\item Les algèbres de convolution sont complètement compatibles avec le théorème de transfert homotopique.
\end{enumerate}
Par la suite, on va appliquer la théorie des algèbres de convolution à la théorie de la déformation dérivée et à la théorie de l'homotopie rationnelle.

\subsection{Théorie de la déformation}

Un problème de déformation est donné en fixant un objet sous-jacent, comme un espace vectoriel, et un type de structure qu'on aimerait mieux comprendre sur cet objet, par exemple les structures d'algèbre associative sur l'espace vectoriel. Une fois donnée une telle structure, on veut la perturber de façon à obtenir une autre structure du même type. Ce processus peut nous donner des informations utiles sur la structure que l'on était en train d'étudier. Il y a une myriade d'exemples de problèmes de déformation : on peut considérer les déformations des structures d'algèbre d'un certain type sur un espace vectoriel, mais aussi les déformations de structures complexes analytiques sur une variété fixée, les déformations de points ou de sous-schémas fermés dans un schéma sous-jacent, les déformations de connexions plates sur un fibré vectoriel, et beaucoup d'autres encore. L'exemple le plus célèbre d'un résultat provenant de la théorie de la déformation est probablement la quantification par déformation des variétés de Poisson, due à Konstevich.

\medskip

Un principe heuristique énoncé par Deligne affirme que chaque problème de déformation (en caractéristique zéro) est équivalent à l'étude de l'espace des éléments de Maurer--Cartan d'une algèbre de Lie différentielle graduée ou plus généralement d'une algèbre de Lie à homotopie près. Cet espace d'éléments de Maurer--Cartan a été décrit par Hinich comme un ensemble simplicial $\MC_\bullet(\g)$ qui peut être fonctoriellement associé à chaque algèbre de Lie à homotopie près $\g$. R\'ecemment ce dernier a été étudié en détail par Getzler. La première application qu'on donnera des algèbres de convolution est la construction explicite d'une algèbre de Lie à homotopie près cosimpliciale $\mcoo_\bullet$ qui représente l'espace de Maurer--Cartan à homotopie près. Plus précisément, pour chaque algèbre de Lie à homotopie près $\g$ qui satisfait certaines conditions de complétude, on a une équivalence homotopique naturelle d'ensembles simpliciaux
\[
\MC_\bullet(\g)\simeq\hom_{\SLoo\text{-}\mathsf{alg}}(\mcoo_\bullet,\g)\ .
\]
On étudiera cette algèbre  $\mcoo_\bullet$ et on l'utilisera pour comprendre certaines propriétés de $\MC_\bullet(\g)$ et d'autres ensembles simpliciaux associés.

\subsection{Homotopie rationnelle}

En homotopie rationnelle on s'intéresse à un invariant des espaces simplement connexes qui est plus faible que leur type d'homotopie mais qui a l'avantage d'être calculable dans beaucoup de cas intéressants. Soit $X$ un espace simplement connexe. Ses groupes d'homotopie sont tous abéliens et donc on peut définir ses groupes d'homotopie rationnelle par
\[
\pi_\bullet(X,x)\otimes_\mathbb{Z}\mathbb{Q}\ ,
\]
où $x\in X$ est un point base. On dit qu'un morphisme est une équivalence rationnelle s'il induit un isomorphisme entre les groupes d'homotopie rationnelle. Par un résultat du travail fondateur de Quillen, l'étude de l'homotopie rationnelle, c'est-à-dire l'étude de la catégorie des espaces simplement connexes modulo les équivalences rationnelles, est la même chose que l'étude de la catégorie des algèbres de Lie différentielles graduées concentrées en degré supérieur ou égal à $1$ modulo les quasi-isomorphismes, qui est elle-même équivalente à l'étude de la catégorie des cogèbres cocommutatives concentrées en degré supérieur ou égal à $2$ modulo les quasi-isomorphismes. Ceci nous dit qu'on peut modéliser les espaces par des algèbres de Lie ou des cogèbres cocommutatives pour avoir des outils calculatoires explicites qui nous permettent de calculer le type d'homotopie rationnelle d'un espace. Les travaux de Sullivan nous donnent une autre approche encore, en modélisant les espaces par des algèbres commutatives différentielles graduées.

\medskip

On peut aussi considérer des types de modèles plus généraux, comme des modèles en algèbres de Lie à homotopie près et des modèles en cogèbres cocommutatives à homotopie près. Un théorème de Berglund nous dit comment construire un modèle en algèbres de Lie à homotopie près pour l'espace des morphismes entre deux espaces en partant d'un modèle commutatif de l'espace de départ et d'un modèle Lie à homotopie près de l'espace d'arrivée. On généralise légèrement ce résultat en démontrant qu'on peut prendre des modèles cocommutatifs à homotopie près pour l'espace de départ et on exprime le modèle résultant pour l'espace d'applications entre les deux espaces comme une algèbre de convolution.

\subsection{Invariance homotopique des espaces de Maurer--Cartan}

Le dernier chapitre de cette thèse est dédié à des résultats seulement partiellement reliés aux algèbres de convolution. Deux résultats importants sur les espaces de Maurer--Cartan sont le théorème de Goldman--Millson et sa généralisation plus récente, le théorème de Dol\-gu\-shev--Rogers. Ces résultats nous disent que certains morphismes entre algèbres de Lie à homotopie près --- les $\infty$-quasi-isomorphismes filtrés --- induisent des équivalences en homotopie entre les espaces de Maurer--Cartan respectifs. Ce résultat a un parfum homotopique : il nous dit qu'une certaine classe de morphismes d'algèbres de Lie à homotopie près est envoyée sur les équivalences faibles d'ensembles simpliciaux par un certain foncteur. Par contre sa preuve n'est pas faite en n'utilisant que de la théorie de l'homotopie et se base sur une démonstration par récurrence sur certaines filtrations, en travaillant directement sur les algèbres de Lie à homotopie près en jeu. On donne une nouvelle approche à la preuve de ce théorème qui n'utilise que de la théorie de l'homotopie. Même si l'on ne récupère pas complètement le théorème de Dolgushev--Rogers, on réussit tout de même à en obtenir une version plus faible et l'on croise des constructions intéressantes le long du chemin. Par exemple cette approche nous fait redécouvrir l'algèbre de Lie cosimpliciale qui représente les espaces de Maurer--Cartan.

\subsection{Résultats ultérieurs}

Cachés dans les rappels, il y a deux résultats originaux supplémentaires dans cette thèse. Le premier est une caractérisation complète des équivalences faibles dans la structure de modèles construite sur la catégorie des cogèbres sur une coopérade par Vallette. Le degré d'originalité de ce résultat n'est pas énorme, car tous les ingrédients dont on a besoin pour sa démonstration étaient déja présents dans les travaux de Vallette. Le second est une démonstration du fait que le dual d'un modèle cocommutatif pour un espace simplement connexe est un modèle commutatif pour le même espaces et \emph{vice versa}, sous certaines hypothèses peu contraignantes de finitude. Ce dernier résultat était certainement bien connu par les experts, mais nous n'avons pas su en trouver une démonstration dans la litterature.

\subsection{Remarques d'ordre général}

Cette thèse veut être le plus autonome possible. Pour cette raison, on consacre une importante première partie  à l'exposition de plusieurs notions préliminaires dont on a besoin dans le reste du texte : la théorie de opérades, les catégories de modèles, l'homotopie simpliciale, l'homotopie rationnelle, les espaces de Maurer--Cartan et la théorie de la déformation. On suppose que le lecteur possède des notions de base d'algèbre homologique, de théorie des catégories et de topologie algébrique.

\medskip

Les r\'esultats nouveaux contenus dans cette thèse sont extraits des articles \cite{rn17tensor}, \cite{rn17cosimplicial}, \cite{rnw17}, \cite{rnw18}, \cite{rn18} et de l'esquisse \cite{rnv18}.

\section{Preamble}

To paraphrase Adrien Douady \cite{Douady}, the goal of this work is to endow the author with the degree of Doctor in Mathematics, and many different spaces with structures of homotopy Lie algebras, as well as studying their properties and apply the developed theory to various fields, such as deformation theory and rational homotopy theory.

\subsection{Operads}

Our main working tool is algebraic operads. An operad is an algebraic structure that codes a ``type of algebra''. For example, there is an operad coding associative algebras, one coding commutative algebras, one for Lie algebras, and so on. One possible intuitive point of view is the following one. Most types of algebras are defined as a vector space endowed with certain operations with a certain number of entries --- their \emph{arity} --- and one exit, which satisfy certain relations between them. For example, an associative algebra is a vector space $A$ together with a binary multiplication
\[
m:A\otimes A\longrightarrow A
\]
which is associative:
\[
m(m(a,b),c) = m(a,m(b,c))
\]
for all $a,b,c\in A$. The associativity relation can be expressed without referring to elements of the underlying vector space as
\[
m(m\otimes1_A) = m(1_A\otimes m)\ .
\]
The idea of operads is to forget the vector spaces and concentrate on the operations. An operad $\P$ is a sequence of vector spaces $\P(n)$, for $n\ge0$, encoding the operations of arity $n$, together with a rule for composing such operations. In the example above, we have $m\in\P(2)$, and a composition rule telling us how composing two copies of $m$ in two different ways gives us the same operation in $\P(3)$. Operads allow us to give general statements about types of algebras and relations between algebras of different types in a clean and elegant way. Dually, one also has cooperads, which are a similar object encoding coalgebras.

\subsection{Homotopical algebra}

Operad theory is closely linked to homotopical algebra, i.e. the study of algebraic structures where one does not identify two algebras only if they are isomorphic, but also when they are related in a more flexible way. Namely, in homotopical algebra one studies algebras up to quasi-isomorphisms, that is up to morphisms inducing isomorphisms in homology. This construction can be formalized as a model category of algebras. Doing homotopical algebra allows us to use some powerful instruments. The ones we will manipulate the most are $\infty$-morphisms of algebras and the homotopy transfer theorem. The former are a ``relaxed'' version of morphisms of algebras, which can be interpreted as maps that are morphisms of algebras, but only ``up to homotopy''. The latter is a theorem that essentially says that if we have two chain complexes that are homotopically the same, and some algebraic structure on one of the two, then we can put an algebraic structure on the other chain complex so that the resulting \emph{algebras} are homotopically the same.

\subsection{Convolution algebras}

Given a certain type of algebra, encoded by an operad $\P$, a certain type of coalgebras, encoded by a cooperad $\C$, and a relation between $\C$ and $\P$ given by what is called a twisting morphisms $\alpha:\C\to\P$, then one can give natural homotopy Lie algebra structures on the spaces of linear maps between $\C$-coalgebras and $\P$-algebras. Namely, if $C$ is a $\C$-coalgebra, and $A$ is a $\P$-algebra, then we have the chain complex of linear maps $\hom(C,A)$. Given the twisting morphism $\alpha$, there is a natural homotopy Lie algebra on $\hom(C,A)$. We denote the resulting algebra by $\homa(C,A)$, and call it the convolution algebra of $C$ and $A$ (relative to $\alpha$). The first goal of the present work is to develop the theory of convolution algebras, and to understand how they behave with respect to the tools of homotopical algebra. We will show what follows.
\begin{enumerate}
	\item Convolution algebras are partially compatible with $\infty$-morphisms. Namely, taking convolution algebras defines a bifunctor from $\C$-coalgebras and $\P$-algebras to homotopy Lie algebras. It is possible to extend this bifunctor either to a bifunctor also accepting $\infty$-morphisms of $\C$-coalgebras in the first slot, or accepting $\infty$-morphisms of $\P$-algebras in the second slot. However, it is not possible to further extend it into a bifunctor accepting $\infty$-morphisms in both slots simultaneously.
	\item If one makes further assumptions about the twisting morphism $\alpha$ --- namely if one asks that $\alpha$ is Koszul --- then one can extend the convolution algebra bifunctor to a bifunctor accepting $\infty$-morphisms in both slots at the same time, but the resulting bifunctor is only defined up to homotopy.
	\item Convolution algebras are completely compatible with the homotopy transfer theorem.
\end{enumerate}
We will then apply the theory of convolution algebras to derived deformation theory and rational homotopy theory.

\subsection{Deformation theory}

A deformation problem is given by a fixed underlying object, such as a vector space, and a type of structure on that object one wants to understand, for example the possible associative algebra structures on the vector space. Given one structure, one wants to perturb it in such a way as to obtain another structure of the same type. This process can yield useful information on the structure one is studying. Examples of deformation problems are legion: one can consider deformations of algebraic structures of a certain type on a vector space, but also deformations of complex analytic structures on a fixed underlying manifold, deformations of points or closed sub-schemes in an underlying scheme, deformations of flat connections on a vector bundle, and many others. Perhaps one of the most celebrated results of deformation theoretical nature is Kontsevich's deformation quantization of Poisson manifolds.

\medskip

It is a heuristic principle due to Deligne that any deformation problem (in characteristic zero) is equivalent to the study of the space of Maurer--Cartan elements of a differential graded Lie algebra, or more generally of a homotopy Lie algebra. This space of Maurer--Cartan elements was expressed by Hinich as a simplicial set $\MC_\bullet(\g)$ functorially associated to any homotopy Lie algebra $\g$, and then studied in depth by Getzler in more recent years. The first application of convolution algebras we will give is the explicit construction of a universal cosimplicial homotopy Lie algebra $\mcoo_\bullet$ representing the Maurer--Cartan space up to homotopy. To be more precise, for any homotopy Lie algebra $\g$ satisfying some completeness condition, we have a natural homotopy equivalence of simplicial sets
\[
\MC_\bullet(\g)\simeq\hom_{\SLoo\text{-}\mathsf{alg}}(\mcoo_\bullet,\g)\ .
\]
We will study the cosimplicial homotopy Lie algebra $\mcoo_\bullet$, and use it to derive some properties of $\MC_\bullet(\g)$ and other related simplicial sets.

\subsection{Rational homotopy theory}

In rational homotopy theory, one is interested in an invariant of simply connected spaces that is weaker than their homotopy type, but which has the advantage of being computable in many interesting cases: its rational homotopy type. Let $X$ be a simply connected space. Then its homotopy groups are all abelian, so that one defines its rational homotopy groups as
\[
\pi_\bullet(X,x)\otimes_\mathbb{Z}\mathbb{Q}\ ,
\]
where $x\in X$ is a basepoint. One says that a morphism of spaces is a rational homotopy equivalence if it induces isomorphisms on the rational homotopy groups. It is a result of the seminal work of Quillen that rational homotopy theory, i.e. the study of the category of simply connected spaces modulo rational equivalences, is the same as the study of the category of differential graded Lie algebra concentrated in degree greater or equal than $1$ modulo quasi-isomorphisms, which is itself equivalent to the study of the category of differential graded cocommutative coalgebras concentrated in degree greater or equal than $2$ modulo quasi-isomorphisms. This tells us that we can model spaces by Lie algebras or cocommutative coalgebras in order to have explicit computational tools to find the rational homotopy type of a space. Sullivan took another approach, modeling spaces with differential graded commutative algebras.

\medskip

One can also consider more general models, such as homotopy Lie algebra models, and homotopy cocommutative coalgebra models. A theorem of Berglund tells us how one can construct a homotopy Lie algebra rational model for the mapping space of two spaces starting from a commutative model of the source space and a homotopy Lie model of the target space. We slightly generalize this result by proving that one can in fact allow homotopy cocommutative models for the source space, and express the resulting model for the mapping space as a convolution algebra.

\subsection{Homotopy invariance of Maurer--Cartan spaces}

The last chapter of this thesis is dedicated to some material only tangentially related to convolution algebras. Two important result on Maurer--Cartan spaces are the Goldman--Millson theorem, and its modern generalization, the Dolgushev--Rogers theorem. They tell us that certain maps between homotopy Lie algebras, namely filtered $\infty$-quasi-isomorphisms, induce homotopy equivalences between the respective Maurer--Cartan spaces. This result has a strong homotopy theoretical flavor: it tells us that a certain class of maps of homotopy Lie algebras is sent to weak equivalences of simplicial sets under a certain functor. However, its proof is not completely homotopy theoretical, and relies on a proof by induction on certain filtrations, working directly with the homotopy Lie algebras in play. We give a new approach to the proof of this theorem, relying purely on homotopy theory. Although we do not recover the full strength of the Dolgushev--Rogers theorem, we are still able to obtain a weaker version, and incur in interesting constructions along the way. For example, this approach allows us to recover the cosimplicial Lie algebra representing Maurer--Cartan spaces.

\subsection{Other results}

There are two further original results in the present work, hiding in the recollections. The first result is a complete characterization of the weak equivalences in the model structure put on coalgebras over a cooperad by Vallette. The degree of originality of this result is not enormous, as all the ingredients needed in the proof were already present in Vallette's work. The second one is a proof of the fact that the dual of a cocommutative rational model for a simply connected space is a commutative model for the same space, and \emph{vice versa}, under some slight finiteness assumptions the dual of a commutative rational model is a cocommutative rational model. This last result was folklore and certainly well-known to experts. However, we were not able to find a proof in the existing literature.

\subsection{General remarks}

Our goal is to be as self-contained as possible. For this reason, we indulge in a lengthy exposition of preliminary notions in the first part of the thesis, explaining the necessary notions of operad theory, model categories, simplicial homotopy theory, rational homotopy theory, Maurer--Cartan spaces, and deformation theory. Still, we assume that the reader is acquainted with the basic notions of homological algebra, category theory, and algebraic topology.

\medskip

All the original material contained in this thesis is extracted from the articles \cite{rn17tensor}, \cite{rn17cosimplicial}, \cite{rnw17}, \cite{rnw18}, \cite{rn18}, and from the draft \cite{rnv18}.

\section{Acknowledgements}

First and foremost, I am immensely grateful to my family for their constant support. Whatever I chose to do, you were always there.

\medskip

A close second is my advisor, Bruno Vallette. This thesis would not have been possible without you, and I want to thank you for all the time you spent explaining things to me, discussing ideas, proofreading my work, and offering comments, critiques and support.

\medskip

Another person who is really important to me mathematically speaking, and as a friend as well, is my collaborator Felix Wierstra. A big thanks goes also to his advisor Alexander Berglund.

\medskip

From my life before Paris, I am grateful to all the teachers, mentors, and professors who pushed me to give my best in what I do and helped me along the way. Starting from Raffaella and Gianluigi in primary school, to Alvaro Zorzi and Gianmarco Zenoni in secondary school, to Fabio Lucchinetti, Fausta Leonardi and Matthias Venzi in high school. Franchino Sonzogni, who teached me to play checks. Christian Pezzatti and Loris Galbusera, who were always an example of dedication and motivation. And finally, Damien Calaque, Mathieu Anel, Giovanni Felder, and Dietmar Salamon, who taught me to do maths.

\medskip

From my time in Paris, the people I would like to thank are legion. I will list them in no particular order, and I am sorry if I forgot anyone. A lot of friends made my period in Paris, and these three years spent doing my PhD in general, a nice experience. The colleagues in the algebraic topology group in Paris 13, and the many other PhD students of the university provided a nice atmosphere in which to work --- despite the \emph{décor} of Paris 13 itself. The secretaries, Isabelle and Yolande, were always a helpful, smiling presence in the department. I had a lot of fun with other PhD students in various conferences around Europe, in particular with Elise, Jérémy, Adélie, Sylvain, Lyne, Brice, Damien, and others still. I also had the occasion of meeting many people outside of the world of mathematics. I spent a lot of time in my first year exploring the city with Angelina, and I hold dear memories of that period of time. I had countless enjoyable moments with the tango dancers: Amrei, Enrique, Auréa, Dominique, Florence, Susan, Viviane, Miguel, Adelaida, and many others. Another group of people I loved to spend time with are the climbers: Vincent, Lou, Bilal, Chakib, Zak, Hugo, Guillaume, Maeva, Laurène, Juju, Elise, and so on. There are also my judo and jujitsu club --- Jannick, Marcel, Steve, Félix, Tom, and so on --- and the salsa dancers at Paris 13.

\medskip

Finally, some miscellaneous people I absolutely don't want to leave out are Matteo, Devis, Simone, Andrea, Davide, Daniele, Federico, Marco, Kai, Aude, Simona, Valerio, Sabrina, Solvène, Charlotte, Karine, and Elise. All of you played an important role in this journey, and I am happy and grateful of having met you.

\medskip

Last, but certainly not least, I would like to thank the Region Île de France and the ANR-14-CE25-0008-01 project SAT for the financial support.

\section{Structure of this work}

This thesis is divided in two parts. The first part is an overview of prerequisites for the second part, which is composed of the original results of this work.

\medskip

The recollections begin in \cref{ch:operads} with an overview of operad theory. We recall all of the necessary notions of operads, cooperads, algebras and coalgebras over them, Koszul duality, twisting morphisms, bar and cobar constructions for (co)operads and (co)algebras, and homotopy theory of algebras up to homotopy. In \cref{ch:model categories}, we recall the basic notions of model categories, which is the framework we will use to do homotopy theory. We give various examples of model categories in \cref{sect:examples model categories}. This chapter contains an original result, \cref{thm:characterization of we in the Vallette model structure}, which gives a complete characterization of the weak equivalences in the Vallette model structure on conilpotent coalgebras over a cooperad. \cref{ch:simplicial homotopy theory} is a review of simplicial homotopy theory. \cref{chapter:RHT} is a recollection on rational homotopy theory. It contains an original result, \cref{thm:dualization of (co)commutative models}, stating that one can dualize cocommutative rational models to obtain commutative rational models, and \emph{vice versa}. \cref{ch:Lie algebras} covers in detail the notions of differential graded Lie algebras, their Maurer--Cartan elements and the equivalence relations between them, the analogous notions for homotopy Lie algebras, and some modern results: the Dolgushev--Rogers theorem and the formal Kuranishi theorem. We conclude the first part of the thesis with a review of deformation theory in \cref{chapter:deformation theory}. This chapter is mostly a motivation of why one is interested in studying the space of Maurer--Cartan elements on homotopy Lie algebras.

\medskip

The second part starts with the introduction and the study of basic properties of $\infty$-morphisms of algebras and coalgebras relative to a twisting morphism, in \cref{chapter:oo-morphisms relative to a twisting morphism}. We continue with \cref{chapter:convolution homotopy algebras}, where we introduce convolution homotopy algebras and study them in depth. They are a central object in this thesis, and we give applications of their theory in the subsequent two chapters. In \cref{ch:representing the deformation oo-groupoid}, we apply it to construct a universal cosimplicial (homotopy) Lie algebra representing the Maurer--Cartan space functor. In \cref{ch:rational models mapping spaces}, we apply it to rational homotopy theory and use it to generalize a theorem due to Berglund. The last chapter of this part is \cref{ch:model structure for GM}. There, we give a completely homotopical approach to the proof of the Goldman--Millson theorem and of the Dolgushev--Rogers theorem.

\medskip

Finally, there are three appendices containing auxiliary notions we need here and there. In \cref{chapter:appendix on trees}, we give the definitions and notations we use for trees (rooted and planar). In \cref{appendix:filtered stuff}, we explain the notions of filtered chain complexes and filtered algebras over operads. In \cref{appendix:formal fixed-pt eq and diff eq}, we define formal fixed-point equations and formal differential equations, and prove some existence and uniqueness results for their solutions.

\section{Notations and conventions}

Before really starting with the main body of this work, we fix some notations and conventions we will use throughout the text.

\subsection{General conventions}

We work over a field of characteristic $0$ unless stated otherwise. The symbol $\S_n$ is reserved for the symmetric group on $n$ elements.

\medskip

All operads, cooperads, algebras and coalgebras are always differential graded, unless explicitly stated otherwise. We usually work over chain complexes, with some exceptions where it is more natural to use cochain complexes, for example in \cref{chapter:deformation theory}.

\medskip

We reserve the symbol $\circ$ for the composite product of $\S$-modules and related concepts, see \cref{def:composite product}, and omit compositions whenever talking about functions, writing $fg$ for the composition of two functions $f$ and $g$.

\subsection{Categories into play}

Here are the notations we will use for some of the most frequently appearing categories.

\begin{tabular}{ll}
	$\Delta$ & The ordinal number category.\\
	$\chain$ & The category of chain complexes and chain maps.\\
	$\Ccog$ & The category of conilpotent coalgebras over a cooperad $\C$.\\
	$\coop$ & The category of (conilpotent) cooperads.\\
	$\SLooalg$ & The category of shifted homotopy Lie algebras.\\
	$\Smod$ & The category of $\S$-modules.\\
	$\op$ & The category of differential graded operads.\\
	$\Palg$ & The category of algebras over an operad $\P$.\\
	$\Phatalg$ & The category of proper complete algebras over an operad $\P$.\\
	$\ssets$ & The category of simplicial sets.\\
	$\Top$ & The category of topological spaces.
\end{tabular}

Moreover, we will denote by either $\cat(x,y)$ or by $\hom_\cat(x,y)$ the set of morphisms in the category $\cat$ with source $x\in\cat$ and target $y\in\cat$.

\subsection{Chain complexes}\label{subsection:intro chain complexes}

We will mainly work over the category $\chain$ of chain complexes over some field $\k$. We fix from the start the notations and conventions we will use in what follows. We assume that the reader is already familiar with the notion of chain complex.

\medskip

We reserve the letter $s$ to denote a formal element of degree $1$. If $V$ is a chain complex, the chain complex $sV$ is the suspension of $V$, that is the chain complex with $(sV)_n \cong V_{n-1}$ and differential
\[
d_{sV}(sv)\coloneqq -sd_Vv
\]
for $v\in V$, coherently with the Koszul sign rule of Subsection \ref{subsection:intro Koszul sign rule}. We denote by $s^{-1}$ the dual of $s$.

\medskip

Let $V$ and $W$ be two chain complexes. We denote by $V\otimes W$ the tensor product\index{Tensor product!of chain complexes} of $V$ and $W$. Its element of degree $n$ are those in
\[
(V\otimes W)_n = \bigoplus_{p+q=n}V_p\otimes W_q\ ,
\]
and its differential is given by
\[
d_{V\otimes W}(v\otimes w)\coloneqq d_Vv\otimes w + (-1)^vv\otimes d_Ww\ ,
\]
which, once again, is coherent with the Koszul sign rule.

\medskip

Let $V$ and $W$ be two chain complexes. By $\hom(V,W)$ we will always mean the inner hom\index{Inner!hom!of chain complexes} of $V$ and $W$, i.e. the chain complex having in degree $n$ the linear maps $V\to W$ of degree $n$, or equivalently the degree zero linear maps $s^nV\to W$. Its differential is given by
\[
\partial(f)\coloneqq d_Wf - (-1)^ffd_V
\]
on $f\in\hom(V,W)$. Notice that the closed elements of degree $0$ are exactly the chain maps, while the exact elements of degree $0$ are chain maps that are homotopic to the zero map.

\medskip

Let $V$ be a chain complex. We denote by $V^\vee\coloneqq\hom(V,\k)$ its dual chain complex, where $\k$ is seen as a chain complex concentrated in degree $0$.

\medskip

A chain complex $V$ is said to be \emph{of finite type} if $V_n$ is finite dimensional for all $n\in\Z$.

\subsection{The Koszul sign rule}\label{subsection:intro Koszul sign rule}\index{Koszul!sign rule}

The Koszul sign rule is a sign convention that is put on the switching maps in the (symmetric monoidal) category of graded vector spaces. Namely, if $V,W$ are two graded vector spaces, then the isomorphism
\[
V\otimes W\longrightarrow W\otimes V
\]
is given by
\[
v\otimes w\longmapsto(-1)^{|v||w|}w\otimes v
\]
on homogeneous elements. This gives an automatic way of obtaining the correct signs in computations. An example of application of the Koszul sign rule is the following. Let $V_1,V_2,W_1,W_2$ be graded vector spaces, and let $f_i:V_i\to W_i$ be linear maps of homogeneous degree. Then the map $f_1\otimes f_2$ is given by
\[
(f_1\otimes f_2)(v_1\otimes v_2) = (-1)^{|f_2||v_1|}f_1(v_1)\otimes f_2(v_2)\ .
\]
In particular, notice that the dual to $s^n\coloneqq s\otimes\cdots\otimes s$ is
\[
(-1)^{\frac{n(n-1)}{2}}s^{-n}\coloneqq(-1)^{\frac{n(n-1)}{2}}s^{-1}\otimes\cdots\otimes s^{-1},
\]
and not simply $s^{-n}$ as one might naively expect.

\subsection{Shuffles}\index{Shuffles}

Shuffles are a special type of permutation of a finite set. One can think of them as taking a deck of cards, cutting it into two, and then shuffling the two parts once --- hence the name. The order of both parts remains unchanged, while the order of the whole deck is not the original one anymore. The formal definition is as follows.

\begin{definition}
	Let $p,q\ge0$, and let $n\coloneqq p+q$. A \emph{$(p,q)$-shuffle}\index{Shuffles} is a permutation $\sigma\in\S_n$ such that
	\[
	\sigma(1)<\cdots<\sigma(p)\quad\text{and}\quad\sigma(p+1)<\cdots<\sigma(n)\ .
	\]
	Equivalently, it is a partition $I\sqcup J = \{1,\ldots,n\}$ such that $|I|=p$ and $|J|=q$. The set of all $(p,q)$-shuffles is denoted by $\sh(p,q)$. A shuffle is \emph{trivial} if either $p=0$ or $q=0$.
\end{definition}

\subsection{Invariants and coinvariants}\label{subsect:invariants and coinvariants}

Let $V$ be a vector space with a left action by a finite group $G$. The \emph{invariants}\index{Invariants} of $V$ are
\[
V^G\coloneqq\{v\in V\mid \forall g\in G:g\cdot v = v\}\subseteq V\ .
\]
The \emph{coinvariants}\index{Coinvariants} of $V$ are
\[
V_G\coloneqq V/(v - g\cdot v)_{v\in V,g\in G}\ .
\]
It is well know and easy to check that the coinvariants of $V$ are the dual space to the invariants of $V$. Moreover, in characteristic $0$, we have an isomorphism
\[
V^G\stackrel{\cong}{\longrightarrow}V_G
\]
given by sending an invariant element of $V$ to its equivalence class. The inverse is given by sending a class $[v]\in V_G$ with representative $v\in V$ to
\[
\frac{1}{|G|}\sum_{g\in G}g\cdot v\in V^G.
\]
We will sometimes implicitly use this identification in the main body of the text.

%% file: AlgebraicOperads.tex
\chapter{Operad theory}\label{ch:operads}

The main tool used throughout the present thesis is algebraic operads. These objects give us an effective way to encode "types of algebras" and to study their properties. They were first introduced in the context of algebraic topology in the works of J. M. Boardman and R. M. Vogt \cite{BoardmanVogt} and J. P. May \cite{May} -- the name ``operad" is due to J. P. May himself, coming from the two words ``operation" and ``monad" --- but have since then found many applications in various domains of mathematics, in the works of many authors. It would be impossible to give a complete survey of the history and applications of operads throughout mathematics, and besides, it would be outside the scope of this thesis, so we will leave it at that.

\medskip

In this chapter, we give a short overview on the subject of operads --- more precisely, algebraic operads, meaning operads in the symmetric monoidal category of chain complexes --- and review the notions we will need in subsequent chapters. Our main reference is the book \cite{LodayVallette}. We will try to stay as close as possible to the notations and conventions used there, and to stress whenever we choose to do things differently.

\medskip

Here and throughout the whole thesis, unless otherwise explicitly stated, we will always work in the differential graded context. We work in the category $\chain$ of chain complexes over a field $\k$ of characteristic $0$. The symbol $\S_n$ is reserved to denote the symmetric groups. All operads and cooperads we will consider are \emph{reduced}, meaning that they are trivial in arity $0$, and that their arity $1$ component is spanned by the identity operation.

\section{(Co)operads and (co)algebras over them}

In this first section, we introduce the basic objects of operad theory: $\S$-modules, operads, cooperads, algebras and coalgebras over operads and cooperads respectively.

\subsection{\texorpdfstring{$\S$}{S}-modules}

The underlying object of all operads and cooperads is the $\S$-module. For further details on the material presented here, see \cite[Sect. 5.1]{LodayVallette}.

\begin{definition}
	An $\S$-module\index{$\S$-module} $M$ is a sequence of chain complexes
	\[
	M=(M(0),M(1),M(2),\ldots)\ ,
	\]
	where for each $n\ge0$, the chain complex $M(n)$ --- called the component of $M$ of \emph{arity}\index{Arity} $n$ --- is a right $\k[\S_n]$-module. A \emph{morphism of $\S$-modules} $f:M\longrightarrow N$ is a collection of $\S_n$-equivariant chain maps $f(n):M(n)\to N(n)$. The category of $\S$-modules and their morphisms is denoted by $\Smod$.
\end{definition}

\begin{example}\label{example:unit and endomorphism S-module}
	The following two $\S$-modules play a central role in the theory of operads and algebras over them.
	\begin{enumerate}
		\item The \emph{unit $\S$-module} is
		\[
		I\coloneqq(0,\k,0,0,\ldots)
		\]
		with the trivial $\S$-action. It is the unit for the composite product of $\S$-modules described below in \cref{def:composite product}.
		\item To any chain complex $V$ we can associate a canonical $\S$-module $\End_V$, the \emph{endomorphism $\S$-module}, defined by
		\[
		\End_V(n)\coloneqq\hom(V^{\otimes n},V)
		\]
		with the $\S_n$-action given by permuting the $n$ starting copies of $V$.
	\end{enumerate}
\end{example}

\begin{example}
	Given a chain complex $V$, one can see it as an $\S$-module concentrated in arity $0$, i.e. $(V,0,0,\ldots)$ with trivial $\S$-action.
\end{example}

There are various natural operations that one would like to be able to do on $\S$-modules, such as direct sums and tensor products. There is a good way to define such operations, which we review here.

\begin{definition}
	Let $M$ and $N$ be two $\S$-modules.The \emph{direct sum} $M\oplus N$ of $M$ and $N$ is the $\S$-module given by
	\[
	(M\oplus N)(k)\coloneqq M(k)\oplus N(k)
	\]
	for $k\ge0$, with the obvious $\S_k$-action.
\end{definition}

\begin{definition}
	Let $M$ and $N$ be two $\S$-modules. The \emph{(Hadamard) tensor product}\index{Tensor product!of $\S$-modules} $M\otimes N$ of $M$ and $N$ is the $\S$-module
	\[
	(M\otimes N)(k)\coloneqq M(k)\otimes N(k)
	\]
	for $k\ge0$, with the diagonal $\S_k$-action.
\end{definition}

\begin{remark}
	In the literature, one sometimes finds specific symbols for the Hadamard tensor product, in order to distinguish it from other notions of tensor product. Since in this thesis a tensor product between $\S$-modules or (co)operads is always meant as a Hadamard tensor product, no confusion is possible and we will simply use the symbol $\otimes$.
\end{remark}

\begin{definition}
	Let $M$ and $N$ be two $\S$-modules. The \emph{inner hom}\index{Inner!hom!of $\S$-modules} $\hom(M,N)$ of $M$ and $N$ is the $\S$-module
	\[
	\hom(M,N)(k)\coloneqq\hom(M(k),N(k))
	\]
	for $k\ge0$, and the $\S_k$ action given by $(f^\sigma)(x)\coloneqq f(x^{\sigma^{-1}})^\sigma$.
\end{definition}

\begin{definition}\label{def:composite product}
	Let $M$ and $N$ be two $\S$-modules. The \emph{composite}\index{Composite of $\S$-modules} $M\circ N$ of $M$ and $N$ is the $\S$-module
	\[
	(M\circ N)(k)\coloneqq\bigoplus_{i\ge0}M(i)\otimes_{\S_i}\left(\bigoplus_{j_1+\cdots j_i = k}\Ind_{\S_{j_1}\times\cdots\times\S_{j_i}}^{\S_k}\left(N(j_1)\otimes\cdots\otimes N(j_i)\right)\right)
	\]
	for $k\ge0$, and the $\S_k$-action induced by the $\S$-actions on $M$ and $N$. The $\S_i$-action on
	\[
	\Ind_{\S_{j_1}\times\cdots\times\S_{j_i}}^{\S_k}\left(N(j_1)\otimes\cdots\otimes N(j_i)\right)
	\]
	is given by permutations of the tuple $(j_1,\ldots,j_i)$.
\end{definition}

The spaces $(M\circ N)(k)$ are spanned by the equivalence classes of the elements
\[
(\mu;\nu_1,\ldots,\nu_i;\sigma)
\]
under the action of $\S_i$, where $\mu\in M(i)$, $\nu_a\in N(j_a)$, and $\sigma\in\Sh(j_1,\ldots,j_i)$ is a shuffle. If $\sigma$ is the identity permutation, then we will denote the element simply by
\[
(\mu;\nu_1,\ldots,\nu_i) = \mu\circ(\nu_1,\ldots,\nu_i)\ .
\]
If $f:M\to M'$ and $g:N\to N'$ are two morphisms of $\S$-modules, then there is an obvious induced morphism of $\S$-modules
\[
f\circ g:M\circ N\longrightarrow M'\circ N'
\]
given by
\[
(f\circ g)(\mu;\nu_1,\ldots,\nu_i)\coloneqq(f(\mu);g(\nu_1),\ldots,g(\nu_i))\ .
\]
Notice that the (arity-wise) homology of an $\S$-module is again an $\S$-module in a natural way. There is the following important result relating the homology of the composite of two $\S$-mo\-du\-les with the composite of the homologies.

\begin{theorem}[Operadic Künneth formula]\label{thm:operadic Kunneth}\index{Operadic!Künneth formula}
	Suppose $M$ is an $\S$-module and assume that the base field $\k$ has characteristic $0$. Then we have
	\[
	H_\bullet(M\circ N)\cong H_\bullet(M)\circ H_\bullet(N)\ .
	\]
\end{theorem}

One can also consider the following version of the composite product.

\begin{definition}
	Let $M$ and $N$ be two $\S$-modules. We define
	\[
	(M\bar{\circ}N)(k)\coloneqq\bigoplus_{i\ge0}\left(M(i)\otimes\left(\bigoplus_{j_1+\cdots j_i = k}\Ind_{\S_{j_1}\times\cdots\times\S_{j_i}}^{\S_k}\left(N(j_1)\otimes\cdots\otimes N(j_i)\right)\right)\right)^{\S_i}.
	\]
\end{definition}

\begin{remark}
	When working over a field $\k$ of characteristic $0$, one has an identification between invariants and coinvariants, as explained in \cref{subsect:invariants and coinvariants}, so that there is a natural isomorphism $M\circ N\cong M\bar{\circ}N$. Since we sill never work in positive characteristic when using symmetric (co)operads, we will implicitly use this identification throughout the rest of the present work.
\end{remark}

As it is often the case, one desires a nice monoidal structure on the category of $\S$-modules. Almost surprisingly, the correct monoidal product one is led to consider if one wants to use $\S$-modules to encode algebraic structures is not the Hadamard tensor product, but the composite product.

\begin{proposition}
	The category of $\S$-modules $(\Smod,\circ,I)$ endowed with the composite product of $\S$-modules and the unit $\S$-module is a monoidal category. The same is true by replacing $\circ$ by $\bar{\circ}$.
\end{proposition}

Given an $\S$-module, one can associate an endofunctor of chain complexes to it.

\begin{definition}
	Let $M$ be an $\S$-module. The \emph{Schur functor}\index{Schur functor}
	\[
	M:\chain\longrightarrow\chain
	\]
	associated to $M$ is defined by
	\[
	M(V)\coloneqq\bigoplus_{n\ge0}M(n)\otimes_{\S_n}V^{\otimes n},
	\]
	with $\S_n$ acting on $V^{\otimes n}$ by permuting the $n$ copies of $V$.
\end{definition}

Notice that the notation is coherent: the Schur functor associated to the composite of two $\S$-modules is the composite of the two Schur functors. Also, morphisms of $\S$-modules induce natural transformations between the associated Schur functors.

\subsection{Operads and cooperads}

We go on by defining the central objects of this chapter --- and indeed the central objects of this whole thesis: operads and cooperads. A possible point of view to interpret and understand operads is the following. A type of algebra --- such as associative algebras, commutative algebras, Lie algebras, and so on --- is usually defined as a vector space endowed with some operations respecting some relations. For example, a commutative algebra is a vector space $A$ together with a commutative binary operation
\[
\mu:A^{\otimes 2}\longrightarrow A
\]
which is associative, i.e. $\mu(\mu\otimes1_A) = \mu(1_A\otimes\mu)$. An operad is an object encoding the operations of a type of algebra, as well as the relations between them. This gives us a way to speak and to prove theorems about all algebras of the same type at once, and to prove general relations between different types of algebras. Cooperads are the analogue notion for coalgebras, and are almost dual to operads. For more details, see e.g. \cite[Sect. 5.2--8]{LodayVallette}.

\medskip

We begin with the definition of an operad.

\begin{definition}
	An \emph{operad}\index{Operad} $\P$ is an $\S$-module $\P$ endowed with two morphisms of $\S$-modules
	\[
	\gamma_\P:\P\circ\P\longrightarrow\P\ ,
	\]
	called the \emph{composition map} of $\P$, and
	\[
	\eta_\P:I\longrightarrow\P\ ,
	\]
	the \emph{unit map} of $\P$, making $\P$ into a monoid. Explicitly, that means that we have natural isomorphisms
	\[
	\gamma_\P(\gamma_\P\circ1_\P) = \gamma_\P(1_\P\circ\gamma_\P),\qquad\gamma_\P(\eta_\P\circ1_\P) = \gamma_\P = \gamma_\P(1_\P\circ\eta_\P)\ .
	\]
	A morphism $f:\P\to\Q$ of operads is a morphism commuting with the compositions and the unit maps of the source and of the target. Explicitly, one asks that
	\[
	f\gamma_\P = \gamma_\Q(f\circ f),\qquad f\eta_\P = \eta_\Q
	\]
	on $\P\circ\P$ and $I$ respectively. The category of operads and their morphisms is denoted by $\op$.
\end{definition}

This means that for any $k\ge0$ and $n_1,\ldots,n_k\ge0$ we can compose
\[
\P(k)\otimes_{\S_n}\P(n_1)\otimes\cdots\otimes\P(n_k)\longrightarrow\P(n_1+\cdots+n_k)\ ,
\]
which can be seen as taking an operation of arity $k$ of an algebra, and $k$ other operations of arity $n_1,\ldots,n_k$, and putting the latter operations in the $k$ slots of the former operation, thus obtaining an operation of arity $n_1+\cdots+n_k$. The original definition by P. May \cite{May} was given in term of these composition maps.

\medskip

Another equivalent way to define an operad is to give just the \emph{partial composition maps}
\[
\circ_j:\P(k)\otimes\P(n)\longrightarrow\P(k+n-1)
\]
for $1\le j\le k$, corresponding to composing an arity $n$ operation in the $j$th slot of an arity $k$ operation, without touching the other slots. Those maps must satisfy certain "associativity" axioms, see e.g. \cite[Sect. 5.3.4]{LodayVallette}.

\medskip

An operad is \emph{augmented}\index{Augmented operad} if there is a morphism $\varepsilon_\P:\P\to I$ such that $\varepsilon_\P\eta_\P=1_I$. If $\P$ is augmented, the kernel $\ker(\varepsilon_\P)$ of $\varepsilon_\P$ is called the \emph{augmentation ideal} of $\P$ and is denoted by $\Pbar$.

\medskip

An operad is \emph{reduced}\index{Reduced!operad}\footnote{This convention differs slightly from the one of \cite{LodayVallette}, where a reduced operad is only required to satisfy $\P(0)=0$.} if $\P(0)=0$ and $\P(1)=\k$. Notice that if $\P$ is reduced, then the unit map $\eta_\P$ is invertible, and its inverse gives a canonical augmentation $\varepsilon_\P$ of $\P$.

\begin{example}
	We look back at the $\S$-modules of \cref{example:unit and endomorphism S-module}. Both can be made into an operad with very natural composition and unit map.
	\begin{enumerate}
		\item The unit $\S$-module $I$ becomes an operad with both maps
		\[
		\gamma_I:I\circ I=I\stackrel{1_I}{\longrightarrow}I,\qquad\eta_I:I\stackrel{1_I}{\longrightarrow}I
		\]
		given by the identity map.
		\item Given a chain complex $V$, the endomorphism $\S$-module $\End_V$ becomes an operad --- the \emph{endomorphism operad}\index{Endomorphism operad} of $V$ --- by defining
		\[
		\gamma_{\End_V}(\mu;\nu_1,\ldots,\nu_k;\sigma)\coloneqq\mu(\nu_1,\ldots,\nu_k)^\sigma
		\]
		given by composing the operations $\nu_1,\ldots,\nu_k$ into the $k$ slots of $\mu:V^{\otimes k}\to V$, and
		\[
		\eta_{\End_V}(1)\coloneqq1_V
		\]
		given by sending $1\in\k$ to the identity map of $V$.
	\end{enumerate}
\end{example}

Given two operads $\P$ and $\Q$ one can endow the tensor product of the underlying $\S$-modules with an operad structure, thus constructing the \emph{(Hadamard) tensor product} $\P\otimes\Q$. Details are given e.g. in \cite[Sect. 5.3.2]{LodayVallette}.

\medskip

A very important construction is the free operad\index{Free!operad} over a given $\S$-module. We give an explicit construction and some properties of this object, but omit all of the proofs. See \cite[Sect. 5.5]{LodayVallette} for a complete treatment.

\begin{definition}
	Let $M$ be an $\S$-module. The \emph{tree module}\index{Tree module} $\T(M)$ over $M$ is defined by the recursion
	\[
	\T_0M\coloneqq I,\qquad\T_{n+1}M\coloneqq I\oplus(M\circ\T_nM)
	\]
	and by setting $\T(M)\coloneqq\colim_n\T_nM$.
\end{definition}

Elements of $\T(M)$ can be visualized as rooted trees with vertices of arity $n$ labeled by an element of $M(n)$. Then one defines the composition map
\[
\gamma_{\T(M)}:\T(M)\circ\T(M)\longrightarrow\T(M)
\]
simply as grafting of trees, and the unit map
\[
\eta_{\T(M)}:I\longrightarrow\T(M)
\]
by sending $1\in\k=I(1)$ to the trivial tree of arity one without any vertex. Moreover, there is a natural inclusion of $\S$-modules
\[
M\longrightarrow\T(M)
\]
given by sending an element $\mu\in M(n)$ to the $n$-corolla with its single vertex labeled by $\mu$.

\begin{theorem}
	The triple $\T(M)\coloneqq(\T(M),\gamma_{\T(M)},\eta_{\T(M)})$ forms an operad, called the \emph{free operad}\index{Free!operad} over $M$. It satisfies the following universal property. For any morphism of $\S$-modules $M\to\P$, where $\P$ is an operad, there exists a unique morphism of operads $\widetilde{f}:\T(M)\to\P$ making the diagram
	\begin{center}
		\begin{tikzpicture}
			\node (a) at (0,0){$M$};
			\node (b) at (2,0){$\T(M)$};
			\node (c) at (2,-1.5){$\P$};
			
			\draw[->] (a) -- (b);
			\draw[->] (a) -- node[below left]{$f$} (c);
			\draw[->,dashed] (b) -- node[right]{$\widetilde{f}$} (c);
		\end{tikzpicture}
	\end{center}
	commute in $\Smod$.
\end{theorem}

Notice that all operads satisfying the universal property in the Theorem above are isomorphic. Moreover, the universal property gives a bijection
\[
\hom_\op(\T(M),\P)\cong\hom_\Smod(M,\P)\ ,
\]
that is to say that the free operad functor is left adjoint to the forgetful functor from operads to $\S$-modules.

\medskip

Almost dual to the notion of operad is the notion of cooperad.

\begin{definition}
	A \emph{cooperad}\index{Cooperad} $\C$ is an $\S$-module $\C$ together with two morphisms of $\S$-modules
	\[
	\Delta_\C:\C\longrightarrow\C\bar{\circ}\C\ ,
	\]
	called the \emph{decomposition map} of $\C$, and
	\[
	\varepsilon_\C:\C\longrightarrow I\ ,
	\]
	the \emph{counit map} of $\C$, making $\C$ into a comonoid.	A morphism of cooperads is a morphism commuting with the decomposition maps and the counit maps of the source and of the target. The category of cooperads and their morphisms is denoted by $\coop$.
\end{definition}

\begin{warning}
	The notion defined above is often called \emph{conilpotent cooperad} in the literature, e.g. in \cite{LodayVallette}. We will never use non-conilpotent cooperads.
\end{warning}

A cooperad $\C$ is \emph{coaugmented}\index{Coaugmented cooperad} if there is a morphism $\eta_\C:I\to\C$ such that $\varepsilon_\C\eta_\C=1_I$. The image of $1\in\k=I(1)$ under $\eta_\C$ is denoted by $\id\in\C(1)$ and is called the \emph{identity cooperation}.

\medskip

A cooperad $\C$ is \emph{reduced}\index{Reduced!cooperad} if $\C(0) = 0$ and $\C(1)=\k$. Again, a reduced cooperad is automatically coaugmented.

\medskip

Similarly to what happens with operads, given an $\S$-module $M$ one can build a cofree cooperad\index{Cofree!cooperad} by taking once again the tree module $\T(M)$ as underlying $\S$-module, defining the decomposition map
\[
\Delta_{\T(M)}:\T(M)\longrightarrow\T(M)\circ\T(M)
\]
by sending a tree to the sum of all its possible decompositions in $2$-leveled trees by cutting the edges of the original tree, and defining the counit
\[
\varepsilon_{\T(M)}:\T(M)\longrightarrow I
\]
as the projection. Moreover, we consider the natural projection of $\S$-modules
\[
\T(M)\longrightarrow M
\]
More details are found in \cite[Sect. 5.8.6--7]{LodayVallette}.

\begin{theorem}
	The triple $\T^c(M)\coloneqq(\T(M),\Delta_{\T(M)},\varepsilon_{\T(M)})$ forms a cooperad, called the \emph{cofree cooperad} over $M$. It satisfies the following universal property, dual to the one of the free operad. For any morphism of $\S$-modules $f:\C\to M$, where $\C$ is a cooperad, there exists a unique morphism of cooperads making the diagram
	\begin{center}
		\begin{tikzpicture}
			\node (a) at (0,0){$\C$};
			\node (b) at (2,0){$\T^c(M)$};
			\node (c) at (2,-1.5){$M$};
			
			\draw[->,dashed] (a) -- node[above]{$\widetilde{f}$} (b);
			\draw[->] (a) -- node[below left]{$f$} (c);
			\draw[->] (b) -- (c);
		\end{tikzpicture}
	\end{center}
	commute in $\Smod$.
\end{theorem}

Observe that all cooperads satisfying the universal property in the Theorem above are isomorphic. Moreover, the universal property gives a natural bijection
\[
\hom_\coop(\C,\T^c(M))\cong\hom_\Smod(\C,M)\ ,
\]
that is to say that the cofree cooperad functor is right adjoint to the forgetful functor from cooperads to $\S$-modules.

\medskip

\label{page:monadic decomposition map}Given a cooperad $\C$, one can iterate the decomposition map $\Delta_\C$ in order to obtain a \emph{monadic decomposition map}\index{Monadic decomposition map $\Delta_\C^\mathrm{mon}$}
\[
\Delta_\C^\mathrm{mon}:\overline{\C}\longrightarrow\T^c(\overline{\C})
\]
sending an element of $\overline{\C}$ to all of its possible decompositions not including the identity element $\id\in\C(1)$. Formally, we define the \emph{reduced decomposition map} $\overline{\Delta}_\C$ of $\C$ as
\[
\overline{\Delta}_\C(c)\coloneqq\Delta_\C-\id\circ c - c\circ\id^{\otimes n}
\]
for $c\in\C(n)$. Now we define iteratively
\[
\widetilde{\Delta}_1\coloneqq1_\C:\overline{\C}\longrightarrow\overline{\C}\ ,
\]
and
\[
\widetilde{\Delta}_n\coloneqq1_\C + (1_\C\circ\widetilde{\Delta}_{n-1})\overline{\Delta}_\C:\overline{\C}\longrightarrow\T_n(\overline{\C})
\]
for $n\ge2$. Then, we set
\[
\Delta_\C^\mathrm{mon}\coloneqq\colim_n\widetilde{\Delta}_n:\overline{\C}\longrightarrow\T(\overline{\C})\ .
\]
One can always dualize a cooperad to obtain an operad, while the converse is not always true. This works as follows. Let $\C$ be a cooperad, then for $n\ge0$ the chain complex $\C(n)^\vee$ is a left $\S_n$-module. We make it into a right $\S_n$-module by $c^\sigma\coloneqq\sigma\cdot c$ for $c\in\C(n)^\vee$. The dual of the decomposition map $\Delta_\C$ gives the composition map $\gamma_{\C^\vee}$, and the dual of the counit map gives the unit. The problem when trying to do the same thing with an operad $\P$ is that the dual of the composition map $\gamma_\P$ could give us infinite sums, and thus not land in $\P^\vee\circ\P^\vee$. However, if we assume that $\P$ is reduced and that either
\begin{itemize}
	\item for all $n\ge2$, the degrees in which $\P(n)$ is non-zero are bounded below, and $\P(n)$ is finite-dimensional in every degree, or that
	\item for all $n\ge2$, the degrees in which $\P(n)$ is non-zero are bounded above, and $\P(n)$ is finite-dimensional in every degree,
\end{itemize}
then everything works out and the dual of an operad gives a cooperad. Notice that The conditions we gave in order to be able to dualize an operad and obtain a cooperad are sufficient, but not necessary. In particular, if an operad is finite dimensional in every arity, then it can be dualized to obtain a cooperad.

\subsection{(Co)algebras over (co)operads}\label{subsection:(co)algebras over (co)operads}

As already mentioned, the principal use of (co)operads is to encode (co)algebras.

\begin{definition}
	Let $\P$ be an operad. A \emph{$\P$-algebra}\index{Algebra over an operad}\index{$\P$-algebra} $A$ is a chain complex $A$ together with a composition map
	\[
	\gamma_A:\P(A)\longrightarrow A
	\]
	satisfying the relations
	\[
	\gamma_A(1_\P\circ\gamma_A) = \gamma_A(\gamma_\P\circ1_A)
	\]
	on $\P(\P(A))$, and
	\[
	\gamma_A(\eta_\P\circ1_A)=\gamma_A
	\]
	on $I(A)\cong A$. A \emph{morphism of $\P$-algebras} $f:A\to B$ is a chain map commuting with the respective composition maps, that is to say
	\[
	f\gamma_A = \gamma_B(1_\P\circ f)
	\]
	on $\P(A)$. The category of $\P$-algebras with their morphisms is denoted by $\Palg$.
\end{definition}

The elements of the operad can be interpreted as operations on the algebra, as is made evident by the following result.

\begin{proposition}\label{prop:P-alg is map P->End}
	The structure of a $\P$-algebra on a chain complex $A$ is equivalent to a map of operads
	\[
	\rho_A:\P\longrightarrow\End_A\ .
	\]
\end{proposition}

\begin{proof}
	Given $\gamma_A$, one obtains $\rho_A$ by
	\[
	\rho_A(p)(a_1,\ldots,a_n)\coloneqq\gamma_A(p\otimes a_1\otimes\cdots\otimes a_n)\ ,
	\]
	and \emph{vice versa}. The details are left to the reader.
\end{proof}

\begin{proposition}\label{prop:free algebra}
	Let $\P$ be an operad and $V$ be a chain complex. Then the map
	\[
	\gamma_{\P(V)}\coloneqq\gamma_\P\circ1_V:\P(\P(V))\cong(\P\circ\P)(V)\longrightarrow\P(V)
	\]
	makes $\P(V)$ into a $\P$-algebra. It is called the \emph{free $\P$-algebra}\index{Free!$\P$-algebra} over $V$, and it satisfied the following universal property. For any morphism of chain complexes $f:V\to A$, where $A$ is a $\P$-algebra, there exists a unique morphism of $\P$-algebras $\widetilde{f}:\P(V)\to A$ making the diagram
	\begin{center}
		\begin{tikzpicture}
			\node (a) at (0,0){$V$};
			\node (b) at (2,0){$\P(V)$};
			\node (c) at (2,-1.5){$A$};
			
			\draw[->] (a) -- (b);
			\draw[->] (a) -- node[below left]{$f$} (c);
			\draw[->,dashed] (b) -- node[right]{$\widetilde{f}$} (c);
		\end{tikzpicture}
	\end{center}
	commute in $\chain$, where $V\to\P(V)$ is the canonical inclusion.
\end{proposition}

Once again, the free $\P$-algebra functor is left adjoint to the forgetful functor from $\P$-algebras to chain complexes.

\begin{lemma}\label{lemma:isomorphisms of P-algebras}
	Let $A,B$ be two $\P$-algebras, and let $f:A\to B$ be a morphism of $\P$-algebras. Suppose that $f$ is an isomorphism of chain complexes. Then $f$ is an isomorphism of $\P$-algebras as well.
\end{lemma}

\begin{proof}
	Write $g:B\to A$ for the inverse of $f$ in chain complexes. We prove that $g$ is a morphism of $\P$-algebras. We have
	\begin{align*}
		g\gamma_B =&\ g\gamma_B(1_\P\circ fg)\\
		=&\ gf\gamma_B(1_\P\circ g)\\
		=&\ \gamma_B(1_\P\circ g)
	\end{align*}
	as maps $\P(B)\to A$, where in the second line we used the fact that $f$ is a morphism of $\P$-algebras.
\end{proof}

Suppose we have a morphism of operads $f:\Q\to\P$. Then it is obvious from \cref{prop:P-alg is map P->End} that every $\P$-algebra $A$ is also a $\Q$-algebra, which we denote by $f^*A$, by $\rho_{f^*A} = f^*\rho_A = \rho_Af$. This defines a functor from $\P$-algebras to $\Q$-algebras, called \emph{restriction of structure}\index{Restriction of structure}. This functor is right adjoint to a functor $f_!$, called \emph{extension of structure}\index{Extension of structure}, i.e.
\[
\adjunction{f_!}{\Qalg}{\Palg}{f^*}.
\]
The $\P$-algebra $f_!A$ associated to a $\Q$-algebra $A$ can be described as a reflexive coequalizer, see \cite[Sect. 3.3.5]{Fresse} for details. These functors define a Quillen adjunction between the respective categories of algebras. For these model categorical aspects, we invite the interested reader to consult \cite[Sect. 16]{Fresse}.

\medskip

Almost dually to the notion of algebra over an operad is the notion of coalgebra over a cooperad.

\begin{definition}
	Let $\C$ be a cooperad. A \emph{$\C$-coalgebra}\index{Coalgebra over a cooperad}\index{$\C$-coalgebra} $C$ is a chain complex $C$ together with a decomposition map
	\[
	\Delta_C:C\longrightarrow\widehat{\C}(C)\coloneqq\prod_{n\ge0}\left(\C(n)\otimes C^{\otimes n}\right)^{\S_n}
	\]
	such that
	\[
	(\Delta_\C\circ1_C)\Delta_C = (1_\C\circ\Delta_C)\Delta_C,\qquad (\varepsilon_\C\circ1_C)\Delta_C=\Delta_C
	\]
	on $C$. A morphism of $\C$-coalgebras $f:C\to D$ is a chain map commuting with the respective decomposition maps, that is
	\[
	(1_\C\circ f)\Delta_C = \Delta_Df
	\]
	on $C$.
\end{definition}

There is no analogue to \cref{prop:P-alg is map P->End} since it is not possible to define an "endomorphism cooperad" playing the dual role of $\End_V$.

\medskip

Let $\C$ be a reduced, coaugmented cooperad, and let $C$ be a $\C$-coalgebra. We denote the image of $x\in C$ under the decomposition map $\Delta_C$ by
\[
\Delta_C(x) = (x_1,x_2,\ldots)\in\prod_{n\ge1}\left(\C(n)\otimes C^{\otimes n}\right)^{\S_n}.
\]
There is a canonical ascending filtration on $C$ given by
\[
\F^0_{\C}C\coloneqq 0\ ,\quad\text{and}\quad\F^n_{\C}C\coloneqq\{x\in C\mid x_k=0\text{ for all }k>n\}
\]
for $n\ge1$. It is called the \emph{coradical filtration}\index{Coradical filtration} of the coalgebra.

\begin{remark}
	If one works more in general with connected weight graded cooperads instead of reduced cooperads, then one must be a little more careful when defining the coradical filtration, see e.g. \cite[Def. 2.2]{val14}. What one desires in the end, is that the coproduct $\Delta_C$ is a morphism of connected weighted cooperads from $C$ to $\C(C)$ endowed with the induced connected weight grading.
\end{remark}

\begin{definition}
	A $\C$-coalgebra $C$ is \emph{conilpotent}\index{Conilpotent coalgebra} if the canonical filtration is exhaustive, that is
	\[
	C = \colim_n\F^n_\C C
	\]
	as $\C$-coalgebras. Conilpotent $\C$-coalgebras form a full subcategory of the category of $\C$-coalgebras and their morphisms, which we will denote by $\Ccog$.
\end{definition}

\begin{remark}
	In the rest of this work, we will always work with conilpotent coalgebras, and never with non-conilpotent ones. For this reason, we will sometimes omit the adjective conilpotent when talking about coalgebras, and simply say $\C$-coalgebras when we really mean conilpotent $\C$-coalgebras.
\end{remark}

The coradical filtration is well-behaved with respect to the coproduct.

\begin{lemma}\label{lemma:property of coradical filtration}
	Let $C$ be a conilpotent $\C$-coalgebra. For each $n\ge0$, we have
	\[
	\Delta_C(\F_\C^nC)\subseteq\bigoplus_{\substack{k\ge1\\n_1+\cdots+n_k=n}}\left(\C(k)\otimes\F_\C^{n_1}C\otimes\cdots\otimes\F_\C^{n_k}C\right)^{\S_k}\ .
	\]
\end{lemma}

\begin{proof}
	The fact that we land in invariants is automatic, so that we can work in the non-symmetric case without loss of generality. The case $n=0$ is trivial. Let $n\ge1$ and $x\in\F^n_\C C$. The relation $(\Delta_\C\circ1_C)\Delta_C=(1_\C\circ\Delta_C)\Delta_C$ implies that, for every $m>n$, we have
	\begin{align*}
		0 =&\ (\Delta_\C\circ1_C)\Delta^m_C(x)\\
		=&\ \sum_{\substack{k\ge1\\n_1+\cdots+n_k=m}}(1_\C\circ(\Delta^{n_1}_C,\ldots,\Delta^{n_k}_C))\Delta^k_C(x)\ .
	\end{align*}
	Since $C$ is conilpotent, the last sum is finite. Every term has a different underlying tree, so that we must have
	\begin{equation}\label{eq:rel on coproducts}
	(1_\C\circ(\Delta^{n_1}_C,\ldots,\Delta^{n_k}_C))\Delta^k_C(x) = 0
	\end{equation}
	for each $k\ge1$ and $n_1+\cdots+n_k=m$. Now fix $k\ge1$ and suppose that $\Delta^k_C(x)$ has a term which is not in the correct space. Since $C$ is conilpotent, that term lives in
	\[
	\C(k)\otimes\F_\C^{\widetilde{n}_1}C\otimes\cdots\otimes\F_\C^{\widetilde{n}_k}C
	\]
	for some $\widetilde{n}_1+\cdots+\widetilde{n}_k > n$, which contradicts relation (\ref{eq:rel on coproducts}).
\end{proof}

\begin{proposition}\label{prop:cofree conil coalgebra}
	Let $\C$ be an operad and $V$ be a chain complex. Then the map
	\[
	\Delta_{\C(V)}\coloneqq\Delta_\C\circ1_V:\C(V)\longrightarrow(\C\circ\C)(V)\cong\C(\C(V))
	\]
	makes $\C(V)$ into a conilpotent $\C$-coalgebra. It is called the \emph{cofree conilpotent $\C$-coalgebra}\index{Cofree!$\C$-coalgebra} over $V$, and it satisfies the following universal property. For any morphism of chain complexes $g:C\to V$, where $C$ is a conilpotent $\C$-coalgebra, there exists a unique morphism of $\C$-coalgebras $\widetilde{g}:C\to \C(V)$ making the diagram
	\begin{center}
		\begin{tikzpicture}
		\node (a) at (0,0){$\C(V)$};
		\node (b) at (0,1.5){$C$};
		\node (c) at (2,0){$V$};
		
		\draw[->] (a) -- (c);
		\draw[->] (b) -- node[above right]{$f$} (c);
		\draw[->,dashed] (b) -- node[left]{$\widetilde{f}$} (a);
		\end{tikzpicture}
	\end{center}
	commute in $\chain$, where the map $\C(V)\to V$ is the canonical projection.
\end{proposition}

\begin{lemma}\label{lemma:isomorphisms of C-coalgebras}
	Let $C,D$ be two conilpotent $\C$-coalgebras, and let $f:C\to D$ be a morphism of $\C$-coalgebras. Suppose that $f$ is an isomorphism of chain complexes. Then $f$ is an isomorphism of $\C$-coalgebras as well.
\end{lemma}

\begin{proof}
	The proof is dual to the proof of \cref{lemma:isomorphisms of P-algebras} and left to the reader.
\end{proof}

\begin{remark}
	One can also define coalgebras over an operad --- see \cite[Sect. 5.2.15]{LodayVallette} --- and algebras over a cooperad. For a modern treatment of these notions and the study of their homotopical behavior, see e.g. \cite{lgl18}.
\end{remark}

\subsection{Limits and colimits of (co)algebras}

The categories of algebras over operads and conilpotent coalgebras over cooperads are complete and cocomplete.

\begin{theorem}[{\cite[Sect. 1.6--7]{gj94}}]
	Let $\P$ be an operad, and let $\C$ be a cooperad.
	\begin{enumerate}
		\item The category $\Palg$ of $\P$-algebras admits all limits and colimits.
		\item The category $\Ccog$ of conilpotent $\C$-coalgebras admits all limits and colimits.
	\end{enumerate}
\end{theorem}

\begin{remark}
	The category of $\P$-algebras is always complete and cocomplete, even when working over a field of positive characteristic. For conilpotent $\C$-coalgebras, one gets all limits and colimits provided one supposes that the cooperad $\C$ satisfies some additional properties.
\end{remark}

Here are some easily described limits and colimits of (co)algebras.

\medskip

\cref{prop:free algebra} tells us that we have a natural bijection
\[
\hom_\chain(V,\forget{A})\cong\hom_\Palg(\P(V),A)\ ,
\]
where $\forget{(-)}$ is the forgetful functor from $\P$-algebras to chain complexes. In other words, we have an adjunction
\[
\adjunction{\P}{\chain}{\Palg}{\forget{(-)}}\ .
\]
In particular, the forgetful functor preserves limits, since it is right adjoint.

\begin{proposition}
	Let $A,B$ be $\P$-algebras. The product $A\times B$\index{Product!of $\P$-algebras} of $A$ and $B$ has the product of the chain complexes $A$ and $B$ as underlying chain complex, which we identify with the direct sum $A\oplus B$, and the $\P$-algebra structure is given by the composite
	\[
	\gamma_{A\times B}\coloneqq\left(\P(A\oplus B)\xrightarrow{\proj_A\oplus\proj_B}\P(A)\oplus\P(B)\xrightarrow{\gamma_A\oplus\gamma_B}A\oplus B\right).
	\]
	Here $\proj_A:\P(A\oplus B)\to\P(A)$ is the canonical projection, and similarly for $\proj_B$.
\end{proposition}

In other words, the product of $A$ and $B$ is given by $A\oplus B$ with the $\P$-algebra structures of $A$ and $B$ on the two factors, and nothing mixing them, i.e. if we put some elements of $A$ and some elements of $B$ in some operation, we get $0$.

\medskip

The description of coproducts is only slightly more complicated.

\begin{proposition}
	Let $A,B$ be $\P$-algebras. The coproduct $A\sqcup B$\index{Coproduct!of $\P$-algebras} of $A$ and $B$ is
	\[
	A\sqcup B\coloneqq\P(A\oplus B)/\sim\ ,
	\]
	where the equivalence relation $\sim$ is generated by
	\begin{align*}
		p\otimes_{\S_n}(a_1\otimes\cdots\otimes a_n)\sim\gamma_A(p\otimes_{\S_n}(a_1\otimes\cdots\otimes a_n))\ ,\\
		p\otimes_{\S_n}(b_1\otimes\cdots\otimes b_n)\sim\gamma_B(p\otimes_{\S_n}(b_1\otimes\cdots\otimes b_n))
	\end{align*}
	for any $p\in\P(n)$, $a_i\in A$, and $b_i\in B$.
\end{proposition}

In other words, the coproduct of $A$ and $B$ is given by a copy of $A$, one copy of $B$, plus all the elements generated freely by mixing elements of $A$ and $B$ and acting with operations from the operad $\P$. Another useful way to express the coproduct of two algebras is as the coequalizer in chain complexes
\begin{center}
	\begin{tikzpicture}
		\node (a) at (0,0){$(\P\circ_{(1)}\P)(A\oplus B)$};
		\node (b) at (6,0){$\P(A\oplus B)$};
		\node (c) at (9,0){$A\sqcup B$};
		
		\draw[->,transform canvas={yshift=3pt}] (a) to node[above]{$\gamma_\P\circ1_{A\oplus B}$} (b);
		\draw[->,transform canvas={yshift=-1pt}] (a) to node[below]{$1_\P\circ_{(1)}(\gamma_A+\gamma_B)$} (b);
		\draw[->] (b) to (c);
	\end{tikzpicture}
\end{center}
where
\[
\gamma_A+\gamma_B:\P(A\oplus B)\longrightarrow A\oplus B
\]
is given by $\gamma_A$ on $\P(A)$, by $\gamma_B$ on $\P(B)$, and by zero on anything containing mixed terms. The $\P$-algebra structure is the one induced by the $\P$-algebra structure of $\P(A\oplus B)$.

\medskip

The situation for cooperads is essentially dual. By \cref{prop:cofree conil coalgebra}, we have the adjunction
\[
\adjunction{\forget{(-)}}{\Ccog}{\chain}{\C}
\]
between the forgetful functor $\forget{(-)}$ and the cofree conilpotent coalgebra functor $\C$. In particular, the forgetful functor preserves colimits.

\begin{proposition}
	Let $C,D$ be conilpotent $\C$-coalgebras. The coproduct $C\sqcup D$\index{Coproduct!of $\C$-coalgebras} of $C$ and $D$ is the conilpotent $\C$-coalgebra
	\[
	C\sqcup D\coloneqq C\oplus D
	\]
	with structure map
	\[
	\Delta_{C\sqcup D}\coloneqq\left(C\oplus D\xrightarrow{\Delta_C\oplus\Delta_D}\C(C)\oplus\C(D)\xhookrightarrow{\quad}\C(C\oplus D)\right).
	\]
\end{proposition}

The product of two conilpotent $\C$-coalgebras $C$ and $D$ is a sub-coalgebra of $\C(C\oplus D)$, which we can express as an equalizer.

\begin{proposition}\label{prop:products of coalgebras}
	Let $C,D$ be conilpotent $\C$-coalgebras. The product $C\times D$\index{Product!of $\C$-coalgebras} of $C$ and $D$ is the conilpotent $\C$-coalgebra given by the equalizer
	\begin{center}
		\begin{tikzpicture}
			\node (a) at (0,0){$C\times D$};
			\node (b) at (3,0){$\C(C\oplus D)$};
			\node (c) at (9,0){$(\C\circ_{(1)}\C)(C\oplus D)$};
			
			\draw[->] (a) to (b);
			\draw[->,transform canvas={yshift=3pt}] (b) to node[above]{$\Delta_{(1)}\circ1_{C\oplus D}$} (c);
			\draw[->,transform canvas={yshift=-1pt}] (b) to node[below]{$1_\C\circ(\Delta_C+\Delta_D)$} (c);
		\end{tikzpicture}
	\end{center}
	where
	\[
	\Delta_C+\Delta_D:C\oplus D\xrightarrow{\Delta_C\oplus\Delta_D}\C(C)\oplus\C(D)\xhookrightarrow{\quad}\C(C\oplus D)\ ,
	\]
	with the $\C$-coalgebra structure induced by the $\C$-coalgebra structure of $\C(C\oplus D)$.
\end{proposition}

\subsection{Quadratic (co)operads}

Many classically arising operads and cooperads can be presented by generators and relations in the way we expose now. More details can be found in \cite[Sect. 7.1]{LodayVallette}.

\begin{definition}
	An \emph{operadic quadratic data}\index{Quadratic!data} $(E,R)$ is a graded $\S$-module $E$ (i.e. an $\S$-module where all the differentials are trivial) together with a graded sub-$\S$-module $R\subseteq\T(E)^{(2)}$ of the graded sub-$\S$-module $\T(E)^{(2)}$ of the tree module $\T(E)$ given by trees with $2$ vertices. The elements of $E$ are called the \emph{generating operations}, and the elements of $R$ the \emph{relations}. A \emph{morphism} of quadratic data
	\[
	f:(E,R)\longrightarrow(F,S)
	\]
	is a morphism $f:E\to F$ of $\S$-modules such that $\T(f)(R)\subseteq S$.
\end{definition}

One can associate an operad $\P(E,R)$ to an operadic quadratic data $(E,R)$ by
\[
\P(E,R)\coloneqq\T(E)/(R)\ ,
\]
where $(R)$ is the smallest operadic ideal of $\T(E)$ containing $R$. Equivalently, $\P(E,R)$ is universal among the quotient operads $\P$ of $\T(E)$ such that the composite
\[
(R)\xrightarrow{\mathrm{incl.}}\T(E)\xrightarrow{\mathrm{proj.}}\P
\]
is zero. An operad obtained in this way is called a \emph{quadratic operad}\index{Quadratic!operad}. If moreover $E$ is concentrated in arity $2$, one speaks of a \emph{binary quadratic operad}\index{Binary!quadratic operad}.

\medskip

Dually, one can associate a cooperad $\C(E,R)$ to an operadic quadratic data $(E,R)$. It is defined through an universal property dual to the one given above, but it doesn't have a nice presentation such as $\P(E,R)$. Explicitly, $\C(E,R)$ is such that for any sub-cooperad $\C$ of $\T^c(E)$ such that the composite
\[
\C\xrightarrow{\mathrm{incl.}}\T^c(E)\xrightarrow{\mathrm{proj.}}\T^c(E)^{(2)}/(R)
\]
is zero, there exists a unique morphism of cooperads $\C\to\C(E,R)$ such that
\[
\C\longrightarrow\C(E,R)\xrightarrow{\mathrm{incl.}}\T^c(E) = \C\xrightarrow{\mathrm{incl.}}\T^c(E)\ .
\]
Any morphism of quadratic data induces a morphism between the associated operads, respectively cooperads.

\subsection{Some classical (co)operads and (co)algebras over them}\label{subsection:three graces}

We will now give some examples of operads and cooperads encoding some classical types of algebras and coalgebras. All of them are induced by quadratic data. We begin by the \emph{three graces}\index{Three graces}, the operads $\com$, $\ass$, and $\lie$.

\medskip

To define the operad $\com$\index{Commutative!algebras}\index{$\com$}, we take the $\S$-module $E$ with $E(2) = \k\mu$ and $E(n)=0$ for all $n\neq2$, with $|\mu|=0$ and the trivial action of $\S_2$ on $\mu$, and the relations $R$ spanned by $\mu\circ(\id\otimes\mu)-\mu\circ(\mu\otimes\id)$. We set
\[
\com\coloneqq\P(E,R)\ .
\]
Notice that, because of the relations, we have $\com(n)=\k\mu_n$ for $n\ge2$, where $\mu_n$ is any composition of copies of $\mu$ giving an operation of arity $n$. All such compositions are equal in $\com$ because of the relations. The algebras over $\com$ are exactly the commutative (not necessarily unital) algebras, i.e. vector spaces $A$ together with a commutative, associative multiplication. Commutativity is encoded in the fact that the $\S_2$-action on $\mu$ is trivial, while associativity is given by $R$.

\medskip

For the operad $\ass$\index{Associative algebras}\index{$\ass$}, we take $E$ with $E(2)=\k[\S_2]=\k a_e\oplus\k a_{(12)}$ with $|a_e|=|a_{(12)}|=0$ and $\S$-action given by $a_e^{(12)}=a_{(12)}$, and $R$ spanned by $a_e\circ(\id\otimes a_e)-a_e\circ(a_e\otimes\id)$ and the elements of $\T(E)^{(2)}$ that can be obtained by it by $\S_3$-action. We set
\[
\ass\coloneqq\P(E,R)\ .
\]
The operad $\ass$ is spanned by operations $\{a_\sigma\mid\sigma\in\S_n\}$ in arity $n\ge2$, corresponding to the multiplication
\[
(x_1,\ldots,x_n)\longmapsto x_{\sigma(1)}\cdots x_{\sigma(n)}
\]
in an associative (not necessarily unital) algebra.

\medskip

There is a non-symmetric version of $\ass$, which is denoted by $\as$\index{$\as$}. It is also a binary quadratic ns operad. It is generated by a single binary operation $\mu_2$ in arity $2$ and has only one relation, given by $m\circ(m\otimes\id)-m\circ(\id\otimes m)$. The resulting operad is $1$-dimensional in every arity $n\ge1$, where it is spanned by the operation $\mu_n$ multiplying $n$ elements in the order they are given. We recover $\ass$ from $\as$ by tensoring by the regular representation $\k[\S]$ of the symmetric groups.

\medskip

The last one of the three graces is the operad $\lie$\index{Lie!algebras}\index{$\lie$}. In order to define it, we define $E$ by $E(2)=\k b$ with $|b|=0$ and the sign representation of $\S_2$, and $R$ spanned by $b\circ(b\otimes\id) + b\circ(b\otimes\id)^{(123)} + b\circ(b\otimes\id)^{(321)}$ and the other elements obtained from this one by $\S_3$-action. An algebra over $\lie$ is the usual notion of (differential graded) Lie algebra: the operation $b$ gives the Lie bracket, which is antisymmetric since $b$ carries the sign representation, and $R$ gives the Jacobi rule.

\medskip

The three graces are related by morphisms of operads
\[
\lie\stackrel{a}{\longrightarrow}\ass\stackrel{u}{\longrightarrow}\com\ ,
\]
all of them coming from morphisms of operadic quadratic data. The morphism $a:\lie\to\ass$ is the \emph{antisymmetrization morphism}\index{Antisymmetrization morphism}, and it is given by sending $b\in\lie(2)$ to $a_e-a_{(12)}\in\ass(2)$. This corresponds to the fact that if $A$ is an associative algebra, then we can see it as a Lie algebra by $[x,y]\coloneqq xy - (-1)^{xy}yx$. The morphism $u:\ass\to\com$ is induced by sending both $a_e$ and $a_{(12)}$ to $\mu$, and correspond to the fact that every commutative algebra is trivially an associative algebra.

\medskip

One can also define operads $\ucom$ and $\uass$ encoding unital commutative, respectively associative algebras simply by setting $E(0)=\k u$ with $|u|=0$ in both cases, and adding the relations $\mu\circ(u\otimes\id) = \id$, respectively $a_e\circ(u\otimes\id)$ and $a_{(12)}\circ(u\otimes\id)$ to $R$. Notice that those relations are now contained in $\T(E)^{(1)}\oplus\T(E)^{(2)}$, so we are not in the quadratic case anymore.

\medskip

On the coalgebra side, we give the example of conilpotent coassociative coalgebras in the non-symmetric setting. These coalgebras are encoded by the cooperad $\mathrm{coAs}$ which is dual to $\as$. In other words, take $E = E(2) = \k a^\vee$, and $R = a^\vee\circ(\id\otimes a^\vee) - a^\vee\circ(a^\vee\otimes\id)$. We set
\[
\mathrm{coAs}\coloneqq\C(E,R)\ .
\]
A conilpotent coalgebra over this cooperad is a vector space\footnote{Or graded vector space, or chain complex. Here we work over vector spaces to avoid signs.} $C$ together with a decomposition map
\[
\Delta:C\longrightarrow\C\otimes C\ ,
\]
corresponding to the binary part of the whole coproduct $\Delta_C$, such that
\[
(\Delta\otimes 1_C)\Delta(x) = (1_C\otimes\Delta)\Delta(x)
\]
and such that any iteration of $\Delta$ on any element eventually terminates (which corresponds to the fact of being conilpotent).

\subsection{Non-symmetric (co)algebras and (co)operads}

There is a version of (co)operads which codes (co)algebras without symmetries (such as commutativity of a multiplication, and so on). It is called \emph{non-symmetric}, or \emph{ns (co)operads}\index{Non-symmetric (co)operads}. It is given by replacing $\S$-modules by sequences of chain complexes without any group action --- also called \emph{arity graded chain complexes} --- and by forgetting all of the group actions in the definitions. For example the composite product of $M$ and $N$ in this context becomes
\[
(M\circ N)(k)\coloneqq\bigoplus_{\substack{i\ge0\\j_1+\cdots+j_i=k}}M(i)\otimes N(j_1)\otimes\cdots\otimes N(j_i)\ ,
\]
and similarly for all the rest. More details are given in \cite[Sect. 5.9]{LodayVallette}.

\medskip

One can always pass from the symmetric world to the ns world by forgetting the $\S$-actions, and go the other way around by tensoring by the regular representation of the symmetric groups. This gives a pair of adjoint functors.

\medskip

It is often true that results holding in the symmetric world are also true in the ns world, while results that are true in the ns world hold in the symmetric world \emph{in characteristic $0$}. This is because e.g. when making operads and cooperads (or algebras and coalgebras) interact, in the symmetric world one has to identify invariants and coinvariants, which one can only do in characteristic $0$. One possible way to go around this restriction is to work with operads with divided powers \cite[Sect. 5.2.9]{LodayVallette}.

\section{Operadic homological algebra}

In order to go on, we have to introduce the constructions and notations allowing us to do homological algebra in the operadic context, such as the correct notion of suspension, the infinitesimal composites corresponding in some sense to "derivatives" of the composite product, and so on.

\subsection{Operadic suspension}

Notice that, given an operad $\P$, there is no natural way to put an operad structure on $s\P$. However, there is a good way to suspend and desuspend (co)operads. We present it here, following \cite[Sect. 7.2.2]{LodayVallette}.

\medskip

Let $\susp\coloneqq\End_{s\k}$ be the endomorphism operad of the chain complex $s\k$, which is of dimension $1$ and concentrated in degree $1$. We have
\[
\susp(n) = \k\susp_n\ ,
\]
where $\susp_n$ is the linear map of degree $1-n$ sending $s^n$ to $s$ and carries the sign representation. We have
\[
\susp_n\circ_1\susp_m = \susp_{n+m-1}
\]
and thus
\[
\susp_n\circ_j\susp_m = (-1)^{(j-1)(1-m)}\susp_{n+m-1}\ .
\]
This determines most of the signs in operad theory, for example in minimal models for Koszul operads, see \cref{section:minimal models and Koszul duality}.

\begin{definition}
	Let $\P$ be an operad. The \emph{(operadic) suspension}\index{Operadic!suspension} of $\P$ is the operad $\susp\otimes\P$.
\end{definition}

In order to desuspend operads, we consider instead $\susp^{-1}\coloneqq\End_{s^{-1}\k}$. It behaves similarly to $\susp$, and we have an isomorphism of operads
\begin{equation}\label{eq:suspension desuspension}
	\susp^{-1}\otimes\susp\cong\End_\k
\end{equation}
with the unit for the Hadamard tensor product given by
\[
\susp^{-1}_n\otimes\susp_n\longmapsto(-1)^{\frac{n(n-1)}{2}}m_n\ ,
\]
where $m_n\in\End_\k(n)$ is the map sending $1_\k^{\otimes n}$ to $1_\k$. The sign can be found by
\begin{align*}
	(-1)^{\frac{n(n-1)}{2}} =&\ s^{-n}s^n\\
	=&\ (-1)^\epsilon\susp^{-1}_n\susp_ns^{-n}s^n\\
	=&\ (-1)^{\epsilon+(1-n)(-n)}\susp^{-1}_ns^{-n}\susp_ns^n\\
	=&\ (-1)^\epsilon s^{-1}s^n\\
	=&\ (-1)^\epsilon\ ,
\end{align*}
with the additional sign appearing in the third line comes from the Koszul sign rule. Explicitly, the composition in $\susp^{-1}$ is given by
\[
\susp^{-1}_n\circ_j\susp^{-1}_m = (-1)^{(j-1)(1-m)}\susp^{-1}_{n+m-1}\ .
\]
A straightforward computation shows that this is compatible with the isomorphism (\ref{eq:suspension desuspension}).

\begin{definition}
	Let $\P$ be an operad. The \emph{(operadic) desuspension}\index{Operadic!desuspension} of $\P$ is the operad $\susp^{-1}\otimes\P$.
\end{definition}

\begin{proposition}
	The structure of a $\P$-algebra on a chain complex $A$ is equivalent to the structure of a $\susp\otimes\P$-algebra on $sA$, respectively to the structure of a $\susp^{-1}\otimes\P$-algebra on $s^{-1}A$.
\end{proposition}

\begin{proof}
	The structure on $sA$ is given by
	\[
	\gamma_{sA}(\susp_np\otimes sa_1\otimes sa_n)\coloneqq(-1)^\epsilon s\gamma_A(p\otimes a_1\otimes\cdots\otimes a_n)
	\]
	with
	\[
	\epsilon = np + \sum_{i=1}^n\sum_{j=1}^{i-1}a_j
	\]
	coming from the Koszul sign rule. We leave the rest of the proof to the reader.
\end{proof}

For cooperads, one can endow the $\S$-module $\End_{s\k}$ with the cooperad structure dual to the operad structure of $\End_{s^{-1}\k}$ to obtain a cooperad $\susp^c$, and similarly a cooperad $(\susp^{-1})^c$, defining cooperadic suspension and desuspension.

\begin{definition}
	The cooperads $\susp^c$ and $(\susp^{-1})^c$ are defined by
	\[
	(\susp^c)^\vee\coloneqq\susp^{-1},\qquad\text{and}\qquad(\susp^{-1})^c\coloneqq\susp^\vee.
	\]
\end{definition}

One can explicitly describe the decomposition map by finding the signs in
\[
\Delta_{(\susp^{-1})^c}(\susp^{-1}_n) = \sum_{\substack{k\ge1\\n_1+\cdots+n_k=n}}(-1)^{\epsilon(k,n_1,\ldots,n_k)}\susp^{-1}_k\otimes\susp^{-1}_{n_1}\otimes\cdots\susp^{-1}_{n_k}
\]
by computing
\begin{align*}
	(-1)^{\epsilon(k,n_1,\ldots,n_k)} =&\ \left\langle\Delta_{(\susp^{-1})^c}(\susp^{-1}_n),\susp_k\otimes\susp_{n_1}\otimes\cdots\susp_{n_k}\right\rangle\\
	=&\ \left\langle\susp^{-1}_n,\gamma_{\susp}(\susp_k\otimes\susp_{n_1}\otimes\cdots\susp_{n_k})\right\rangle\ .
\end{align*}

\subsection{Connected weight gradings}

Sometimes, in order to make certain homological arguments, one needs an additional grading on (co)operads. This is encoded in the notion of \emph{connected weight graded}\index{Connected weight grading} (co)operads, see e.g. \cite[Sect. 1.7]{val14}. When working with reduced (co)operads, one can usually consider the grading given by putting in weight $w$ the elements of arity $w+1$. We will always implicitly consider this canonical additional grading on reduced (co)operads. Many results presented in this thesis for reduced (co)operads also hold in the context of connected weight graded (co)operads.

\subsection{Infinitesimal composites}

The composite product of $\S$-modules is linear in the first slot, i.e.
\[
(M_1\oplus M_2)\circ N\cong(M_1\circ N)\oplus(M_2\circ N)
\]
for $\S$-modules $M_1,M_2$ and $N$, but not in the second one. In order to do homological algebra, we need a version of the composite product which is linear also on the right. This is the material of \cite[Sect. 6.1]{LodayVallette}.

\medskip

Let $M,N_1$ and $N_2$ be three $\S$-modules. Then we can consider the sub-module $M\circ(N_1;N_2)$ of $M\circ(N_1\oplus N_2)$ where $N_2$ appears exactly once in each summand. This construction is linear in both $M$ and $N_2$. Notice that $M\circ(N;N)$ is not isomorphic to $M\circ N$, because in $M\circ(N_1;N_2)$ we always remember the position of the copy of $N_2$. However, we have an obvious forgetful map $M\circ(N;N)\to M\circ N$.

\begin{definition}
	Let $f:M_1\to M_2$ and $g:N_1\to N_2$ be two morphisms of $\S$-modules. The \emph{infinitesimal composite of morphisms}\index{Infinitesimal!composite!of morphisms} is defined as
	\[
	f\circ'g:M_1\circ N_1\longrightarrow M_2\circ(N_1;N_2)
	\]
	by
	\[
	(f\circ'g)(\mu;\nu_1,\ldots,\nu_n) = \sum_{i=1}^n(f(\mu);\nu_1,\ldots,g(\nu_i),\ldots,\nu_n)
	\]
	for $(\mu;\nu_1,\ldots,\nu_n)\in(M_1\circ N_1)(n)$.
\end{definition}

\begin{example}
	Let $M$ and $N$ be two $\S$-modules, then the differential of $M\circ N$ can be seen as
	\[
	d_{M\circ N} = d_M\circ1_N + 1_M\circ'd_N
	\]
	after applying the canonical forgetful map $M\circ(N;N)\to M\circ N$ to the last term.
\end{example}

\begin{definition}
	Let $M$ and $N$ be two $\S$-modules. The \emph{infinitesimal composite}\index{Infinitesimal!composite!of $\S$-modules} of $M$ and $N$ is the $\S$-module
	\[
	M\infcomp N\coloneqq M\circ(I;N)\ .
	\]
	If $f:M_1\to M_2$ and $g:N_1\to N_2$ are two morphisms of $\S$-modules, we define
	\[
	f\infcomp g:M_1\infcomp N_1\longrightarrow M_2\infcomp N_2
	\]
	is defined by
	\[
	(f\infcomp g)(\mu;\id,\ldots,\nu,\ldots,\id) = (f(\mu);\id,\ldots,g(\nu),\ldots,\id)
	\]
	for $\mu\in M_1$ and $\nu\in N_1$.
\end{definition}

Notice that we always have
\[
M\infcomp M\cong \T(M)^{(2)}.
\]

\begin{lemma}
	The infinitesimal composite product $\infcomp$ is linear in both variables.
\end{lemma}

If $\P$ is an operad, we define the \emph{infinitesimal composition map}\index{Infinitesimal!composition map}
\[
\gamma_{(1)}:\P\infcomp\P\longrightarrow\P
\]
as the composite
\[
\P\infcomp\P = \P\circ(I;\P)\longrightarrow\P\circ(I\oplus\P)\xrightarrow{1_\P\circ(\eta_\P+1_\P)}\P\circ\P\xrightarrow{\gamma_\P}\P\ .
\]
It is simply the composition map $\gamma_\P$ restricted to the composition of two operations in $\P$.

\medskip

Dually, if $\C$ is a cooperad we define its \emph{infinitesimal decomposition map}\index{Infinitesimal!decomposition map}
\[
\Delta_{(1)}:\C\longrightarrow\C\infcomp\C
\]
by
\[
\C\xrightarrow{\Delta_\C}\C\circ\C\xrightarrow{1_\C\circ'1_\C}\C\circ(\C;\C)\xrightarrow{1_\C\circ(\varepsilon_\C;1_\C)}\C\circ(I;\C)=\C\infcomp\C\ .
\]
It can be seen as the decomposition of elements of $\C$ into two parts.

\subsection{Convolution operads}

Given a cooperad $\C$ and an operad $\P$, one can give a natural operad structure to $\hom(\C,\P)$. This was first proven in \cite[p. 3]{bm03} and is detailed in \cite[Sect. 6.4.1]{LodayVallette}.

\medskip

The composition map on $\hom(\C,\P)$ is defined as follows. Let $f\in\hom(\C(k),\P(k))$, and let $g_i\in\hom(\C(n_i),\P(n_i))$ for $1\le i\le k$. Then $\gamma_{\hom(\C,\P)}(f;g_1,\ldots,g_k)$ is given by the composite
\begin{align*}
	\C(n)&\xrightarrow{\Delta_\C}(\C\bar{\circ}\C)(n)\xrightarrow{\mathrm{proj.}}\C(k)\otimes\C(n_1)\otimes\cdots\otimes\C(n_k)\otimes\k[\S_n]\\
	&\xrightarrow{f\otimes g_1\otimes\cdots\otimes g_k\otimes1_{\k[\S_n]}}\P(k)\otimes\P(n_1)\otimes\cdots\otimes\P(n_k)\otimes\k[\S_n]\\
	&\xrightarrow{\mathrm{incl.}}(\P\circ\P)(n)\xrightarrow{\gamma_\P}\P(n)\ ,
\end{align*}
where $n=n_1+\cdots+n_k$. The projection map sends an element of $(\C\bar{\circ}\C)(n)$ to all its components that live in $\C(k)\otimes\C(n_1)\otimes\cdots\otimes\C(n_k)$, with possibly a permutation of the arguments encoded by an $(n_1,\ldots,n_k)$-shuffle. Moreover, the composite
\[
\C\xrightarrow{\varepsilon_\C}I\xrightarrow{\eta_\P}\P
\]
defines. the unit map $\eta_{\hom(\C,\P)}$.

\begin{theorem}
	Let $\C$ be a cooperad, and let $\P$ be an operad. With the maps described above, the $\S$-module $\hom(\C,\P)$ becomes an operad, which is called the \emph{convolution operad}\index{Convolution!operad} of $\C$ and $\P$.
\end{theorem}

\begin{proof}
	By inspection and left to the reader.
\end{proof}

Notice that if $\P$ is reduced, then so is $\hom(\C,\P)$.

\medskip

Now suppose that $C$ is a conilpotent $\C$-coalgebra, and that $A$ is a $\P$-algebra. Then we can endow the chain complex $\hom(C,A)$ with a $\hom(\C,\P)$-algebra structure as follows. Let $\phi\in\hom(\C,\P)(n)$, and let $f_1,\ldots,f_n\in\hom(C,A)$, then $\gamma_{\hom(C,A)}(\phi\otimes_{\S_n}(f_1\otimes\cdots\otimes f_n))$ is given by the composite
\begin{align*}
	C\xrightarrow{\Delta_C^n}\left(\C(n)\otimes C^{\otimes n}\right)^{\S_n}&\xhookrightarrow{\quad}\C(n)\otimes C^{\otimes n}\xrightarrow{\sum_{\sigma\in\S_n}(-1)^\epsilon\phi\otimes f_{\sigma(1)}\otimes\cdots\otimes f_{\sigma(n)}}\P(n)\otimes A^{\otimes n}\\
	&\xrightarrow{\quad}\mathrel{\mkern-14mu}\rightarrow\P(n)\otimes_{\S_n} A^{\otimes n}\xrightarrow{\gamma_A}A\ ,
\end{align*}
where $\epsilon$ is the Koszul sign appearing from the permutation of the $f_i$.

\begin{theorem}\label{thm:hom of algebras is algebras over hom}
	Let $\C$ be a cooperad, let $\P$ be an operad, let $C$ be a conilpotent $\C$-coalgebra, and let $A$ be a $\P$-algebra. With the structure described above, $\hom(C,A)$ is a $\hom(\C,\P)$-algebra, which we call the \emph{convolution algebra}\index{Convolution!algebra} of $C$ and $A$.
\end{theorem}

\begin{proof}
	By inspection and left to the reader.
\end{proof}

\subsection{Twisting morphisms}\label{subsect:operadic twisting morphisms}

A central notion for the theory of operads is that of twisting morphisms. More details on the subject can be found in \cite[Sect. 6.4]{LodayVallette}.

\medskip

To any reduced operad $\P$ it is possible to associate a (pre-)Lie algebra structure on the chain complex $\prod_{n\ge0}\P(n)$, see \cite[Sect. 5.4.3]{LodayVallette}. If the operad is a convolution operad $\hom(\C,\P)$ with $\C$ and $\P$ reduced\footnote{So that $\hom(\C,\P)$ is also reduced.}, then the pre-Lie product is simply given by the composite
\[
f\star g = \big(\C\xrightarrow{\Delta_{(1)}}\C\infcomp\C\xrightarrow{f\infcomp g}\P\infcomp\P\xrightarrow{\gamma_{(1)}}\P\big)
\]
for $f,g\in\prod_{n\ge0}\hom(\C,\P)(n)$, and it preserves the subspace of $\S$-equivariant maps, i.e. (possibly graded) morphisms of $\S$-modules.

\begin{definition}
	Let $\C$ be a reduced cooperad, and let $\P$ be a reduced operad. An \emph{(operadic) twisting morphism}\index{Twisting morphism!operadic} $\alpha:\C\to\P$ from $\C$ to $\P$ is a morphism of $\S$-modules of degree $-1$ satisfying the Maurer--Cartan equation
	\[
	\partial(\alpha) + \alpha\star\alpha = 0\ .
	\]
	We denote the set of twisting morphisms from $\C$ to $\P$ by $\Tw(\C,\P)$.
\end{definition}

Given a morphism of $\S$-modules $\alpha:\C\to\P$ of degree $-1$, we can consider the unique derivation $d_\alpha^r$ on $\C\circ\P$ extending 
\[
\C\xrightarrow{\Delta_{(1)}}\C\infcomp\C\xrightarrow{1_\C\infcomp\alpha}\C\infcomp\P\longrightarrow\C\circ\P\ ,
\]
cf. \cite[Prop. 6.3.9]{LodayVallette}. Explicitly, it is given by $d_\alpha^r = (1_\C\circ\gamma_\P)(((1_\C\circ\alpha)\Delta_{(1)})\circ1_\P)$. Define $d_\alpha\coloneqq d_{\C\circ\P}+d_\alpha^r$ on $\C\circ\P$. Dually, on $\P\circ\C$ we consider the unique derivation $d_\alpha^\ell$ extending
\[
\C\xrightarrow{\Delta_\C}\C\circ\C\xrightarrow{\alpha\circ1_\C}\P\circ\C
\]
and define $d_\alpha\coloneqq d_{\P\circ\C}+d_\alpha^\ell$.

\begin{lemma}
	On $\C\circ\P$ we have $d_\alpha^2 = d^r_{\partial(\alpha)+\alpha\star\alpha}$, and on $\P\circ\C$ we have $d_\alpha^2 = d^\ell_{\partial(\alpha)+\alpha\star\alpha}$. Therefore, in both cases we have that $\alpha\in\Tw(\C,\P)$ if, and only if $d_\alpha^2=0$.
\end{lemma}

Thus, given a twisting morphism $\alpha\in\Tw(\C,\P)$, one can define two chain complexes
\[
\C\circ_\alpha\P\coloneqq\big(\C\circ\P,d_\alpha\big)\qquad\text{and}\qquad\P\circ_\alpha\C\coloneqq\big(\P\circ\C,d_\alpha\big)\ ,
\]
called the \emph{left} and \emph{right twisted composite product}\index{Twisted composite products} of $\C$ and $\P$ respectively. These constructions are functorial in the sense that if we have a commutative square
\begin{center}
	\begin{tikzpicture}
		\node (a) at (0,1.5){$\C$};
		\node (b) at (2,1.5){$\C'$};
		\node (c) at (0,0){$\P$};
		\node (d) at (2,0){$\P'$};
		
		\draw[->] (a) -- node[above]{$f$} (b);
		\draw[->] (a) -- node[left]{$\alpha$} (c);
		\draw[->] (b) -- node[right]{$\alpha'$} (d);
		\draw[->] (c) -- node[above]{$g$} (d);
	\end{tikzpicture}
\end{center}
where $\alpha,\alpha'$ are twisting morphisms and $f,g$ are a morphism of cooperads and a morphism of operad respectively, then
\[
f\circ g:\C\circ_\alpha\P\longrightarrow\C'\circ_{\alpha'}\P'
\]
is a chain map, and similarly for $g\circ f$.

\begin{theorem}[Comparison lemma for twisted composite products]
	Suppose that we are in the situation above.
	\begin{enumerate}
		\item If two among $f$, $g$ and $f\circ g:\C\circ_\alpha\P\longrightarrow\C'\circ_\alpha\P'$ are quasi-isomorphisms, then so is the third.
		\item Dually, if two among $f$, $g$ and $g\circ f:\P\circ_\alpha\C\longrightarrow\P'\circ_\alpha\C'$ are quasi-isomorphisms, then so is the third.
	\end{enumerate}
\end{theorem}

The proof of this theorem uses a spectral sequence argument. It is found in \cite[Sect. 6.7]{LodayVallette} and it is based on \cite[Sect. 2]{fre04}.

\begin{remark}
	This result also holds for (co)operads that are not reduced, but in that case it is crucial to assume that everything is connected weight graded.
\end{remark}

\subsection{Operadic bar and cobar construction}\label{subsection:operadic bar and cobar}

It is natural ask whether the functors $\Tw(\C,-)$ and $\Tw(-,\P)$ are representable or not, and whether one can find canonical resolutions of operads and cooperads in general or not. Both questions are answered in the positive by the two constructions we present here, following \cite[Sect. 6.5]{LodayVallette}. The argument will be completed in \cref{sect:koszul twisting morphisms,sect:bar-cobar resolution of operads}.

\medskip

We start by defining the \emph{operadic bar construction}\index{Bar construction!operadic}, which is a functor
\[
\Bar:\augOp\longrightarrow\coaugCoop\ .
\]
Let $\P$ be an augmented operad. The functor $\Bar$ associates to $\P$ the quasi-free cooperad
\[
\Bar\P\coloneqq\big(\T^c(s\overline{\P}),d\coloneqq d_1+d_2\big)\ ,
\]
where $d_1$ is the natural differential induced on $\T^c(s\overline{\P})$ by the differential $d_\P$ of $\P$, and $d_2$ is defined as the unique coderivation extending the composite
\begin{align*}
	\T^c(s\overline{\P})&\xrightarrow{\mathrm{proj.}}\T^c(s\overline{\P})^{(2)}\cong(\k s\otimes\overline{\P})\infcomp(\k s\otimes\overline{\P})\\
	&\cong(\k s\otimes\k s)\otimes(\overline{\P}\infcomp\overline{\P})\xrightarrow{\susp_2\otimes\gamma_{(1)}}s\overline{\P}\ .
\end{align*}
Here we see $\susp_2$ as the operation sending $s^2$ to $s$. It is straightforward to prove that $d_2$ squares to zero and anticommutes with $d_1$, so that also $d_1+d_2$ squares to zero.

\begin{proposition}
	Suppose the characteristic of the base field $\k$ is $0$. The operadic bar construction $\Bar$ preserves quasi-isomorphisms.
\end{proposition}

Dually, the \emph{operadic cobar construction}\index{Cobar construction!operadic} is the functor
\[
\Cobar:\coaugCoop\longrightarrow\augOp
\]
defined as
\[
\Cobar\C\coloneqq\big(\T(s^{-1}\overline{\C}),d\coloneqq d_1+d_2\big)\ ,
\]
where again $d_1$ is induced by $d_\C$, and $d_2$ is the unique derivation extending the composite
\begin{align*}
	s^{-1}\overline{\C}&\xrightarrow{\Delta_s\otimes\Delta_{(1)}}(\k s^{-1}\otimes\k s^{-1})\otimes(\overline{\C}\infcomp\overline{\C})\\
	&\cong(\k s^{-1}\otimes\overline{\C})\infcomp(\k s^{-1}\otimes\overline{\C})\cong\T(s^{-1}\overline{\C})^{(2)}\xrightarrow{\mathrm{incl.}}\T(s^{-1}\overline{\C})\ .
\end{align*}
Here, $\Delta_{s^{-1}}$ is the dual of $\susp_2$, and is given by sending $s^{-1}$ to $-s^{-1}\otimes s^{-1}$. Once again, one easily proves that $d_2$ squares to zero and anticommutes with $d_1$, so that $d_1+d_2$ also squares to zero. We also have a compatibility with quasi-isomorphisms, but with rather more stringent assumptions.

\begin{proposition}
	The operadic cobar construction $\Cobar$ preserves quasi-isomorphisms between cooperads that are non-negatively graded, $I$ in degree $0$, and $0$ in degree $1$.
\end{proposition}

The operadic bar and cobar constructions form a pair of adjoint functors.

\begin{theorem}\label{thm:Rosetta stone operads}
	Let $\P$ be an augmented operad, and let $\C$ be a coaugmented cooperad. There are natural isomorphisms
	\[
	\hom_\op(\Cobar\C,\P)\cong\Tw(\C,\P)\cong\hom_\coop(\C,\Bar\P)\ .
	\]
	In particular, the functors $\Bar$ and $\Cobar$ form an adjoint pair.
\end{theorem}

The unit and the counit of the adjunction give rise to canonical twisting morphisms. Namely, let $\C$ be a cooperad, then the unit of the adjunction is a map
\[
\C\longrightarrow\Bar\Cobar\C
\]
of cooperads and as such it corresponds to a canonical twisting morphism
\[
\iota:\C\longrightarrow\Cobar\C\ .
\]
It is given by the composite
\[
\iota = \left(\C\xrightarrow{s^{-1}}s^{-1}\C\xrightarrow{\mathrm{incl.}}\Cobar\C\right).
\]
Dually, for $\P$ an operad, the counit of the adjunction is a map
\[
\Cobar\Bar\P\longrightarrow\P
\]
of operads, and gives us a canonical twisting morphism
\[
\pi:\Bar\P\longrightarrow\P
\]
given by the composite
\[
\pi = \left(\Bar\P\xrightarrow{\mathrm{proj.}}s\P\xrightarrow{s^{-1}}\P\right).
\]
These two twisting morphisms are universal, in the sense that every twisting morphism splits through them.

\begin{theorem}\label{thm:universal twisting morphisms}
	Let $\alpha:\C\to\P$ be a twisting morphism. There exist a unique morphism of cooperads $f_\alpha:\C\to\Bar\P$ and a unique morphism of operads $g_\alpha:\Cobar\C\to\P$ such that the diagram
	\begin{center}
		\begin{tikzpicture}
			\node (a) at (0,0){$\C$};
			\node (b) at (3,0){$\P$};
			\node (c) at (1.5,1.2){$\Cobar\C$};
			\node (d) at (1.5,-1.2){$\Bar\P$};
			
			\draw[->] (a) -- node[above]{$\alpha$} (b);
			\draw[->] (a) -- node[above left]{$\iota$} (c);
			\draw[->,dashed] (c) -- node[above right]{$g_\alpha$} (b);
			\draw[->,dashed] (a) -- node[below left]{$f_\alpha$} (d);
			\draw[->] (d) -- node[below right]{$\pi$} (b);
		\end{tikzpicture}
	\end{center}
	commutes.
\end{theorem}

\begin{notation}\label{notation:iota and pi}
	From now on, we reserve the Greek letters $\iota$ and $\pi$ for the canonical twisting morphisms we just described.
\end{notation}

\subsection{Koszul twisting morphisms}\label{sect:koszul twisting morphisms}

Certain twisting morphisms behave especially well with respect to the homotopy theory of (co)o\-perads and (co)algebras over them. The archetype for such twisting morphisms are the universal twisting morphisms $\iota$ and $\pi$.

\begin{lemma}\label{lemma:universal twisting morphisms are Koszul}
	Let $\C$ be a coaugmented cooperad and let $\P$ be an augmented operad. The chain complexes
	\[
	\C\circ_\iota\Cobar\C\ ,\quad\Cobar\C\circ\iota\C\ ,\quad\P\circ_\pi\Bar\P\quad\text{and}\quad\Bar\P\circ_\pi\P
	\]
	are acyclic.
\end{lemma}

\begin{proof}
	This is \cite[Lemma 6.5.9]{LodayVallette}.
\end{proof}

Using this fact, one can prove the following important fact \cite[Thm. 6.6.1]{LodayVallette}.

\begin{theorem}\label{thm:fundamental thm of operadic twisting morphisms}
	Let $\C$ be a connected weight graded cooperad, let $\P$ be a connected weight graded operad, and let $\alpha:\C\to\P$ be a twisting morphism. The following are equivalent.
	\begin{enumerate}
		\item The right twisted composite product $\C\circ_\alpha\P$ is acyclic.
		\item The left twisted composite product $\P\circ_\alpha\C$ is acyclic.
		\item The morphism of cooperads $f_\alpha:\C\to\Bar\P$ of \cref{thm:universal twisting morphisms} is a quasi-isomorphism.
		\item The morphism of operads $g_\alpha:\Bar\P\to\C$ of \cref{thm:universal twisting morphisms} is a quasi-isomorphism.
	\end{enumerate}
\end{theorem}

\begin{definition}
	Let $\C$ be a connected weight graded cooperad and let $\P$ be a connected weight graded operad. A twisting morphism $\alpha:\C\to\P$ is \emph{Koszul}\index{Koszul!twisting morphism} if any of the equivalent conditions of \cref{thm:fundamental thm of operadic twisting morphisms} is satisfied.
\end{definition}

\begin{example}
	By \cref{lemma:universal twisting morphisms are Koszul}, the universal twisting morphisms $\pi$ and $\iota$ are Koszul.
\end{example}

\subsection{Bar-cobar resolution}\label{sect:bar-cobar resolution of operads}

Finally, we can prove that the bar-cobar adjunction gives a canonical resolution of operads and cooperads, following \cite[Thm. 6.6.3]{LodayVallette}.

\begin{theorem}
	Let $\C$ be a cooperad. The unit
	\[
	\C\longrightarrow\Bar\Cobar\C
	\]
	of the bar-cobar adjunction is a quasi-isomorphism of cooperads. Dually, let $\P$ be an operad. The counit
	\[
	\Cobar\Bar\P\longrightarrow\P
	\]
	of the bar-cobar adjunction is a quasi-isomorphism of operads.
\end{theorem}

\begin{proof}
	This is a direct corollary of \cref{thm:fundamental thm of operadic twisting morphisms} and \cref{lemma:universal twisting morphisms are Koszul} in the connected weight graded case. The more general case can be found in \cite{fre04}.
\end{proof}

\begin{remark}
	In particular, $\Cobar\Bar\P$ provides a functorial cofibrant resolution of $\P$, see \cref{cor:cobar-bar is cofibrant resolution}. One should notice that if $\C$ is a cooperad, then in general $\Cobar\C$ is not cofibrant. However, this is the case if one has a connected weight graded cooperad, and in particular $\Cobar\C$ is always cofibrant if $\C$ is a reduced cooperad.
\end{remark}

\section{Minimal models and Koszul duality}\label{section:minimal models and Koszul duality}

The bar-cobar resolution of operads is very useful, but almost always really big. Fortunately, in some cases it is possible to find smaller resolutions for operads, such as minimal models. One way to do that is Koszul duality, which recall briefly in this section. It is a theory originally developed by V. Ginzburg and M. Kapranov \cite{gk94} for operads. Koszul duality for associative and Lie algebras existed prior to it, and other versions for related concepts have been developed since, for example in \cite{val07} for props.

\subsection{Minimal models}

Let $\P$ be an operad. A \emph{model} for $\P$ is an operad $\Q$ together with a surjective quasi-isomorphism of operads $\Q\to\P$.

\medskip

An operad $\Q$ is \emph{minimal} if it is \emph{quasi-free}\index{Quasi-free operad}, i.e. of the form $\Q=(\T(M),d)$ for some graded $\S$-module $M$ (that is, for each $n\ge0$ the object $M(n)$ is just a graded vector space, not a chain complex), and such that the differential satisfies the following two conditions.
\begin{enumerate}
	\item The differential $d$ of $\Q$ is \emph{decomposable}, that is $d(E)\subseteq\T(E)^{(\ge2)}$.
	\item The graded $\S$-module $E$ admits a decomposition
	\[
	E = \bigoplus_{k\ge1}E^{(k)}
	\]
	such that
	\[
	d(E^{(k+1)})\subseteq\T\left(\bigoplus_{i=1}^kE^{(i)}\right)
	\]
	for any $k\ge0$.
\end{enumerate}

\begin{definition}
	A \emph{minimal model}\index{Minimal!models for operads} for an operad $\P$ is a model $\Q\to\P$ such that $\Q$ is a minimal operad.
\end{definition}

\begin{theorem}[{\cite[Prop. 3.7]{dcv13}}]
	If an operad $\P$ admits a minimal model, then the model is unique up to (non-unique) isomorphism.
\end{theorem}

\subsection{Koszul duality}\label{subsect:Koszul duality}

For the rest of this section, we work with quadratic (co)operads. We refer the reader to \cite[Sect. 7.2 and 7.4]{LodayVallette} for more details.

\begin{definition}
	Let $\P=\P(E,R)$ be a quadratic operad. The \emph{Koszul dual cooperad}\index{Koszul!dual (co)operad} of $\P$ is the quadratic cooperad
	\[
	\P^{\antishriek}\coloneqq\C(sE,s^2R)\ .
	\]
	The \emph{Koszul dual operad} op $\P$ is the operad
	\[
	\P^{\shriek}\coloneqq(\susp^c\otimes\P^{\antishriek})^\vee.
	\]
\end{definition}

The Koszul dual cooperad of an operad is of greater theoretical importance than the Koszul dual operad. However, the latter has the advantage of being often easy to describe explicitly, as the following result explains.

\begin{proposition}\label{prop:formula for Koszul dual operad}
	Let $\P=\P(E,R)$ be a quadratic operad generated by an $\S$-module $E$ which is reduced and finite dimensional in every arity. Then the Koszul dual operad of $\P$ is quadratic with presentation
	\[
	\P^{\shriek} = \P(s^{-1}\susp^{-1}\otimes E^\vee,R^{\perp})\ ,
	\]
	where $R^\perp$ is the subspace of $\T(s^{-1}\susp^{-1}\otimes E^\vee)^{(2)}$ obtained by taking the subspace orthogonal to $s^2R$ in $\T(sE)^{(2)}$ and desuspending its elements in the obvious way. Moreover, we have
	\[
	(\P^{\shriek})^{\shriek} = \P\ .
	\]
\end{proposition}

\begin{proof}
	This is \cite[Prop. 7.2.1 and 7.2.2]{LodayVallette}.
\end{proof}

Let $(E,R)$ be a quadratic data. Then both $\P(E,R)^{(1)}$ and $\C(E,R)^{(1)}$ are given by $E$, and thus we can define the map
\[
\kappa:\C(sE,s^2R)\xrightarrow{\mathrm{proj}}sE\xrightarrow{s^{-1}}E\xrightarrow{\mathrm{incl.}}\P(E,R)\ .
\]

\begin{lemma}
	The map $\kappa$ described above is a twisting morphism.
\end{lemma}

\begin{proof}
	This is \cite[Lemma 7.4.1]{LodayVallette}.
\end{proof}

\begin{notation}
	From now on, we reserve the Greek letter $\kappa$ as notation for the canonical twisting morphism defined above.
\end{notation}

\begin{definition}
	A quadratic operad $\P$ is \emph{Koszul}\index{Koszul!operad} if the canonical twisting morphism
	\[
	\kappa:\P^{\antishriek}\longrightarrow\P
	\]
	is Koszul.
\end{definition}

Consider the canonical inclusion
\[
i\coloneqq f_\kappa:\P^{\antishriek}\longrightarrow\Bar\P
\]
and the canonical projection
\[
p\coloneqq g_\kappa:\Cobar\P^{\antishriek}\longrightarrow\P
\]
given by applying \cref{thm:universal twisting morphisms} to $\kappa$. We have the following version of \cref{thm:fundamental thm of operadic twisting morphisms} for $\kappa$. Notice that we don't need a connected weight grading for it to hold. In particular, it also works without restrictions even if the quadratic data is not reduced.

\begin{theorem}
	Let $\P=\P(E,R)$ a quadratic operad. The following are equivalent.
	\begin{enumerate}
		\item The \emph{right Koszul complex} $\P^{\antishriek}\circ_\kappa\P$ is acyclic.
		\item The \emph{left Koszul complex} $\P\circ_\kappa\P^{\antishriek}$ is acyclic.
		\item The canonical inclusion $i:\P^{\antishriek}\to\Bar\P$ is a quasi-isomorphism.
		\item The canonical projection $p:\Cobar\P^{\antishriek}\to\P$ is a quasi-isomorphism.
	\end{enumerate}
\end{theorem}

\begin{proof}
	This is \cite[Thm. 7.4.2]{LodayVallette}.
\end{proof}

\begin{corollary}
	Suppose $\P$ is a Koszul operad. Then the operad $\Cobar\P^{\antishriek}$ is the minimal model of $\P$. It will often be denoted by $\P_\infty$.
\end{corollary}

Therefore, we have a minimal model for any Koszul operad. Checking that an operad is Koszul looks like a difficult problem, but one has various methods of doing it which do not involve checking if the right or left Koszul complex are acyclic. This is the topic of \cite[Ch. 8]{LodayVallette}. It is outside the scope of the present work, and we will not mention this problematic again.

\subsection{Homotopy algebras and homotopy morphisms}\label{subsection:homotopy algebras and homotopy morphisms}

Given an arbitrary operad $\P$, often the category of $\P$-algebras does not have very good homotopical properties. However, whenever $\P$ is Koszul there is a notion of $\P$-algebra up to homotopy which is much better behaved. The material presented here is contained in \cite[Ch. 10]{LodayVallette}.

\medskip

For this section, we fix a Koszul operad $\P$.

\begin{definition}
	A \emph{homotopy $\P$-algebra}\index{Homotopy!$\P$-algebras}\index{$\P_\infty$-algebra} is an algebra over the operad $\P_\infty = \Cobar\P^{\antishriek}$. We also use the name \emph{$\P_\infty$-algebras} for homotopy $\P$-algebras.
\end{definition}

Notice that every $\P$ algebra is also a $\P_\infty$-algebra via restriction of structure along the projection $\P_\infty\to\P$. The following point of view on $\P_\infty$-algebras is often useful.

\begin{proposition}
	A structure of $\P_\infty$-algebra on a chain complex $A$ is equivalent to a twisting morphism $\varphi_A\in\Tw(\P^{\antishriek},\End_A)$.
\end{proposition}

\begin{proof}
	This is an immediate consequence of \cref{prop:P-alg is map P->End,thm:Rosetta stone operads}.
\end{proof}

\begin{proposition}
	A $\P_\infty$-algebra $A$ is a $\P$-algebra if, and only if the twisting morphism $\varphi_A$ is concentrated in weight $1$.
\end{proposition}

\begin{proof}
	This is \cite[Prop. 10.1.4]{LodayVallette}.
\end{proof}

\begin{proposition}\label{prop:Poo-alg is coderivation of Pi-cog}
	Let $A$ be a graded vector space. A structure of $\P_\infty$-algebra on $A$ is equivalent to a square-zero coderivation on the (non-differential) graded cofree $\P^{\antishriek}$-coalgebra $\P^{\antishriek}(A)$.
\end{proposition}

\begin{proof}
	This is \cite[Prop. 10.1.11]{LodayVallette}
\end{proof}

One can of course consider the category of $\P_\infty$-algebras with the morphisms of $\P_\infty$-algebras between them. However, there is a notion of morphism of $\P_\infty$-algebras "relaxed up to homotopy" which has better homotopical properties.

\begin{definition}
	Let $A$ and $A'$ be two $\P_\infty$-algebras. An \emph{$\infty$-morphism}\index{$\infty$-morphism} $\Psi$ of $\P_\infty$-algebras from $A$ to $A'$, denoted by $\Psi:A\rightsquigarrow B$, is a morphism of $\P^{\antishriek}$-coalgebras
	\[
	\Psi:\P^{\antishriek}(A)\longrightarrow\P^{\antishriek}(A')
	\]
	between the $\P^{\antishriek}$-coalgebras associated to $A$ and $A'$ through \cref{prop:Poo-alg is coderivation of Pi-cog}. The composition of $\infty$-morphisms is induced by the usual composition of morphisms of $\P^{\antishriek}$-coalgebras. The category of $\P_\infty$-algebras with $\infty$-morphisms as morphisms is denoted by $\ooPooAlg$.
\end{definition}

Since an $\infty$-morphism $\Psi$ of $\P_\infty$-algebras is a morphism between cofree coalgebras, it is completely determined by its projection
\[
\P^{\antishriek}(A)\longrightarrow\P^{\antishriek}(A')\xrightarrow{\mathrm{proj.}}A',
\]
which we will again denote by $\Psi$, abusing notation. Therefore, such an $\infty$-morphism is equivalent to a collection of maps
\[
\psi_n:\P^{\antishriek}(n)\otimes_{\S_n}A^{\otimes n}\longrightarrow A'
\]
satisfying certain relations.

\medskip

From now on, we suppose that the quadratic data $(E,R)$ describing $\P$ is such that $E(0)=E(1)=0$. Then the cooperad $\P^{\antishriek}$ is reduced, and $\psi_1$ is a chain map
\[
\psi_1:A\longrightarrow A'.
\]
There is an analogous point of view for general Koszul operads treated in \cite[p.373]{LodayVallette}, where the map $\psi_1$ above corresponds to $\psi_{(0)}$, but we will not use it.

\medskip

Notice that every morphism of $\P_\infty$-algebras $\psi:A\to A'$ is an $\infty$-morphism by $\Psi\coloneqq\P^{\antishriek}(\psi)$. For the converse direction, we have the following.

\begin{proposition}
	An $\infty$-morphism $\Psi:A\rightsquigarrow A'$ of $\P_\infty$-algebras is a strict morphism of $\P_\infty$-algebras if, and only if $\psi_n=0$ for all $n\ge2$.
\end{proposition}

\begin{proof}
	This is \cite[Prop. 10.2.5]{LodayVallette}.
\end{proof}

Finally, we introduce the important notion of $\infty$-quasi-isomorphism.

\begin{definition}
	An $\infty$-morphism of $\P_\infty$-algebras $\Psi:A\rightsquigarrow A'$ is an \emph{$\infty$-quasi-isomorphism}\index{$\infty$-quasi-isomorphism} if $\psi_1:A\to A'$ is a quasi-isomorphism.
\end{definition}

The notions of $\infty$-morphisms in general and $\infty$-quasi-isomorphisms in particular will be studied some more later, in \cref{sect:homotopy theory of homotopy algebras}, as well as in \cref{chapter:oo-morphisms relative to a twisting morphism}.

\subsection{The homotopy transfer theorem}\label{subsect:HTT}

The homotopy transfer theorem is an important result on the homotopical behavior of algebras over (cofibrant) operads. It tells us that if we have a homotopy retraction between two chain complexes --- which in particular implies being homotopically the same --- and an algebra structure over one of them, then we can produce in a very natural way an algebra structure on the other one which has the same homotopical information.

\medskip

The first notions we need is those of homotopy retractions and contractions of chain complexes.

\begin{definition}
	Let $V,W$ be two chain complexes. A \emph{homotopy retraction}\index{Homotopy!retraction} of $V$ into $W$ is three maps
	\begin{center}
		\begin{tikzpicture}
			\node (a) at (0,0){$V$};
			\node (b) at (2,0){$W$};
			
			\draw[->] (a)++(.3,.1)--node[above]{\mbox{\tiny{$p$}}}+(1.4,0);
			\draw[<-,yshift=-1mm] (a)++(.3,-.1)--node[below]{\mbox{\tiny{$i$}}}+(1.4,0);
			\draw[->] (a) to [out=-150,in=150,looseness=4] node[left]{\mbox{\tiny{$h$}}} (a);
		\end{tikzpicture}
	\end{center}
	such that $pi=1_W$ and $1-ip=d_Vh+hd_V$. It is a \emph{contraction}\index{Contraction} if the \emph{side conditions}\index{Side conditions}
	\[
	h^2=0\ ,\qquad ph=0\ ,\qquad\text{and}\qquad hi=0
	\]
	are satisfied.
\end{definition}

\begin{remark}
	It is possible to construct a sensible category whose objects are the homotopy retractions, cf. \cref{subsect:complete contractions and htt}.
\end{remark}

Let $\P$ be a Koszul operad, let $A$ be a $\P_\infty$-algebra, and suppose that we have a contraction of chain complexes
\begin{center}
	\begin{tikzpicture}
	\node (a) at (0,0){$A$};
	\node (b) at (2,0){$B$};
	
	\draw[->] (a)++(.3,.1)--node[above]{\mbox{\tiny{$p$}}}+(1.4,0);
	\draw[<-,yshift=-1mm] (a)++(.3,-.1)--node[below]{\mbox{\tiny{$i$}}}+(1.4,0);
	\draw[->] (a) to [out=-150,in=150,looseness=4] node[left]{\mbox{\tiny{$h$}}} (a);
	\end{tikzpicture}
\end{center}
The \emph{transfer problem}\index{Transfer problem} asks under what conditions it is possible to produce a $\P_\infty$-algebra structure on $B$ such that $i$ extends to an $\infty$-morphism of $\P_\infty$-algebras. It turns out that this is always true. This result is known as the \emph{homotopy transfer theorem}, and has been developed throughout the years by many authors. It has been known for a long time for $\A_\infty$-algebras --- see \cite[Thm. 1]{kad80}, and e.g. \cite[Sect. 6.4]{ks00} for an explicit formula in terms of trees --- and $\L_\infty$-algebras. It was proven for $\P_\infty$-algebras --- and more generally for algebras over $\Cobar\C$ for $\C$ a reduced cooperad\footnote{This slightly more general case will be explained later, in \cref{subsect:generalized HTT}.} --- in \cite[Thm. 1.5]{ber14transfer} with explicit formul{\ae}, but the existence part was already known by \cite{bm03} and \cite{fre09}. Another approach using pre-Lie deformation theory was recently given in \cite[Sect. 8]{dsv16}. For a more extensive review of the literature on this result, we refer the reader to the introduction of \cite[Sect. 10.3]{LodayVallette} and to the survey \cite{sta10}, which also provides a nice historical perspective, as well as links with theoretical physics. The formul{\ae} we will present here are found in \cite[Sect. 10.3]{LodayVallette}.

\medskip

The main tool we need to give explicit formul{\ae} for the transfered structure and the induced $\infty$-morphisms is the \emph{van der Laan morphism}\index{Van der Laan morphism} \cite{vdL03} associated to the contraction. It is a morphism of cooperads
\[
\vdl_B:\Bar(\End_A)\longrightarrow\Bar(\End_B)
\]
defined as follows. An element of $\Bar(\End_A)=\T^c(s\End_A)$ is given by a rooted tree $\tau\in\RT$ with vertices labeled by elements of $s\End_A$. Using the notations of \cref{chapter:appendix on trees}, we write $\tau(f)$ for such a tree, where $f:V\to s\End_A$ is a function from the set of vertices of $\tau$ to $s\End_A$. Given such a tree $\tau(f)$ with $\tau\neq\emptyset$, we denote by $\tau^h(f)$ the element of $\End_A$ defined as follows. If $\tau=c_n$ is the $n$-corolla, then
\[
c_n(f)\coloneqq f(*)\ .
\]
Else, we write $\tau = c_k\circ(\tau_1,\ldots,\tau_k)$, where $\tau_i\in\RT_{n_i}$ are possibly empty, and define
\[
\tau^h(f)\coloneqq \gamma_{\End_A}\big(f(*)\circ(hs^{-1}\tau_1^h(f|_{V_1}),\ldots,hs^{-1}\tau_k^h(f|_{V_k})\big)\ ,
\]
where $*$ is the unique vertex of $c_k$, $V_i$ is the set of vertices of $\tau_i$, and where we formally set $\emptyset^h = h^{-1}$. Then we define
\[
\vdl_B(\tau(f))\coloneqq s(p\tau^h(f)i^{\otimes n})
\]
to obtain a map
\[
\vdl_B:\Bar(\End_A)\longrightarrow s\End_B\ .
\]
The drawing in \cite[p. 378]{LodayVallette} might prove illuminating to the reader confused by the exposition above. By universal property of the bar construction, this map extends to give a morphism
\[
\Bar(\End_A)\longrightarrow\Bar(\End_B)
\]
of cooperads in graded vector spaces, which we denote again by $\vdl_B$, abusing notation.

\begin{theorem}[{\cite[Thm. 5.2]{vdL03}}]
	The map
	\[
	\vdl_B:\Bar(\End_A)\longrightarrow\Bar(\End_B)
	\]
	described above is a morphism of cooperads.
\end{theorem}

A proof of this result can also be found in \cite[Prop. 10.3.2]{LodayVallette}. The transfered structure on $B$ can now be expressed as follows. The $\P_\infty$-algebra structure on $A$ is given by a twisting morphism $\varphi_A\in\Tw(\P^{\antishriek},\End_A)$. This is equivalent to a morphism of cooperads
\[
f_A:\P^{\antishriek}\longrightarrow\Bar(\End_A)
\]
by \cref{thm:Rosetta stone operads}. We compose this with $\vdl_B$ to get a morphism of cooperads
\[
\vdl_Bf_A:\P^{\antishriek}\longrightarrow\Bar(\End_B)\ ,
\]
which again is equivalent to a twisting morphism $\varphi_B\in\Tw(\P^{\antishriek},\End_B)$ defining a $\P_\infty$-algebra structure on $B$. Of course, this is nothing else than $\varphi_B = \vdl_Bf_A$. Explicitly, the structure is given by
\[
\varphi_B = \left(\P^{\antishriek}\xrightarrow{\Delta_{\P^{\antishriek}}^\mathrm{monadic}}\T^c\big(\overline{\P^{\antishriek}}\big)\xrightarrow{\T^c(s\varphi_A)}\Bar(\End_A)\xrightarrow{\vdl_B}\End_B\right)\ .
\]
This explicit formulation for the transferred structure first appeared in \cite{gctv12}.

\medskip

Next, one defines a map
\[
i_\infty:\P^{\antishriek}(B)\longrightarrow A
\]
by
\small
\[
i_\infty\coloneqq\left(\P^{\antishriek}\circ B\xrightarrow{\Delta_{\P^{\antishriek}}^\mathrm{monadic}\circ1_B}\T^c\big(\overline{\P^{\antishriek}}\big)\circ B\xrightarrow{\T^c(s\varphi_A)\circ1_B}\Bar(\End_A)\circ B\xrightarrow{\vdl_B^i\circ1_B}\End^B_A\circ B\longrightarrow A\right),
\]
\normalsize
where $\End^B_A$ is the $\S$-module given by $\End^B_A(n)\coloneqq\hom_\chain(B^{\otimes n},A)$, the map
\[
\vdl_B^i:\Bar(\End_A)\longrightarrow\End^B_A
\]
is given on an element of $\Bar(\End_A)(n)$ by
\[
\vdl_B^i(\tau(f))\coloneqq h\tau^h(f)i^{\otimes n},
\]
and where the last arrow is the obvious evaluation map.

\begin{theorem}[Homotopy transfer theorem]\index{Homotopy!transfer theorem}
	The map $i_\infty$ defines an $\infty$-quasi-isomorphism from $B$ with the $\P_\infty$-algebra structure defined above to $A$.
\end{theorem}

\begin{proof}
	E.g. \cite[Thm. 10.3.6]{LodayVallette}.
\end{proof}

The following fact is also important and will prove useful later on.

\begin{proposition}
	The map $p:A\to B$ can also be extended to an $\infty$-quasi isomorphism $p_\infty:A\rightsquigarrow B$. Moreover, it can be take such that the composite
	\[
	p_\infty i_\infty = 1_B
	\]
	is the identity on $B$.
\end{proposition}

An explicit formula for $p_\infty$, together with a proof of this result, can be found in \cite[Prop. 10.3.9]{LodayVallette}, \cite{ber14transfer}, and \cite[Thm. 5]{dsv16}.

\begin{remark}
	As remarked in \cite{ber14transfer}, the extension of $p$ to an $\infty$-morphism is not unique. In the present work, when talking about the extension $p_\infty$ we will always mean the $\infty$-morphism obtained using the formul{\ae} of \cite[Prop. 10.3.9]{LodayVallette}.
\end{remark}

We will now give some examples of algebras up to homotopy and their homotopy transfer theorems.

\subsection{Homotopy associative algebras}\label{subsect:A-infinity algebras}

Our first example is that of homotopy associative algebras, i.e. $\A_\infty$-algebras\index{$\A_\infty$-algebras}\index{Homotopy!associative algebras}. This kind of algebra has a long history in the literature. For example, they appear naturally when one considers the Massey products of topological spaces, cf. \cref{subsect:Massey products}. For this subsection, we work in the non-symmetric setting, so that by associative algebra we mean an algebra over the operad $\as$, cf. \cref{subsection:three graces}.

\medskip

A straightforward computation using \cref{prop:formula for Koszul dual operad} shows that $\as^{\shriek} = \as$. Then, we have
\[
\as^{\antishriek}\cong(\susp^{-1})^c\otimes(\P^{\shriek})^\vee = (\susp^{-1})^c\otimes\as^\vee.
\]
This cooperad is spanned in arity $n\ge2$ by $\susp^{-1}_n\mu_n^\vee$. The operad $\as$ is Koszul, see e.g. \cite[Thm. 9.15]{LodayVallette}, and thus, the operad $\A_\infty\coloneqq\Cobar\as^{\antishriek}$ is a minimal model for $\as$. It is freely generated by the operations
\[
m_n\coloneqq s^{-1}\susp^{-1}_n\mu_n^\vee\ .
\]
The only thing left to determine in order to understand $\A_\infty$-algebras is the differential of the operad $\A_\infty$. As defined in \cref{subsection:operadic bar and cobar}, this is given on $m_n$ by
\[
d_{\A_\infty}(m_n) = d_1(m_n) + d_2(m_n)\ ,
\]
where $d_1 = 0$ since the differential on $\as$ is trivial, and where $d_2(m_n)$ is obtained by
\begin{align*}
	m_n\longmapsto&\ -(s^{-1}\otimes s^{-1})\otimes\sum_{\substack{n_1+n_2 = n+1\\1\le j\le n_1}}(-1)^{(j-1)(1-n_2) + (n_2-1)(1-n_1)}\susp^{-1}_{n_1}\mu_{n_1}^\vee\otimes_j\susp^{-1}_{n_2}\mu_{n_2}^\vee\\
	& = \sum_{\substack{n_1+n_2 = n+1\\1\le j\le n_1}}(-1)^{(1-n_2)(j+n_1) + (1-n_1)}s^{-1}\susp^{-1}_{n_1}\mu_{n_1}^\vee\otimes_js^{-1}\susp^{-1}_{n_2}\mu_{n_2}^\vee\\
	& = \sum_{\substack{n_1+n_2 = n+1\\1\le j\le n_1}}(-1)^{n_2(n_1-j) + j + 1}m_{n_1}\otimes_jm_{n_2}\ .
\end{align*}
where the signs come from the Koszul sign rule\footnote{In particular, to compute the sign associated to the decomposition of $\susp^{-1}_n$, one knows that the part with underlying tree $c_{n_1}\circ_jc_{n_2}$ is of the form $(-1)^\epsilon\susp_{n_1}^{-1}\otimes_j\susp_{n_2}^{-1}$. To find $\epsilon$, one computes on one side $\langle\Delta_{(1)}(\susp_n^{-1}),\susp_{n_1}\otimes_j\susp_{n_2}\rangle = (-1)^\epsilon\langle\susp_{n_1}^{-1}\otimes_j\susp_{n_2}^{-1},\susp_{n_1}\otimes_j\susp_{n_2}\rangle = (-1)^{\epsilon+(n_2-1)(1-n_1)}$, and on the other side $\langle\Delta_{(1)}(\susp_n^{-1}),\susp_{n_1}\otimes_j\susp_{n_2}\rangle = \langle\susp_n^{-1},(-1)^{(j-1)(1-n_2)}\susp_n\rangle = (-1)^{(j-1)(1-n_2)}$.}. Summarizing, an $\A_\infty$-algebra is defined as follows.

\begin{definition}
	An $\A_\infty$-algebra $A$ is a chain complex together with operations
	\[
	m_n:A^{\otimes n}\longrightarrow A\qquad\text{for }n\ge2
	\]
	satisfying
	\[
	\partial(m_n) = \sum_{\substack{n_1+n_2 = n+1\\1\le j\le n_1}}(-1)^{n_2(n_1-j) + j + 1}m_{n_1}\otimes_jm_{n_2}
	\]
	for all $n\ge2$.
\end{definition}

If we write explicitly what that means for $n=2,3$, we get that
\[
\partial(m_2) = 0\ ,
\]
i.e. $d_Am_2(x,y) = m_2(d_Ax,y) + (-1)^xm_2(x,d_Ay)$, and
\[
\partial(m_3) = m_2\otimes_1m_2 - m_2\otimes_2m_2\ ,
\]
which tells us that the binary operation $m_2$ is not associative, but it is up to a homotopy given by the ternary operation $m_3$. The relations for $n\ge4$ are higher compatibilities between the operations.

\medskip

If $A$ and $B$ are $\A_\infty$-algebras, an $\infty$-morphism $\Psi:A\rightsquigarrow B$ between them\index{$\infty$-morphism!of $\A_\infty$-algebras} is a collection
\[
\psi_n:A^{\otimes n}\longrightarrow B
\]
of linear maps of degrees $|\psi_n| = n-1$ which put together form a coherent morphism of $\as^{\antishriek}$-coalgebras
\[
\Psi:\as^{\antishriek}(A)\longrightarrow\as^{\antishriek}(B)\ ,
\]
where the differential of $\as^{\antishriek}(A)$ and $\as^{\antishriek}(B)$ is the one given by \cref{prop:Poo-alg is coderivation of Pi-cog}. Explicitly, let $n\ge2$ and $a_1,\ldots,a_n\in A$, and for brevity write $a\coloneqq a_1\otimes\cdots\otimes a_n$. Then
\[
\Psi(\susp_n^{-1}\mu_n^\vee\otimes a) = \sum_{\substack{k\ge2\\m_1+\cdots+m_k=n}}\pm\,\susp_k^{-1}\mu_k^\vee\otimes \psi_{m_1}(a_1,\ldots,a_{m_1})\otimes\cdots\otimes\psi_{m_k}(a_{m_k+1},\ldots,a_n)\ ,
\]
where the sign comes both from the decomposition of $\susp_n^{-1}$, and from a Koszul sign coming from the fact that we have to switch the maps $\psi_{m_i}$, which have degree $1-m_i$, and some of the elements of $A$. The differential of $\as^{\antishriek}(A)$ is explicitly given by
\begin{align*}
	d_{\as^{\antishriek}(A)}&(\susp_n^{-1}\mu_n^\vee\otimes a) =\\ =\ &(-1)^{n-1}\susp_n^{-1}\mu_n^\vee\otimes d_{A^{\otimes n}}(a) +\\
	& + \sum_{\substack{n_1+n_2 = n+1\\1\le j\le n_1}}\pm\,\susp_{n_1}^{-1}\mu_{n_1}^\vee\otimes a_1\otimes\cdots a_{j-1}\otimes m_{n_2}^A(a_j,\ldots,a_{j+n_2})\otimes a_{j+n_2+1}\otimes\cdots\otimes a_n\ ,
\end{align*}
where again the sign comes both from the decomposition of $\susp_n^{-1}$, and from the Koszul rule. Imposing
\[
d_{\as^{\antishriek}(B)}\Psi = \Psi d_{\as^{\antishriek}(A)}
\]
and projecting on $A$, we obtain the relations
\[
\partial(\psi_n) + \sum_{\substack{k\ge2\\m_1+\cdots+m_k=n}}\pm\,m_k^B\circ(\psi_{m_1},\ldots,\psi_{m_k}) = \sum_{\substack{n_1+n_2 = n+1\\1\le j\le n_1}}\pm\,\psi_{n_1}\circ_jm_{n_2}^A\ .
\]
The homotopy transfer theorem\index{Homotopy!transfer theorem!for $\A_\infty$-algebras} is given as follows. Let
\begin{center}
	\begin{tikzpicture}
	\node (a) at (0,0){$A$};
	\node (b) at (2,0){$B$};
	
	\draw[->] (a)++(.3,.1)--node[above]{\mbox{\tiny{$p$}}}+(1.4,0);
	\draw[<-,yshift=-1mm] (a)++(.3,-.1)--node[below]{\mbox{\tiny{$i$}}}+(1.4,0);
	\draw[->] (a) to [out=-150,in=150,looseness=4] node[left]{\mbox{\tiny{$h$}}} (a);
	\end{tikzpicture}
\end{center}
be a contraction, and suppose that $A$ is an $\A_\infty$-algebra with operations $m_n^A$ for $n\ge2$. The monadic decomposition map in $\as^\vee$ is given by
\[
\Delta^\mathrm{mon}_{\as^\vee}(\mu_n^\vee) = \sum_{t\in\PT_n}t(v\mapsto\mu_{|v|}^\vee)\ ,
\]
and thus the transfered structure on $B$ is given by
\[
m_n^B = \sum_{t\in\PT_n}\pm\,pt^h(v\mapsto m_{|v|}^A)i^{\otimes n},
\]
where the signs come from the monadic decomposition of $\susp_n^{-1}$.

\subsection{Homotopy commutative algebras}\label{subsect:homotopy commutative algebras}

Another type of algebra up to homotopy which we will encounter later, e.g. in \cref{sect:cosimplicial model for Lie algebras}, are commutative algebras up to homotopy, often called $\C_\infty$-algebras\index{$\C_\infty$-algebras}\index{Homotopy!commutative algebras}. These also appeared in the literature long before the theory of Koszul duality for operads, see \cite{kad88}.

\medskip

An easy computation shows that we have
\[
\com^{\shriek} = \lie\quad\text{and}\quad\lie^{\shriek}=\com\ ,
\]
see e.g. \cite[Sect. 7.6.4]{LodayVallette}, while the methods of \cite[Ch. 8]{LodayVallette} show that both operads $\com$ and $\lie$ are Koszul. Therefore, its minimal model is given by
\[
\C_\infty\coloneqq\Cobar\com^{\antishriek}\ ,
\]
where
\[
\com^{\antishriek}\cong(\susp^{-1})^c\otimes\lie^\vee.
\]
However, the operad $\lie$ is complicated and difficult to treat, due to the Jacobi relation. Thus, one has to go another way in order to understand $\C_\infty$-algebras a bit better.

\medskip

Let $A$ be a chain complex, and let $p,q\ge1$. A \emph{non-trivial $(p,q)$-shuffle} of elements of $A$ is an element of $A^{\otimes(p+q)}$ of the form
\[
\sum_{\sigma\in\sh(p,q)}(-1)^\epsilon a_{\sigma(1)}\otimes\cdots\otimes a_{\sigma(p+q)}
\]
for $a_1,\ldots,a_{p+q}\in A$, where $\epsilon$ is the Koszul sign. In other words, it is the sum over all the ways of shuffling $a_1\otimes\cdots\otimes a_p$ and $a_{p+1}\otimes\cdots\otimes a_{p+q}$.

\begin{proposition}[{\cite[Prop. 13.1.6]{LodayVallette}}]
	A $\C_\infty$-algebra\index{$\C_\infty$-algebras} is an $\A_\infty$-algebra such that each one of the generating operations $m_n$ vanishes on all non-trivial $(p,q)$-shuffle of elements of $A$, for $p+q=n$.
\end{proposition}

The $\infty$-morphisms of $\C_\infty$-algebras are similarly characterized.

\begin{proposition}
	Let $A,B$ be two $\C_\infty$-algebras. An $\infty$-morphism $\Psi:A\rightsquigarrow B$ of $\C_\infty$-algebras\index{$\infty$-morphism!of $\C_\infty$-algebras} is an $\infty$-morphism of $\A_\infty$-algebras such that each of its components
	\[
	\psi_n:A^{\otimes n}\longrightarrow B
	\]
	vanishes over all non-trivial $(p,q)$-shuffle of elements of $A$, for $p+q=n$.
\end{proposition}

One may also wonder if the homotopy transfer theorem for $\C_\infty$-algebra structures can be recovered from the one for $\A_\infty$-algebras. The answer is positive.

\begin{theorem}[{\cite[Thm. 12]{cg08}}]
	Suppose we obtain an $\A_\infty$-algebra structure on a chain complex $B$ by homotopy transfer from an $\A_\infty$-algebra $A$. If $A$ were a $\C_\infty$-algebra, then so is $B$ with the transfered $\A_\infty$-algebra structure, and this structure corresponds with the one obtained by homotopy transfer for $\C_\infty$-algebras.\index{Homotopy!transfer theorem!for $\C_\infty$-algebras}
\end{theorem}

\subsection{Homotopy algebras over the dual numbers and spectral sequences}

One can recover the spectral sequence associated to a bicomplex via an application of the homotopy transfer theorem. This example is extracted from \cite[Sect. 1]{dsv15}.

\medskip

The operad of \emph{dual numbers}\index{Dual numbers} --- which we will denote by $\D$ in this section, but which will not make any other appearances in the rest of the present work --- is the quadratic operad
\[
\D\coloneqq\P(\k\Delta,\Delta\circ\Delta)\ ,
\]
where $\Delta$ is an arity $1$ element of degree $1$. An algebra over $\D$ is a chain complex $A$ together with an operation
\[
\Delta:A\longrightarrow A
\]
such that $\Delta^2 = 0$ and $d\Delta + \Delta d = 0$. In other words, a $\D$-algebra is nothing else than a bicomplex.

\medskip

It can be shown that the operad $\D$ is Koszul. Thanks to \cref{prop:formula for Koszul dual operad}, one sees that its Koszul dual operad is
\[
\D^{\shriek} = \P(s^{-1}\Delta^\vee)\ ,
\]
with no relations, and thus its Koszul dual cooperad is concentrated in arity $1$, where it is given by
\[
\D^{\antishriek}(1) = \bigoplus_{n\ge1}\k\delta_n
\]
with $\delta_n$ corresponding to the dual of $(s^{-1}\Delta^\vee)^n$. We have $|\delta_n| = 2n$ and
\[
\Delta_{(1)}(\delta_n) = \sum_{n_1+n_2 = n}\delta_{n_1}\circ\delta_{n_2}\ .
\]
It follows that the minimal model for $\D$ is the operad $\D_\infty$ freely generated by $\Delta_1,\Delta_2,\ldots$, where $\Delta_n$ corresponds to $s^{-1}\delta_n$ in $\Cobar\D^{\antishriek}$ and has degree $2n-1$. These operators satisfy
\[
d\Delta_n + \Delta_nd = -\sum_{n_1+n_2 = n}\Delta_{n_1}\circ\Delta_{n_2}\ ,
\]
or, writing $\Delta_0\coloneqq d$,
\[
\sum_{n_1+n_2 = n}\Delta_{n_1}\circ\Delta_{n_2} = 0\ .
\]
A $\D_\infty$-algebra is also known as a \emph{multicomplex}\index{Multicomplex}.

\medskip

Let $A$ be a bicomplex. Since we work over a field, one can always choose a contraction
\begin{center}
	\begin{tikzpicture}
	\node (a) at (0,0){$A$};
	\node (b) at (2.3,0){$H(A)$};
	
	\draw[->] (a)++(.3,.1)--node[above]{\mbox{\tiny{$p$}}}+(1.4,0);
	\draw[<-,yshift=-1mm] (a)++(.3,-.1)--node[below]{\mbox{\tiny{$i$}}}+(1.4,0);
	\draw[->] (a) to [out=-150,in=150,looseness=4] node[left]{\mbox{\tiny{$h$}}} (a);
	\end{tikzpicture}
\end{center}
from $A$ to its homology $H(A)$. The homotopy transfer theorem endows $H(A)$ with a multicomplex structure, where
\[
\Delta_n\coloneqq p\underbrace{\Delta h\Delta h\cdots h\Delta}_{n\text{ copies of }\Delta}i\ .
\]
The operator $\Delta_n$ is essentially the differential of the $n$th page of the spectral sequence associated to the bicomplex $A$.

\subsection{Homotopy Lie algebras}\label{subsect:homotopy Lie algebras}

This last example is the most important for the original results of this work. Homotopy Lie algebras --- also known as strong homotopy Lie algebras, or $\L_\infty$-algebras\index{Homotopy!Lie algebras}\index{$\L_\infty$-algebras} --- have a long history in the literature, where they appeared in a multitude of subjects. For example, they arise naturally in Kontsevich's proof of deformation quantization of Poisson manifolds \cite{kon03}, in string field theory \cite{zwi93}, in derived deformation theory \cite{pri10}, \cite{ss12} and others, as algebras of symmetries for conformal field theories \cite{bft17}, in symplectic topology \cite{Kontsevich}, \cite{Seidel}, \cite{FukayaOhOhtaOno} and others, as rational models for mapping spaces \cite{ber15}, and in many other places.

\medskip

As already stated before, we have
\[
\lie^{\shriek} = \com\ ,
\]
and the usual methods prove that $\lie$ is Koszul. Therefore, its minimal model is given by
\[
\L_\infty\coloneqq\Cobar\lie^{\antishriek}\ ,
\]
and
\[
\lie^{\antishriek}\cong(\susp^{-1})^c\otimes\com^\vee.
\]
Since $\com$ is $1$-dimensional in every arity, $\L_\infty$ is freely generated by the operations
\[
\ell_n\coloneqq s^{-1}\susp^{-1}_n\mu_n^\vee\in\L_\infty(n),\qquad n\ge2\ ,
\]
which have degree $|\ell_n| = n-2$. The only thing left to study is the differential, which is given by
\[
d_{\L_\infty}(\ell_n) = d_1(\ell_n) + d_2(\ell_n)\ ,
\]
where $d_1=0$ since the differential of $\com$ is trivial. The $d_2(\ell_n)$ is given by
\begin{align*}
	\ell_n\mapsto&\ -(s^{-1}\otimes s^{-1})\otimes\sum_{\substack{n_1+n_2 = n+1\\\sigma\in\sh(n_1,n_2-1)}}(-1)^{(n_2-1)(1-n_1) + \sigma}(\susp_{n_1}^{-1}\mu_{n_1}^\vee\otimes_1\susp_{n_2}^{-1}\mu_{n_2}^\vee)^\sigma\\
	&=\sum_{\substack{n_1+n_2 = n+1\\\sigma\in\sh(n_1,n_2-1)}}(-1)^{n_2(n_1-1)+\sigma+1}(\ell_{n_1}\otimes_1\ell_{n_2})^\sigma.
\end{align*}

In summary, an $\L_\infty$-algebra is defined as follows.

\begin{definition}
	An $\L_\infty$-algebra $\g$ is a chain complex together with graded antisymmetric operations (also called brackets)
	\[
	\ell_n:\g^{\otimes n}\longrightarrow\g\qquad\text{for }n\ge2
	\]
	satisfying
	\[
	\partial(\ell_n) = \sum_{\substack{n_1+n_2 = n+1\\\sigma\in\sh(n_1-1,n_2)}}(-1)^{n_2(n_1-1)+\sigma +1}(\ell_{n_1}\otimes_1\ell_{n_2})^\sigma
	\]
	for all $n\ge2$.
\end{definition}

Proceeding as we did in \cref{subsect:A-infinity algebras}, we obtain that an $\infty$-morphism of $\L_\infty$-algebras\index{$\infty$-morphism!of $\L_\infty$-algebras} $\Psi:\g\rightsquigarrow\h$ is a collection of linear maps
\[
\psi_n:\g^{\otimes n}\longrightarrow\h
\]
of degree $|\psi_n| = n-1$ that are antisymmetric and satisfy
\[
\partial(\psi_n) + \sum_{\substack{k\ge2\\m_1+\cdots+m_k=n\\\sigma\in\sh(m_1,\ldots,m_k)}}\pm\,\ell_k^B\circ(\psi_{m_1},\ldots,\psi_{m_k})^\sigma = \sum_{\substack{n_1+n_2 = n+1\\\theta\in\sh(n_1,n_2-1)}}\pm\,(\psi_{n_1}\circ_1\ell_{n_2}^A)^\theta.
\]
Finally, for the homotopy transfer theorem\index{Homotopy!transfer theorem!for $\L_\infty$-algebras}, if
\begin{center}
	\begin{tikzpicture}
	\node (a) at (0,0){$\g$};
	\node (b) at (2,0){$\h$};
	
	\draw[->] (a)++(.3,.1)--node[above]{\mbox{\tiny{$p$}}}+(1.4,0);
	\draw[<-,yshift=-1mm] (a)++(.3,-.1)--node[below]{\mbox{\tiny{$i$}}}+(1.4,0);
	\draw[->] (a) to [out=-150,in=150,looseness=4] node[left]{\mbox{\tiny{$h$}}} (a);
	\end{tikzpicture}
\end{center}
is a contraction and $\g$ is an $\L_\infty$-algebra, then the transferred $\L_\infty$-algebra structure on $\h$ is given by
\[
\ell_n^\h = \sum_{\tau\in\RT_n}\pm\,p\tau^h(v\mapsto\ell_{|v|}^\g)i^{\otimes n}.
\]
As usual, the signs come from the decomposition map of $(\susp^{-1})^c$.

\subsection{Operadic Massey products}\label{subsect:Massey products}

Given a chain complex $V$, one can always obtain a (non-canonical) contraction
\begin{center}
	\begin{tikzpicture}
	\node (a) at (0,0){$V$};
	\node (b) at (2.3,-.05){$H(V)$};
	
	\draw[->] (a)++(.3,.1)--node[above]{\mbox{\tiny{$p$}}}+(1.4,0);
	\draw[<-,yshift=-1mm] (a)++(.3,-.1)--node[below]{\mbox{\tiny{$i$}}}+(1.4,0);
	\draw[->] (a) to [out=-150,in=150,looseness=4] node[left]{\mbox{\tiny{$h$}}} (a);
	\end{tikzpicture}
\end{center}
from $V$ to its homology $H(V)$. It can be done as follows. For each $n\in\mathbb{Z}$, choose a complement of the subspace of cycles $\mathcal{Z}_n(V)$ and notice that it is isomorphic to the boundaries $\mathcal{B}_{n-1}(V)$ via the differential. Thus,
\[
V_n\cong\mathcal{Z}_n(V)\oplus\mathcal{B}_{n-1}(V)\ .
\]
Now choose a complement of $\mathcal{B}_n(V)$ in $\mathcal{Z}_n(V)$ and notice that it is isomorphic to $H_n(V)$. Therefore, we have
\[
V_n\cong\mathcal{B}_n(V)\oplus H_n(V)\oplus\mathcal{B}_{n-1}(V)\ .
\]
Define $i$ by sending $H(V)$ to the chosen copy of $H(V)$ in $V$, $p$ by projecting onto the copy of $H(V)$ in $V$, and $h$ on $V_n$ by first projecting onto $\mathcal{B}_n(V)$ and then identifying it with the copy of $\mathcal{B}_n(V)$ contained in $V_{n+1}$.

\begin{lemma}
	The three maps described above form a contraction from $V$ to $H(V)$.
\end{lemma}

\begin{proof}
	By inspection.
\end{proof}

Let $A$ be an $\A_\infty$-algebra. For example, one can take the the singular cochain complex\footnote{The theory works exactly the same if we exchange chain and cochain complexes, of course.} of a topological space $X$ with coefficients in the field $\k$ together with the cup product --- an associative algebra. Then the choice of a contraction as above gives an $\A_\infty$-algebra structure on the homology $H(A)$ of $A$.

\medskip

There is a well known classical construction of higher products --- the Massey products\index{Massey products} --- on the homology of an $\A_\infty$-algebra due to Massey \cite{mas58} and May \cite{may69}. As one might expect, the induced $\A_\infty$-algebra structures on $H(A)$ are strictly related to the Massey product. The exact relationship has been studied in \cite{mf17}.

\section{Bar and cobar construction for (co)algebras}

Given an operadic twisting morphism $\alpha:\C\to\P$, it is possible to define a bar-cobar adjunction relating the category of conilpotent $\C$-coalgebras and the category of $\P$-algebras. This helps for example to give a cleaner definition of the notion of $\infty$-morphisms, and will be the base for a generalization of the notion which will be exposed in \cref{chapter:oo-morphisms relative to a twisting morphism}. One also has a notion of twisting morphism relative to $\alpha$ giving a result analogous to \cref{thm:Rosetta stone operads} for (co)algebras. The material presented here comes from \cite[Sect. 11.1--3]{LodayVallette}.

\subsection{Bar and cobar construction relative to a twisting morphism}\label{subsection:relative bar and cobar}

We begin with the definition of the bar and cobar construction relative to a twisting morphism. For the rest of this section, fix a twisting morphism $\alpha:\C\to\P$.

\medskip

Let $A$ be a $\P$ algebra, then we define a conilpotent $\C$-coalgebra by
\[
\Bar_\alpha A\coloneqq(\C(A),d_{\Bar_\alpha A}\coloneqq d_1+d_2)\ ,
\]
where $d_1\coloneqq d_\C\circ1_A + 1_\C\circ'd_A$, and $d_2$ is the unique coderivation extending the composite
\[
\C(A)\xrightarrow{\alpha\circ1_A}\P(A)\xrightarrow{\gamma_A}A\ .
\]
That is to say, the full expression for $d_2$ is given by the composite
\[
\C(A)\xrightarrow{\Delta_{(1)}\circ1_A}(\C\circ_{(1)}\C)(A)\xrightarrow{(1_\C\circ_{(1)}\alpha)\circ1_A}(\C\circ_{(1)}\P)(A)\cong\C\circ(A;\P(A))\xrightarrow{1_\C\circ(1_A;\gamma_A)}\C(A)\ .
\]
The coderivation $d_{\Bar_\alpha A}$ squares to zero by \cite[Lemma 11.2.1]{LodayVallette}. Given a morphism $f:A\to A'$ of $\P$-algebras, we obtain a morphism of conilpotent $\C$-coalgebras by
\[
\Bar_\alpha f\coloneqq\C(f):\Bar_\alpha A\longrightarrow\Bar_\alpha A'\ ,
\]
and this assignment is functorial.

\begin{definition}
	The functor
	\[
	\Bar_\alpha:\Palg\longrightarrow\cCcog
	\]
	defined above is called the \emph{bar construction (relative to $\alpha$)}\index{Bar construction!relative to $\alpha$}.
\end{definition}

Dually, let $C$ be a conilpotent $\C$-coalgebra. We define a $\P$-algebra by
\[
\Cobar_\alpha C\coloneqq(\P(C),d_{\Cobar_\alpha C}\coloneqq d_1+d_2)\ ,
\]
where $d_1\coloneqq d_\P\circ1_D + 1_\P\circ'd_D$ and $-d_2$ is the unique derivation extending the composite
\[
C\xrightarrow{\Delta_C}\C(C)\xrightarrow{\alpha\circ1_C}\P(C)\ .
\]
Similarly to the previous case, the full expression for $-d_2$ is given by the composite
\begin{align*}
\P(C)&\xrightarrow{1_\P\circ'\Delta_C}\P\circ(C;\C(C))\\
&\xrightarrow{1_\P\circ(1_C;\alpha\circ1_C)}\P\circ(C;\P(C))\cong(\P\circ_{(1)}\P)(C)\xrightarrow{\gamma_{(1)}\circ1_C}\P(C)\ .
\end{align*}
Once again, the derivation $d_{\Cobar_\alpha C}$ squares to zero, see \cite[Lemma 11.2.4]{LodayVallette}. Given a morphism $g:C'\to C$ of $\C$-coalgebras, we obtain a morphism of $\P$-algebras by
\[
\Cobar_\alpha g\coloneqq\P(g):\Cobar_\alpha C'\longrightarrow\Cobar_\alpha C\ ,
\]
and this assignment is functorial.

\begin{definition}
	The functor
	\[
	\Cobar_\alpha:\cCcog\longrightarrow\Palg
	\]
	defined above is called the \emph{cobar construction (relative to $\alpha$)}\index{Cobar construction!relative to $\alpha$}.
\end{definition}

The relative bar construction gives us a cleaner way to define $\P_\infty$-algebras for $\P$ a Koszul operad and $\infty$-morphisms of $\P_\infty$-algebras.

\begin{proposition}[{\cite[Prop. 11.4.1]{LodayVallette}}]\label{prop:extended bar to oo-morphisms}
	Let $\P$ be a Koszul operad. The bar construction $\Bar_\iota$ associated to the twisting morphism $\iota:\P^{\antishriek}\to\Cobar\P^{\antishriek}$ extends to an isomorphism of categories
	\[
	\Bar_\iota:\ooPooAlg\longrightarrow\mathsf{quasi\text{-}free}\ \P^{\antishriek}\text{-}\mathsf{cog}
	\]
	from the category of $\P_\infty$-algebras with their $\infty$-morphisms to the full subcategory of the conilpotent $\P^{\antishriek}$-coalgebras given by the quasi-free coalgebras, i.e. the ones whose underlying $\P^{\antishriek}$-coalgebra in graded vector spaces is of the form $\P^{\antishriek}(V)$.
\end{proposition}

In other words, if $A$ is a $\P_\infty$-algebra, then $\Bar_\iota A$ is the $\P^{\antishriek}$-coalgebra given by \cref{prop:Poo-alg is coderivation of Pi-cog}, and an $\infty$-morphism\index{$\infty$-morphism} of $\P_\infty$-algebras is a morphism of coalgebras between the bar constructions.

\subsection{Relative bar and cobar constructions and quasi-isomorphisms}

Quasi-isomorphisms behave really well with respect to the bar construction.

\begin{proposition}[{\cite[Prop. 11.2.3]{LodayVallette}}]\label{prop:bar of qi is qi}
	Let $\alpha:\C\to\P$ be a twisting morphism, and let $f:A\to A'$ be a quasi-isomorphism of $\P$-algebras. Then
	\[
	\Bar_\alpha f:\Bar_\alpha A\longrightarrow\Bar_\alpha A'
	\]
	is a quasi-isomorphism.
\end{proposition}

They behave a bit worse with respect to the cobar construction.

\begin{proposition}[{\cite[Prop. 11.2.6]{LodayVallette}}]\label{prop:cobar construction and qis}
	Let $\alpha:\C\to\P$ be a twisting morphism, and let $g:C'\to C$ be a quasi-isomorphism between connected conilpotent $\C$-coalgebras, i.e. $\C$-coalgebras that are $0$ in degrees smaller or equal to $0$. Then
	\[
	\Cobar_\alpha g\coloneqq\P(g):\Cobar_\alpha C'\longrightarrow\Cobar_\alpha C
	\]
	is a quasi-isomorphism.
\end{proposition}

There are quasi-isomorphisms between conilpotent $\C$-coalgebras that are not sent to quasi-iso\-morphisms by the cobar construction $\Cobar_\alpha$, see e.g. \cite[Prop. 2.4.3]{LodayVallette}.

\subsection{Relative twisting morphisms}

Let $C$ be a conilpotent $\C$-coalgebra, and let $A$ be a $\P$-algebra. We consider the operator $\star_\alpha$ of degree $-1$ acting on the chain complex $\hom(C,A)$ given by the composite
\[
\star_\alpha(\varphi)\coloneqq\left(C\xrightarrow{\Delta_C}\C\circ C\xrightarrow{\alpha\circ\varphi}\P\circ A\xrightarrow{\gamma_A}A\right)
\]
for $\varphi\in\hom(C,A)$.

\begin{definition}\label{def:relative twisting morphisms}
	A \emph{twisting morphism relative to $\alpha$}\index{Twisting morphism!relative to $\alpha$} is an element $\varphi\in\hom(C,A)$ of degree $0$ satisfying the Maurer--Cartan equation
	\begin{equation}\label{equation:MC for relative twisting morphisms}
		\partial(\varphi) + \star_\alpha(\varphi) = 0\ .
	\end{equation}
	We denote the set of all such relative twisting morphisms by $\Tw_\alpha(C,A)$.
\end{definition}

\subsection{Bar-cobar adjunction for (co)algebras}

Similarly to the operadic case, the bar and cobar functors form an adjoint pair, cf. \cref{thm:Rosetta stone operads}.

\begin{theorem}\label{thm:Rosetta stone algebras}
	Let $\alpha:\C\to\P$ be an operadic twisting morphism, let $C$ be a conilpotent $\C$-coalgebra, and let $A$ be a $\P$-algebra. There are bijections
	\[
	\hom_{\Palg}(\Cobar_\alpha C,A)\cong\Tw_\alpha(C,A)\cong\hom_{\Ccog}(C,\Bar_\alpha A)\ ,
	\]
	natural both in $C$ and $A$. In particular, $\Cobar_\alpha$ and $\Bar_\alpha$ form an adjoint pair.
\end{theorem}

\begin{proof}
	We prove the first bijection, the second one being dual. A morphism of $\P$-algebras $f:\Cobar_\alpha C\to A$ is in particular a morphism of algebras in graded vector spaces from $\P(C)\to A$, and is therefore completely determined by its restriction $\varphi\coloneqq f|_C$ to $C$. In the other direction, given $\varphi$ we recover $f$ as
	\[
	f = \gamma_A(1_\P\circ\varphi)\ .
	\]
	We only need to show that $f$ commuting with the differentials is equivalent to $\varphi$ satisfying the Maurer--Cartan equation (\ref{equation:MC for relative twisting morphisms}). The restriction to $C$ of the relation
	\[
	d_Af = fd_{\Cobar_\alpha C}
	\]
	gives
	\begin{align*}
		d_A\varphi =&\ -\varphi d_C - f(\alpha\circ1_C)\Delta_C\\
		=&\ -\varphi d_C - \gamma_A(1_\P\circ\varphi)(\alpha\circ1_C)\Delta_C\\
		=&\ -\varphi d_C - \gamma_A(\alpha\circ\varphi)\Delta_C\\
		=&\ -\varphi d_C - \star_\alpha(\varphi)\ .
	\end{align*}
	Therefore, if $f$ commutes with the differentials, then $\varphi$ satisfies the Maurer--Cartan equation. The other direction also follows from this computation and from the fact that the restriction to $C$ determines the whole morphism.
\end{proof}

\subsection{Bar-cobar resolutions of (co)algebras}

Now we would like to use the bar-cobar adjunction for (co)algebras to give (functorial) resolutions of the same, as we did for operads in \cref{sect:bar-cobar resolution of operads}. This is indeed possible, but we have to further require that $\alpha$ is Koszul.

\begin{theorem}\label{thm:bar-cobar resolution for algebras}
	Let $\alpha:\C\to\P$ be a twisting morphism. The following are equivalent.
	\begin{enumerate}
		\item $\alpha$ is a Koszul morphism.
		\item The counit
		\[
		\epsilon_A:\Cobar_\alpha\Bar_\alpha A\longrightarrow A
		\]
		of the bar-cobar adjunction is a quasi-isomorphism for any $\P$-algebra $A$.
	\end{enumerate}
\end{theorem}

\begin{remark}
	In particular, the bar-cobar adjunction provides a cofibrant resolution for $\P$-algebras, cf. \cref{subsection:Hinich model structure on algebras}.
\end{remark}

\begin{proof}
	This is \cite[Thm. 11.3.3]{LodayVallette}.
\end{proof}

\begin{remark}
	This is true as stated because we supposed that all of our (co)operads are reduced, and thus are canonically connected weight graded by the arity. In more generality, one has to assume that $\C$ and $\P$ are connected weight graded, and that $\alpha$ preserves this additional grading.
\end{remark}

A similar statement is true for coalgebras.

\begin{theorem}\label{thm:alpha Koszul iff unit is qi}
	Let $\alpha:\C\to\P$ be a twisting morphism. The following are equivalent.
	\begin{enumerate}
		\item\label{pt:1} $\alpha$ is a Koszul morphism.
		\item\label{pt:2} The unit
		\[
		\eta_C:C\longrightarrow\Bar_\alpha\Cobar_\alpha C
		\]
		is a quasi-isomorphism for any conilpotent $\C$-coalgebra $C$.
	\end{enumerate}
\end{theorem}

\begin{remark}
	A refinement version of the implication $(\ref{pt:1})\implies(\ref{pt:2})$ has been given in \cite[Thm. 2.6(2)]{val14}\footnote{It is stated there only for the twisting morphism $\kappa$ given by Koszul duality, but the result holds in general for twisting morphisms between connected weight graded (co)operads. In particular, it always hold in our setting.}. We will present it in more details in \cref{subsection:Vallette model structure on coalgebras}, see \cref{corollary:unit of bar-cobar is a we in the Vallette model structure}.
\end{remark}

\begin{proof}
	This is \cite[Thm. 11.3.4]{LodayVallette}.
\end{proof}

\section{Homotopy theory of homotopy algebras}\label{sect:homotopy theory of homotopy algebras}

In this section, we give some results on the homotopy theory of $\P_\infty$-algebras, for $\P$ a Koszul operad. The homotopy theory of homotopy algebras, and more generally of (co)algebras related by an operadic Koszul morphisms will be studied again in more detail in \cref{subsection:Vallette model structure on coalgebras} and \cref{sect:homotopy theory of oo-morphisms of coalgebras}. The material presented here is extracted from \cite[Sect. 11.4]{LodayVallette}. For this section, fix a Koszul operad $\P$, and as usual let $\kappa:\P^{\antishriek}\to\P$ be the Koszul morphism given by Koszul duality, and let $\iota:\P^{\antishriek}\to\P_\infty$ be the canonical twisting morphism, which is also Koszul.

\subsection{Rectification of \texorpdfstring{$\P_\infty$}{Poo}-algebras}\label{subsection:rectification for Poo-algebras}

Every $\P$-algebra is in particular a $\P_\infty$-algebra. This can be done as follows. Let $g_\kappa:\P_\infty\to\P$ be the morphism of operads given by \cref{thm:universal twisting morphisms} applied to $\alpha=\kappa$. Notice that it is a quasi-isomorphism by \cref{thm:fundamental thm of operadic twisting morphisms}. Then the restriction of structure $g_\kappa^*$ gives the desired functor
\[
g_\kappa^*:\Palg\longrightarrow\PooAlg\ .
\]
Denote by
\[
i:\Palg\longrightarrow\ooPooAlg
\]
the composite of $g_\kappa^*$ with the inclusion of $\PooAlg$ into $\ooPooAlg$. On the other hand, we have the functor
\[
\Cobar_\kappa\Bar_\iota:\ooPooAlg\longrightarrow\Palg\ ,
\]
where $\Bar_\iota$ is the extension of the bar construction of \cref{prop:extended bar to oo-morphisms}.

\begin{proposition}
	The functors described above form an adjoint pair
	\[
	\adjunction{\Cobar_\kappa\Bar_\iota}{\ooPooAlg}{\Palg}{i}\ .
	\]
\end{proposition}

\begin{proof}
	This is \cite[Prop. 11.4.3]{LodayVallette}.
\end{proof}

\begin{theorem}
	Let $A$ be a $\P_\infty$-algebra, then there is an $\infty$-quasi-isomorphism
	\[
	A\stackrel{\sim}{\rightsquigarrow}\Cobar_\kappa\Bar_\iota A\ ,
	\]
	natural in $A$. Moreover, the $\P$-algebra $\Cobar_\kappa\Bar_\iota A$ is unique, up to isomorphism, with respect to the universal property that any $\infty$-morphism with $A$ as domain and a $\P$-algebra as target factors into the $\infty$-quasi-isomorphism above followed by a strict morphism of $\P$-algebras.
\end{theorem}

\begin{proof}
	This result is \cite[Thm. 11.4.4, and Prop. 11.4.5 and 11.4.6]{LodayVallette}.
\end{proof}

The functor $\Cobar_\kappa\Bar_\iota$ is called the \emph{rectification functor}\index{Rectification functor!for $\P_\infty$-algebras} for $\P_\infty$-algebras. In fact, it gives an equivalence between the homotopy categories of $\ooPooAlg$ and $\Palg$, where in $\ooPooAlg$ the weak equivalences are the $\infty$-quasi-isomorphism. See \cite[Sect. 11.4]{LodayVallette} for details.

\subsection{Relation between quasi-isomorphisms and \texorpdfstring{$\infty$}{oo}-quasi-isomorphisms}

Quasi-isomorphisms and $\infty$-quasi-isomorphisms are closely related. The following result is of fundamental importance, for example in interpreting certain classical definitions in rational homotopy theory, cf. \cref{chapter:RHT}.

\begin{theorem}\label{thm:oo-qi of algebras is zig-zag of qi}
	Let $A$ and $B$ be two $\P_\infty$-algebras. The following are equivalent.
	\begin{enumerate}
		\item\label{pt:1qi} There is a zig-zag of quasi-isomorphisms of $\P_\infty$-algebras
		\[
		A\stackrel{\sim}{\longleftarrow}\bullet\stackrel{\sim}{\longrightarrow}\bullet\stackrel{\sim}{\longleftarrow}\bullet\cdots\bullet\stackrel{\sim}{\longrightarrow}B\ .
		\]
		\item There are two quasi-isomorphisms of $\P_\infty$-algebras
		\[
		A\stackrel{\sim}{\longleftarrow}\bullet\stackrel{\sim}{\longrightarrow}B\ .
		\]
		\item\label{pt:3qi} There exists an $\infty$-quasi-isomorphism of $\P_\infty$-algebras
		\[
		A\stackrel{\sim}{\rightsquigarrow}B\ .
		\]
	\end{enumerate}
\end{theorem}

\begin{proof}
	The case where $A$ and $B$ are $\P$-algebras is proven in \cite[Thm. 11.4.9]{LodayVallette}. The more general case of $\P_\infty$-algebras is done similarly, but in order to prove $(\ref{pt:1qi})\implies(\ref{pt:3qi})$ one has to use a model categorical argument, cf. the case for coalgebras of \cref{prop:alpha-we of coalgebras is zig-zag of we}.
\end{proof}

An immediate consequence of this result is the following one.

\begin{corollary}
	Let $A$ and $B$ be two $\P_\infty$-algebras. If there is an $\infty$-quasi-isomorphism
	\[
	A\stackrel{\sim}{\rightsquigarrow}B\ ,
	\]
	then there exists an $\infty$-quasi-isomorphism in the other direction
	\[
	B\stackrel{\sim}{\rightsquigarrow}A\ .
	\]
\end{corollary}

%% file: ModelCategories.tex
\chapter{Model categories}\label{ch:model categories}

Model categories were introduced by Quillen in \cite{QuillenMC} in order to do ``non-linear homological algebra'', also known as homotopical algebra. They give a generalized framework in which to study homotopy theory in some category, which consists into taking a category and formally inverting some morphisms that one would like to consider as equivalences. Two motivating examples are:
\begin{enumerate}
	\item One considers the category of topological spaces and wants to study their homotopy groups. Therefore, one wants to formally invert continuous maps that induce isomorphisms on all homotopy groups, so that the isomorphisms classes of the new category correspond to the existing homotopy types of topological spaces.
	\item One considers the category of chain complexes and wants to study their homology. Therefore, one wants to invert quasi-isomorphisms, i.e. the chain maps inducing isomorphisms in homology. Trying to do this, one essentially recovers classical homological algebra.
\end{enumerate}
In this chapter, we will give a fast introduction to model categories and the concepts surrounding them, and then give examples, some for motivation, and some because we will need them later on. Our main references are the books \cite{Hovey}, and \cite[Ch. II]{GoerssJardine}, but the reader should be aware that there are many good items on this subject in the literature, such as the already mentioned seminal work \cite{QuillenMC}, and \cite{ds95}. One should also mention that model categories have been somewhat superseded by $\infty$-categories\footnote{More precisely, $(\infty,1)$-categories, of which one model are quasicategories.} in the recent years, although model categories remain an important tool for the homotopy theorists. The topic is outside the scope of this work, but the interested reader should have no problems finding references on the subject, e.g. starting with \cite{Lurie}.

\section{Model categories}

We begin by giving the basic definitions of model categories and explaining the notion of a homotopy between morphisms in a model category.

\subsection{Definitions}

Without further ado, we give the definition of a model category.

\begin{definition}
	A \emph{model category}\index{Model category}\index{Fibrations}\index{Cofibrations}\index{Weak equivalences}\footnote{Sometimes, this is called a \emph{closed} model category.} is a category $\cat$, together with three classes of arrows $W,F$ and $C$ --- called respectively \emph{weak equivalences}\index{Weak equivalences}, \emph{fibrations}\index{Fibrations}, and \emph{cofibrations}\index{Cofibrations} --- satisfying the following properties.
	\begin{enumerate}[label = \textbf{M\arabic*}]
		\item The category $\cat$ has all finite limits and colimits.
		\item\label{axiom:2-out-of-3} Given two arrows $f,g$ in $\cat$ such that the target of $f$ is the domain of $g$, if any two of $f,g$ and $gf$ are weak equivalences, then so is the third. We say that $W$ satisfies the \emph{$2$-out-of-$3$ property}\index{$2$-out-of-$3$ property}.
		\item The classes of weak equivalences, fibrations, and cofibrations are all closed under retracts.
		\item\label{axiom:lifting properties} Suppose we are given the following solid arrow diagram
		\begin{center}
			\begin{tikzpicture}
				\node (a) at (0,0){};
				\node (b) at (2,0){};
				\node (c) at (0,-2){};
				\node (d) at (2,-2){};
				
				\draw[->] (a) to (b);
				\draw[>->] (a) to node[left]{$i$} (c);
				\draw[->>] (b) to node[right]{$p$} (d);
				\draw[->] (c) to (d);
				
				\draw[->,dashed] (c) to (b);
			\end{tikzpicture}
		\end{center}
		where $i$ is a cofibration, and $p$ is a fibration. If either one of $i$ or $p$ is also a weak equivalence, then there exist a diagonal filler (the dashed arrow in the diagram).
		\item\label{axiom:factorization} Every arrow $f$ in $\cat$ can be factored both as $f = pi$ with $i$ a cofibration which is also a weak equivalence and $p$ a fibration, and as $f=qj$ with $j$ a cofibration, and $q$ a fibration that is also a weak equivalence.
	\end{enumerate}
	We often abuse of notation and just talk of the model category $\cat$ when no confusion is possible about the three classes of maps.
\end{definition}

\begin{remark}
	There are many slight variations on the definition of a model category in the literature. The definition given above is the one of \cite{QuillenMC}, while for example \cite{Hovey} requires that all limits and colimits exist in $\cat$, and  moreover requires that the factorizations of (\ref{axiom:factorization}) be functorial. As a rule of thumb, essentially all results that can be proven with one version of the definition of a model category hold for all other sensible versions of the definition, with at most minor changes if needed.
\end{remark}

We will often emphasize the fact that an arrow is a cofibration by adding a tail to it, that it is a fibration by giving it a double head, and that it is a weak equivalence by writing a $\sim$ next to it.

\medskip

It is usual to call \emph{trivial fibrations}\index{Trivial!fibrations}, respectively \emph{trivial cofibrations}\index{Trivial!cofibrations}, the fibration, resp. cofibrations, that are also weak equivalences. One often rephrases the axiom (\ref{axiom:lifting properties}) by saying that trivial cofibrations have the \emph{left lifting property}\index{Left!lifting property} with respect to fibrations, or dually, that fibrations have the \emph{right lifting property}\index{Right!lifting property} with respect to trivial cofibrations, and the same for the relation between cofibrations and trivial fibrations.

\medskip

Given a model category, one can characterize the fibrations and cofibrations.

\begin{lemma}
	Let $\cat$ be a model category.
	\begin{enumerate}
		\item A morphism in $\cat$ is a fibration if, and only if it has the right lifting property with respect to all trivial cofibrations.
		\item A morphism in $\cat$ is a trivial fibration if, and only if it has the right lifting property with respect to all cofibrations.
		\item A morphism in $\cat$ is a cofibration if, and only if it has the left lifting property with respect to all trivial fibrations.
		\item A morphism in $\cat$ is a trivial cofibration if, and only if it has the left lifting property with respect to all fibrations.
	\end{enumerate}
\end{lemma}

\begin{corollary}
	Let $\cat$ be a model category.
	\begin{enumerate}
		\item Cofibrations and trivial cofibrations are closed under pushouts and compositions, and all isomorphisms are cofibrations.
		\item Fibrations and trivial fibrations are closed under pullbacks and compositions, and all isomorphisms are fibrations.
	\end{enumerate}
\end{corollary}

One can also say something about weak equivalences.

\begin{lemma}
	Let $\cat$ be a model category. All isomorphisms are weak equivalences.
\end{lemma}

There are two sets of objects that play special roles in model categories.

\begin{definition}
	Let $\cat$ be a model category.
	\begin{enumerate}
		\item An object $X\in\cat$ is \emph{cofibrant}\index{Cofibrant!object} if the unique map
		\[
		\emptyset\longrightarrow X
		\]
		from the initial object to $X$ is a cofibration.
		\item An object $X\in\cat$ is \emph{fibrant}\index{Fibrant!object} if the unique map
		\[
		X\longrightarrow *
		\]
		from $X$ to the final object is a fibration.
	\end{enumerate}
	We call an object \emph{bifibrant}\index{Bifibrant object} if it is both fibrant and cofibrant.
\end{definition}

Notice that using (\ref{axiom:factorization}), for any object $X\in\cat$ one can find another object $\widetilde{X}$ which is cofibrant and weakly equivalent to $X$, or fibrant and weakly equivalent to $X$. These new objects are called \emph{cofibrant}, resp. \emph{fibrant}, \emph{replacements}\index{Cofibrant!replacement}\index{Fibrant!replacement} of $X$.

\medskip

A very useful result about model categories is Ken Brown's lemma.

\begin{lemma}[Ken Brown's lemma]\index{Ken Brown's lemma}\label{lemma:Ken Brown}
	Let $\cat$ be a model category, and let $\cat'$ be a category with a class of weak equivalences which satisfies (\ref{axiom:2-out-of-3}). Let $F:\cat\to\cat'$ be a functor.
	\begin{enumerate}
		\item If $F$ takes trivial cofibrations between cofibrant objects to weak equivalences, then $F$ takes all weak equivalences between cofibrant objects to weak equivalences.
		\item Dually, if $F$ takes trivial fibrations between fibrant objects to weak equivalences, then $F$ takes all weak equivalences between fibrant objects to weak equivalences.
	\end{enumerate}
\end{lemma}

\subsection{Duality}

Fibrations and cofibrations of a model category play dual roles, in a way which is made precise by the following result.

\begin{proposition}
	Suppose that $\cat$ is a model category. Then the opposite category $\cat^\mathrm{op}$ is also a model category with the same weak equivalences, the fibrations of $\cat$ as cofibrations, and the cofibrations of $\cat$ as fibrations.
\end{proposition}

Therefore, it is usually only necessary to prove a result for fibrations in order to have a dual result for cofibrations, and \emph{vice versa}.

\subsection{Path objects, cylinder objects, and homotopies}

Let $\cat$ category with finite limits and colimits. If $X\in\cat$ is an object, then the identity induces two canonical maps. The \emph{fold map}\index{Fold map}
\[
\nabla:X\sqcup X\longrightarrow X\ ,
\]
and the \emph{diagonal map}\index{Diagonal map}
\[
\Delta:X\longrightarrow X\times X\ .
\]
Using these maps, one can define two notions of homotopy between morphisms in a model category.

\begin{definition}
	Let $\cat$ be a model category, let $X,Y\in\cat$ be two objects, and let $f,g:X\to Y$ be two morphisms.
	\begin{enumerate}
		\item A \emph{cylinder object}\index{Cylinder object} for $X$ is a commutative diagram
		\begin{center}
			\begin{tikzpicture}
				\node (a) at (0,0){$X\sqcup X$};
				\node (b) at (3,0){$\Cyl(X)$};
				\node (c) at (6,0){$X$};
				
				\draw[>->] (a) to node[below]{$i$} (b);
				\draw[->] (b) to node[below]{$w$} node[label={[label distance=-.2cm]90:$\sim$}]{} (c);
				\draw[->] (a) to[out=20,in=160] node[above]{$\nabla$} (c);
			\end{tikzpicture}
		\end{center}
		with $i$ a cofibration, and $w$ a weak equivalence. We often abuse of notation and speak of the cylinder object $\Cyl(X)$.
		\item The two morphisms $f$ and $g$ are \emph{left homotopic}\index{Left!homotopy} if there exists a cylinder object $\Cyl(X)$ for $X$ and a morphism $H:\Cyl(X)\to Y$, called a \emph{left homotopy} between $f$ and $g$, such that
		\begin{center}
			\begin{tikzpicture}
			\node (a) at (0,0){$X\sqcup X$};
			\node (b) at (3,0){$\Cyl(X)$};
			\node (c) at (6,0){$Y$};
			
			\draw[>->] (a) to node[below]{$i$} (b);
			\draw[->] (b) to node[below]{$H$} (c);
			\draw[->] (a) to[out=20,in=160] node[above]{$f\sqcup g$} (c);
			\end{tikzpicture}
		\end{center}
		We write $f\sim_\ell g$.
		\item A \emph{path object}\index{Path object} for $Y$ is a commutative diagram
		\begin{center}
			\begin{tikzpicture}
			\node (a) at (0,0){$Y$};
			\node (b) at (3,0){$\Path(Y)$};
			\node (c) at (6,0){$Y\times Y$};
			
			\draw[->] (a) to node[below]{$w$} node[label={[label distance=-.2cm]90:$\sim$}]{} (b);
			\draw[->>] (b) to node[below]{$p$} (c);
			\draw[->] (a) to[out=20,in=160] node[above]{$\Delta$} (c);
			\end{tikzpicture}
		\end{center}
		with $p$ a fibration, and $w$ a weak equivalence. We often abuse of notation and speak of the path object $\Path(X)$.
		\item The two morphisms $f$ and $g$ are \emph{right homotopic}\index{Right!homotopy} if there exists a path object $\Path(Y)$ for $Y$ and a morphism $H:X\to\Path(Y)$, called a \emph{right homotopy} between $f$ and $g$, such that
		\begin{center}
			\begin{tikzpicture}
			\node (a) at (0,0){$X$};
			\node (b) at (3,0){$\Path(Y)$};
			\node (c) at (6,0){$Y\times Y$};
			
			\draw[>->] (a) to node[below]{$H$} (b);
			\draw[->] (b) to node[below]{$p$} (c);
			\draw[->] (a) to[out=20,in=160] node[above]{$f\times g$} (c);
			\end{tikzpicture}
		\end{center}
		We write $f\sim_r g$.
	\end{enumerate}
\end{definition}

\begin{remark}\label{remark:cylinders and paths in topological spaces}
	The notions of cylinder and path objects are inspired by topology. Indeed, there is a model structure on the category of topological spaces for which a choice of cylinder and path objects for a space $X$ are the usual ones, namely $X\times I$ gives a cylinder object, and $X^I$ a path object, cf. \cref{subsect:model structure on topological spaces}. This way, one recovers the usual notions of homotopy between continuous maps.
\end{remark}

The following result tells us that the notions of left and right homotopy given above are well defined and independent from the choice of cylinder, resp. path object.

\begin{lemma}
	Let $\cat$ be a model category, and let $X\in\cat$ be an object.
	\begin{enumerate}
		\item A cylinder object for $X$ always exists. Any two cylinder objects for $X$ are weakly equivalent. If two morphisms are left homotopic with respect to a cylinder object, then they are for all cylinder objects.
		\item A path object for $X$ always exists. Any two path objects for $X$ are weakly equivalent. If two morphisms are right homotopic for a path object, then they are for all path objects.
	\end{enumerate}
\end{lemma}

\begin{proof}
	We will only prove the first result, the second one being dual. Existence of a cylinder object is guaranteed by (\ref{axiom:factorization}). Applying it to the fold map, we factor it into
	\begin{center}
		\begin{tikzpicture}
		\node (a) at (0,0){$X\sqcup X$};
		\node (b) at (3,0){$\Cyl(X)$};
		\node (c) at (6,0){$X$};
		
		\draw[>->] (a) to node[below]{$i$} (b);
		\draw[->>] (b) to node[below]{$w$} node[sloped,yshift=.1cm]{$\sim$} (c);
		\draw[->] (a) to[out=20,in=160] node[above]{$\nabla$} (c);
		\end{tikzpicture}
	\end{center}
	with $w$ being not only a weak equivalence, but a trivial fibration. Suppose we are given a second cylinder object
	\begin{center}
		\begin{tikzpicture}
		\node (a) at (0,0){$X\sqcup X$};
		\node (b) at (3,0){$\Cyl(X)'$};
		\node (c) at (6,0){$X$};
		
		\draw[>->] (a) to node[below]{$i'$} (b);
		\draw[->] (b) to node[below]{$w'$} node[sloped,yshift=.1cm]{$\sim$} (c);
		\draw[->] (a) to[out=20,in=160] node[above]{$\nabla$} (c);
		\end{tikzpicture}
	\end{center}
	Then we have the following commutative diagram.
	\begin{center}
		\begin{tikzpicture}
			\node (a) at (0,0){$X\sqcup X$};
			\node (b) at (3,0){$\Cyl(X)$};
			\node (c) at (0,-2.){$\Cyl(X)'$};
			\node (d) at (3,-2){$X$};
			
			\draw[>->] (a) to node[above]{$i$} (b);
			\draw[->>] (b) to node[left]{$w$} node[sloped,yshift=.1cm]{$\sim$} (d);
			\draw[>->] (a) to node[left]{$i'$} (c);
			\draw[->] (c) to node[below]{$w'$} node[sloped,yshift=.1cm]{$\sim$} (d);
			
			\draw[->,dashed] (c) to node[sloped,yshift=.1cm]{$\sim$} (b);
		\end{tikzpicture}
	\end{center}
	The dashed arrow exists by (\ref{axiom:lifting properties}), and it is a weak equivalence by (\ref{axiom:2-out-of-3}). So any two cylinder objects are weakly equivalent. The last statement follows in a straightforward manner.
\end{proof}

If the domain of the arrows we consider is cofibrant, and if the target is fibrant, then the homotopy relations are very well behaved.

\begin{proposition}
	Let $\cat$ be a model category, and let $X,Y\in\cat$ be two objects.
	\begin{enumerate}
		\item If $X$ is cofibrant, then being left homotopic is an equivalence relation on $\hom_\cat(X,Y)$.
		\item If $Y$ is fibrant, then being right homotopic is an equivalence relation on $\hom_\cat(X,Y)$.
		\item If $X$ is cofibrant, and $Y$ is fibrant, then two maps $X\to Y$ in $\cat$ are left homotopic if, and only if they are right homotopic. In this case, we simply say that the two maps are \emph{homotopic}.
	\end{enumerate}
\end{proposition}

One can now give an analogue of the Whitehead theorem in the general context of model categories.

\begin{definition}
	Let $\cat$ be a model category, and let $X,Y\in\cat$ be two bifibrant objects. We say that $f:X\to Y$ is a \emph{homotopy equivalence}\index{Homotopy!equivalence} if there exists a morphism $g:Y\to X$ such that $fg\sim1_Y$ and $gf\sim1_X$.
\end{definition}

\begin{theorem}[Whitehead]\index{Whitehead theorem}
	Let $\cat$ be a model category, and let $X,Y\in\cat$ be two bifibrant objects. If $f:X\to Y$ is a weak equivalence, then it is a homotopy equivalence.
\end{theorem}

\section{Homotopy categories}

As explained at the beginning of the chapter, one wants to formally invert some class of maps to obtain the ``homotopy category'' of the original category. We explain the naive way to do that and the problems in which one might incur, and then expose how one can do the desired process if one wants to invert the class of weak equivalences in a model category.

\subsection{Localization at a class of maps}

Suppose that $\cat$ is a category, and that $W$ is a class of morphisms in $\cat$ that one wishes to formally invert.

\begin{definition}
	The \emph{localization of $\cat$ at $W$}\index{Localization of a category}, if it exists, is a category $\cat[W^{-1}]$ together with a functor $F:\cat\to\cat[W^{-1}]$ satisfying the following universal property. If $\cat'$ is another category, and $G:\cat\to\cat'$ is a functor, then $G$ factors through $F$ if, and only if every morphism in $W$ is sent into an isomorphism by $G$.
\end{definition}

\begin{remark}
	This is very closely related with localizations in rings and modules over rings.
\end{remark}

One can try to construct the localized category $\cat[W^{-1}]$ by taking the category with the same objects as $\cat$, and as morphisms the words formed by composable morphisms of $\cat$ and formal inverses of morphisms in $W$, and then identifying the word $fg$ with the composite of $f$ and $g$ whenever $f,g$ are in $\cat$, the words $ww^{-1}$ and $w^{-1}w$ with the identity whenever $w\in W$ and $w^{-1}$ is its formal inverse, and the letter $w^{-1}$ with the inverse of $w$ whenever $w\in W$ is invertible. The problem with this process is that it adds a potentially huge amount of morphisms to $\cat$. So many, in fact, that one easily incurs in set theoretical issues that make it so that the structure $\cat[W^{-1}]$ defined this way is no longer a category.

\subsection{The homotopy category of a model category}

Now suppose that $\cat$ is a model category.

\begin{definition}
	The localization of $\cat$ at the class $W$ of weak equivalences is called the \emph{homotopy category}\index{Homotopy!category} of $\cat$, and it is denoted by
	\[
	\ho{\cat}\coloneqq\cat[W^{-1}]\ .
	\]
\end{definition}

It is an important result of \cite{QuillenMC} that this gives a well-defined category, which furthermore admits a much smaller, equivalent description as the quotient of the subcategory of bifibrant objects of $\cat$ by the homotopy relation.

\begin{theorem}\label{thm:homotopy category as quotient of category of bifibrant objects}
	Let $\cat$ be a model category. The homotopy category $\ho{\cat}$ of $\cat$ is equivalent to the category whose objects are the bifibrant objects of $\cat$, and whose morphisms between two such objects $X, Y\in\cat$ are given by the quotient of $\hom_\cat(X,Y)$ by the homotopy equivalence relation. In particular, the homotopy category of a model category is a well-defined category.
\end{theorem}

\begin{remark}
	Because of this result, we will usually implicitly abuse of notation and write $\ho{\cat}$ and talk of the homotopy category of a model category when in fact meaning the smaller presentation presented above.
\end{remark}

As a consequence, we have the following result.

\begin{lemma}\label{lemma:we induce bijections of hom sets}
	Suppose $X_1,X_2,Y\in\cat$ are bifibrant, and let $\phi:X_1\to X_2$ be a weak equivalence. Then the maps
	\[
	\phi^*:\hom_{\ho{\cat}}(X_2,Y)\longrightarrow\hom_{\ho{\cat}}(X_1,Y)\quad\text{and}\quad\phi_*:\hom_{\ho{\cat}}(Y,X_1)\longrightarrow\hom_{\ho{\cat}}(Y,X_2)
	\]
	are bijections.
\end{lemma}

\begin{proof}
	This looks trivial with the statement to \cref{thm:homotopy category as quotient of category of bifibrant objects}, but it is in fact a step of the proof of that result. See \cite[Prop. 1.2.5(iv)]{Hovey} for a proof.
\end{proof}

\subsection{Quillen functors and Quillen equivalences}

One is now interested in understanding under what conditions a functor between two model categories induces a functor on the homotopy categories, and in particular when two categories have equivalent homotopy categories\footnote{This can happen even when the two model categories are not equivalent themselves.}. The correct notions giving an answer to these questions are \emph{Quillen functors}\index{Quillen!functor}, \emph{Quillen adjunctions}\index{Quillen!adjunction}, and \emph{Quillen equivalences}\index{Quillen!equivalence}. We will not use these concepts a lot, and so we refer the reader to \cite[Sect. 1.3]{Hovey} for details.

\section{Examples}\label{sect:examples model categories}

Examples of model categories abound. In this section, we give a few ones. Some of them are standard and we write them down for completeness, while others are less well known. The model structures on chain complexes, operads, cooperads, algebras over operads, and coalgebras over cooperads will be of interest in the rest of the present work.

\subsection{Topological spaces}\label{subsect:model structure on topological spaces}

The first example we give, which is the one on which model categories themselves are modeled upon, is the category of topological spaces.

\begin{theorem}[Quillen]
	The following three classes of continuous maps make $\Top$ into a model category.\index{Model category!of topological spaces}
	\begin{enumerate}
		\item A continuous map $f:X\to Y$ is a \emph{weak equivalence}\index{Weak equivalences!of topological spaces} if all of the induced maps on the homotopy groups
		\[
		\pi_0(f):\pi_0(X)\longrightarrow\pi_0(Y)\qquad\text{and}\qquad\pi_n(f,x):\pi_n(X,x)\longrightarrow\pi_n(Y,f(x))
		\]
		are isomorphisms, for all $x\in X$ and all $n\ge1$.
		\item A continuous map is a \emph{fibration}\index{Fibrations!of topological spaces}\index{Serre fibrations}\footnote{Also known as a \emph{Serre fibration} in the literature.} if it has the right lifting property with respect to all the inclusions $D^n\to D^n\times I$ of the $n$-disk into the product of the $n$-disk with an interval given by sending $x\in D^n$ to $(x,0)\in D^n\times I$.
		\item A continuous map is a \emph{cofibration}\index{Cofibrations!of topological spaces} if it has the left lifting property with respect to all trivial fibrations.
	\end{enumerate}
\end{theorem}

As already mentioned in \cref{remark:cylinders and paths in topological spaces}, a cylinder object for a topological space $X$ is given by $X\times I$, and a path object is given by $X^I$.

\subsection{Simplicial sets}

Simplicial sets are certain combinatorial objects whose homotopy theory is equivalent to the one of topological spaces. Namely, the category $\ssets$ of simplicial sets admits a model structure\footnote{In fact, it admits more than one. Here, we mean the classical --- or Quillen --- model structure, as opposed to the Joyal model structure, which is very important in the world of $\infty$-categories.} and a Quillen equivalence with the model category of topological spaces. This will be discussed in more detail in \cref{ch:simplicial homotopy theory}.

\subsection{Chain complexes}

Since model categories give us ``non-linear homological algebra\footnote{See the very beginning of the introduction to \cite{QuillenMC}.}'', one would like to recover the usual homological algebra when treating chain complexes. This is indeed possible, by considering the following model structure on $\chain$.

\begin{theorem}
	The following three classes of chain maps make the category $\chain$ of chain complexes\footnote{Chain complexes do not have any boundedness assumption throughout this work, unless explicitly stated.} into a model category.\index{Model category!of chain complexes}
	\begin{enumerate}
		\item A chain map $f:V\to W$ is a \emph{weak equivalence}\index{Weak equivalences!of chain complexes} if it is a quasi-isomorphism\index{Quasi-isomorphism}, i.e. if all the induced maps
		\[
		H_n(f):H_n(V)\to H_n(W)
		\]
		are isomorphisms, for all $n\in\mathbb{Z}$.
		\item A chain map is a \emph{fibration}\index{Fibrations!of chain complexes} if it is surjective in every degree.
		\item A chain map is a \emph{cofibration}\index{Cofibrations!of chain complexes} if it has the left lifting property with respect to all trivial fibrations.
	\end{enumerate}
	In particular, all chain complexes are fibrant.\index{Fibrant!chain complex}
\end{theorem}

Since we are working over a field, we have even more.

\begin{proposition}
	All chain complexes are cofibrant.\index{Cofibrant!chain complex}
\end{proposition}

\begin{proof}
	For $n\in\mathbb{Z}$, denote by
	\[
	D(n)\coloneqq\k x_n\oplus\k y_{n-1}\qquad\text{and}\qquad S(n)\coloneqq\k z_n
	\]
	the chain complexes with $|x_n| = |z_n| = n$, $|y_n| = n$, and $dx_n = y_{n-1}$. Since $\k$ is a field, every chain complex can be written as a colimit of copies of $D(n)$ and $S(n)$, with varying $n$. Since the colimit of cofibrant objects is cofibrant, it is enough to show that $D(n)$ and $S(n)$ are cofibrant. Fix any trivial fibration $f:V\to W$, i.e. a surjective chain map which induces an isomorphism in homology. Given
	\begin{center}
		\begin{tikzpicture}
			\node (a) at (0,0){$S(n)$};
			\node (b) at (2.5,0){$W$};
			\node (c) at (2.5,2){$V$};
			
			\draw[->] (a) to node[below]{$g$} (b);
			\draw[->>] (c) to  node[sloped,yshift=.1cm]{$\sim$} node[left]{$f$} (b);
			\draw[->,dashed] (a) to node[above left]{$h$} (c);
		\end{tikzpicture}
	\end{center}
	we have the existence of the dashed lift $h$. Indeed, $g(z_n)$ represents a class in the homology of $W$. Since $f$ is bijective in homology, there exists a closed element $v\in V$ such that $f(v) = g(z_n) + dw$, for some $w\in W_{n+1}$. Let $\widetilde{w}\in V_{n+1}$ be any preimage of $w$ under $f$, whose existence is guaranteed by the fact that $f$ is surjective. Then $h(z_n)\coloneqq v-d\widetilde{w}$ is a well-defined chain map lifting $g$.
	
	\medskip
	
	Similarly, given
	\begin{center}
		\begin{tikzpicture}
		\node (a) at (0,0){$D(n)$};
		\node (b) at (2.5,0){$W$};
		\node (c) at (2.5,2){$V$};
		
		\draw[->] (a) to node[below]{$g$} (b);
		\draw[->>] (c) to  node[sloped,yshift=.1cm]{$\sim$} node[left]{$f$} (b);
		\draw[->,dashed] (a) to node[above left]{$h$} (c);
		\end{tikzpicture}
	\end{center}
	the dashed lift $h$ exists. To see this, simply take any preimage $v$ of $g(x_n)$ under $f$, and set $h(x_n)\coloneqq v$, and $h(y_n)\coloneqq dv$.
\end{proof}

For details about this model structure, as well as references to the original literature, we invite the reader to take a look at \cite[Sect. 2.3]{Hovey}.

\subsection{Operads}

In his article \cite[Sect. 6]{hin97homological}, Hinich introduced a model structure on the category of operads in chain complexes. It is given as follows.

\begin{theorem}[{\cite[Thm. 6.1.1]{hin97homological}}]
	There is a closed model structure on the category $\op$ of operads\index{Model category!of operads}\index{Hinich model structure!on operads} with
	\begin{itemize}
		\item the arity-wise quasi-isomorphisms as weak equivalences\index{Weak equivalences!of operads}, and
		\item the arity-wise surjections as fibrations\index{Fibrations!of operads}.
	\end{itemize}
	The cofibrations\index{Cofibrations!of operads} are the morphisms that have the left lifting property with respect to the class of trivial fibrations. In particular, all operads are fibrant\index{Fibrant!operads} in this model structure.
\end{theorem}

We call this model structure the \emph{Hinich model structure} on operads.

\begin{proposition}[{\cite[Prop. 38\footnote{This is Proposition 95 in the arXiv version of the article.}]{mv09}}]
	The cofibrant\index{Cofibrant!operads} objects in the category of operads are the retracts of quasi-free operads\footnote{Recall that a quasi-free operad is an operad of the form $\T(M)$ for some $\S$-module $M$, endowed with some differential.} whose generating $\S$-module $M$ admits an exhaustive filtration
	\[
	M_0\coloneqq\{0\}\subseteq\F^1M\subseteq\F^2M\subseteq\cdots\subseteq M = \colim_n\F^nM
	\]
	such that $d(\F^iM)\subseteq\T(\F^{i-1}M)$ and such that the inclusions $\F^iM\xhookrightarrow{}\F^{i+1}M$ are split monomorphisms whose cokernel is isomorphic to a free $\S$-module.
\end{proposition}

The following corollary motivates various constructions seen in \cref{ch:operads}.

\begin{corollary}\label{cor:cobar-bar is cofibrant resolution}
	Let $\P$ be an operad. The bar-cobar resolution\index{Bar-cobar resolution!of operads} $\Cobar\Bar\P$ is a functorial cofibrant resolution of $\P$. Moreover, if $\P$ is Koszul, then the minimal resolution $\P_\infty$ of $\P$ also is a cofibrant resolution of $\P$.
\end{corollary}

\subsection{Cooperads}

There are various model structures one can put on the category of cooperads. A good example, in the same spirit as the Vallette model structure that we will see in \cref{subsection:Vallette model structure on coalgebras}, can be found in \cite[Sect. 3]{leg17}.

\subsection{Algebras over operads}\label{subsection:Hinich model structure on algebras}

A model structure on algebras over an operad was also given by Hinich in \cite[Sect. 4]{hin97homological}. Notice that, since our base field $\k$ has characteristic $0$, every operad is $\S$-split. Fix an operad $\P$.

\begin{theorem}[{\cite[Thm. 4.1.1]{hin97homological}}]
	There is a closed model structure on the category $\Palg$ of $\P$-algebras\index{Model category!of $\P$-algebras}\index{Hinich model structure!on $\P$-algebras} with
	\begin{itemize}
		\item the quasi-isomorphisms of $\P$-algebras as weak equivalences\index{Weak equivalences!of $\P$-algebras}, and
		\item the surjective morphisms of $\P$-algebras as fibrations\index{Fibrations!of $\P$-algebras}.
	\end{itemize}
	The cofibrations\index{Cofibrations!of $\P$-algebras} are the morphisms that have the left lifting property with respect to the class of trivial fibrations. In particular, all $\P$-algebras are fibrant\index{Fibrant!$\P$-algebras} in this model structure.
\end{theorem}

We call this model structure the \emph{Hinich model structure} on $\P$-algebras.

\begin{lemma}
	Let $\alpha:\C\to\P$ be a Koszul twisting morphism. Then every $\P$-algebra of the form $\Cobar_\alpha C$, where $C$ is a conilpotent $\C$-coalgebra, is cofibrant.
\end{lemma}

\begin{proof}
	This follows immediately from \cref{thm:Vallette model structure}.
\end{proof}

\subsection{The Vallette model structure for coalgebras over a cooperad}\label{subsection:Vallette model structure on coalgebras}

Let $\alpha:\C\to\P$ be an operadic twisting morphism. The Hinich model structure on $\P$-algebra can be transfered along the relative cobar functor $\Cobar_\alpha$ to give a closed model structure on $\C$-coalgebras. This was done by Vallette \cite{val14} in the case where $\P$ is a Koszul operad and $\alpha=\kappa$ is the twisting morphism given by Koszul duality, and then generalized by Drummond-Cole--Hirsch \cite{dch16} and Le Grignou \cite{leg16}.

\medskip

The main results are the following ones.

\begin{theorem}\label{thm:Vallette model structure}
	Let $\alpha:\C\to\P$ be an operadic twisting morphism.
	\begin{enumerate}
		\item\label{pt:all coalgebras are cofibrant} \emph{\cite[Thm. 2.1(1)]{val14} and \cite[Thm. 10 and Prop. 26]{leg16}} There is a closed model structure on the category $\cCcog$ of conilpotent $\C$-coalgebras\index{Model category!of conilpotent $\C$-coalgebras}\index{Vallette model structure} with
		\begin{itemize}
			\item the morphisms $g:C\to D$ of $\C$-coalgebras such that $\Cobar_\alpha g$ is a quasi-isomorphism as weak equivalences\index{Weak equivalences!of conilpotent $\C$-coalgebras}, and
			\item the injective morphisms of $\C$-coalgebras as cofibrations\index{Cofibrations!of conilpotent $\C$-coalgebras}.
		\end{itemize}
		The fibrations are the morphisms that have the right lifting property with respect to the class of trivial cofibrations. In particular, all $\C$-coalgebras are cofibrant in this model structure.\index{Cofibrant!conilpotent $\C$-coalgebras}
		\item\label{pt:Vms2} \emph{\cite[Prop. 32]{leg16}} If $\alpha'=f\alpha$ with $f:\P\to\P'$ a quasi-isomorphism of operads, then the model structure induced by $\alpha'$ coincides with the model structure induced by $\alpha$.
		\item\label{pt:Vms3} \emph{\cite[Thm. 2.1(2)]{val14} and \cite[Thm. 14]{leg16}} Let $\iota:\C\to\Cobar\C$ be the canonical twisting morphism. If $\alpha$ is Koszul, then the fibrant objects\index{Fibrant!conilpotent $\C$-coalgebras} are exactly the conilpotent $\C$-coalgebras isomorphic to a $\C$-coalgebra of the form $\Bar_\iota A$ for some $\Cobar\C$-algebra $A$.
		\item\label{pt:Vms4} Every $\C$-coalgebra of the form $\Bar_\alpha A$ for $A$ any $\P$-algebra is fibrant.
		\item \emph{\cite[Thm. 2.1(3)]{val14} and \cite[Thm. 13]{leg16}} The bar-cobar adjunction relative to $\iota:\C\to\Cobar\C$
		\[
		\adjunction{\Cobar_\iota}{\cCcog}{\Cobar\C\text{-}\mathsf{alg}}{\Bar_\iota}
		\]
		is a Quillen equivalence.
	\end{enumerate}
\end{theorem}

\begin{proof}
	The only point which might not be immediately clear from the references given above is the point (\ref{pt:Vms4}). By \cref{thm:fundamental thm of operadic twisting morphisms}, the twisting morphism $\alpha$ splits into $\alpha = g_\alpha\iota$. Then one has
	\[
	\Bar_\alpha A = \Bar_\iota(g_\alpha^*A)
	\]
	by \cref{lemma:equality of oo-morphisms}, and we conclude by point (\ref{pt:Vms3}).
\end{proof}

We call this model structure the \emph{Vallette model structure} on conilpotent $\C$-coalgebras. When $\alpha$ is Koszul\footnote{In which case, by \cref{thm:fundamental thm of operadic twisting morphisms} and \cref{thm:Vallette model structure}(\ref{pt:Vms2}), reduces to studying the model structure induced by the Koszul morphism $\iota:\C\to\Cobar\C$.}, we can also completely characterize the class of weak equivalences. All the accessory results we need were proven\footnote{For $\P$ a Koszul morphism and $\alpha=\kappa:\P^{\antishriek}\to\P_\infty$, but the proofs readily generalize to the slightly more general case we desire.} in \cite{val14}, but the conclusion is original work.

\begin{definition} \label{def:cofiltered coalgebras}
	Let $C$ be a conilpotent $\C$-coalgebra. A \emph{cofiltration}\index{Cofiltration on a $\C$-coalgebras} on $C$ is a sequence of sub-chain complexes
	$$0=\F^0C\subseteq \F^1C\subseteq \F^2C\subseteq\ldots\subseteq C$$
	satisfying \cite[Prop. 2.2]{val14}. Namely, we require that
	\begin{enumerate}
		\item it respects the coproduct, that is
		\[
		\overline{\Delta}_C(\F^nC)\subseteq\bigoplus_{\substack{k\ge1\\n_1+\cdot+n_k = n}}\left(\C(k)\otimes \F^{n_1}C\otimes\cdots\otimes \F^{n_k}C\right)^{\S_k},
		\]
		where $\overline{\Delta}_C(c)\coloneqq\Delta_C(c)-c$, and
		\item it is preserved by the differential: $d_C(\F^nC)\subseteq \F^nC$.
		\setcounter{counter}{\value{enumi}}
	\end{enumerate}
	Moreover, if
	\begin{enumerate}
		\setcounter{enumi}{\value{counter}}
		\item the cofiltration is exhaustive: $\colim_n \F^nC\cong C$,
	\end{enumerate}
	then we say that $C$ is a cocomplete coalgebra.
\end{definition}

\begin{remark}
	This is also called a filtered coalgebra in the literature. To attempt to have more clarity, we decided to call filtrations the descending filtrations and cofiltrations the ascending filtrations.
\end{remark}

\begin{example}
	A cofiltration that exists for any conilpotent $\P^{\antishriek}$-coalgebra $C$ is the coradical filtration introduced in \cref{subsection:(co)algebras over (co)operads}.
\end{example}

\begin{definition}
	Let $C_1,C_2$ be two conilpotent $\C$-coalgebras, and let $\phi:C_1\to C_2$ be a morphism of $\C$-coalgebras.
	\begin{enumerate}
		\item Let $\F^\bullet C_1$ and $\F^\bullet C_2$ be filtrations. The morphism $\phi$ is \emph{cofiltered}\index{Cofiltered!morphism of coalgebras} with respect to the cofiltrations if for each $n\ge0$ we have $\phi(\F^nC_1)\subseteq \F^nC_2$.
		\item The morphism $\phi$ is a \emph{cofiltered quasi-isomorphism}\index{Cofiltered!quasi-isomorphism of co\-al\-ge\-bras} if it is a quasi-isomorphism and if there are cocomplete cofiltrations $\F^\bullet C_1$ and $\F^\bullet C_2$ such that $\phi$ is cofiltered with respect to the cofiltrations, and such that for each $n\ge0$, the morphism $\phi$ induces a quasi-isomorphism
		\[
		\F^nC_1/\F^{n-1}C_1\longrightarrow \F^nC_2/\F^{n-1}C_2\ .
		\]
	\end{enumerate}
\end{definition}

\begin{remark}
	Notice that the coradical filtration is final, in the sense that if we put the coradical filtration on $C_1$ and any cofiltration on $C_2$ then any morphism of $\C$-coalgebras will be cofiltered.
\end{remark}

The following result was proven in \cite{rn18}, even though all the ingredient of the proof were already present in \cite{val14}.

\begin{theorem}[{\cite[Thm. 4.9]{rn18}}] \label{thm:characterization of we in the Vallette model structure}
	The class $W$ of weak equivalence\index{Weak equivalences!of conilpotent $\C$-coalgebras} is the smallest class of arrows of $\cCcog$ containing all cofiltered quasi-isomorphisms and which is closed under the $2$-out-of-$3$ property.
\end{theorem}

The proof of this theorem is similar to what found in \cite[Sect. 9.3]{pos11}. Before going on, we need a couple of preliminary results, all of which come from \cite{val14}.

\begin{lemma} \label{lemma:OmegaOfFQIisQI}
	Let $f$ be a cofiltered quasi-isomorphism of conilpotent $\C$-coalgebras. Then the morphism of $\Cobar\C$-algebras $\Omega_\iota f$ is a quasi-isomorphism.
\end{lemma}

\begin{proof}
	This follows from the proof of \cite[Prop. 2.3]{val14}.
\end{proof}

\begin{lemma} \label{lemma:BarOfQIisFQI}
	Let $f$ be a quasi-isomorphism of $\Cobar\C$-algebras. Then $\Bar_\iota f$ is a cofiltered quasi-iso\-mor\-phism of conilpotent $\C$-coalgebras.
\end{lemma}

\begin{proof}
	This is \cite[Prop. 2.4]{val14}.
\end{proof}

\begin{lemma} \label{lemma:counitIsFQI}
	Let $C$ be a conilpotent $\C$-coalgebra. Then the unit map
	$$\eta_C:\Bar_\iota\Cobar_\iota C\longrightarrow C$$
	is a cofiltered quasi-isomorphism.
\end{lemma}

\begin{proof}
	This is a consequence of the proof of \cite[Thm. 2.6]{val14}.
\end{proof}

\begin{proof}[Proof of \cref{thm:characterization of we in the Vallette model structure}]
	We denote by $\fqi$ the smallest class of arrows in $\cCcog$ which contains all cofiltered quasi-isomorphisms and which is closed under the $2$-out-of-$3$ property.
	
	\medskip
	
	\cref{lemma:OmegaOfFQIisQI} implies that that $\fqi\subseteq W$. To prove the other inclusion, let
	$$f:C\longrightarrow D$$
	be a morphism of conilpotent $\C$-coalgebras such that $\Omega_\iota f$ is a quasi-isomorphism, that is to say $f\in W$. We consider the diagram
	\begin{center}
		\begin{tikzpicture}
		\node (a) at (0,0){$C$};
		\node (b) at (3,0){$D$};
		\node (c) at (0,2){$\Bar_\iota\Cobar_\iota C$};
		\node (d) at (3,2){$\Bar_\iota\Cobar_\iota D$};
		
		\draw[->] (a)-- node[above]{$f$} (b);
		\draw[->] (c)-- node[above]{$\Bar_\iota\Cobar_\iota f$} (d);
		\draw[->] (c)-- node[left]{$\eta_C$} (a);
		\draw[->] (d)-- node[right]{$\eta_D$} (b);
		\end{tikzpicture}
	\end{center}
	where by \cref{lemma:BarOfQIisFQI} the arrow $\Bar_\iota\Cobar_\iota f$ is a filtered quasi-isomorphism, and by \cref{lemma:counitIsFQI} both vertical arrows are also cofiltered quasi-isomorphisms. A double application of the $2$-out-of-$3$ property proves that $f\in\fqi$, concluding the proof.
\end{proof}

The following result is a refinement of one direction of \cref{thm:alpha Koszul iff unit is qi}, and provides the "good" dual version to \cref{thm:bar-cobar resolution for algebras}.

\begin{corollary}\label{corollary:unit of bar-cobar is a we in the Vallette model structure}
	Let $\alpha:\C\to\P$ be a Koszul morphism. Then the unit
	\[
	\eta_C:C\longrightarrow\Bar_\alpha\Cobar_\alpha C
	\]
	of the bar-cobar adjunction relative to $\alpha$ is a weak equivalence in the Vallette model structure, for any conilpotent $\C$-coalgebra $C$. In particular, the bar-cobar adjunction provides a fibrant resolution for conilpotent $\C$-coalgebras.\index{Bar-cobar resolution!of coalgebras}
\end{corollary}

\begin{proof}
	The proof of \cite[Thm. 2.6]{val14} works in this case and shows that the unit $\eta_C$ is a cofiltered quasi-isomorphism. \cref{thm:characterization of we in the Vallette model structure} concludes the proof.
\end{proof}

\subsection{Algebras with \texorpdfstring{$\infty$}{infinity}-morphisms}\label{subsect:homotopy theory of algebras with oo-moprhisms}

Let $\P$ be a Koszul operad, and let $\P_\infty$ be its minimal model. The category of $\P_\infty$-algebra and their $\infty$-morphisms can be identified as the full subcategory of the category of conilpotent $\P^{\antishriek}$-coalgebras given by quasi-free $\P^{\antishriek}$-coalgebras via the bar functor $\Bar_\iota$, cf. \cref{prop:extended bar to oo-morphisms}. This way, one immediately sees that the category of $\P_\infty$-algebra and their $\infty$-morphisms cannot be a model category: it is not complete since for example products of quasi-free $\P^{\antishriek}$-coalgebras are not necessarily quasi-free, cf. \cref{prop:products of coalgebras}. However, one can use the Vallette model structure on coalgebras to speak of the homotopy theory of $\infty$-morphisms of $\P_\infty$-algebras. One simply defines two $\infty$-morphisms to be homotopic if they are when seen as morphisms of $\P^{\antishriek}$-coalgebras. More details of this theory are given in \cite[Sect. 3]{val14}.

\section{Framings}\label{sect:framings}

The idea of framings is to give a simplicial or cosimplicial resolution of an object in a category, i.e. a ``nice'' simplicial, resp. cosimplicial, object in that category whose $0$-simplices are given by the original object we wanted to study. We present here the theory of simplicial framings, as it is what we will need later. Cosimplicial framings are the dual concept. We assume some knowledge about simplicial sets, see also \cref{ch:simplicial homotopy theory}. The material presented here is extracted from \cite[Ch. 5]{Hovey}.

\subsection{The model structures on simplicial objects}

Let $\cat$ be a category. Recall that a simplicial object is a functor
\[
X_\bullet:\Delta^\mathrm{op}\longrightarrow\cat
\]
from the ordinal number category to $\cat$. Simplicial objects in $\cat$ form a category $\mathsf{s}\cat$, whose morphisms are given by the natural transformations.

\begin{definition}
	Let $\cat$ be a category with all finite limits and colimits. Let $X_\bullet\in\mathsf{s}\cat$ be a simplicial object in $\cat$, and let $n\ge0$ be an integer.
	\begin{enumerate}
		\item The \emph{$n$th latching object}\index{Latching object} $L_nX_\bullet\in\cat$ of $X_\bullet$ is the union of all the degenerate $n$-cells of $X_\bullet$. In particular, $L_0X_\bullet$ is the initial object, and $L_1X_\bullet = X_0$.
		\item The \emph{$n$th matching object}\index{Matching object} $M_nX_\bullet\in\cat$ of $X_\bullet$ is $M_nX_\bullet\coloneqq K_\bullet^{\partial\Delta[n]}$, i.e. the limit of $X_\bullet$ over the diagram of all simplices of $\partial\Delta[n]$. In particular, $M_0X_\bullet$ is the terminal object, and $M_1X_\bullet = X_0\times X_0$.
	\end{enumerate}
	For each $n\ge0$, we have natural maps $L_nX_\bullet\to X_n\to M_nX_\bullet$. Moreover, any morphism of simplicial objects induces natural morphisms between the latching and matching objects.
\end{definition}

One can use latching and matching objects to define a model structure on the category of simplicial objects in a model category.

\begin{theorem}
	Let $\cat$ be a model category. The following classes of maps define a model structure --- the \emph{Reedy model structure}\index{Reedy model structure}\index{Model category!of simplicial objects} --- on the category $\mathsf{s}\cat$ of simplicial objects in $\cat$.
	\begin{itemize}
		\item The weak equivalences\index{Weak equivalences!of simplicial objects} are the level-wise weak equivalences. That is to say that a morphism $f_\bullet:X_\bullet\to Y_\bullet$ is a weak equivalence if $f_n:X_n\to Y_n$ is a weak equivalence in $\cat$ for every $n\ge0$.
		\item The (trivial) cofibrations\index{Cofibrations!of simplicial objects} are the morphisms $f_\bullet:X_\bullet\to Y_\bullet$ such that the induced morphisms
		\[
		X_i\sqcup_{L_iX_\bullet}L_iY_\bullet\longrightarrow Y_i
		\]
		are (trivial) cofibrations for all $n\ge0$.
		\item The (trivial) fibrations\index{Fibrations!of simplicial objects} are the morphisms $f_\bullet:X_\bullet\to Y_\bullet$ such that the induced morphisms
		\[
		X_i\longrightarrow Y_i\times_{M_iY_\bullet}M_iX_\bullet
		\]
		are (trivial) fibrations for all $n\ge0$.
	\end{itemize}
\end{theorem}

\subsection{Simplicial framings}

Let $\cat$ be a model category. Suppose we are given an object $x\in\cat$, which we would like to see as $X_0$ for a simplicial object $X_\bullet\in\mathsf{s}\cat$. There are two natural choices for the whole object $X_\bullet$.
\begin{enumerate}
	\item One defines $\ell_\bullet x\in\mathsf{s}\cat$ by
	\[
	\ell_nx\coloneqq x
	\]
	for all $n$, with the identity map for all boundary and degeneracy maps.
	\item One defines $r_\bullet x\in\mathsf{s}\cat$ by
	\[
	r_nx\coloneqq x\times\cdots\times x\ ,
	\]
	the product of $n+1$ copies of $x$. The face maps are given by the projections forgetting one of the factors, and the degeneracy maps are induced by the diagonal map $x\to x\times x$.
\end{enumerate}

\begin{definition}
	A \emph{simplicial frame}\index{Simplicial!frame} $X_\bullet$ for $x\in\cat$ is a factorization
	\begin{center}
		\begin{tikzpicture}
			\node (a) at (0,0){$\ell_\bullet x$};
			\node (b) at (2,0){$X_\bullet$};
			\node (c) at (4,0){$r_\bullet x$};
			
			\draw[->] (a) to node[sloped,yshift=.1cm]{$\sim$} (b);
			\draw[->>] (b) to (c);
		\end{tikzpicture}
	\end{center}
	of the natural map $\ell_\bullet x\to r_\bullet x$ into a weak equivalence followed by a fibration.
\end{definition}

Notice that simplicial frames always exist by (\ref{axiom:factorization}). We will need the following result in \cref{ch:model structure for GM}.

\begin{proposition}\label{prop:framings}
	Suppose $x$ is a fibrant object in $\cat$ and let $X_\bullet$ be a simplicial frame on $x$. Then the functor
	\[
	\hom_\cat(-,X_\bullet):\cat^{op}\longrightarrow\ssets
	\]
	preserves fibrations, trivial fibrations, and weak equivalences between fibrant objects.
\end{proposition}

\begin{proof}
	The fact that the functor preserves fibrations and trivial fibrations is proven in \cite[Cor. 5.4.4(2)]{Hovey}. The fact that it preserves weak equivalences between fibrant objects then follows by Ken Brown's lemma, \cref{lemma:Ken Brown}.
\end{proof}

%% file: SimplicialHomotopy.tex
\chapter{Simplicial homotopy theory}\label{ch:simplicial homotopy theory}

Simplicial sets are combinatorial objects that can be used to study the homotopy theory of topological spaces. In fact, they carry a model structure with which they are Quillen equivalent to topological spaces. Therefore, in modern algebraic topology they are often used as models for spaces instead of topological spaces themselves.

\medskip

In this chapter, we begin by giving a rapid definition of the category of simplicial sets. Then we present the model structure on simplicial sets and the Quillen equivalence between simplicial sets and topological spaces. Finally, we give a rapid overview of the Dold--Kan correspondence.

\medskip

The material of this chapter is extracted from \cite{GoerssJardine}, and in particular Chapters I and III in \emph{op. cit.}

\section{Simplicial sets}

We begin by giving an introduction to simplicial sets as the category of presheaves of sets, i.e. the category of contravariant functors from the ordinal number category $\Delta$ to the category $\sets$ of sets.

\subsection{The ordinal number category}

The most basic object encoding all of the combinatorial information of simplicial sets is the following category.

\begin{definition}
	The \emph{ordinal number category}\index{Ordinal number category} $\Delta$ is the category whose objects are the ordered sets
	\[
	[n]\coloneqq\{0<1<\cdots<n\}\ ,
	\]
	for $n\in\mathbb{N}$, and whose morphisms are the order preserving\footnote{Not strictly, meaning that if $i<j$, then we only require that such a map $\phi$ satisfies $\phi(i)\le\phi(j)$.} maps.
\end{definition}

There are two classes of special morphisms in the ordinal number category.

\begin{definition}
	Let $n\ge1$. The \emph{$i$th coface map}\index{Coface maps}, $0\le i\le n$, is the unique morphism
	\[
	d^i:[n-1]\longrightarrow[n]
	\]
	which is injective and skips $i$. The \emph{$j$th codegeneracy map}\index{Codegeneracy maps}, $0\le j\le n$, is the unique morphism
	\[
	s^j:[n]\longrightarrow[n+1]
	\]
	which is surjective and takes the value $j$ twice.
\end{definition}

The cofaces and codegeneracies generate all morphisms in $\Delta$, in the sense that any morphism in $\Delta$ can be written in terms of compositions of cofaces and codegeneracies. They satisfy certain relations between them, commonly called the \emph{cosimplicial identities}, see \cite[p. 4]{GoerssJardine}.

\subsection{Simplicial objects in a category and simplicial sets}

Functors from the ordinal number category $\Delta$ and its opposite category to any category have a nice combinatorial structure.

\begin{definition}
	Let $\cat$ be a category. A \emph{simplicial object}\index{Simplicial!objects} in $\cat$ is a functor
	\[
	\Delta^\mathrm{op}\longrightarrow\cat\ .
	\]
	The category $\mathsf{s}\cat$ of simplicial objects in $\cat$ is the category whose objects are simplicial objects in $\cat$, and whose morphisms are the natural transformations between simplicial objects.
\end{definition}

\begin{definition}
	Dually, a \emph{cosimplicial object}\index{Cosimplicial objects} in $\cat$ is a functor
	\[
	\Delta\longrightarrow\cat\ .
	\]
\end{definition}

The case of interest to us for the moment is when we take $\cat = \sets$.

\begin{definition}
	The category $\ssets$ of \emph{simplicial sets}\index{Simplicial!sets} is the category of simplicial objects in the category of sets.
\end{definition}

Let $K_\bullet\in\ssets$ be a simplicial set. We denote by
\[
K_n\coloneqq K_\bullet([n])
\]
the \emph{set of $n$-simplices} of $K_\bullet$. The cofaces and codegeneracies of the category $\Delta$ naturally induce maps between the various sets of simplices of a simplicial set, called the \emph{face maps}\index{Face maps} $d_i$ and \emph{degeneracy maps}\index{Degeneracy maps} $s_j$ respectively. A simplex in $K_n$ is \emph{degenerate} if it is in the image of a degeneracy map.

\subsection{Limits and colimits}

Limits and colimits are well behaved in the category of simplicial sets.

\begin{proposition}
	The category $\ssets$ of simplicial sets is complete and cocomplete, and limits and colimits are taken level-wise.
\end{proposition}

In other words, if $I$ is an index category, and $L:I\to\ssets$ is a functor, then
\[
\left(\lim_{i\in I}L(i)\right)_n = \lim_{i\in I}L(i)_n\ ,
\]
and similarly for colimits. In particular, the initial object in the category of simplicial sets is the empty simplicial set $\emptyset$, and the final object is the point $*\coloneqq\Delta[0]$.

\subsection{The Yoneda embedding, boundaries and horns}

The Yoneda embedding embeds $\Delta$ as a subcategory of $\ssets$.

\begin{definition}
	We define a functor
	\[
	\Delta[-]:\Delta\longrightarrow\ssets
	\]
	by
	\[
	\Delta[n]\coloneqq\hom_\Delta(-,[n])\ .
	\]
\end{definition}

The simplicial sets $\Delta[n]$ are the ``building blocks" of all simplicial sets.

\begin{lemma}
	Let $K_\bullet$ be a simplicial set. Then
	\[
	K_n\cong\hom_{\ssets}(\Delta[n],K_\bullet)\ .
	\]
\end{lemma}

\begin{proof}
	This follows immediately from the Yoneda lemma, e.g. \cite[p. 61]{MacLane}.
\end{proof}

There are some other very important simplicial sets that can be built from the $\Delta[n]$.

\begin{definition}
	The \emph{boundary}\index{Boundary of a simplicial set} $\partial\Delta[n]$ of the simplicial set $\Delta[n]$ is the simplicial set obtained by gluing all of the non-degenerate $(n-1)$-simplices of $\Delta[n]$ along of their faces, whenever the faces of two such $(n-1)$-simplices coincide. In other words, it is given as the coequalizer
	\[
	\bigsqcup_{[n-2]\to[n]\text{ inj.}}\Delta[n-2]\rightrightarrows\bigsqcup_{[n-1]\to[n]\text{ inj.}}\Delta[n-1]\longrightarrow\partial\Delta[n]\ .
	\]
\end{definition}

\begin{definition}
	The \emph{$i$th horn}\index{Horns of a simplicial set} $\Lambda^i[n]$ of $\Delta[n]$ is the simplicial set defined the same way as $\partial\Delta[n]$ but without the $(n-1)$-simplex opposite to the $i$th $0$-simplex of $\Delta[n]$. In other words, it is given as the coequalizer
	\[
	\bigsqcup_{\substack{[n-2]\to[n]\text{ inj.}\\i\text{ in the image}}}\Delta[n-2]\rightrightarrows\bigsqcup_{\substack{[n-1]\to[n]\text{ inj.}\\i\text{ in the image}}}\Delta[n-1]\longrightarrow\Lambda^i[n]\ .
	\]
\end{definition}

\section{The model structure on simplicial sets}

In this section, we present the model structure on simplicial sets. The main theorem (which we will not prove here) is the following.

\begin{theorem}
	The classes of cofibrations, fibrations, and weak equivalences that will be defined in \cref{def:cofibrations of ssets,def:fibrations of ssets,def:we of ssets} define a closed model structure on the category $\ssets$ of simplicial sets.\index{Model category!of simplicial sets}
\end{theorem}

\subsection{Cofibrations}

Cofibrations are easy.

\begin{definition}\label{def:cofibrations of ssets}
	A morphism $f:K_\bullet\to L_\bullet$ of simplicial sets is a \emph{cofibration}\index{Cofibrations!of simplicial sets} if it is level-wise injective, i.e. if all of the maps of sets $f_n:K_n\to L_n$ are injective.
\end{definition}

In particular, all simplicial sets are cofibrant.

\subsection{Fibrations and Kan complexes}

Fibrations are defined by right lifting property with respect to a certain set of generating cofibrations, namely the inclusions of horns.

\begin{definition}\label{def:fibrations of ssets}
	A morphism $f:K\to L$ of simplicial sets is a \emph{fibration}\index{Fibrations!of simplicial sets} if it satisfies the right lifting property with respect to all the natural inclusions $\Lambda^i[n]\xhookrightarrow{}\Delta[n]$ for $0\le i\le n$ and all $n\ge2$.
\end{definition}

\begin{definition}
	Fibrant simplicial sets are called \emph{Kan complexes}.\index{Kan complexes}
\end{definition}

Many naturally arising simplicial sets are Kan. Examples are the nerve of a group, the underlying simplicial set of a simplicial abelian group, and the singular complex of a topological space --- which we will see in more detail later. This last fact motivates the heuristic that Kan complex correspond to ``spaces".

\subsection{Realization and weak equivalences}

Let $\cghaus$ be the category of compactly generated Hausdorff\footnote{Also known as $T_2$.} topological spaces. Homotopically speaking, the category $\cghaus$ is essentially the same as the whole category $\Top$ of topological spaces, since CW-complexes are in $\cghaus$. However, notice that for example products in $\cghaus$ are not the same as products in $\Top$.

\begin{definition}
	For $n\ge0$, the \emph{geometric $n$-simplex}\index{Geometric!simplices} is the topological space
	\[
	\Delta^n\coloneqq\left\{(x_0,\ldots,x_n)\ \vline\ \sum_{i=0}^nx_i = 1\text{ and }x_i\ge0\ \forall0\le i\le n\right\}\subset\mathbb{R}^{n+1}\ .
	\]
\end{definition}

Notice that $\Delta^n\in\cghaus$. The topological space $\Delta^n$ is the geometric analogue of the simplicial set $\Delta[n]$, and since all simplicial sets can be constructed by gluing spaces of the form $\Delta[n]$ together, we try to recover topological spaces by gluing geometric simplices together.

\begin{definition}
	The \emph{geometric realization}\index{Geometric!realization} is the functor
	\[
	|-|:\ssets\longrightarrow\cghaus
	\]
	given by sending a simplicial set $K\in\ssets$ to the topological space
	\[
	|K|\coloneqq\colim_{\Delta[n]\to K}\Delta^n.
	\]
\end{definition}

In other words, given $K\in\ssets$ we take one copy of $\Delta^n$ for each $n$-simplex in $K$, and then glue all those geometric simplices together according to the face maps in $K$.

\begin{definition}\label{def:we of ssets}
	A morphism $f:K\to L$ of simplicial sets is a \emph{weak equivalence}\index{Weak equivalences!of simplicial sets} if its geometric realization
	\[
	|f|:|K|\longrightarrow|L|
	\]
	is a weak equivalence of simplicial sets, i.e. if it induces an isomorphism on all homotopy groups.
\end{definition}

\subsection{Quillen equivalence between simplicial sets and topological spaces}

The geometric realization functor has a right adjoint, the \emph{singular set} functor\index{Singular set functor}
\[
S_\bullet:\cghaus\longrightarrow\ssets\ ,
\]
which is defined as
\[
S_\bullet(X)\coloneqq\hom_\Top(\Delta^\bullet,X)\ ,
\]
the simplicial set of singular simplices of a topological space $X\in\cghaus$.

\begin{proposition}
	Let $K\in\ssets$ and $X\in\cghaus$. There is a natural isomorphism
	\[
	\hom_\Top(|K|,X)\cong\hom_\ssets(K,S_\bullet(X))\ .
	\]
\end{proposition}

There is more. This adjunction induces an equivalence of categories between the homotopy categories of simplicial sets and compactly generated, Hausdorff space, showing that their homotopy theory is ``the same''.

\begin{theorem}
	The adjunction given by the geometric realization and the singular set functor is a Quillen equivalence.
\end{theorem}

\begin{proof}
	See e.g. \cite[Thm. 3.6.7]{Hovey}.
\end{proof}

This motivates the fact that one can usually replace topological spaces by simplicial sets when doing homotopy theory, in order to work with objects that have a nicer combinatorial behavior.

\subsection{Homotopy groups of simplicial sets}\index{Homotopy!groups of simplicial sets}

Let $K$ be a Kan complex, i.e. a fibrant simplicial set.  Define $\pi_0(K)$ to be the set of homotopy classes of vertices of $K$. Let $x\in K_0$ be a vertex of $K$. For $n\ge1$, we define $\pi_n(K,x)$ to be the set of homotopy classes rel $\partial\Delta[n]$ of maps $\alpha:\Delta[n]\to K$ fitting into the diagram
\begin{center}
	\begin{tikzpicture}
		\node (a) at (0,0){$\partial\Delta[n]$};
		\node (b) at (2.5,0){$\Delta[0]$};
		\node (c) at (0,-2){$\Delta[n]$};
		\node (d) at (2.5,-2){$K$};
		
		\draw[->] (a) to (b);
		\draw[right hook->] (a) to (c);
		\draw[->] (b) to node[right]{$x$} (d);
		\draw[->] (c) to node[below]{$\alpha$} (d);
	\end{tikzpicture}
\end{center}
In other words, $\alpha$ must map the boundary $\partial\Delta[n]$ of $\Delta[n]$ to the basepoint $x$. The sets $\pi_n(K,x)$, $n\ge1$, are endowed with a group structure as follows. Let $\alpha,\beta:\Delta[n]\to K$ be as above. Then one constructs a morphism of simplicial sets
\[
(\alpha,\beta):\Lambda^n[n+1]\longrightarrow K
\]
by setting it to be constant with value $x$ on all faces, except the $(n-1)$th face, where it has value $\alpha$, and the $(n+1)$th, where it has value $\beta$. Since $K$ is Kan, we have a lift
\begin{center}
	\begin{tikzpicture}
		\node (a) at (0,0){$\Lambda^n[n+1]$};
		\node (b) at (3,0){$K$};
		\node (c) at (0,-2){$\Delta[n+1]$};
		
		\draw[->] (a) to node[above]{$(\alpha,\beta)$} (b);
		\draw[right hook->] (a) to (c);
		\draw[->,dashed] (c) to node[below right]{$\omega$} (b);
	\end{tikzpicture}
\end{center}
It is straightforward to check that $d_n\omega$, the ``missing'' $n$-face of $\Lambda^n[n+1]$ recovered by filling, is constant on its boundary with value $x$, and thus defines an element of $\pi_n(K,x)$, which we will denote by $\alpha*\beta$. The following lemma tells us that $\alpha*\beta$ is well defined as an element of $\pi_n(K,x)$.

\begin{lemma}
	The homotopy class of $\alpha*\beta$ is independent of the choice of the lift $\omega$.
\end{lemma}

We also define $e\in\pi_n(K,x)$ to be the constant map with value $x$.

\begin{theorem}
	Taking $*$ as multiplication, and $e$ as identity element, the $\pi_n(K,x)$ are groups for all $n\ge1$, and $\pi_n(K,x)$ are abelian for every $n\ge2$. We call $\pi_n(K,x)$ the \emph{$n$th homotopy group of $K$} with basepoint $x$.
\end{theorem}

We have the following very useful fact.

\begin{theorem}
	There is a canonical isomorphism
	\[
	\pi_0(K)\cong\pi_0(|K|)\ ,\qquad\text{and}\qquad\pi_n(K,x)\cong\pi_n(|K|,x)
	\]
	for all $n\ge1$. In other words, the homotopy groups of a Kan complex are the same as the homotopy groups of its geometric realization.
\end{theorem}

\subsection{The long exact sequence associated to a fibration}\label{subsect:LES of homotopy groups}

It is possible to associate a long exact sequence of homotopy groups to a fibration of simplicial sets.

\begin{definition}
	Let $p:K\to L$ be a fibration of simplicial sets. The \emph{fibre}\index{Fibre} of $p$ over a point $x\in L_0$ is $i:F\to K$ defined by the pullback
	\begin{center}
		\begin{tikzpicture}
			\node (a) at (0,0){$F$};
			\node (b) at (2.5,0){$K$};
			\node (c) at (0,-2){$\Delta[0]$};
			\node (d) at (2.5,-2){$L$};
			
			\draw[->] (a) to node[above]{$i$} (b);
			\draw[->>] (b) to node[right]{$p$} (d);
			\draw[->] (a) to (c);
			\draw[->] (c) to node[above]{$x$} (d);
		\end{tikzpicture}
	\end{center}
\end{definition}

Suppose that $p:K\to L$ is a fibration with fibre $i:F\to K$ over $x\in L_0$. Fix a point $y\in F_0$, and by abuse of notation denote $i(y)\in K_0$ also by $y$. Let $\alpha:\Delta[n]\to L$ represent a homotopy class $[\alpha]\in\pi_n(L,x)$, and define an $n$-horn $\Lambda^0[n]\to L$ by sending everything to $y$. We have a lift
\begin{center}
	\begin{tikzpicture}
	\node (a) at (0,0){$\Lambda^0[n]$};
	\node (b) at (2.5,0){$K$};
	\node (c) at (0,-2){$\Delta[n]$};
	\node (d) at (2.5,-2){$L$};
	
	\draw[->] (a) to (b);
	\draw[->>] (b) to node[right]{$p$} (d);
	\draw[>->] (a) to node[sloped, yshift=.1cm]{$\sim$} (c);
	\draw[->] (c) to node[above]{$\alpha$} (d);
	\draw[->,dashed] (c) to node[above left]{$\theta$} (b);
	\end{tikzpicture}
\end{center}
and one can check that the element $\partial(\alpha)\coloneqq[d_0\theta]\in\pi_{n-1}(F,y)$ does not depend on the choice of the particular lift. Thus, we obtain a map
\[
\partial:\pi_n(L,x=f(y))\longrightarrow\pi_{n-1}(F,y)\ ,
\]
called the \emph{boundary map}.

\begin{theorem}\label{thm:LES of a fibration}\index{Long exact sequence of a short exact sequence}
	For all $n\ge1$, the maps
	\[
	\partial:\pi_n(L,x)\longrightarrow\pi_{n-1}(F,y)
	\]
	are group homomorphisms. They fit in a long sequence
	\begin{align*}
		\cdots\xrightarrow{\partial}\pi_n(F,y)\xrightarrow{i}\pi_n(K,y)&\xrightarrow{p}\pi_n(L,x)\xrightarrow{\partial}\pi_{n-1}(F,y)\xrightarrow{i}\cdots\\
		\cdots&\xrightarrow{p}\pi_1(L,x)\xrightarrow{\partial}\pi_0(F,y)\xrightarrow{i}\pi_0(K,y)\xrightarrow{p}\pi_0(L,x)
	\end{align*}
	which is exact in the sense that kernel equals image everywhere. Furthermore, $\pi_1(L,x)$ acts on $\pi_0(F)$, and two elements of $\pi_0(F)$ have the same image under $i$ if, and only if they are in the same orbit of the $\pi_1(L,x)$-action.
\end{theorem}

\section{The Dold--Kan correspondence}

The celebrated Dold--Kan correspondence gives an equivalence of categories between simplicial abelian groups and non-negatively graded chain complexes. In this section, we give a rapid overview of this very important result.

\subsection{The normalized chain complex and the Moore complex of a simplicial abelian group}

We start by reminding an important fact about simplicial groups, due to \cite[Thm. 3]{moo54}.

\begin{theorem}[Moore]
	Let $G$ be a simplicial group. Then the underlying simplicial set of $G$ is a Kan complex.
\end{theorem}

For the rest this section, $A$ will always denote a simplicial abelian group.

\begin{definition}
	The \emph{normalized chain complex}\index{Normalized chain complex} $NA$ of a simplicial abelian group $A$ is the non-negatively graded chain complex which is given in degree $n\ge0$ by
	\[
	NA_n\coloneqq\bigcap_{i=0}^{n-1}\ker(d_i:A_n\to A_{n-1})\subseteq A_n\ ,
	\]
	and whose differential is given by $d_n$.
\end{definition}

\begin{remark}
	For this section only, ``chain complex'' means chain complex of $\mathbb{Z}$-modules. Later, we will apply these constructions to simplicial $\k$-vector spaces instead of just simplicial abelian groups, thus obtaining non-negatively graded $\k$-chain complexes.
\end{remark}

It is straightforward to check that $NA$ is a chain complex, i.e. that its differential squares to zero. This assignment is functorial, giving
\[
N:\mathsf{sAb}\longrightarrow\chain_+
\]
from simplicial abelian groups to non-negatively graded chain complexes.

\medskip

Another closely related construction is the Moore chain complex.

\begin{definition}
	The \emph{Moore complex}\index{Moore complex} $\mathcal{M}(A)$ of a simplicial abelian group $A$ is the non-negatively graded chain complex given in degree $n\ge0$ by
	\[
	\mathcal{M}(A)_n\coloneqq A_n\ ,
	\]
	and whose differential is given by $\sum_{i=0}^n(-1)^id_i$.
\end{definition}

The relation between these constructions is given by the following result.

\begin{theorem}
	The natural inclusion
	\[
	i:NA\longrightarrow\mathcal{M}(A)
	\]
	is a chain homotopy equivalence.
\end{theorem}

Therefore, the normalized chain complex and the Moore complex can be used interchangeably when doing homotopy theory. Moreover, these constructions allow one to readily compute the homotopy groups of $A$.

\begin{theorem}\label{thm:homotopy of simplicial vector space is homology of Moore complex}
	Let $A$ be a simplicial abelian group. There are natural isomorphisms of groups
	\[
	\pi_0(A)\cong H_0(\mathcal{M}(A))\ ,\qquad\text{and}\qquad\pi_n(A,0)\cong H_n(\mathcal{M}(A))
	\]
	for all $n\ge1$.
\end{theorem}

\subsection{The Dold--Kan correspondence}\label{subsect:Dold-Kan correspondence}

This chapter would not be complete without at least mentioning the Dold--Kan correspondence.

\medskip

There is a functor
\[
\Gamma:\chain_+\longrightarrow\mathsf{sAb}
\]
defined by
\[
\Gamma_n(C)\coloneqq\bigoplus_{[n]\to[k]\text{ surj.}}C_k
\]
and endowed with certain natural simplicial structure maps.

\medskip

The following celebrated result is due to A. Dold and D. Kan, and is widely known as the Dold--Kan correspondence.

\begin{theorem}[Dold--Kan]\label{thm:Dold-Kan}\index{Dold--Kan correspondence}
	The functors
	\[
	N:\mathsf{sAb}\longrightarrow\chain_+\qquad\text{and}\qquad\Gamma:\chain_+\longrightarrow\mathsf{sAb}
	\]
	form an equivalence of categories.
\end{theorem}

%% file: RationalHomotopyTheory.tex
\chapter{Rational homotopy theory}\label{chapter:RHT}

Studying the homotopy type of topological spaces is a very difficult problem --- even for well-behaved ones: suffices to say that at the present time we don't know all the homotopy groups of spheres. The idea of rational homotopy theory is to simplify the problem by studying the \emph{rational} homotopy groups of (nice) spaces. If $X$ is a simply connected space, then all of its homotopy groups are abelian (since the fundamental group is trivial), so one considers
\[
\pi_n(X)\otimes_\mathbb{Z}\mathbb{Q}\quad\text{for}\quad n\ge2\ .
\]
These groups are much easier to compute and to study than the homotopy groups $\pi_n(X)$. They are closely related to the rational homology groups $H_n(X;\mathbb{Q})$ of the space, the relationship being stricter than the one between the usual homotopy groups and integer homology, and it is possible to give algebraic models for the spaces that completely encode their rational homotopy theory.

\medskip

In other words, rational homotopy theory is ``the study of the rational homotopy category, that is the category obtained from the category of $1$-connected pointed spaces by localizing with respect to the family of those maps which are isomorphisms modulo the class in the sense of Serre of torsion abelian groups" --- \cite[p. 205]{Quillen}.

\medskip

The author does not make any claims of being an expert on the domain of rational homotopy theory. This section is a naive introduction to the subject, based on the reference books \cite{Quillen} and \cite{FelixHalperinThomas}.

\section{Conventions and basic definitions}

For the rest of this chapter, the base field will be the field of rational numbers $\k=\mathbb{Q}$. All vector spaces, chain complexes, algebras, etc. will always be over this field.

\subsection{The rational homotopy category}

We place ourselves in the category $\Top_{*,1}$ of simply-connected, pointed topological spaces and continuous pointed maps\footnote{This category is denoted by $\Top_2$ in \cite{Quillen}, for ``spaces beginning in dimension $2$".}.

\begin{definition}
	Let $X\in\Top_{*,1}$. The \emph{rational homotopy}\index{Rational!homotopy} of $X$ is the graded $\mathbb{Q}$-vector space $\pi_\bullet(X)\otimes_{\mathbb{Z}}\mathbb{Q}$ starting in degree $2$. The \emph{rational homology}\index{Rational!homology} of $X$ is the graded $\mathbb{Q}$-vector space\footnote{Also starting in degree $2$ by the Hurewicz theorem.} $H_\bullet(X;\mathbb{Q})$.
\end{definition}

In this category, maps inducing isomorphisms in rational homotopy and maps inducing isomorphisms in rational homology are the same, as proven by Serre in \cite[Thm. 3]{ser53}.

\begin{theorem}[Serre]\label{thm:Serre rational equivalences}
	Let $f:X\to Y$ be a map in $\Top_{*,1}$. The following are equivalent.
	\begin{enumerate}
		\item The map $f$ induces an isomorphism
		\[
		\pi_\bullet(f)\otimes_{\mathbb{Z}}\mathbb{Q}:\pi_\bullet(X)\otimes_{\mathbb{Z}}\mathbb{Q}\longrightarrow\pi_\bullet(Y)\otimes_{\mathbb{Z}}\mathbb{Q}
		\]
		in rational homotopy.
		\item The map $f$ induces an isomorphism
		\[
		H_\bullet(f;\mathbb{Q}):H_\bullet(X;\mathbb{Q})\longrightarrow H_\bullet(Y;\mathbb{Q})
		\]
		in rational homology.
	\end{enumerate}
\end{theorem}

\begin{definition}
	A map $f$ in $\Top_{*,1}$ inducing an isomorphism in rational homotopy or, equivalently, in rational homology is called a \emph{rational equivalence}\index{Rational!equivalence}. We denote by $\rhcat$ --- the \emph{rational homotopy category}\index{Rational!homotopy category} --- the localization of $\Top_{*,1}$ at the rational equivalences.
\end{definition}

The category $\Top_{*,1}$ cannot be made into a model category for a trivial reason: it is not complete. However, if one takes the set of rational equivalences as ``weak equivalences" and define appropriate sets of ``fibrations" and ``cofibrations", then the resulting category behaves very similarly to a model category, see \cite[Thm. II.6.1(a)]{Quillen}. In particular, $\rhcat$ is a well defined category.

\medskip

One similarly defines a category $\ssets_{1}$ given by the full subcategory of $\ssets$ spanned by $2$-reduced simplicial sets, i.e. the simplicial sets that have a single $0$-simplex and a single $1$-simplex. This category admits a model structure where the weak equivalences are the rational equivalences, making it into a model category, see \cite[Thm. II.2.2]{Quillen}. The homotopy theory of this model category is then equivalent to the homotopy theory of $\Top_{*,1}$, with one of the equivalences being given by geometric realization of simplicial sets, see \cite[Thm. II.6.1(b)]{Quillen}.

\begin{remark}
	Because of this equivalence, we will often speak of ``spaces" without specifying if we mean topological spaces or simplicial sets.
\end{remark}

The category $\rhcat$ has the same objects as $\Top_{*,1}$, but morphisms behave a bit differently. For example, if two morphisms $f,g:X\to Y$ are homotopic in $\Top_{*,1}$, then they are the same morphism in $\rhcat$. However, it is not true that two maps $f,g:X\to Y$ in $\Top_{*,1}$ inducing the same maps in rational homotopy or rational homology are identified in $\rhcat$.

\subsection{Conventions on cocommutative coalgebras and Lie algebras}

We will mainly be dealing with algebras and coalgebras representing the singular chain complexes of topological spaces. Therefore, we consider the category $\chain_{\ge0}$ of chain complexes concentrated in degree greater or equal than $0$ instead of the category $\chain$.

\medskip

For Lie algebras and cocommutative coalgebras, we consider
\begin{itemize}
	\item the category $\dgl_{\ge1}$ of Lie algebras concentrated in degrees greater or equal than $1$, and
	\item the category $\cocom_{\ge2}$ of cocommutative coalgebras concentrated in degrees greater or equal than $2$.
\end{itemize}

\begin{remark}
	In \cite{Quillen}, one considers \emph{counital, coaugmented} cocommutative coalgebras, but directly looking at non-counital cocommutative coalgebras is equivalent.
\end{remark}

By the K\"unneth formula, the homology of an object in $\dgl_{\ge1}$, respectively $\cocom_{\ge2}$, is again in $\dgl_{\ge1}$, respectively $\cocom_{\ge2}$, albeit with trivial differential. A morphism in $\dgl_{\ge1}$ or $\cocom_{\ge2}$ is a \emph{weak equivalence} if it is a quasi-isomorphism. One can also define classes of fibrations and cofibrations, making both categories into model categories, cf. \cite[Thm. 5.1 and 5.2]{Quillen}.

\begin{remark}
	Notice that this is perfectly coherent with the model structures on algebras and coalgebras we looked at in \cref{ch:model categories}. The quasi-isomorphisms are exactly the weak equivalences in the category algebras over an operad, as seen in \cref{subsection:Hinich model structure on algebras}, and while they are not the same as the weak equivalences for coalgebras in general, cf. \cref{subsection:Vallette model structure on coalgebras}, they are indeed for coalgebras concentrated in degree $\ge2$ by \cref{prop:cobar construction and qis}.
\end{remark}

\subsection{Quillen's main theorem}\label{subsect:Quillen main thm}

Let $X\in\Top_{*,1}$, then the rational homotopy of $X$ can be made into a graded\footnote{By which we mean differential graded with trivial differential.} Lie algebra $\pi(X)\in\mathsf{grLie}_{\ge1}$ by
\[
\pi(X)_n\coloneqq\pi_{n+1}(X)\otimes_{\mathbb{Z}}\mathbb{Q}\ ,
\]
i.e. $\pi(X)\coloneqq s^{-1}\pi_\bullet(X)\otimes_{\mathbb{Z}}\mathbb{Q}$ as a chain complex, and using the Whitehead product to define the Lie bracket. This yields a functor
\[
\pi:\Top_{*,1}\longrightarrow\mathsf{grLie}_{\ge1}\ .
\]
By the definition of $\rhcat$, the functors $X\mapsto\pi(X)$ and $X\mapsto H_\bullet(X;\mathbb{Q})$ from $\Top_{*,1}$ to graded Lie algebras and graded cocommutative algebras extend uniquely to functors with $\rhcat$ as domain.

\medskip

The main theorem of \cite{Quillen}, Theorem I.1 in loc. cit., is the following.

\begin{theorem}[Quillen]
	There are equivalences of categories
	\[
	\rhcat\stackrel{\lambda}{\longrightarrow}\ho{\dgl_{\ge1}}\xrightarrow{s\Bar_\kappa}\ho{\cocom_{\ge2}}\ ,
	\]
	where the second functor is the suspension of the usual bar construction with respect to the twisting morphism
	\[
	\kappa:\susp^{-1}\otimes\com^\vee = \lie^{\antishriek}\longrightarrow\lie\ .
	\]
	Moreover, there are isomorphisms of functors
	\[
	\pi(X)\longrightarrow H_\bullet(\lambda(X))\,\quad\text{and}\quad H_\bullet(X)\longrightarrow H_\bullet(s\Bar_\kappa\lambda(X))
	\]
	from $\rhcat$ to graded Lie algebras and graded cocommutative coalgebras respectively.
\end{theorem}

The functor $\lambda:\rhcat\to\ho{\dgl_{\ge1}}$ comes from a sequence of adjunctions leading from $\Top_{*,1}$ to $\dgl_{\ge1}$, see \cite[p. 211]{Quillen}. A full description of this functor is out of the scope of the present work.

\begin{corollary}
	If $\g$ is a graded Lie algebra concentrated in degrees $\ge1$, then $\g\simeq\pi(X)$ for some $X\in\Top_{*,1}$. If $\C$ is a graded cocommutative coalgebra concentrated in degrees $\ge2$, then $C\simeq H_\bullet(X;\mathbb{Q})$ for some $X\in\Top_{*,1}$.
\end{corollary}

This is a first hint to the fact that, in rational homotopy theory, one should be able to model spaces by Lie algebras and cocommutative coalgebras.

\subsection{Rationalization}

This part is based on \cite[Ch. 9]{FelixHalperinThomas}.

\begin{definition}
	An abelian group $G$ is called \emph{rational}\index{Rational!abelian group} if multiplication by $k$ is an automorphism for each $k\in\mathbb{Z}\backslash\{0\}$.
\end{definition}

A rational abelian group is therefore canonically a $\mathbb{Q}$-vector space. Notice that for any abelian group $G$, the abelian group $G\otimes_{\mathbb{Z}}\mathbb{Q}$ is rational. It is called the \emph{rationalization}\index{Rationalization!of an abelian group} of $G$. If $G$ was rational to start with, then
\[
G\otimes_{\mathbb{Z}}\mathbb{Q}\cong G
\]
canonically via the obvious map induced by $g\otimes1\mapsto g$.

\medskip

Something analogous can be done with spaces.

\begin{theorem}[{\cite[Thm. 9.3]{FelixHalperinThomas}}]
	Let $X\in\Top_{*,1}$. The following are equivalent.
	\begin{enumerate}
		\item $\pi_\bullet(X)$ is rational.
		\item $H_\bullet(X,\pt;\mathbb{Z})$ is rational.
	\end{enumerate}
\end{theorem}

\begin{definition}
	A space $X\in\Top_{*,1}$ is \emph{rational}\index{Rational!space} if $\pi_\bullet(X)$ is rational, or equivalently if $H_\bullet(X,\pt;\mathbb{Z})$ is rational.
\end{definition}

Let $f:X\to Y$ be a morphism in $\Top_{*,1}$, and suppose that $Y$ is rational. Then the morphism $\pi_\bullet(f)$ of abelian groups canonically extends to a morphism
\[
\pi_\bullet(X)\otimes_{\mathbb{Z}}\mathbb{Q}\longrightarrow\pi_\bullet(Y)
\]
of $\mathbb{Q}$-vector spaces.

\begin{definition}
	A \emph{rationalization}\index{Rationalization!of a space} of a space $X\in\Top_{*,1}$ is a morphism $f:X\to X_\mathbb{Q}$ to a rational space $X_\mathbb{Q}$ such that the induced map
	\[
	\pi_\bullet(X)\otimes_{\mathbb{Z}}\mathbb{Q}\longrightarrow\pi_\bullet(X_\mathbb{Q})
	\]
	is an isomorphism.
\end{definition}

\begin{theorem}[{\cite[Thm. 9.7]{FelixHalperinThomas}}]\label{thm:rationalization exists}
	For every $X\in\Top_{*,1}$ there exists a relative CW complex $(X_\mathbb{Q},X)$ with no $0$-cells and no $1$-cells such that the inclusion
	\[
	X\longrightarrow X_\mathbb{Q}
	\]
	is a rationalization. This rationalization is unique up to homotopy equivalence $\mathrm{rel}\,X$.
\end{theorem}

In particular, the weak homotopy type of $X_\mathbb{Q}$ is the rational homotopy type of $X$.

\section{Sullivan's approach and the rational de Rham theorem}

Sullivan's \cite{sul77} approach to rational homotopy theory focuses on rational cohomology rather than rational homotopy or homology. This way, one can use commutative algebras as models for spaces, instead of cocommutative algebras or Lie algebras.

\medskip

In this section, we work with cochain complexes instead of chain complexes, as it is more conventional to work with differentials of degree $1$ on differential forms.

\medskip

Most of the material present in this section works over an arbitrary field $\k$ of characteristic $0$. In particular, PL differential forms and Dupont's contraction do not need the base field to be $\mathbb{Q}$. What we present here is extracted from \cite[Sect. 10]{FelixHalperinThomas} and \cite{dup76}.

\subsection{PL differential forms and rational cohomology}\label{subsect:PL differential forms}

Instead of focusing on rational homotopy or rational homology, one can look at rational cohomology. The fact that this approach gives the same information as the other ones follows from the following result.

\begin{proposition}
	A map $f:X\to Y$ in $\Top_{*,1}$ is a rational equivalence if, and only if it induces an isomorphism
	\[
	H^\bullet(f;\mathbb{Q}):H^\bullet(Y;\mathbb{Q})\longrightarrow H^\bullet(X;\mathbb{Q})
	\]
	in rational cohomology.
\end{proposition}

\begin{proof}
	This is an immediate consequence of \cref{thm:Serre rational equivalences} and of the fact that
	\[
	H^n(X;\mathbb{Q})\cong H_n(X;\mathbb{Q})^\vee,
	\]
	since $\mathbb{Q}$ is a field.
\end{proof}

One now tries to approach cohomology in an algebraic way, using commutative algebras\footnote{One should notice that singular cochains $C^\bullet(X;\mathbb{Q})$ are not a commutative algebra, but only an $\mathbb{E}_\infty$-algebra. However, since the characteristic of the base field is $0$, the category of $\mathbb{E}_\infty$-algebras is equivalent to the category of commutative algebras, so that we can look for an equivalent strictly commutative model for the cochain algebra.}. We want to work combinatorially, so we will use simplicial sets rather than topological spaces. One defines a good commutative algebra of \emph{polynomial differential forms}\index{Polynomial!differential forms} --- more commonly, \emph{PL differential forms} --- on the basic simplices, before extending this construction to arbitrary simplicial sets.

\begin{definition}\label{def:Sullivan algebra}
	Let $n\ge0$. One defines a commutative algebra\index{Polynomial differential forms}
	\[
	\APL(\Delta[n])\coloneqq\frac{\k[t_0,\ldots,t_n,dt_0,\ldots,dt_n]}{\left(\sum_{i=0}^nt_i - 1,\sum_{i=0}^ndt_i\right)}\ ,
	\]
	where $|t_i| = 0$ and $|dt_i| = 1$, with differential $d(t_i)\coloneqq dt_i$. The collection of the $\APL(\Delta[n])$ forms a simplicial commutative algebra $\Omega_\bullet$ by
	\[
	\Omega_n\coloneqq\APL(\Delta[n])
	\]
	with face maps
	\[
	\partial_i:\APL(\Delta[n+1])\longrightarrow\APL(\Delta[n]),\qquad t_i\longmapsto\begin{cases}
	t_k&\text{if }k<i\ ,\\
	0&\text{if }k=i\ ,\\
	t_{k-1}&\text{if }k>i\ ,
	\end{cases}\qquad\text{for}\quad0\le i\le n\ ,
	\]
	and degeneracy maps
	\[
	s_j:\APL(\Delta[n])\longrightarrow\APL(\Delta[n+1]),\qquad t_i\longmapsto\begin{cases}
	t_k&\text{if }k<j\ ,\\
	t_k+t_{k+1}&\text{if }k=j\ ,\\
	t_{k+1}&\text{if }k>j\ ,
	\end{cases}\qquad\text{for}\quad0\le j\le n\ .
	\]
	We call the simplicial commutative algebra $\Omega_\bullet$ the \emph{Sullivan algebra}\index{Sullivan!algebra}\footnote{This terminology is not standard, and one should be careful not to confuse it with e.g. the notion of Sullivan algebra given in \cite{FelixHalperinThomas}, where it indicates instead a certain type of rational models for spaces, cf. \cref{def:Sullivan models}}.
\end{definition}

One now extends this assignment to a contravariant functor
\[
\APL:\ssets\longrightarrow\comalg
\]
from simplicial sets to commutative algebras by
\[
\APL(K)\coloneqq\lim_{\Delta[n]\to K}\APL(\Delta[n]).
\]
In other words, we take a copy of $\APL(\Delta[n])$ for every $n$-simplex of $K$, and then glue them together according to the face maps of $K$. Notice that this is very similar to the construction of the geometric realization functor. The action on morphisms is given by pullback of differential forms.

\medskip

We will not see how $\APL$ gives the same cohomology as the usual one.

\subsection{Dupont's contraction and the PL de Rham theorem}\index{Dupont's contraction}\label{subsect:Dupont contraction}

In \cite{dup76}, Dupont describes a contraction
\begin{center}
	\begin{tikzpicture}
	\node (a) at (0,0){$\Omega_\bullet$};
	\node (b) at (2,0){$C_\bullet$};
	
	\draw[->] (a)++(.3,.1)--node[above]{\mbox{\tiny{$p_\bullet$}}}+(1.4,0);
	\draw[<-,yshift=-1mm] (a)++(.3,-.1)--node[below]{\mbox{\tiny{$i_\bullet$}}}+(1.4,0);
	\draw[->] (a) to [out=-150,in=150,looseness=4] node[left]{\mbox{\tiny{$h_\bullet$}}} (a);
	\end{tikzpicture}
\end{center}
from the Sullivan algebra to a subcomplex $C_\bullet$. A consequence of the existence of such a contraction is the ``de Rham theorem" telling us that the cohomology of $\APL$ is the same as the usual cohomology. We will now describe in detail the elements appearing in this contraction, and state the theorem.

\medskip

Other good reference for Dupont's contraction are \cite{get09} and \cite{cg08}.

\medskip

Fix $n\ge0$, let $k\ge1$, and let $0\le i_0,\ldots,i_k\le n$ be pairwise different. One defines a differential form
\[
\omega_{i_0\ldots i_k}\coloneqq k!\sum_{j=0}^k(-1)^jt_{i_j}dt_{i_0}\cdots\widehat{dt_{i_j}}\cdots dt_{i_k}\in\Omega_n\ ,
\]
where the hat means that we omit the term. Then $C_n$ is defined as the span of all forms appearing this way, that is
\[
C_n\coloneqq\mathrm{span}_\k\{\omega_{i_0\ldots i_k}\mid k\ge1\text{ and }0\le i_0<i_1<\cdots<i_k\le n\}\ .
\]
The fact that $C_n\subseteq\Omega_n$ is a subcomplex is implied by the following lemma, while the fact that it gives rise to a simplicial cochain complex $C_\bullet$ is straightforward.

\begin{lemma}
	Let $k\ge1$, and let $0\le i_0,\ldots,i_k\le n$ be pairwise different. Then
	\[
	d(\omega_{i_0\ldots i_k}) = \sum_{i=0}^n\omega_{ii_0\ldots i_k}\ ,
	\]
	where $\omega_{ii_0\ldots i_k} = 0$ whenever $i=i_j$ for some $0\le j\le k$.
\end{lemma}

The easiest map of the contraction is the morphism $i_\bullet:C_\bullet\to\Omega_\bullet$. It is simply given by the inclusion of $C_\bullet$ into $\Omega_\bullet$.

\medskip

Next in line is the morphism $p_\bullet:\Omega_\bullet\to C_\bullet$, which is given by integration. Namely, if $\omega\in\Omega_n$ is a polynomial differential form, one defines
\[
p_n(\omega)\coloneqq\sum_{\substack{0\le i_0<\cdots<i_k\le n\\f:\{i_0,\ldots,i_k\}\xhookrightarrow{}[n]}}\left(\int_{\Delta[p]}f^*\omega\right)\omega_{i_0\ldots i_k}\ .
\]
Finally, the contracting homotopy $h_\bullet:\Omega_\bullet\to\Omega_\bullet$ is defined as follows. Let
\[
\varphi_i:[0,1]\times\Delta^n\longrightarrow\Delta^n
\]
be defined by
\[
\varphi_i(u,t_0,\ldots,t_n)\coloneqq((1-u)t_0,\ldots,(1-u)t_i + u,\ldots(1-u)t_n)\ .
\]
Geometrically, it is the map contracting the standard geometric $n$-simplex to its $i$th vertex. For $0\le i\le n$, define
\[
h_{(i)}:\Omega_n\longrightarrow\Omega_n
\]
as the map taking a form in $\Omega_n$, pulling it back by $\varphi_i$, and then integrating the resulting form along the fiber of $\varphi_i$. Then the map $h_n$ is defined by
\[
h_n\coloneqq\sum_{\substack{0\le k\le n-1\\0\le i_0<\cdots<i_k\le n}}\omega_{i_0\ldots i_k}h_{(i_k)}\cdots h_{(i_0)}\ .
\]

\begin{proposition}
	The maps defined above are all simplicial, and they form a contraction from $\Omega_\bullet$ to $C_\bullet$.
\end{proposition}

Similarly to what done for $\APL$, one extends the assignment $\Delta[n]\mapsto C_n$ to a functor
\[
\CPL:\ssets\longrightarrow\cochain
\]
from simplicial sets to cochain complexes.

\medskip

To conclude, we have the following theorem.

\begin{theorem}
	Let $X\in\ssets$. There is a natural quasi-isomorphism of cochain complexes
	\[
	\APL(X)\xrightarrow{\simeq}\CPL(X)\ ,
	\]
	where the first map is induced by $p_\bullet$. In particular,
	\[
	H^\bullet(\APL(X))\cong H^\bullet(X;\k)\ .
	\]
\end{theorem}

\section{Algebraic models for spaces}

Given a space $X$, one models its rational homotopy, homology and cohomology by the Lie algebra $\lambda(X)$, the cocommutative coalgebra $s\Bar_\kappa\lambda(X)$, and the commutative algebra $\APL(X)$ respectively. However, one could also consider other algebraic models for a space, by taking (co)algebras having the same properties as the previous three. Such objects are called \emph{rational models} for the space $X$. We review their definitions and some basic existence results in this section.

\medskip

The material of this section is extracted from \cite{FelixHalperinThomas} and \cite[Ch. 4]{Majewski}, with the exception of \cref{subsect:dualizing (co)commutative models}, which is original work.

\subsection{Commutative models}

We begin with commutative models. The idea is that $\APL(X)$ for a space $X$ is huge, and has a relatively complicated algebraic structure, so one tries to replace it by a smaller, simpler commutative algebra.

\begin{definition}
	Let $X\in\ssets_1$ be a $1$-reduced simplicial set. A \emph{commutative rational model}\index{Commutative!rational model} for $X$ is an augmented unital\footnote{As already mentioned previously, because of the augmentation hypothesis, one can equivalently work with non-unital commutative algebras.} commutative algebra $A$ linked to $\APL(X)$ by a zig-zag of quasi-isomorphisms of commutative algebras
	\[
	A\xleftarrow{\sim}\bullet\xrightarrow{\sim}\APL(X)\ .
	\]
	Equivalently, it is a commutative algebra $A$ together with an $\infty$-morphism\footnote{By which we mean an $\infty$-morphism or $\C_\infty$-algebras, cf. \cref{subsect:homotopy commutative algebras}.} from $A$ to $\APL(X)$, cf. \cref{thm:oo-qi of algebras is zig-zag of qi}.
\end{definition}

There is a particular kind of commutative models that are of special interest in rational homotopy, namely Sullivan models, and more specifically minimal Sullivan models.

\begin{definition}\label{def:Sullivan models}
	A \emph{Sullivan model}\index{Sullivan!model} for a space $X$ is a rational commutative model which is a quasi-free commutative algebra $(\com(V),d)$, where $V$ is a graded vector space, such that
	\[
	V = \bigcup_{k=0}^\infty V(k)\ ,
	\]
	where $V(0)\subseteq V(1)\subseteq\cdots$ is an increasing sequence of graded subspaces of $V$, and the differential satisfies
	\[
	d(V(0)) = 0\ ,\quad\text{and}\quad d(V(k))\subseteq\com(V(k-1))\quad\text{for}\quad k\ge1\ .
	\]
	A Sullivan model is \emph{minimal}\index{Minimal!Sullivan model} if moreover we have
	\[
	\mathrm{Im}(d)\subseteq\overline{\com(V)}\cdot\overline{\com(V)}\ .
	\]
\end{definition}

Minimal Sullivan models for spaces exist under some not too restrictive assumptions.

\begin{definition}
	A simply connected topological space $X\in\Top_{*,1}$ is \emph{of finite type} if $H_i(X;\mathbb{Q})$ is finite dimensional for all $i$.
\end{definition}

\begin{theorem}[{\cite[p. 146]{FelixHalperinThomas}}]\label{thm:existence of minimal Sullivan model}
	Let $X\in\Top_{*,1}$ be of finite type. Then $X$ admits a minimal Sullivan model
	\[
	M_X\coloneqq(\com(V),d)\xrightarrow{\sim}\APL(X)
	\]
	such that $V$ is concentrated in degree $\ge2$ and finite dimensional in every degree. It is unique up to (non-canonical) isomorphism.
\end{theorem}

\subsection{Lie models}\label{subsect:Lie models for spaces}

A similar idea comes into play for Lie algebras. Quillen's theorem gives us a canonical candidate for a Lie algebra modeling a simply connected space $X\in\Top_{*,1}$, namely the Lie algebra $\lambda(X)\in\dgl_{\ge1}$.

\begin{definition}
	A Lie algebra $\g\in\dgl_{\ge1}$ is a \emph{Lie rational model}\index{Lie!rational model} for $X$ if there is a zig-zag of quasi-isomorphisms of Lie algebras
	\[
	\g\xleftarrow{\sim}\bullet\xrightarrow{\sim}\lambda(X)\ .
	\]
	Equivalently, it is a Lie algebra $\g$ together with an $\infty$-morphism from $\g$ to $\lambda(X)$.
\end{definition}

Lie models have many nice properties. We invite the interested reader to consult \cite[Part IV]{FelixHalperinThomas} for them.

\begin{remark}
	In \cite[p. 322]{FelixHalperinThomas}, Lie models are defined as those Lie algebras $\g\in\dgl_{\ge1}$ such that the commutative algebra $(s\Bar_\kappa(\g))^\vee$ is a commutative rational model. By \cref{thm:dualization of (co)commutative models}, this definition is equivalent to the one we gave in this section.
\end{remark}

A very important property of Lie models is the following result, due to Berglund. Here, $\MC_\bullet(\g)$ denotes the Maurer--Cartan space of the Lie algebra $\g$. It will be formally introduced in \cref{subsect:Maurer-Cartan space}.

\begin{theorem}[{\cite[Prop. 6.1]{ber15}}]\label{thm:g Lie model, then MC(g) we to X_Q}
	Let $X\in\Top_{*,1}$, and let $\g$ be a Lie model for $X$. Then
	\[
	\MC_\bullet(\g)\simeq X_\mathbb{Q} .
	\]
	In other words, the Maurer--Cartan space of $\g$ is rationally homotopic to $X$.
\end{theorem}

\subsection{Cocommutative models}

Once again with the same ideas in mind, the definition of a cocommutative rational model for a space is as follows.

\begin{definition}
	Let $X\in\Top_{*,1}$ be a simply connected topological space. A cocommutative coalgebra $C\in\cocom_{\ge2}$ is a \emph{cocommutative rational model}\index{Cocommutative rational model} for $X$ if there is a zig-zag of weak equivalences (in the Vallette model structure on coalgebras)
	\[
	C\xrightarrow{\sim}\bullet\xleftarrow{\sim}s\Bar_\kappa\lambda(X)\ .
	\]
\end{definition}

We will see in \cref{prop:alpha-we of coalgebras is zig-zag of we} that this condition is also equivalent to the existence of an $\infty$-morphism of coalgebras from $C$ to $s\Bar_\kappa\lambda(X)$.

\subsection{Dualizing (co)commutative models}\label{subsect:dualizing (co)commutative models}

Commutative and cocommutative models are strictly related, as one would expect. We will now provide a proof of the fact that the dual of a cocommutative model is a commutative model, and \emph{vice versa} that the dual of a commutative model of finite type is a cocommutative model. This result is certainly well-known to experts and part of the folklore, but we haven't been able to find a proof in the literature.

\medskip

We begin by recalling the following theorem, due to Majewski, which relates commutative and Lie models.

\begin{theorem}[{\cite[Thm. 4.90]{Majewski}}]\label{thm:majewski}
	Let $X\in\ssets_{1}$ be a simply-connected space of finite $\mathbb{Q}$-type. Let
	\[
	\ell_X:M_X\stackrel{\sim}{\longrightarrow}A_{PL}^*(X)
	\]
	be a simply-connected commutative model of finite type for $X$, and let
	\[
	\mu_X:\g_X\stackrel{\sim}{\longrightarrow}\lambda(X)
	\]
	be a quasi-free Lie model for $X$. There exists a canonical homotopy class of quasi-isomorphisms of Lie algebras
	\[
	\alpha_X:\g_X\stackrel{\sim}{\longrightarrow}\Omega_\kappa(M_X^\vee)\ .
	\]
\end{theorem}

Using this, we can prove the result we wished.

\begin{theorem}\label{thm:dualization of (co)commutative models}
	Let $X$ be a simply-connected space of finite $\mathbb{Q}$-type.
	\begin{enumerate}
		\item \label{1} Let $A$ be a commutative model of finite type for $X$. Then its dual $A^\vee$ is a cocommutative model for $X$.
		\item \label{2} Dually, let $C$ be a cocommutative model for $X$. Then its linear dual $C^\vee$ is a commutative model for $X$.
	\end{enumerate}
	Notice that we do not have any finiteness assumption on our cocommutative models.
\end{theorem}

\begin{proof}
	We begin by proving (\ref{1}). Let $A$ be a simply-connected commutative model of finite type for $X$. Every simply-connected space $X$ of finite $\mathbb{Q}$-type admits a minimal commutative model $M_X$, which in particular is simply-connected and of finite type, see \cref{thm:existence of minimal Sullivan model}. It follows that we have a zig-zag of quasi-isomorphisms\footnote{In fact, since $M_X$ is cofibrant there is a direct quasi-isomorphism $M_X\to A$.}
	\[
	A\longleftarrow\bullet\longrightarrow M_X
	\]
	by \cref{thm:oo-qi of algebras is zig-zag of qi}. Inspecting the proof in loc. cit. we notice that we can take $\Cobar_\iota\Bar_\iota M_X$ as intermediate algebra, which is again simply-connected and of finite type. Dualizing linearly, we obtain a zig-zag of quasi-isomorphisms
	\begin{equation}\label{eq:zigzag}
	A^\vee\longrightarrow\bullet\longleftarrow M_X^\vee\ ,
	\end{equation}
	where all terms are well-defined coalgebras thanks to the fact that they are of finite type. Finally, we obtain a zig-zag
	\[
	A^\vee\longrightarrow\Bar_\pi\Cobar_\pi A^\vee\longrightarrow\Bar_\pi\Cobar_\pi(\bullet)\longleftarrow\Bar_\pi\Cobar_\pi M_X^\vee\longleftarrow\Bar_\pi\Cobar_\pi\Bar_\pi\lambda(X)\longrightarrow\Bar_\kappa\lambda(X)\ ,
	\]
	where the first arrow is the unit of the bar-cobar adjunction and is a weak equivalence by \cite[Thm. 2.6(2)]{val14}, the second and third arrows are obtained by the arrows of the zig-zag (\ref{eq:zigzag}) by applying $\Bar_\pi\Cobar_\pi$ and are also weak equivalences. The last two arrows are obtained as follows. The Lie algebra $\lambda(X)$ is a Lie model for $X$, and so is its bar-cobar resolution $\Cobar_\pi\Bar_\pi\lambda(X)$ thanks to the counit of the bar-cobar adjunction. Moreover, this last Lie algebra is quasi-free. Therefore, by \cref{thm:majewski} there is a quasi-isomorphism
	\[
	\alpha_X:\Cobar_\pi\Bar_\pi\lambda(X)\stackrel{\sim}{\longrightarrow}\Omega_\pi(M_X^\vee)\ .
	\]
	Applying the bar construction we obtain the desired zig-zag of weak equivalences
	\[
	\Bar_\pi\lambda(X)\stackrel{\sim}{\longleftarrow}\Bar_\pi \Cobar_\pi\Bar_\pi\lambda(X)\xrightarrow[\Bar_\pi\alpha_X]{\sim}\Bar_\pi\Cobar_\pi(M_X^\vee)\ ,
	\]
	concluding the first part of the proof.
	
	\medskip
	
	For point (\ref{2}), let $C$ be a cocommutative model for $X$. By point (\ref{1}), we have in particular that the dual $M_X^\vee$ of the minimal commutative model $M_X$ is a cocommutative model for $X$. Therefore, we have a zig-zag of weak equivalences
	\[
	C\longrightarrow\bullet\longleftarrow M_X^\vee\ ,
	\]
	which in particular are quasi-isomorphisms. Dualizing linearly, we obtain a zig-zag of quasi-isomor\-phisms
	\[
	C^\vee\longleftarrow\bullet\longrightarrow M_X^{\vee\vee}\cong M_X\ ,
	\]
	where the last isomorphism holds because $M_X$ is of finite type. Therefore, the commutative algebra $C^\vee$ is a commutative model for $X$.
\end{proof}

%% file: LieAlgebras.tex
\chapter{Maurer--Cartan spaces}\label{ch:Lie algebras}

Maurer--Cartan elements of Lie algebras, and more generally homotopy Lie algebras, appear naturally throughout the whole of mathematics. The reason of this fact comes from deformation theory, and we will try to explain it later, namely in \cref{chapter:deformation theory}. Since the study of the spaces of Maurer--Cartan elements will be one of the central topics of the present thesis, we dedicate this chapter to the definition of these objects and of all the notions surrounding them, as well as the statement and the study of some central theorems in this area.

\section[Lie algebras]{Lie algebras: the Deligne groupoid and the Goldman--Mill\-son theorem}

In this section, we introduce the Maurer--Cartan set of a Lie algebra and some of its properties. We look at the Baker--Campbell--Hausdorff formula, which makes it into a groupoid -- the Deligne groupoid. This leads us naturally to consider complete Lie algebras in the sense of \cref{sect:proper complete algebras}. Finally, we present the Goldman--Millson theorem, which is the precursor of the Dolgushev--Rogers theorem.

\subsection{The Maurer--Cartan set of a Lie algebra}

Let $\g$ be a Lie algebra\index{Lie!algebras}, that is a chain complex together with a bracket
\[
[-,-]:\g\otimes\g\longrightarrow\g
\]
of degree $0$ which
\begin{itemize}
	\item is antisymmetric, i.e.
	\[
	[x,y] = (-1)^{|x||y|+1}[y,x]\ ,
	\]
	and
	\item satisfies the Jacobi rule\index{Jacobi rule}, that is
	\[
	(-1)^{|x||z|}[x,[y,z]]
+ (-1)^{|x||y|}[y,[z,x]] + (-1)^{|y||z|}[z,[x,y]] = 0\ .
\]
\end{itemize}

\begin{definition}
	A \emph{Maurer--Cartan element}\index{Maurer--Cartan!elements!of a Lie algebra} of a Lie algebra $\g$ is an element $x\in\g_{-1}$ of degree $-1$ satisfying the \emph{Maurer--Cartan equation}\index{Maurer--Cartan!equation!in a Lie algebra}
	\begin{equation}\label{eq:MC for Lie algebras}
		dx + \frac{1}{2}[x,x] = 0\ .
	\end{equation}
	The set of all Maurer--Cartan elements of $\g$ will be denoted by $\MC(\g)$.
\end{definition}

\begin{example}
	In any Lie algebra $\g$, the element $0$ is always a Maurer--Cartan element.
\end{example}

Notice that whenever $\g$ is finite dimensional, then $\MC(\g)$ is an algebraic variety. More specifically, it is an intersection of quadrics in $\g_{-1}$.

\subsection{Gauges between Maurer--Cartan elements}

Let $\g$ be a Lie algebra, and let $\lambda\in\g_0$ be a degree $0$ element. One can consider the ``vector field" given by
\[
x\in\g\longmapsto d\lambda + [x,\lambda]\in\g\ .
\]
In the finite dimensional case, it makes sense to identify the target copy of $\g$ with $T_x\g$, and to take the flow of such a vector field, but in general we have to proceed formally. One considers the differential equation in $\g[[t]]\coloneqq\g\otimes\k[[t]]$ given by
\begin{equation}\label{eq:gauge Lie algebras}
\frac{d}{dt}x(t) = d\lambda + [x(t),\lambda]\ ,
\end{equation}
where $\lambda\in\g_0$ is seen as a constant element in $\g[[t]]$.

\begin{lemma}\label{lemma:formula for gauges in Lie algebra}
	Let $x_0\in\g$, then the unique solution of equation (\ref{eq:gauge Lie algebras}) with initial value $x(0) = x_0$ is given by
	\[
	x(t) = \frac{e^{t\ad_\lambda}-\id}{\ad_\lambda}(d\lambda) + e^{t\ad_\lambda}(x_0)\ ,
	\]
	where $\ad_\lambda(x)\coloneqq[x,\lambda]$, and the exponential has to be understood as a formal power series.
\end{lemma}

A conceptual way to derive the formula above is presented in \cite[Sect. 1]{dsv16}, cf. also the differential trick exposed in \cref{subsect:BCH}. Another way is provided by \cref{prop:formula for formal ODE}, cf. \cref{subsect:explicit formula for gauges}.

\begin{proof}
	Obviously, we have that $x(0) = x_0$. Differentiating formally, we have
	\begin{align*}
		\frac{d}{dt}x(t) =&\ \frac{d}{dt}\left(\sum_{n\ge1}\frac{t^n}{n!}\ad_\lambda^{n-1}(d\lambda) + \sum_{n\ge0}\frac{t^n}{n!}\ad_\lambda^n(x_0)\right)\\
		=&\ \sum_{n\ge1}\frac{t^{n-1}}{(n-1)!}\ad_\lambda^{n-1}(d\lambda) + \sum_{n\ge0}\frac{t^{n-1}}{(n-1)!}\ad_\lambda^n(x_0)\\
		=&\ d\lambda + [x(t),\lambda]\ ,
	\end{align*}
	where $\ad_\lambda^n\coloneqq\ad_\lambda\ad_\lambda^{n-1}$ for $n\ge1$, and $\ad_\lambda^0=\id$.
\end{proof}

We would like to evaluate this formula at a $t\in\k$, in order that the solution of (\ref{eq:MC for Lie algebras}) makes sense in $\g$, and not just formally in $\g[[t]]$. However, the formula contains infinite sums, which do not make sense in a mere chain complex. Therefore, one has to ask that the Lie algebra is proper complete\footnote{See \cref{appendix:filtered stuff} for the formal definitions. One can also drop the properness assumption and only consider those Maurer--Cartan elements and $\lambda$ that are in $(\F_1\g)_0$.}, so that $\lambda\in\F_1\g$ and the iterated brackets with $\lambda$ live in increasing degrees of the filtration, and the infinite sums make sense.

\begin{lemma}
	Suppose $\g$ is a proper complete Lie algebra, and let $x_0\in\MC(\g)$ be a Maurer--Cartan element. Then the evaluation of the solution to equation (\ref{eq:gauge Lie algebras}) always exists for any time $t\in\k$, and it remains in $\MC(\g)$.
\end{lemma}

\begin{proof}
	Existence is trivial. To see that the solution remains in the Maurer--Cartan set, we differentiate the Maurer--Cartan equation applied to the solution:
	\begin{align*}
		\frac{d}{dt}\left(dx(t) + \frac{1}{2}[x(t),x(t)]\right) =&\ d\dot{x}(t) + \frac{1}{2}([\dot{x}(t),x(t)] + [x(t),\dot{x}(t)])\\
		=&\ d(d\lambda + [x(t),\lambda]) + \frac{1}{2}([d\lambda + [x(t),\lambda],x(t)] + [x(t),d\lambda + [x(t),\lambda]])\\
		=&\ 0\ ,
	\end{align*}
	where one has to use the Jacobi rule to conclude.
\end{proof}

\begin{definition}
	Let $\g$ be a proper complete Lie algebra. Two Maurer--Cartan elements $x_0,x_1\in\MC(\g)$ of $\g$ are \emph{gauge equivalent}\index{Gauge equivalence!in a Lie algebra} if there exists $\lambda\in\g_0$ such that the solution $x(t)$ of equation (\ref*{eq:MC for Lie algebras}) with initial value $x(0)=x_0$ satisfies $x(1) = x_1$. We write $x_0\gaugeeq x_1$, and say that $\lambda$ is a \emph{gauge} between $x_0$ and $x_1$.
\end{definition}

\begin{remark}
	Since the gauge equation (\ref{eq:gauge Lie algebras}) is autonomous, any two Maurer--Cartan elements on a solution of the equation are gauge equivalent.
\end{remark}

By taking $\lambda=0$, one sees that for any $x\in\MC(\g)$ we have $x\gaugeeq x$, and if $\lambda$ is a gauge from $x_0$ to $x_1$, then $-\lambda$ is a gauge from $x_1$ to $x_0$. Next, we will show that the gauge relation is transitive, effectively making it into an equivalence relation.

\subsection{The Lawrence--Sullivan algebra}\label{subsect:LS algebra}

In \cite{ls10}, R. Lawrence and D. Sullivan defined a Lie algebra representing the interval.

\begin{definition}
	The Lawrence--Sullivan algebra\index{Lawrence--Sullivan algebra} is the free complete Lie algebra on elements $x_0,x_1$ of degree $-1$ and $\lambda$ of degree $0$ with
	\begin{align*}
		dx_i =&\ -\frac{1}{2}[x_i,x_i]\ ,\\
		d\lambda =&\ \sum_{n\ge0}\frac{B_n}{n!}\ad_\lambda^n(x_1-x_0) - \ad_\lambda(x_0)\ ,
	\end{align*}
	where $B_n$ is the $n$th Bernoulli number.
\end{definition}

This Lie algebra is the free complete Lie algebra over two Maurer--Cartan elements $x_0$ and $x_1$, representing the endpoints of the interval, and a gauge $\lambda$ from $x_0$ to $x_1$, which corresponds to the interval itself. It is easily derived from \cref{lemma:formula for gauges in Lie algebra}. Indeed, we want
\[
x_1 = \frac{e^{\ad_\lambda}-\id}{\ad_\lambda}(d\lambda) + e^{\ad_\lambda}(x_0) = \frac{e^{\ad_\lambda}-\id}{\ad_\lambda}(d\lambda + \ad_\lambda(x_0)) + x_0\ .
\]
Thus, we have
\[
d\lambda = \left(\frac{e^{\ad_\lambda}-\id}{\ad_\lambda}\right)^{-1}(x_1-x_0) - \ad_\lambda(x_0)\ ,
\]
and using the formal power series
\[
\frac{t}{e^t-1} = \sum_{n\ge0}\frac{B_n}{n!}t^n\ ,
\]
we find that we must have
\[
d\lambda = \sum_{n\ge0}\frac{B_n}{n!}\ad_\lambda^n(x_1-x_0) - \ad_\lambda(x_0)\ .
\]

\subsection{The Baker--Campbell--Hausdorff formula}\label{subsect:BCH}

Let $\g$ be a Lie algebra. The Baker--Campbell--Hausdorff formula\index{Baker--Campbell--Hausdorff formula}
\[
\BCH:\g_0\times\g_0\longrightarrow\g_0
\]
is a formula making $\g_0$ into a group, and the gauge action into a group action. It has a long history, going back to F. Schur \cite{sch90}, with the first explicit formula is attributed to E. Dynkin. An extensive reference for this formula is the book \cite{BonfiglioliFulci}.

\medskip

It can be approached with what is called the \emph{differential trick}, which we learned from \cite[Sect. 2]{dsv16}. We extend the Lie algebra $\g$ by an element $\delta$ of degree $-1$, defining
\[
\g^+\coloneqq\g\oplus\k\delta\ ,
\]
with the original differential and bracket on $\g$, and imposing
\[
d(\delta) = 0\ ,\qquad[\delta,\delta] = 0\ ,\qquad\text{and}\qquad[\delta,x] = dx
\]
for any $x\in\g$. Then we can look at the Maurer--Cartan elements $x\in\MC(\g)$ of $\g$ as
\[
\overline{x}\coloneqq\delta + x\in\g^+,
\]
with the Maurer--Cartan equation becoming just
\[
[\overline{x},\overline{x}] = 0\ .
\]
If we have $\lambda\in\g_0$, then the differential equation (\ref{eq:gauge Lie algebras}) defining the gauge action simply becomes
\[
\dot{\overline{x}}(t) = \ad_\lambda(\overline{x}(t))\ ,
\]
which is solved by
\[
\overline{x}(t) = e^{t\ad_\lambda}(\overline{x_0})\ .
\]
This immediately recovers \cref{lemma:formula for gauges in Lie algebra}. Moreover, we can derive the Baker--Campbell--Haus\-dorff formula from it. Let $\lambda,\mu\in\g_0$. We want an element $\BCH(\lambda,\mu)\in\g_0$ such that
\[
e^{\ad_{\BCH(\lambda,\mu)}}(\overline{x}) = e^{\ad_\lambda}(e^{\ad_\mu}(\overline{x}))\ .
\]
It is not obvious that such an element $\BCH(\lambda,\mu)$ exists. However, one can then proceed as in \cite[Sect. 5.3--6]{Hall} to obtain that
\[
\BCH(\lambda,\mu) = \lambda + \int_0^1g(e^{\ad_\lambda}e^{t\ad_\mu})(\mu)dt\ ,
\]
where
\[
g(z)\coloneqq\frac{\log z}{1-\frac{1}{z}}
\]
is a holomorphic function on the disk (and thus we consider its formal power series to perform the integration). If one performs the integration, one obtains for the first terms
\[
\BCH(\lambda,\mu) = \lambda + \mu + \frac{1}{2}[\lambda,\mu] + \frac{1}{12}\left([\lambda,[\lambda,\mu]] + [\mu,[\mu,\lambda]]\right) + \cdots
\]
With this definition for $\BCH$, we immediately have the following result.

\begin{proposition}
	Let $\g$ be a proper complete Lie algebra, and let $x_0,x_1,x_2\in\MC(\g)$ be three Maurer--Cartan elements. If $\lambda\in\g_0$ is a gauge from $x_0$ to $x_1$, and $\mu\in\g_0$ is a gauge from $x_1$ to $x_2$, then $\BCH(\lambda,\mu)$ is a gauge from $x_0$ to $x_2$. In particular, being gauge equivalent is an equivalence relation.
\end{proposition}

\begin{definition}
	Let $\g$ be a Lie algebra. The \emph{Deligne groupoid}\index{Deligne groupoid} $\Del(\g)$ of $\g$ is the groupoid with the Maurer--Cartan elements $\MC(\g)$ as objects and the gauges as morphisms.
\end{definition}

In other words, the Deligne groupoid is the action groupoid associated to the gauge action of $\g_0$ --- which is made into a group by the Baker--Campbell--Hausdorff formula --- on the set of Maurer--Cartan elements of $\g$.

\begin{remark}
	The Deligne groupoid admits an extension to a $2$-groupoid. This was first done in a letter \cite{letterDeligneBreen} sent by P. Deligne to L. Breen in 1994. See \cite[Sect. 6]{yek12} or \cite[Sect. 3.3]{bgnt15} for a clean definition.
\end{remark}

An object of interest in deformation theory is the quotient of the set of Maurer--Cartan elements by the gauge equivalence relation.

\begin{definition}
	The \emph{moduli space of Maurer--Cartan elements}\index{Moduli space of Maurer--Cartan elements} $\MCbar(\g)$ of $\g$ is the quotient
	\[
	\MCbar(\g)\coloneqq\MC(\g)/\gaugeeq
	\]
	of the set of Maurer--Cartan elements of $\g$ by the gauge equivalence relation.
\end{definition}

\subsection{The Goldman--Millson theorem}

The Goldman--Millson theorem\footnote{Which W. M. Goldman and J. J. Millson attribute to Deligne and Schlessinger--Stasheff, see \cite[p.46]{gm88}.} \cite{gm88} tells us that, under some conditions, quasi-isomorphic Lie algebras give rise to equivalent Deligne groupoids.

\medskip

The original statement of the result is as follows.

\begin{theorem}[Goldman--Millson]\label{thm:Goldman-Millson}\index{Goldman--Millson theorem}
	Let $(A,\mathfrak{m})$ be an Artinian local $\k$-algebra with maximal ideal $\mathfrak{m}$. Let $\g$ and $\h$ be two Lie algebras, and let $\phi:\g\to\h$ be a morphism of Lie algebras such that
	\[
	H_i(\phi):H_i(\g)\longrightarrow H_i(\h)
	\]
	is an isomorphism for $i=0,1$, and is injective for $i=2$. Then the map
	\[
	\Del(\phi\otimes1):\Del(\g\otimes\mathfrak{m})\longrightarrow\Del(\h\otimes\mathfrak{m})
	\]
	is an equivalence of categories. In particular, it induces a bijection between the respective moduli spaces of Maurer--Cartan elements.
\end{theorem}

The tensorization by the maximal ideal of an Artinian local $\k$-algebra corresponds to some kind of localization. This result has been generalized in various ways during the years, e.g. by Yekutieli \cite{yek12}, the most general version being given by the Dolgushev--Rogers theorem, which we will treat in detail in \cref{sect:Dolgushev-Rogers thm}. The proof is done by ``Artinian induction". We will not treat it here, but we will see similar ideas appear in the proof of the Dolgushev--Robers theorem, and then again in \cref{ch:representing the deformation oo-groupoid}.

\section{Homotopy Lie algebras: the deformation \texorpdfstring{$\infty$}{oo}-groupoid}

Since $\L_\infty$-algebras are a natural generalization of Lie algebras, it is natural to wonder what is the correct generalization of the theory treated above in this context. The answer was given by V. Hinich \cite{hin97descent} with the introduction of the deformation $\infty$-groupoid for Lie algebras\footnote{Called the \emph{contents} of a Lie algebra in \emph{op. cit}.}, which was generalized to $\L_\infty$-algebras and further studied by various authors, such as E. Getzler \cite{get09}, and V. A. Dolgushev and C. L. Rogers \cite{dr15}.

\subsection{Maurer--Cartan elements in a homotopy Lie algebras}\label{subsect:MC el of Loo algebras}

Recall from \cref{subsect:homotopy Lie algebras} that an $\L_\infty$-algebra\index{$\L_\infty$-algebras}\index{Homotopy!Lie algebras} $\g$ is a chain complex equipped with graded antisymmetric operations
\[
\ell_n:\g^{\otimes n}\longrightarrow\g
\]
of degree $n-2$, for all $n\ge2$, satisfying the relations
\begin{equation}\label{eq:rel Loo-algebra}
\sum_{\substack{n_1+n_2 = n+1\\\sigma\in\sh(n_1-1,n_2)}}(-1)^{n_2(n_1-1) + \sigma}(\ell_{n_1}\circ_1\ell_{n_2})^\sigma = 0
\end{equation}
for all $n\ge1$, where we use the short-hand notation $\ell_1\coloneqq d$.

\medskip

Let's explore these relations a bit. For $n=1$, equation (\ref{eq:rel Loo-algebra}) becomes
\[
\ell_1^2 = 0\ ,
\]
which is the same as to say that $d_\g$ squares to zero, i.e. that $\g$ is a chain complex. For $n = 2$, we have
\[
d\ell_2(x,y) = \ell_2(dx,y) + (-1)^{|x|}\ell_2(x,dy)\ ,
\]
i.e. $d$ is a derivation with respect to $\ell_2$. For $n=3$, we obtain
\begin{align*}
\partial(\ell_3)&(x,y,z) = -\ell_2(\ell_2(x,y),z) - (-1)^{|z|(|x|+|y|)}\ell_2(\ell_2(z,x),y) - (-1)^{|x|(|y|+|z|)}\ell_2(\ell_2(y,z),x)\\
=&\ (-1)^{|x||z|+1}\left((-1)^{|x||z|}\ell_2(\ell_2(x,y),z) + (-1)^{|z||y|}\ell_2(\ell_2(z,x),y) + (-1)^{|y||x|}\ell_2(\ell_2(y,z),x)\right)\ ,
\end{align*}
which tells us that the Jacobi rule is satisfied up to a homotopy given by the ternary bracket $\ell_3$. For $n\ge4$, we have higher compatibility relations between the brackets.

\begin{definition}
	A \emph{Maurer--Cartan element}\index{Maurer--Cartan!elements!of an $\L_\infty$-algebra} of a proper complete $\L_\infty$-algebra is an element $x\in\g_{-1}$ of degree $-1$ satisfying the Maurer--Cartan equation\index{Maurer--Cartan!equation!in an $\L_\infty$-algebra}
	\begin{equation}\label{eq:MC for Loo-algebras}
		dx + \sum_{n\ge2}\frac{1}{n!}\ell_n(x,\ldots,x) = 0\ .
	\end{equation}
	The set of all Maurer--Cartan elements of $\g$ will be denoted by $\MC(\g)$.
\end{definition}

\begin{remark}
	If $\g$ is a Lie algebra, then the Maurer--Cartan equation (\ref{eq:MC for Loo-algebras}) reduces to the Maurer--Cartan equation for Lie algebras (\ref{eq:MC for Lie algebras}).
\end{remark}

\begin{remark}
	While in the case of Lie algebras we only needed the algebras to be proper complete in order for the gauge relation to be well-defined, in the case of $\L_\infty$-algebras we already need this condition for the Maurer--Cartan equation to make sense.
\end{remark}

\subsection{The deformation \texorpdfstring{$\infty$}{oo}-groupoid}\label{subsect:Maurer-Cartan space}

One now would like a notion of gauge equivalence between Maurer--Cartan elements of an $\L_\infty$-algebra. As a matter of fact, one obtains a whole hierarchy of relations thanks to the following object defined by V. Hinich \cite[Def. 2.1.1]{hin97descent}.

\begin{definition}
	Let $\g$ be a proper complete $\L_\infty$-algebra. The \emph{deformation $\infty$-groupoid}\index{Deformation!$\infty$-groupoid} of $\g$ --- also known as its \emph{Maurer--Cartan space}\index{Maurer--Cartan!space} --- is the simplicial set
	\[
	\MC_\bullet(\g)\coloneqq\lim_n\MC\big(\g/\F_n\g\otimes\Omega_\bullet\big)\in\ssets\ ,
	\]
	where $\Omega_\bullet$ is the Sullivan algebra, cf. \cref{def:Sullivan algebra}, with inverted degrees (so that it is a chain complex instead of a cochain complex).
\end{definition}

\begin{remark}
	The name of the object $\MC_\bullet(\g)$ is not standard and one finds different nomenclatures depending on the author. Examples are \emph{Deligne--Hinich $\infty$-groupoid}, \emph{Deligne--Hinich--Getzler $\infty$-groupoid}, and \emph{contents of $\g$}. One should also be careful with the term $\infty$-groupoid, as it can find different meanings in the literature. We take the $\infty$-categorical approach and use it to mean Kan complex, but e.g. \cite{get09} uses a stronger definition.
\end{remark}

The elements of $\MC_1(\g)$ are the analogue of the gauges between Maurer--Cartan elements in this context. We will look at them in more detail in \cref{subsection:homotopies and gauges}.

\medskip

The simplicial set $\MC_\bullet(\g)$ has a lot of nice properties. It is extends to a functor
\[
\MC_\bullet:\infty\text{-}\widehat{\L}_\infty\text{-}\mathsf{alg}\longrightarrow\ssets
\]
from proper complete $\L_\infty$-algebras with filtered $\infty$-morphisms to simplicial sets. Its action on filtered $\infty$-morphisms is induced by the following fact.

\begin{lemma}
	If $\Phi:\g\rightsquigarrow\h$ is a filtered $\infty$-morphism between proper complete $\L_\infty$-algebras, and $x\in\MC(\g)$ is a Maurer--Cartan element of $\g$, then
	\[
	\MC(\Phi)(x)\coloneqq\sum_{n\ge1}\frac{1}{n!}\phi_n(x,\ldots,x)
	\]
	is a Maurer--Cartan element of $\h$. Moreover, if $\Phi:\g_1\to\g_2$ and $\Psi:\g_2\to\g_3$ are two such filtered $\infty$-morphisms, then
	\[
	\MC(\Psi)\MC(\Phi) = \MC(\Psi\Phi)\ .
	\]
\end{lemma}

\begin{proof}
	This is just developing the Maurer--Cartan equation for $\MC(\Phi)(x)$ and using the explicit relations satisfied by the components of an $\infty$-morphism that we wrote down in \cref{subsect:homotopy Lie algebras} to recover the Maurer--Cartan equation for $x$. Proving that $\MC(-)$ respects compositions is also straightforward.
\end{proof}

We have the following result, motivating the name ``$\infty$-groupoid" and ``space".

\begin{theorem}[{\cite[Thm. 2]{rog16}}] \label{thm:filtered surjections induce fibrations of MC}
	Let $\g$ and $\h$ be proper complete $\L_\infty$-algebras, and let $\Phi:\g\rightsquigarrow\h$ be a filtered $\infty$-morphism that induces a surjection at every level of the filtrations. Then
	\[
	\MC_\bullet(\Phi):\MC_\bullet(\g)\longrightarrow\MC_\bullet(\h)
	\]
	is a fibration of simplicial sets. In particular, for any proper complete $\L_\infty$-algebra $\g$, the simplicial set $\MC_\bullet(\g)$ is a Kan complex.
\end{theorem}

This result was originally proven by Hinich \cite[Th. 2.2.3]{hin97descent} for strict surjections between nilpotent Lie algebras concentrated in positive degrees, and then successively generalized by E. Getzler \cite[Prop. 4.7]{get09} to strict surjections between nilpotent $\L_\infty$-algebras and by C. L. Rogers to the versions stated above. A precursor to the version of \cite{rog16} is \cite[Prop. 3.1]{yal16}.

\subsection{Homotopies and gauges}\label{subsection:homotopies and gauges}

There are two possible --- and equivalent, as we will see --- definitions of equivalence between Maurer--Cartan elements of an $\L_\infty$-algebra.

\begin{definition}
	Let $\g$ be a proper complete $\L_\infty$-algebra. Two Maurer--Cartan elements $x_0,x_1\in\MC(\g) = \MC_0(\g)$ are \emph{homotopy equivalent}\index{Homotopy!equivalence of Maurer--Cartan elements} if there is an element $\alpha\in\MC_1(\g)$ such that
	\[
	d_0(\alpha) = x_0,\ \quad\text{and}\quad d_1(\alpha) = x_1\ .
	\]
\end{definition}

Let's write down explicitly what it means for two Maurer--Cartan elements to be homotopy equivalent. For simplicity, assume that the $\L_\infty$-algebra $\g$ is nilpotent. The general case works similarly.

\medskip

Under the isomorphism $\Omega_1\cong\k[t,dt]$ of unital commutative algebras given by sending $t_1$ to $t$, a generic element of $(\g\otimes\Omega_1)_{-1}$ is given by
\[
\alpha = x(t) + \lambda(t)dt\ ,
\]
where $x(t)\in\g_{-1}[t]$, and $\lambda(t)\in\g_0[t]$. The Maurer--Cartan equation (\ref{eq:MC for Loo-algebras}) becomes
\begin{align*}
	0 =&\ d\alpha + \sum_{n\ge2}\frac{1}{n!}\ell_n(\alpha,\ldots,\alpha)\\
	=&\ d(x(t)+\lambda(t)dt) + \sum_{n\ge2}\left(\frac{1}{n!}\ell_n(x(t),\ldots,x(t)) + \frac{1}{(n-1)!}\ell_n(x(t),\ldots,x(t),\lambda(t))\right)\\
	=&\ d_\g x(t) + \sum_{n\ge2}\frac{1}{n!}\ell_n(x(t),\ldots,x(t)) +\\
	&\qquad + \left(-\frac{d}{dt}x(t) + d_\g\lambda(t) + \sum_{n\ge2}\frac{1}{(n-1)!}\ell_n(x(t),\ldots,x(t),\lambda(t))\right)dt\ .
\end{align*}
To do this computation, we used the following facts:
\begin{itemize}
	\item Since $|dt| = -1$, we have that $(dt)^2 = 0$. In particular, if the term $\lambda(t)dt$ appears twice in a bracket, then the bracket gives $0$.
	\item For $1\le i\le n$, we have
	\begin{align*}
		\ell_n(\underbrace{x(t),\ldots,x(t)}_{i},\lambda(t)dt,\underbrace{x(t),\ldots,x(t)}_{n-(i+1)}) =&\ \ell_n(x(t),\ldots,x(t),\lambda(t)dt)\\
		=&\ \ell_n(x(t),\ldots,x(t),\lambda(t))dt\ .
	\end{align*}
	\item Let $x\in\g_{-1}$ and let $k\ge0$. Then
	\[
	d(xt^k) = d_\g(x)t^k - kxt^{k-1}dt\ .
	\]
\end{itemize}
Therefore, $\alpha$ is a Maurer--Cartan element if, and only if we have:
\begin{enumerate}
	\item $x(t)\in\MC(\g)$ for all times $t\in\k$, and
	\item $x(t)$ satisfies the differential equation
	\[
	\frac{d}{dt}x(t) = d_\g\lambda(t) + \sum_{n\ge2}\frac{1}{(n-1)!}\ell_n(x(t),\ldots,x(t),\lambda(t))\ .
	\]
\end{enumerate}
The two boundaries are given by evaluation of $x(t)$ at $t=0$ and $t=1$, that is
\[
\partial_0(\alpha) = x(0)\quad\text{and}\quad\partial_1(\alpha) = x(1)\ .
\]

\begin{lemma}
	Being homotopy equivalent is an equivalence relation on the set of Maurer--Cartan elements of a proper complete $\L_\infty$-algebra.
\end{lemma}

\begin{proof}
	This is an immediate consequence of the fact that $\MC_\bullet(\g)$ is a Kan complex for any proper complete $\L_\infty$-algebra.
\end{proof}

By taking the quotient of the set of Maurer--Cartan elements by the homotopy equivalence relation, we obtain the zeroth homotopy group $\pi_0\MC_\bullet(\g)$ of the deformation $\infty$-groupoid of $\g$.

\medskip

By keeping $\lambda(t)$ constant, we get a second notion of equivalence of Maurer--Cartan elements.

\begin{definition}
	Let $\g$ be a proper complete $\L_\infty$-algebra. Two Maurer--Cartan elements $x_0,x_1\in\MC(\g)$ are \emph{gauge equivalent}\index{Gauge equivalence!in an $\L_\infty$-algebra} if there exists a $\lambda\in\g_0$ such that the solution $x(t)\in\g[[t]]$ of the formal differential equation
	\begin{equation}\label{eq:gauge Loo-algebras}
		\frac{d}{dt}x(t) = d\lambda + \sum_{n\ge2}\frac{1}{(n-1)!}\ell_n(x(t),\ldots,x(t),\lambda)
	\end{equation}
	with initial value $x(0) = x_0$ is such that $x(1) = x_1$. The element $\lambda\in\g_0$ is called a \emph{gauge} from $x_0$ to $x_1$.
\end{definition}

If $\lambda\in\g_0$ is a gauge from $x_0$ to $x_1$, then one immediately obtains a homotopy equivalence between the two Maurer--Cartan elements by taking
\[
\alpha\coloneqq x(t) + \lambda dt\ ,
\]
where $x(t)$ is the solution of the gauge equation (\ref{eq:gauge Loo-algebras}). Therefore, if two elements are gauge equivalent, then they are homotopy equivalent. The other direction is also known in the literature, see e.g. \cite[Prop. 9]{dp16}. In particular, gauge equivalence is an equivalence relation. We will give a new proof of this fact and explicit formul{\ae} to obtain a gauge equivalence from a homotopy equivalence in \cref{subsect:properties of MC(g x C.)}.

\subsection{An explicit formula for gauges}\label{subsect:explicit formula for gauges}

We can use \cref{prop:formula for formal ODE} to give an explicit formula for the action of a gauge $\lambda$ on a Maurer--Cartan element $x_0$. For each $n\ge1$, we fix
\[
f_{n,1}(y_1,\ldots,y_n)\coloneqq\frac{1}{n!}\ell_{n+1}(y_1,\ldots,y_n,\lambda)\ ,
\]
and $f_{0,1}(1)\coloneqq d\lambda$. We set $f_{n,k} = 0$ for all $k\ge2$. Then equation (\ref{eq:formal differential equation}) becomes equation (\ref{eq:gauge Loo-algebras}), and \cref{prop:formula for formal ODE} gives us what we wanted. A similar formula is already present in \cite[Prop. 5.7]{get09}.

\medskip

The combinatorics of planar trees is rather complicated, especially since we allow vertices of valence $0$ and $1$. Therefore, we do not know of any way to express the resulting formula in a simpler way than the formula already given. However, if we work with a Lie algebra instead of an $\L_\infty$-algebra, then things simplify a lot, and we recover the formula we gave in \cref{lemma:formula for gauges in Lie algebra}. This gives an alternative proof of that result.

\begin{proof}[Proof of \cref{lemma:formula for gauges in Lie algebra} (alternative version)]
	The only non-vanishing operators are $f_{0,1} = d\lambda$, and $f_{1,1} = \ad_\lambda$. Therefore, we only have to work with trees that are a composition of $0$-corollas and $1$-corollas, and which have the weight $1$ at each vertex. They all fall in the following two categories.
	\begin{enumerate}
		\item A linear composition of $1$-corollas. We denote such trees by $a_n$, with
		\[
		a_0\coloneqq\emptyset,\quad\text{and}\quad a_n\coloneqq c_1\circ a_{n-1}
		\]
		for $n\ge1$.
		\item A linear composition of $1$-corollas with a $0$-corolla instead of the free leaf. We denote such trees by $b_n$, with
		\[
		b_1\coloneqq c_0,\quad\text{and}\quad b_n\coloneqq c_1\circ b_{n-1}
		\]
		for $n\ge2$.
	\end{enumerate}
	One easily computes the coefficients of those trees to be
	\[
	F(a_n) = F(b_n) = n!\ ,
	\]
	and
	\[
	a_n(x_0) = t^n\ad_\lambda^n(x_0)\ ,\quad b_n(x_0) = t^n\ad_\lambda^{n-1}(d\lambda)\ .
	\]
	Thus, the solution of the gauge equation is given by
	\begin{align*}
		x(t) =&\ \sum_{n\ge0}\frac{1}{F(a_n)}a_n(x_0) + \sum_{n\ge1}\frac{1}{F(b_n)}b_n(x_0)\\
		=&\ \sum_{n\ge0}\frac{t^n}{n!}\ad_\lambda^n(x_0) + \sum_{n\ge1}\frac{t^n}{n!}\ad_\lambda^{n-1}(d\lambda)\\
		=&\ e^{t\ad_\lambda}(x_0) + \frac{e^{t\ad_\lambda}-\id}{\ad_\lambda}(d\lambda)\ ,
	\end{align*}
	concluding the proof.
\end{proof}

\subsection{Getzler's functor}

In his article \cite{get09}, E. Getzler introduced an object which is smaller, but homotopically equal to the Maurer--Cartan space of an $\L_\infty$-algebra.

\medskip

Recall Dupont's contraction
\begin{center}
	\begin{tikzpicture}
	\node (a) at (0,0){$\Omega_\bullet$};
	\node (b) at (2,0){$C_\bullet$};
	
	\draw[->] (a)++(.3,.1)--node[above]{\mbox{\tiny{$p_\bullet$}}}+(1.4,0);
	\draw[<-,yshift=-1mm] (a)++(.3,-.1)--node[below]{\mbox{\tiny{$i_\bullet$}}}+(1.4,0);
	\draw[->] (a) to [out=-150,in=150,looseness=4] node[left]{\mbox{\tiny{$h_\bullet$}}} (a);
	\end{tikzpicture}
\end{center}
from \cref{subsect:Dupont contraction}.

\begin{definition}
	\emph{Getzler's functor}\index{Getzler's functor} is the functor
	\[
	\gamma_\bullet:\widehat{\L_\infty}\text{-}\mathsf{alg}\longrightarrow\ssets
	\]
	given by sending a proper complete $\L_\infty$-algebra $\g$ to the simplicial set
	\[
	\gamma_\bullet(\g)\coloneqq\MC_\bullet(\g)\cap\ker(1_\g\otimes h_\bullet)\ ,
	\]
	where $h_\bullet$ comes from the Dupont contraction, cf. \cref{subsect:Dupont contraction}. The action on morphisms\footnote{Here, only strict morphisms are considered. An extension to an action of $\infty$-morphisms is possible using the results of \cref{sect:compatibility of convolution algebras with oo-morphisms} and \cref{sect:properties and comparison with Getzler}.} is given by
	\[
	\gamma_\bullet(\phi)(x) = \phi(x)
	\]
	for $\phi:\g\to\h$ a morphism of $\L_\infty$-algebras.
\end{definition}

\begin{theorem}[{\cite[Thm. 5.8]{get09}}]\label{thm:Getzler}
	Let $\g,\h$ be proper complete $\L_\infty$-algebras, and let $\phi:\g\to\h$ be a filtered morphism that induces surjections at every level of the filtration. Then the morphism
	\[
	\gamma_\bullet(\phi):\gamma_\bullet(\g)\longrightarrow\gamma_\bullet(\h)
	\]
	is a fibration of simplicial sets. In particular, the simplicial set $\gamma_\bullet(\g)$ is a Kan complex for any $\L_\infty$-algebra $\g$.
\end{theorem}

\begin{remark}
	This was originally stated for surjections between nilpotent $\L_\infty$-algebras. The more general statement above follows from a straightforward limit argument.
\end{remark}

\begin{theorem}[{\cite[Thm. 5.9]{get09}}]\label{thm:getzler gamma is we to MC.}
	Let $\g$ be a proper complete $\L_\infty$-algebra. The inclusion
	\[
	\gamma_\bullet(\g)\xhookrightarrow{\quad}\MC_\bullet(\g)
	\]
	is a homotopy equivalence of simplicial sets.
\end{theorem}

In fact, the simplicial set $\gamma_\bullet(\g)$ has a property which is stronger than simply being a Kan complex.

\begin{definition}
	An $n$-simplex $x\in\gamma_n(\g)$ is \emph{thin}\index{Thin simplex in $\gamma_\bullet(\g)$} if
	\[
	\int_{\Delta[n]}x = 0\ ,
	\]
	where integration acts on the part of $x$ living in $\Omega_n$.
\end{definition}

\begin{theorem}[{\cite[Thm. 5.4]{get09}}]\label{thm:Getzler's notion of oo-group}
	Let $\g$ be a nilpotent $\L_\infty$-algebra. Then $\gamma_\bullet(\g)$ has the following properties.
	\begin{enumerate}
		\item Every degenerate simplex is thin.
		\item Every horn has a unique thin simplex filling it.
	\end{enumerate}
\end{theorem}

\begin{remark}
	In \cite{get09}, an $\infty$-groupoid is a Kan complex which moreover has a set of thin simplices which satisfies \cref{thm:Getzler's notion of oo-group}. We adopt a more $\infty$-categorical point of view and use the terms Kan complex and $\infty$-groupoid interchangeably.
\end{remark}

\section{Paradigm change: shifted homotopy Lie algebras}\label{sect:shifted Loo-alg}

The operadic suspension functor
\[
\susp\otimes-:\op\longrightarrow\op
\]
is an automorphism of categories. Moreover, for any chain complex $V$ there is a canonical isomorphism
\[
\susp\otimes\End_V\cong\End_{sV}\ ,
\]
which implies that if $\P$ is an operad, then a $\susp\otimes\P$-algebra is equivalent to a $\P$-algebra via a suspension, and \emph{vice versa}.

\medskip

Relying on these facts, we will often work with $\SLoo$-algebras --- shifted homotopy Lie algebras --- instead of $\L_\infty$-algebras in what follows. This gives the same exact theory, but has the great advantage to substantially reduce the amount of signs appearing.

\medskip

Here are some basic, useful facts. A $\SLoo$-algebra\index{Shifted homotopy Lie algebras}\index{$\SLoo$-algebras} is an algebra over\index{$\SLoo$}
\[
\SLoo\coloneqq\susp\otimes\L_\infty\cong\Cobar\com^\vee\ .
\]
The operations on it are generated by graded \emph{symmetric} arity $n$ brackets
\[
\ell_n\coloneqq s^{-1}\mu_n^\vee\ ,
\]
for all $n\ge2$, all of which have degree $-1$. They satisfy the relations
\begin{equation*}\label{eq:rel SLoo-algebra}
\sum_{\substack{n_1+n_2 = n+1\\\sigma\in\sh(n_1,n_2-1)}}(\ell_{n_2}\circ_1\ell_{n_1})^\sigma = 0\ .
\end{equation*}
A Maurer--Cartan element of a proper complete $\SLoo$-algebra $\g$ is a degree $0$ element $x\in\g_0$ satisfying the equation
\[
dx + \sum_{n\ge2}\frac{1}{n!}\ell_n(x,\ldots,x) = 0\ .
\]

\begin{remark}
	Shifted Lie algebras have already appeared in the literature. For example, the Whitehead product endows the rational homotopy groups $\pi_\bullet(X)\otimes\mathbb{Q}$ of a simply connected space $X$ with a shifted Lie algebra structure, c.f. the beginning of \cref{subsect:Quillen main thm}. Because of this, the term \emph{Whitehead algebras} has been used by some authors to indicate shifted Lie algebras.
\end{remark}

\section{The Dolgushev--Rogers theorem}\label{sect:Dolgushev-Rogers thm}

In this section, we will present the Dolgushev--Rogers theorem \cite[Thm. 2.2]{dr15} in some detail. The ideas of the proof will return in some applications, most notably in \cref{ch:representing the deformation oo-groupoid}. The proof we present here is very close to the original one, with some slight variations in the presentation, namely by putting more stress on the obstruction theoretical aspects of one step of the demonstration, see \cref{subsect:induction step DR}.

\medskip

The main theorem of this section is the following one.

\begin{theorem}[Dolgushev--Rogers]\label{thm:Dolgushev-Rogers}\index{Dolgushev--Rogers theorem}
	Let $\g,\h$ be two complete proper $\SLoo$-algebras, and let
	\[
	\Phi:\g\rightsquigarrow\h
	\]
	be a filtered quasi-isomorphism of $\SLoo$-algebras. Then the induced map
	\[
	\MC_\bullet(\Phi):\MC_\bullet(\g)\longrightarrow\MC_\bullet(\h)
	\]
	is a weak equivalence of simplicial sets.
\end{theorem}

This generalizes the Goldman--Millson in the following sense. Suppose that $\g,\h$ are any two Lie algebras, and let $(A,\mathfrak{m})$ be an Artinian local $\k$-algebra with maximal ideal $\mathfrak{m}$. Then $\g\otimes\mathfrak{m}$ has a natural filtration given by
\[
\F_n(\g\otimes\mathfrak{m})\coloneqq\g\otimes\mathfrak{m}^n,
\]
making it into a proper complete Lie algebra\footnote{Since $A$ is Artinian, there is an $N$ such that $\mathfrak{m}^N = 0$. Therefore, $\g\otimes\mathfrak{m}$ is in fact nilpotent, and the filtration only has finitely many levels.}. The same thing holds for $\h$. If $\phi:\g\to\h$ is a quasi-isomorphism of Lie algebras, then the induced morphism
\[
\phi\otimes1_\mathfrak{m}:\g\otimes\mathfrak{m}\longrightarrow\h\otimes\mathfrak{m}
\]
is a filtered (strict) quasi-isomorphism. By applying \cref{thm:Dolgushev-Rogers} to this situation and looking at the $0$th homotopy group of the spaces of Maurer--Cartan elements, we get the natural bijection
\[
\MCbar(\phi):\MCbar(\g)\longrightarrow\MCbar(\h)
\]
between the moduli spaces of Maurer--Cartan elements of $\g$ and $\h$.

\medskip

The strategy of the proof of \cref{thm:Dolgushev-Rogers} is the following. First one proves that the statement holds for abelian $\SLoo$-algebras with a certain trivial filtration and strict quasi-isomorphisms, and the statement holds in full generality at the level of the $0$th homotopy groups. The following step is to prove that we can always twist the $\SLoo$-algebras into play to set the basepoint of the higher homotopy groups at $0$. Then one uses these statements as the base case for an induction proving that the theorem holds in full generality for complete proper $\SLoo$-algebras whose filtration terminates at a finite filtration degree --- and which are therefore necessarily nilpotent. Finally, a limit argument concludes the proof.

\subsection{The case of abelian \texorpdfstring{$\SLoo$-}{shifted homotopy Lie }algebras}

The first step is to prove that the theorem holds in the case of abelian $\SLoo$-algebras, i.e. $\SLoo$-algebras in which all brackets are constantly zero. Let $\g$ be such an $\SLoo$-algebra, and equip it with the filtration
\begin{equation}\label{eq:trivial filtration on abelian Loo-algebras}
	\F_1\g\coloneqq\g\supseteq\F_2\g=\F_3\g=\ldots\coloneqq 0
\end{equation}
making it into a proper complete $\SLoo$-algebra. A Maurer--Cartan element is nothing else than a $0$-cycle:
\[
\MC_n(\g) = \cyc_0(\g\otimes\Omega_n)\ ,
\]
and in particular it follows that $\MC_\bullet(\g)$ is a simplicial vector space.

\begin{proposition}[{\cite[Prop. 2.4]{dr15}}]\label{prop:abelian Loo-alg DR}
	Let $\g,\h$ be two abelian $\SLoo$-algebras endowed with filtrations as in (\ref{eq:trivial filtration on abelian Loo-algebras}), and let $\phi:\g\to\h$ be a quasi-isomorphism. Then the induced map
	\[
	\MC_\bullet(\phi):\MC_\bullet(\g)\longrightarrow\MC_\bullet(\h)
	\]
	is a weak equivalence of simplicial sets.
\end{proposition}

\begin{proof}
	By \cref{thm:homotopy of simplicial vector space is homology of Moore complex}, we know that the homotopy groups of a simplicial vector space $V_\bullet$ are given by the homology of the Moore complex:
	\[
	\pi_0(V_\bullet) = H_0(\moore(V_\bullet))\quad\text{and}\quad\pi_k(V_\bullet,0)=H_k(\moore(V_\bullet))
	\]
	for $k\ge1$. It follows that a map $f:V_\bullet\to W_\bullet$ of simplicial vector spaces is a weak equivalence if, and only if the induced map
	\[
	\moore(f):\moore(V_\bullet)\longrightarrow\moore(W_\bullet)
	\]
	is a quasi-isomorphism of chain complexes.
	
	\medskip
	
	Recall Dupont's contraction from \cref{subsect:Dupont contraction}
	\begin{center}
		\begin{tikzpicture}
		\node (a) at (0,0){$\Omega_\bullet$};
		\node (b) at (2,0){$C_\bullet$};
		
		\draw[->] (a)++(.3,.1)--node[above]{\mbox{\tiny{$p_\bullet$}}}+(1.4,0);
		\draw[<-,yshift=-1mm] (a)++(.3,-.1)--node[below]{\mbox{\tiny{$i_\bullet$}}}+(1.4,0);
		\draw[->] (a) to [out=-150,in=150,looseness=4] node[left]{\mbox{\tiny{$h_\bullet$}}} (a);
		\end{tikzpicture}
	\end{center}
	We extend it to a contraction
	\begin{center}
		\begin{tikzpicture}
		\node (a) at (0,0){$\g\otimes\Omega_\bullet$};
		\node (b) at (2.6,0){$\g\otimes C_\bullet$};
		
		\draw[->] (a)++(.6,.1)--node[above]{\mbox{\tiny{$1_\g\otimes p_\bullet$}}}+(1.4,0);
		\draw[<-,yshift=-1mm] (a)++(.6,-.1)--node[below]{\mbox{\tiny{$1_\g\otimes i_\bullet$}}}+(1.4,0);
		\draw[->] (a) to [out=-150,in=150,looseness=4] node[left]{\mbox{\tiny{$1_\g\otimes h_\bullet$}}} (a);
		\end{tikzpicture}
	\end{center}
	We claim that the following sequence is exact:
	\[
	0\longrightarrow d(1_\g\otimes h_\bullet)(\g\otimes\Omega_\bullet)_0\xhookrightarrow{\quad}\cyc_0(\g\otimes\Omega_\bullet)\xrightarrow{1_\g\otimes p_\bullet}\cyc_0(\g\otimes C_\bullet)\longrightarrow0\ ,
	\]
	where $d$ denotes the differential of $\g\otimes\Omega_1$. The only thing that is not immediately obvious is the fact that
	\begin{equation}\label{eq:equality giving SES}
		\ker(1_\g\otimes p_\bullet)|_{\cyc_0(\g\otimes\Omega_\bullet)} = d(1_\g\otimes h_\bullet)(\g\otimes\Omega_\bullet)_0\ .
	\end{equation}
	We have
	\begin{align*}
		\ker(1_\g\otimes p_\bullet)|_{\cyc_0(\g\otimes\Omega_\bullet)} =&\ \ker(1 - d(1_\g\otimes h_\bullet) - (1_\g\otimes h_\bullet)d)|_{\cyc_0(\g\otimes\Omega_\bullet)}\\
		=&\ \ker(1 - d(1_\g\otimes h_\bullet))|_{\cyc_0(\g\otimes\Omega_\bullet)}\\
		=&\ \left\{x\in\cyc_0(\g\otimes\Omega_\bullet)\mid d(1_\g\otimes h_\bullet)(x) = x\right\}\ .
	\end{align*}
	An easy algebraic manipulation using the fact that we are working with a contraction shows that $(d(1_\g\otimes h_\bullet))^2 = d(1_\g\otimes h_\bullet)$ on $\g\otimes\Omega_\bullet$. Together with the fact that $d(1_\g\otimes h_\bullet)$ takes image in $\cyc_0(\g\otimes\Omega_\bullet)$, this implies (\ref{eq:equality giving SES}).
	
	\medskip
	
	Now we show that the first space of the short exact sequence is acyclic. Recall \cite[Lemma 3.2]{get09}, which states that the simplicial set $(\Omega_\bullet)_k$ is contractible for all $k$. If follows that $(\g\otimes\Omega_\bullet)_k$ is also contractible for all $k$, and in particular for $k=0$. Since
	\[
	d(1_\g\otimes h_\bullet):(\g\otimes\Omega_\bullet)_0\longrightarrow d(1_\g\otimes h_\bullet)(\g\otimes\Omega_\bullet)_0
	\]
	is a retraction, it follows that $d(1_\g\otimes h_\bullet)(\g\otimes\Omega_\bullet)_0$ is contractible, and thus that the Moore complex
	\[
	\moore(d(1_\g\otimes h_\bullet)(\g\otimes\Omega_\bullet)_0)
	\]
	is acyclic. By applying the Moore complex to the short exact sequence defined above, we obtain that
	\[
	\moore(1_\g\otimes p_\bullet):\moore(\cyc_0(\g\otimes\Omega_\bullet))\longrightarrow\moore(\cyc_0(\g\otimes C_\bullet))
	\]
	is a quasi-isomorphism.
	
	\medskip
	
	To conclude, consider the commutative square
	\begin{center}
		\begin{tikzpicture}
			\node (a) at (0,0){$\moore(\cyc_0(\g\otimes\Omega_\bullet))$};
			\node (b) at (5.5,0){$\moore(\cyc_0(\g\otimes C_\bullet))$};
			\node (c) at (0,-2){$\moore(\cyc_0(\h\otimes\Omega_\bullet))$};
			\node (d) at (5.5,-2){$\moore(\cyc_0(\h\otimes C_\bullet))$};
			
			\draw[->] (a) to node[above]{$\moore(1_\g\otimes p_\bullet)$} (b);
			\draw[->] (a) to node[left]{$\moore(\MC_\bullet(\phi))$} (c);
			\draw[->] (b) to node[right]{$\moore(\phi\otimes1_{C_\bullet})$} (d);
			\draw[->] (c) to node[above]{$\moore(1_\h\otimes p_\bullet)$} (d);
		\end{tikzpicture}
	\end{center}
	We already know that the horizontal arrows are quasi-isomorphisms. The right vertical arrow is also a quasi-isomorphisms. Indeed, the simplicial vector space $\cyc_0(\g\otimes C_\bullet)$ is the result of applying the Dold--Kan functor $\Gamma$ of \cref{subsect:Dold-Kan correspondence} to the truncation
	\[
	\cdots\stackrel{d}{\longrightarrow}\g_{-2}\stackrel{d}{\longrightarrow}\g_{-1}\stackrel{d}{\longrightarrow}\cyc_0(\g)\ .
	\]
	By the Dold--Kan correspondence --- \cref{thm:Dold-Kan} --- we have natural isomorphisms
	\[
	\g\otimes C_\bullet\cong\moore(\cyc_0(\g\otimes C_\bullet))\ .
	\]
	Since $\phi$ is a quasi-isomorphism, it induces a quasi-isomorphism between the respective truncations of $\g$ and $\h$. It follows that $\moore(\phi\otimes1_{C_\bullet})$ is a quasi-isomorphism. Therefore, the morphism $\moore(\MC_\bullet(\phi))$ must also be a quasi-isomorphism, which implies that the morphism $\MC_\bullet(\phi)$ itself is a weak equivalence by \cref{thm:homotopy of simplicial vector space is homology of Moore complex}.
\end{proof}

\subsection{The theorem holds for the zeroth homotopy group}

One now proves that the statement holds for the $0$th homotopy group, without any further assumption on the algebras or on the morphism than the ones of the theorem.

\begin{theorem}[{\cite[Sect. 3]{dr15}}]\label{thm:bijection of pi0 (DR)}
	Let $\g,\h$ and $\Phi:\g\rightsquigarrow\h$ be as in the statement of \cref{thm:Dolgushev-Rogers}. The induced map
	\[
	\pi_0(\MC_\bullet(\Phi)):\pi_0(\MC_\bullet(\g))\longrightarrow\pi_0(\MC_\bullet(\h))
	\]
	is bijective.
\end{theorem}

The proof of this result is technical and not so important for the rest of the present work, and will therefore be omitted.

\begin{remark}
	\cref{prop:rectification for deformation oo-groupoid} provides a clean alternative to the proof of \cite[Lemma B.2]{dr15}, a technical lemma of fundamental importance for the proof of \cref{thm:bijection of pi0 (DR)}, partially simplifying the proof of this result. Notice that the proof of \cref{prop:rectification for deformation oo-groupoid} is independent of \cref{thm:bijection of pi0 (DR)}.
\end{remark}

\subsection{Setting the basepoint to zero}

In order to avoid having to worry about the basepoints, one shows that these can always be set to be $0$. This step can probably be avoided, but it has the advantage of making the rest of the proof much cleaner.

\begin{lemma} \label{lemma:shift to zero DR}
	Let $\alpha\in\MC(\g)$, and let $\g^\alpha$ be the $\SLoo$-algebra obtained by twisting $\g$ by $\alpha$, that is the $\SLoo$-algebra with the same underlying graded vector space, but with differential
	\[
	d^\alpha(x) \coloneqq dx + \sum_{n\ge2}\frac{1}{(n-1)!}\ell_n(\alpha,\ldots,\alpha,x)
	\]
	and brackets
	\[
	\ell^\alpha(x_1,\ldots,x_m) \coloneqq \sum_{n\ge m}\frac{1}{(n-m)!}\ell_n(\alpha,\ldots,\alpha,x_1,\ldots,x_m)\ .
	\]
	Let
	\[
	\mathrm{Shift}_\alpha:\MC_\bullet(\g^\alpha)\longrightarrow\MC_\bullet(\g)
	\]
	be the isomorphism of simplicial sets\footnote{Notice that this assignment does not define a morphism of $\SLoo$-algebras. It only makes sense on the spaces of Maurer--Cartan elements.} given by
	$$\beta\in\g^\alpha\longmapsto\alpha+\beta\in\g\ .$$
	Then the following diagram commutes
	\begin{center}
		\begin{tikzpicture}
			\node (a) at (0,0){$\MC_\bullet(\g^\alpha)$};
			\node (b) at (4,0){$\MC_\bullet(\g)$};
			\node (c) at (0,-2){$\MC_\bullet(\h^\alpha)$};
			\node (d) at (4,-2){$\MC_\bullet(\h)$};
			
			\draw[->] (a) to node[above]{$\mathrm{Shift}_\alpha$} (b);
			\draw[->] (a) to node[left]{$\Phi^\alpha$} (c);
			\draw[->] (b) to node[right]{$\Phi$} (d);
			\draw[->] (c) to node[above]{$\mathrm{Shift}_{\MC(\Phi)(\alpha)}$} (d);
		\end{tikzpicture}
	\end{center}
	where
	$$\MC_\bullet(\Phi^\alpha)(\beta)\coloneqq \sum_{k\ge1}\phi^\alpha_k(\beta^{\otimes k})$$
	and
	$$\phi^\alpha_k(\beta_1\otimes\cdots\otimes\beta_k)\coloneqq\sum_{j\ge0}\frac{1}{j!}\phi_{k+j}(\alpha^{\otimes j}\otimes\beta_1\otimes\ldots\otimes\beta_k)$$
	is the twist of $\Phi$ by the Maurer--Cartan element $\alpha$. Here, we identified $\alpha\in\g$ with $\alpha\otimes1\in\g\otimes\Omega_\bullet$.
\end{lemma}

The proof is a straightforward unwinding of definitions, and will be omitted here. A conceptual approach to this kind of twisting procedures is given in \cite{dsv18}. Given $\SLoo$-algebras $\g,\h$ with basepoints $\alpha\in\MC(\g)$ and $\MC(\Phi)(\alpha)\in\MC(\h)$, then the shifts in the lemma above change the basepoints to $0\in\g$ and $0\in\h$.

\subsection{The induction step}\label{subsect:induction step DR}

We are now set to do the induction step. Let $n\ge2$, then we denote
\[
\g^{(n)}\coloneqq\g/\F_n\g\ ,
\]
and similarly for $\h$, and also
\[
\Phi^{(n)}:\g^{(n)}\rightsquigarrow\h^{(n)}
\]
the induced map, which is well defined because $\Phi$ is filtered.

\begin{proposition}\label{prop:induction step DR}
	For each $n\ge2$, the filtered quasi-isomorphism $\Phi$ induces a weak equivalence of simplicial sets
	\[
	\MC_\bullet(\Phi^{(n)}):\MC_\bullet(\g^{(n)})\longrightarrow\MC_\bullet(\h^{(n)})\ .
	\]
\end{proposition}

\begin{proof}
	The statement is true for $n=2$ by \cref{prop:abelian Loo-alg DR}, because the $\SLoo$-algebras $\g^{(2)}$ and $\h^{(2)}$ are abelian. Now suppose that the statement is true up to some $n$. We will prove that then it also holds for $n+1$.
	
	\medskip
	
	Consider the commutative diagram
	\begin{center}
		\begin{tikzpicture}
			\node (a) at (0,2.5){$0$};
			\node (b) at (2.5,2.5){$\F_n\g/\F_{n+1}\g$};
			\node (c) at (5.5,2.5){$\g^{(n+1)}$};
			\node (d) at (8.5,2.5){$\g^{(n)}$};
			\node (e) at (11,2.5){$0$};
			
			\node (A) at (0,0){$0$};
			\node (B) at (2.5,0){$\F_n\h/\F_{n+1}\h$};
			\node (C) at (5.5,0){$\h^{(n+1)}$};
			\node (D) at (8.5,0){$\h^{(n)}$};
			\node (E) at (11,0){$0$};
			
			\draw[->] (a) to (b);
			\draw[->] (b) to (c);
			\draw[->] (c) to (d);
			\draw[->] (d) to (e);
			
			\draw[->] (A) to (B);
			\draw[->] (B) to (C);
			\draw[->] (C) to (D);
			\draw[->] (D) to (E);
			
			\draw[->,line join=round,decorate,decoration={zigzag,segment length=4,amplitude=.9,post=lineto,post length=2pt}] (b) to (B);
			\draw[->,line join=round,decorate,decoration={zigzag,segment length=4,amplitude=.9,post=lineto,post length=2pt}] (c) to node[right]{$\Phi^{(n+1)}$} (C);
			\draw[->,line join=round,decorate,decoration={zigzag,segment length=4,amplitude=.9,post=lineto,post length=2pt}] (d) to node[right]{$\Phi^{(n)}$} (D);
		\end{tikzpicture}
	\end{center}
	where the horizontal maps are the natural inclusions and projections, and the vertical ones are all induced by $\Phi$. Notice that the leftmost vertical arrow is in fact a strict morphism, induced by the linear part $\phi_1$ of $\Phi$. In particular, it is a quasi-isomorphism between abelian $\SLoo$-algebras, and thus induces a weak equivalence of simplicial sets between the respective Maurer--Cartan spaces by \cref{prop:abelian Loo-alg DR}. The $\infty$-morphism $\Phi^{(n)}$ also induces a weak equivalence between Maurer--Cartan spaces by induction hypothesis. Therefore, applying the Maurer--Cartan functor we have
	\begin{center}
		\begin{tikzpicture}
		\node (b) at (0,2.5){$\MC_\bullet(\F_n\g/\F_{n+1}\g)$};
		\node (c) at (4,2.5){$\MC_\bullet(\g^{(n+1)})$};
		\node (d) at (8,2.5){$\MC_\bullet(\g^{(n)})$};
		
		\node (B) at (0,0){$\MC_\bullet(\F_n\h/\F_{n+1}\h)$};
		\node (C) at (4,0){$\MC_\bullet(\h^{(n+1)})$};
		\node (D) at (8,0){$\MC_\bullet(\h^{(n)})$};
		
		\draw[->] (b) to (c);
		\draw[->] (c) to (d);
		
		\draw[->] (B) to (C);
		\draw[->] (C) to (D);
		
		\draw[->] (b) to node[sloped,below]{$\sim$} (B);
		\draw[->] (c) to node[right]{$\MC_\bullet(\Phi^{(n+1)})$} (C);
		\draw[->] (d) to node[sloped,below]{$\sim$} node[right]{$\MC_\bullet(\Phi^{(n)})$} (D);
		\end{tikzpicture}
	\end{center}
	Since the canonical projection $\g^{(n+1)}\to\g^{(n)}$ is surjective, the induced map on the Maurer--Cartan spaces is a fibration of simplicial sets by \cref{thm:filtered surjections induce fibrations of MC}, and the term on the left in the diagram above is easily checked to be the fibre above $0$. The same thing holds for $\h$. We pass to the long exact sequences of homotopy groups --- which is in fact exact everywhere except at $\pi_0$ --- and using the $5$-lemma conclude that $\MC_\bullet(\Phi^{(n+1)})$ induces bijections between the respective $k$th homotopy groups, for all $k\ge2$. The case $k=0$ was covered by \cref{thm:bijection of pi0 (DR)}. We are left to check the case $k=1$. The relevant part of the long sequence is
	\begin{center}
		\small
		\begin{tikzpicture}
			\node (a) at (0,2.3){$\pi_2\MC_\bullet(\g^{(n)})$};
			\node (b) at (3.1,2.5){$\pi_1\MC_\bullet\left(\frac{\F_n\g}{\F_{n+1}\g}\right)$};
			\node (c) at (6.2,2.3){$\pi_1\MC_\bullet(\g^{(n+1)})$};
			\node (d) at (9.3,2.3){$\pi_1\MC_\bullet(\g^{(n)})$};
			\node (e) at (12.4,2.3){$\pi_0\MC_\bullet\left(\frac{\F_n\g}{\F_{n+1}\g}\right)$};
			
			\node (A) at (0,0){$\pi_2\MC_\bullet(\h^{(n)})$};
			\node (B) at (3.1,0){$\pi_1\MC_\bullet\left(\frac{\F_n\h}{\F_{n+1}\h}\right)$};
			\node (C) at (6.2,0){$\pi_1\MC_\bullet(\h^{(n+1)})$};
			\node (D) at (9.3,0){$\pi_1\MC_\bullet(\h^{(n)})$};
			\node (E) at (12.4,0){$\pi_0\MC_\bullet\left(\frac{\F_n\h}{\F_{n+1}\h}\right)$};
			
			\draw[->] (a) to node[above]{$\partial$} (b);
			\draw[->] (b) to (c);
			\draw[->] (c) to (d);
			\draw[->] (d) to node[above]{$\partial$} (e);
			
			\draw[->] (A) to node[above]{$\partial$} (B);
			\draw[->] (B) to (C);
			\draw[->] (C) to (D);
			\draw[->] (D) to node[above]{$\partial$} (E);
			
			\draw[->] (a) to node[sloped, above]{$\cong$} (A);
			\draw[->] (b) to node[sloped, above]{$\cong$} (B);
			\draw[->] (c) to (C);
			\draw[->] (d) to node[sloped, above]{$\cong$} (D);
			\draw[->] (e) to node[sloped, above]{$\cong$} (E);
		\end{tikzpicture}
		\normalsize
	\end{center}
	We need to prove that the central map is an isomorphism, but we cannot directly use the $5$-lemma to do it as the rightmost terms are just sets, not groups, and the map between them is only a bijection of sets. The proof, however, is an obstruction theoretical argument similar to the proof of the $5$-lemma.
	
	\medskip
	
	\emph{The map is surjective.} Let $y\in\pi_1\MC_\bullet(\h^{(n+1)})$. Denote by $\overline{y}$ its image in $\pi_1\MC_\bullet(\h^{(n+1)})$. Since the relevant vertical map is an isomorphism, there exists an $\overline{x}\in\pi_1\MC_\bullet(\g^{(n)})$ mapping to $\overline{y}$. Moreover, since $\overline{y}$ maps to $0\in\pi_0\MC_\bullet\left(\frac{\F_n\h}{\F_{n+1}\h}\right)$, then $\overline{x}$ maps to $0\in\pi_0\MC_\bullet\left(\frac{\F_n\g}{\F_{n+1}\g}\right)$. By ``exactness" of the sequence, there exists an $x\in\pi_1\MC_\bullet(\g^{(n+1)})$ mapping to $\overline{x}$. Denote by $y'$ the image of $x$ in $\pi_1\MC_\bullet(\h^{(n+1)})$. It is not necessarily equal to $y$. However, we have that $y'y^{-1}$ maps to $0\in\pi_1\MC_\bullet(\h^{(n)})$, and thus there exists an element $z\in\pi_1\MC_\bullet\left(\frac{\F_n\g}{\F_{n+1}\g}\right)$ mapping to it by doing first the vertical arrow, and then the horizontal arrow. Denote by $x'\in\pi_1\MC_\bullet(\g^{(n+1)})$ the image of $z$ under the horizontal arrow. Then $(x')^{-1}x$ maps to $y$, proving that the middle vertical map is surjective.
	
	\medskip
	
	\emph{The map is injective.} Suppose $x\in\pi_1\MC_\bullet(\g^{(n+1)})$ maps to $0\in\pi_1\MC_\bullet(\h^{(n+1)})$. Then, since the map $\pi_1\MC_\bullet(\g^{(n)})\to\pi_1\MC_\bullet(\h^{(n)})$ is an isomorphism, it follows that $x$ maps to $0\in\pi_1\MC_\bullet(\g^{(n)})$, and thus that	there exists $y\in\pi_1\MC_\bullet\left(\frac{\F_n\g}{\F_{n+1}\g}\right)$ mapping to $x$. Let $\overline{y}\in\pi_1\MC_\bullet\left(\frac{\F_n\g}{\F_{n+1}\g}\right)$ be the image of $y$. Then $\overline{y}$ maps to $0\in\pi_1\MC_\bullet(\h^{(n+1)})$, and it follows that there exists $\overline{z}\in\pi_2\MC_\bullet(\h^{(n)})$ mapping to it. Let $z\in\pi_2\MC_\bullet(\g^{(n)})$ be the preimage of $\overline{z}$. Then $z$ maps to $y$, and by exactness of the long sequence, it follows that $x=0$, concluding the proof.
\end{proof}

\subsection{Conclusion of the proof}

The only step missing for the proof of \cref{thm:Dolgushev-Rogers} is a passage to the limit.

\medskip

\cref{prop:induction step DR}, together with all we have said before, shows that $\MC_\bullet(\Phi^{(n)})$ is a weak equivalence for all $n\ge2$. Therefore, we have the following commutative diagram:
\begin{center}
	\begin{tikzpicture}
		\node (a) at (0,0){$\vdots$};
		\node (b) at (3,0){$\vdots$};
		\node (c) at (0,-2){$\MC_\bullet(\g^{(4)})$};
		\node (d) at (3,-2){$\MC_\bullet(\h^{(4)})$};
		\node (e) at (0,-4){$\MC_\bullet(\g^{(3)})$};
		\node (f) at (3,-4){$\MC_\bullet(\h^{(3)})$};
		\node (g) at (0,-6){$\MC_\bullet(\g^{(2)})$};
		\node (h) at (3,-6){$\MC_\bullet(\h^{(2)})$};
		
		\draw[->] (c) to node[above]{$\sim$} (d);
		\draw[->] (e) to node[above]{$\sim$} (f);
		\draw[->] (g) to node[above]{$\sim$} (h);
		\draw[->] (a) to (c);
		\draw[->] (c) to (e);
		\draw[->] (e) to (g);
		\draw[->] (b) to (d);
		\draw[->] (d) to (f);
		\draw[->] (f) to (h);
	\end{tikzpicture}
\end{center}
where all objects are Kan complexes, all horizontal arrows are weak equivalences, and all vertical arrows are fibrations of simplicial sets by \cref{thm:filtered surjections induce fibrations of MC}. It follows that the collection of horizontal arrows defines a weak equivalence between fibrant objects in the model category of tower of simplicial sets, see \cite[Sect. VI.1]{GoerssJardine}. The functor from towers of simplicial sets to simplicial sets given by taking the limit is right adjoint to the constant tower functor, which trivially preserves cofibrations and weak equivalences. Thus, the constant tower functor is a left Quillen functor, and it follows that the limit functor is a right Quillen functor. In particular, it preserves weak equivalences between fibrant objects. Applying this to the diagram above proves that $\MC_\bullet(\Phi)$ is a weak equivalence.

\section{The formal Kuranishi theorem}\label{sect:formal Kuranishi}

In this section, we present some results of Bandiera \cite{Bandiera}, \cite{ban17}, which in particular give another construction of Getzler's $\infty$-groupoid $\gamma_\bullet(\g)$ for a $\SLoo$-algebra $\g$. In particular, the results presented here will give a very important link between the results of \cref{ch:representing the deformation oo-groupoid} and the theory developed in the present chapter.

\subsection{Complete contractions and the homotopy transfer theorem}\label{subsect:complete contractions and htt}

The first results of \cite{ban17} we present give a categorical point of view on the homotopy transfer theorem. They were originally stated for $\L_\infty$-algebras, but we give a slightly more general version, using the generalized homotopy transfer theorem, see \cref{thm:generalized HTT}.

\medskip

We place ourselves in the context of proper complete chain complexes and algebras, cf. \cref{appendix:filtered stuff}.

\begin{definition}
	A \emph{complete contraction}\index{Complete!contraction} is a contraction of chain complexes
	\begin{center}
		\begin{tikzpicture}
		\node (a) at (0,0){$V$};
		\node (b) at (2,0){$W$};
		
		\draw[->] (a)++(.3,.1)--node[above]{\mbox{\tiny{$p$}}}+(1.4,0);
		\draw[<-,yshift=-1mm] (a)++(.3,-.1)--node[below]{\mbox{\tiny{$i$}}}+(1.4,0);
		\draw[->] (a) to [out=-150,in=150,looseness=4] node[left]{\mbox{\tiny{$h$}}} (a);
		\end{tikzpicture}
	\end{center}
	such that $(V,\F_\bullet V)$ is a proper complete chain complex, and the chain maps $h$ and $ip$ are filtered.
\end{definition}

Given such a complete contraction, one defines the filtration
\[
\F_nW\coloneqq i^{-1}(\F_nV)
\]
on $W$.

\begin{lemma}
	The filtration $\F_\bullet W$ is unique with respect to the property that both $p$ and $i$ are filtered morphisms. The filtered chain complex $(W,\F_\bullet W)$ is proper complete.
\end{lemma}

\begin{proof}
	The fact that $i,p$ are both filtered (with respect to $\F_\bullet W$) is obvious. Conversely, let $\F_\bullet W$ be any filtration of $W$ such that $i,p$ are filtered. Then for every $x\in\F_nW$ we must have $i(x)\in\F_nV$. Conversely, if $x\in i^{-1}(\F_nV)$, then $i(x)\in\F_nV$, and thus
	\[
	x = pi(x)\in\F_nW\ .
	\]
	The first statement follows.
	
	\medskip
	
	To show that $W$ is complete, notice that since they are filtered, the maps $i$ and $p$ induce chain maps
	\[
	i:W/\F_nW\longrightarrow V/\F_nV\qquad\text{and}\qquad p:V/\F_nV\longrightarrow W/\F_nW
	\]
	at every level of the filtration, which still satisfy $pi=1_{W/\F_nW}$. In particular, $p$ is surjective at all levels of the filtration. Then
	\[
	W = p(V)\cong p\left(\lim_n V/\F_nV\right)\cong \lim_nW/\F_nW\ ,
	\]
	where in the last identification we used the fact that $p$ is filtered to switch it over the limit.
\end{proof}

\begin{definition}
	A \emph{morphism $\phi$ between complete contractions} is
	\begin{center}
		\begin{tikzpicture}
		\node (a) at (0,0){$V$};
		\node (b) at (2,0){$W$};
		
		\draw[->] (a)++(.3,.1)--node[above]{\mbox{\tiny{$p$}}}+(1.4,0);
		\draw[<-,yshift=-1mm] (a)++(.3,-.1)--node[below]{\mbox{\tiny{$i$}}}+(1.4,0);
		\draw[->] (a) to [out=-150,in=150,looseness=4] node[left]{\mbox{\tiny{$h$}}} (a);
		
		\node (A) at (0,-2){$V'$};
		\node (B) at (2,-2){$W'$};
		
		\draw[->] (A)++(.3,.1)--node[above]{\mbox{\tiny{$p'$}}}+(1.4,0);
		\draw[<-,yshift=-1mm] (A)++(.3,-.1)--node[below]{\mbox{\tiny{$i'$}}}+(1.4,0);
		\draw[->] (A) to [out=-150,in=150,looseness=4] node[left]{\mbox{\tiny{$h'$}}} (A);
		
		\draw[->] (a) to node[left]{$\phi$} (A);
		\end{tikzpicture}
	\end{center}
	such that $\phi:V\to V'$ is filtered, and such that $h'\phi=\phi h$. The composite of two such morphisms is simply given by composing the vertical maps of the diagram above. This notion of morphism makes complete contractions into a category, which we denote by $\cctr$.
\end{definition}

There are two obvious ``projection" functors
\[
\pr_{1,2}:\cctr\longrightarrow\cchain
\]
giving back $V$ and $W$ respectively. Their action on morphisms of complete contractions are given by
\[
\pr_1(\phi)\coloneqq\phi:V\longrightarrow V',\quad\text{and}\quad\pr_2(\phi)\coloneqq p'\phi i:W\to W'
\]
respectively. The fact that $\pr_2$ commutes with compositions follows from the fact that in a contraction we have $pi=1_W$.

\medskip

The homotopy transfer theorem is well behaved with respect to complete contractions.

\begin{theorem}
	Let $\C$ be a cooperad, and suppose we have a complete contraction
	\begin{center}
		\begin{tikzpicture}
		\node (a) at (0,0){$A$};
		\node (b) at (2,0){$B$};
		
		\draw[->] (a)++(.3,.1)--node[above]{\mbox{\tiny{$p$}}}+(1.4,0);
		\draw[<-,yshift=-1mm] (a)++(.3,-.1)--node[below]{\mbox{\tiny{$i$}}}+(1.4,0);
		\draw[->] (a) to [out=-150,in=150,looseness=4] node[left]{\mbox{\tiny{$h$}}} (a);
		\end{tikzpicture}
	\end{center}
	where $A$ is a proper complete $\Cobar\C$-algebra with respect to the given filtration on $A$. Then the transfered structure makes $B$ in a proper complete $\Cobar\C$-algebra, and the induced $\infty$-quasi-isomorphisms $i_\infty$ and $p_\infty$ are filtered $\infty$-quasi-isomorphisms. Moreover, if we have two such contractions and a morphism $\phi$ of complete contractions which is a strict morphism of $\Cobar\C$-algebras, then $\pr_2(\phi)$ also is a strict morphism of $\Cobar\C$-algebras.
\end{theorem}

A couple further compatibility results can be found in \cite[Sect. 1]{ban17}. These imply the following result. Consider the filtered product $\widehat{\Cobar\C}\text{-}\mathsf{alg}\times_{\cchain}\cctr$ given by
\begin{center}
	\begin{tikzpicture}
		\node (a) at (0,0) {$\widehat{\Cobar\C}\text{-}\mathsf{alg}\times_{\cchain}\cctr$};
		\node (b) at (4,0) {$\cctr$};
		\node (c) at (0,-2.5) {$\widehat{\Cobar\C}\text{-}\mathsf{alg}$};
		\node (d) at (4,-2.5) {$\cchain$};
		
		\draw[->] (a) to (b);
		\draw[->] (a) to (c);
		\draw[->] (b) to node[right]{$\pr_1$} (d);
		\draw[->] (c) to node[above]{$(-)^\#$} (d);
	\end{tikzpicture}
\end{center}

\begin{proposition}
	The homotopy transfer theorem gives a functor
	\[
	\widehat{\Cobar\C}\text{-}\mathsf{alg}\times_{\cchain}\cctr\longrightarrow \widehat{\Cobar\C}\text{-}\mathsf{alg}\ .
	\]
\end{proposition}

\subsection{The formal Kuranishi theorem}\label{subsect:formal Kuranishi theorem}

We specialize now to $\SLoo$-algebras. Suppose we have a complete contraction
\begin{center}
	\begin{tikzpicture}
	\node (a) at (0,0){$\g$};
	\node (b) at (2,0){$\h$};
	
	\draw[->] (a)++(.3,.1)--node[above]{\mbox{\tiny{$p$}}}+(1.4,0);
	\draw[<-,yshift=-1mm] (a)++(.3,-.1)--node[below]{\mbox{\tiny{$i$}}}+(1.4,0);
	\draw[->] (a) to [out=-150,in=150,looseness=4] node[left]{\mbox{\tiny{$h$}}} (a);
	\end{tikzpicture}
\end{center}
with $\g$ a $\SLoo$-algebra, and consider the induced $\SLoo$-algebra structure on $\h$.

\begin{theorem}[{\cite[Thm. 1.13]{ban17}}]\label{thm:formal Kuranishi theorem}\index{Formal!Kuranishi theorem}
	The map of sets
	\[
	\MC(\g)\longrightarrow\MC(\h)\times h(\g_0)
	\]
	given by
	\[
	x\longmapsto(\MC(p_\infty)(x),h(x))
	\]
	is bijective.
\end{theorem}

The proof is done by considering and solving recursively certain fixed point equations in nilpotent $\SLoo$-algebras, and then concluding by passing to the limit. We will not do do it here, and we refer the reader to the original references \cite[Thm. 1.13]{ban17} and \cite[Thm. 2.3.3]{Bandiera}.

\begin{remark}
	The reason behind the name of this theorem is the fact that it generalizes a result by Kuranishi \cite{kur62}. See \cite[p. 3]{get18} for a nice exposition.
\end{remark}

We are mostly interested by the following consequence. Consider the contraction
\begin{center}
	\begin{tikzpicture}
	\node (a) at (0,0){$\g\otimes\Omega_\bullet$};
	\node (b) at (2.6,0){$\g\otimes C_\bullet$};
	
	\draw[->] (a)++(.6,.1)--node[above]{\mbox{\tiny{$1_\g\otimes p_\bullet$}}}+(1.4,0);
	\draw[<-,yshift=-1mm] (a)++(.6,-.1)--node[below]{\mbox{\tiny{$1_\g\otimes i_\bullet$}}}+(1.4,0);
	\draw[->] (a) to [out=-150,in=150,looseness=4] node[left]{\mbox{\tiny{$1_\g\otimes h_\bullet$}}} (a);
	\end{tikzpicture}
\end{center}
induced by Dupont's contraction. \cref{thm:formal Kuranishi theorem} implies the following.

\begin{corollary}
	The map
	\[
	\gamma_\bullet(\g)\longrightarrow\MC(\g\otimes C_\bullet)
	\]
	given by
	\[
	x\longmapsto\MC((1_\g\otimes p)_\infty)(x)
	\]
	is an isomorphism of simplicial sets.
\end{corollary}

\begin{proof}
	The result is a direct consequence of applying \cref{thm:formal Kuranishi theorem} to the contraction above, and considering the restriction of the bijection to $\gamma_\bullet(\g) = \MC_\bullet(\g)\cap\ker(1_\g\otimes h_\bullet)$.
\end{proof}

%% file: DeformationTheory.tex
\chapter{Deformation theory}\label{chapter:deformation theory}

The goal of this chapter is to motivate the interest in the space of Maurer--Cartan elements of Lie and homotopy Lie algebras. The archetypal deformation problem goes like this. One is given some fixed object, say a vector space $V$, or a manifold $M$, and a structure on this object, such as an associative algebra structure on $V$, or a complex structure on $M$. Then the goal is to understand how one can ``deform'' the given structure in such a way as to obtain another structure of the same kind on the same base object, and to understand which deformations give us structures that are isomorphic to the original one, and which ones don't. This way, one can obtain information about the original structure.

\medskip

Lie algebras, their Maurer--Cartan elements, and gauge equivalences appear naturally in deformation theory, and provide a generalized framework to study deformation problems. In fact, it is a ``philosophical principle'' attributed to Deligne \cite{letterDeligneMillson} that, in characteristic zero, every deformation problem corresponds to the study of Maurer--Cartan elements and the gauges between them in a (possibly homotopy) Lie algebra. This principle has been formalized in recent years by Pridham \cite{pri10} and Lurie \cite{lur11} to give a theorem in the context of $\infty$-categories.

\medskip

Deformation theory is a deep and beautiful subject with a long history and plenty of applications. Its study goes back to the works of Grothendieck, Artin, Quillen and many others, up until the more recent works of Kontsevich in deformation quantization and mirror symmetry, and the already mentioned formalization of Deligne's principle by Pridham and Lurie.

\medskip

We will introduce deformation theory by the means of an algebraic example: deformations of associative algebra structures on a vector space. We will begin by talking about infinitesimal deformations, before passing to general deformations, introducing the deformation complex, and state the fundamental principle of deformation theory.  We will then conclude the section by presenting various other examples of deformation problems, both algebraic and geometrical, and the Lie algebras governing them.

\medskip

In this section, we work in cochain complexes. Some good introductory references for deformation theory are e.g. the notes of Kontsevich \cite{notesKontsevich} and \cite{sze99}.

\section{Deformation theory through an example}

In this section, we attempt to give an overview of deformation theory through the archetypal example of deformations of associative algebra structures on a vector space. We begin by explaining what an infinitesimal deformation is, before passing to more general deformations, and then presenting the differential graded Lie algebra governing this specific deformation problem. We conclude the section by presenting the fundamental principle of deformation theory, due to P. Deligne, which relates deformation problems in characteristic $0$ to Maurer--Cartan elements of (homotopy) Lie algebras.

\medskip

The exposition given here is heavily based on the introductory note to deformation theory \cite{sze99} by B. Szendr\H{o}i.

\subsection{Infinitesimal deformations}\label{subsect:infinitesimal deformations}

Let $A$ be an associative algebra, not differential graded. In other words, $A$ is a finite dimensional vector space together with a linear map
\[
m:A\otimes A\longrightarrow A
\]
satisfying the associativity condition
\[
m(m(a,b),c) = m(a,m(b,c))
\]
for all $a,b,c\in A$.

\medskip

One wants now to ``infinitesimally deform'' $m$ in order to obtain another associative algebra structure on $A$. In order to do this, one considers structures of the form $\widetilde{m}\coloneqq m+\epsilon f$ with $f\in\hom(A\otimes A,A)$ and $\epsilon$ a formal parameter satisfying $\epsilon^2 = 0$. This can be formalized by saying that we are looking at elements of
\[
\hom(A\otimes A,A)\otimes\k[\epsilon]/(\epsilon^2)\ ,
\]
which corresponds to look at the ``tangent space'' of our space of structures\footnote{From an algebro-geometrical point of view, $\k[\epsilon]/(\epsilon^2)$ models tangent vectors.}. The associativity condition now reads
\begin{align*}
	0 =&\ \widetilde{m}(\widetilde{m}(a,b),c) - \widetilde{m}(a,\widetilde{m}(b,c))\\
	=&\ m(m(a,b),c) - m(a,m(b,c)) + \epsilon\big(f(m(a,b),c) - f(a,m(b,c)) + m(f(a,b),c)-m(a,f(b,c)\big)\\
	=&\ \epsilon\big(f(m(a,b),c) - f(a,m(b,c)) + m(f(a,b),c)-m(a,f(b,c)\big)
\end{align*}
where in the second line we used the fact that $\epsilon^2=0$, and in the third line the fact that $m$ was associative to begin with. Thus, $\widetilde{m}$ being associative is equivalent to
\[
f(m(a,b),c) - f(a,m(b,c)) + m(f(a,b),c)-m(a,f(b,c) = 0\ .
\]
One also wants to understand when a deformed structure is isomorphic to the original one. In order to do this, one considers automorphisms of $A$ of the form $T\coloneqq1_A+\epsilon g$ for $g\in\hom(A,A)$. Notice that such a morphism is indeed invertible, with its inverse being given by $T^{-1} = 1_A-\epsilon g$. Then one looks at the pullback of the multiplication $m$ by $T$:
\begin{align*}
	Tm(T^{-1}x,T^{-1}y) =&\ m(x,y) + \epsilon\big(\underbrace{g(m(x,y)) - m(g(x),y) - m(x,g(y))}_{f_T\coloneqq}\big)\ .
\end{align*}
This tells us that a deformed multiplication $\widetilde{m} = m+\epsilon f$ is isomorphic to the original multiplication $m$ if, and only if
\[
f = g(m(x,y)) - m(g(x),y) - m(x,g(y))
\]
for some $g\in\hom(A,A)$.

\medskip

This can be wrapped up together nicely as the first levels of a homology theory. Indeed, one considers the Hochschild cochain complex\index{Hochschild complex}
\[
\Hoch(A,m)\coloneqq\Big(\hom(\k,A)\stackrel{d}{\longrightarrow}\hom(A,A)\stackrel{d}{\longrightarrow}\hom(A^{\otimes 2},A)\stackrel{d}{\longrightarrow}\hom(A^{\otimes 2},A)\stackrel{d}{\longrightarrow}\cdots\Big)\ ,
\]
where we see $\hom(A^{\otimes n},A)$ as living in degree $n-1$, with the differential being given by
\begin{align*}
	d(f)&(a_1,\ldots,a_{n+1}) =\\
	&= m(a_1,f(a_2,\ldots,a_{n+1})) + \sum_{k=1}^n(-1)^kf(a_1,\ldots,m(a_k,a_{k+1}),\ldots,a_{n+1}) +\\
	&\qquad+ (-1)^{n+1}m(f(a_1,\ldots,a_n),a_{n+1})\ .
\end{align*}
This complex was introduced by G. Hochschild in \cite{hoc45}. It is an easy exercise to check that this differential squares to zero. Then the infinitesimally deformed associative structures are exactly the cycles in $\Hoch_1(A,m)$, and two such structures are isomorphic if, and only if the respective cycles are cohomologous. Thus, we have
\[
\{\text{essentially distinct infinitesimal deformations of }m\}\cong H^1(\Hoch(A,m))\ .
\]
This is what normally happens in general when one tries to deform infinitesimally some kind of structure. One is led to some cochain complex where the cycles at some level correspond to the deformed structures, and which are cohomologous if, and only if the deformed structures are equivalent. This happens because working infinitesimally --- i.e. over the ring $\k[\epsilon]/(\epsilon^2)$ --- correspond to linearizing the higher structure hiding behind more general deformations: the structure of a Lie algebra, where the cycle equation will be replaced by the Maurer--Cartan equation. One can already see this happening when trying to understand deformations to a higher order, working over $\k[\epsilon]/(\epsilon^n)$, or over an artinian ring in general, and also with formal deformations, i.e. working over the ring of formal power series $\k[[t]]$.

\subsection{Deformations in general}\label{subsect:deformations in general}

Now one wants to understand deformations in general. A deformation of the original associative structure $m$ is just
\[
\widetilde{m}\coloneqq m+f
\]
with $f\in\hom(A\otimes A,A)$. The associativity condition now reads
\[
f(m(a,b),c) - f(a,m(b,c)) + m(f(a,b),c) - m(a,f(b,c)) + f(f(a,b),c) - f(a,f(b,c)) = 0\ ,
\]
where now the last two terms of the left-hand side --- which are quadratic in $f$ --- do not necessarily vanish. If one writes
\[
[f_1,f_2](a,b,c)\coloneqq f_1(f_2(a,b),c) - f_1(a,f_2(b,c)) + f_2(f_1(a,b),c) - f_2(a,f_1(b,c))\ ,
\]
then $[-,-]$ is graded antisymmetric (if one considers $f_1,f_2$ as elements of degree $1$) and satisfies the Jacobi rule. With this in mind, we can write the associativity condition as
\[
df +\frac{1}{2}[f,f] = 0\ ,
\]
where $d$ is the differential of the Hochschild complex defined in \cref{subsect:infinitesimal deformations}, which can also be written as $d\coloneqq[m,-]$. Thus, assuming that we can extend this bracket to define a Lie bracket on the whole Hochschild complex, deformations of $m$ correspond to Maurer--Cartan elements of $\Hoch_\bullet(A,m)$.

\medskip

To understand what deformations are equivalent, suppose that $A$ is finite dimensional. To any linear map $\hom(A,A)$ we associate an automorphism of $A$ by
\[
e^g = \sum_{n\ge0}\frac{1}{n!}g^n\in\hom(A,A)\ ,
\]
whose inverse is given by $e^{-g}$. We consider the path $t\mapsto e^{tg}$ and the algebra structure
\[
m_t(x,y)\coloneqq e^{tg}m(e^{-tg}x,e^{-tg}y)\ .
\]
This can be written as
\[
m_t = m + f_t\ ,
\]
for $f_t\coloneqq m_t-m\in\hom(A\otimes A,A)$. Then
\[
\frac{d}{dt}f_t(x,y) = gm(x,y) + m(g(x),y) + m(x,g(y)) + gf_t(x,y) + f_t(g(x),y) + f_t(x,g(y))\ .
\]
Defining
\[
[g,f](x,y)\coloneqq g(f(x,y)) + f(g(x),y) + f(x,g(y))
\]
for $f\in\hom(A\otimes A,A)$ and $g\in\hom(A,A)$, we can rewrite this as
\[
\frac{d}{dt}f_t = df_t + [g,f_t]\ ,
\]
which one recognizes as the differential equation saying that $g$ is a gauge from $m$ to the pullback $e^gm(e^{-g},e^{-g})$ of $m$ by $e^g$.

\subsection{The deformation complex}

Consider the following intrinsic version of the Hochschild complex, corresponding to the trivial algebra structure $m=0$ on $A$. The cochain complex is given by
\[
\Hoch_n(A)\coloneqq\hom(A^{\otimes(n+1)},A)
\]
for $n\ge-1$, with trivial differential\footnote{If $A$ is itself a cochain complex, then one takes the differential induced by the differential of $A$.}. One defines a Lie bracket on this complex by
\begin{align*}
	[f,g](a_1,\ldots,a_{n+m-1})\coloneqq&\ \sum_{i=1}^n(-1)^{im}f(a_1,\ldots,g(a_i,\ldots,a_{i+m-1}),\ldots,a_{n+m-1})+\\
	&+ \sum_{j=1}^n(-1)^{jn}g(a_1,\ldots,g(a_j,\ldots,a_{j+n-1}),\ldots,a_{n+m-1})
\end{align*}
for $f\in\hom(A^{\otimes n},A)$ and $g\in\hom(A^{\otimes m},A)$. This bracket was introduced by M. Gerstenhaber \cite{ger63}, and is therefore usually called the Gerstenhaber bracket\index{Gerstenhaber bracket}. Notice that this recovers the bracket given in \cref{subsect:deformations in general}. The Maurer--Cartan elements of $\Hoch(A)$ are the elements
\[
m\in\hom(A\otimes A,A) = \Hoch_1(A)
\]
such that the Maurer--Cartan equation
\[
0 = \frac{1}{2}[m,m] = \frac{1}{2}(m(m\otimes1_A) - m(1_A\otimes A))
\]
is satisfied, i.e. the associative algebra structures on $A$. Moreover, by what we have seen before, two such algebraic structures are isomorphic if, and only if the corresponding Maurer--Cartan elements are gauge equivalent. Therefore, we have
\[
\{\text{isomorphism classes of associative algebra structures on }A\}\cong\MCbar(\Hoch(A))\ .
\]
If one wants to look at deformations of a given associative algebra structure $m$, i.e. study the elements $f\in\hom(A\otimes A,A)$ such that $m+f$ is again an associative algebra structure on $A$, then one can do so by means of a \emph{twist}. Namely, one considers the Lie algebra $\Hoch(A,m)\coloneqq\Hoch(A)^m$ which has the same underlying graded vector space, the same Lie bracket, but differential
\[
d^m\coloneqq[m,-]\ .
\]
Since $m$ is a Maurer--Cartan element, this is again a Lie algebra, and $f$ is a Maurer--Cartan element in $\Hoch(A,m)$ if, and only if $m+f$ is a Maurer--Cartan element in $\Hoch(A)$. Cf. also \cref{lemma:shift to zero DR}.

\medskip

Finally, if one wants to recover infinitesimal deformations, it suffices to consider the Lie algebra
\[
\Hoch(A)\otimes\k[\epsilon]/(\epsilon^2)\ .
\]
If $m$ is a Maurer--Cartan element of $\Hoch(A)$, i.e. an associative algebra structure on $A$, then $m+\epsilon f$ is a Maurer--Cartan element of $\Hoch(A)\otimes\k[\epsilon]/(\epsilon^2)$ if, and only if it is an infinitesimal deformation of $m$.

\subsection{The fundamental principle of deformation theory}

One might think that the situation above is specific to the case of associative algebras, or maybe of some specific algebraic situations. This is not the case, and a plethora of examples have been found during the years, both in the algebraic setting, in geometrical problems, and even in situations strictly linked with physics. We will expose some of those in \cref{sect:other examples and applications}. In fact, this situation is so ubiquitous that it led P. Deligne \cite{letterDeligneMillson} to formulate the following heuristic principle, often referred to as the \emph{fundamental principle of deformation theory}:
\begin{center}
	\begin{minipage}{10cm}
		``When working over a field of characteristic zero, all deformation problems can be formulated as the study of the Maurer--Cartan elements in a differential graded Lie algebra.''
	\end{minipage}
\end{center}
More generally, one might want to consider homotopy Lie algebras instead of strict Lie algebras. To be more precise, the principle above means that for every deformation problem there exists a (homotopy) Lie algebra $\g$ such that one can establish the following dictionary:
\begin{center}
	\begin{tabular}{|c|c|}
		\hline
		\textbf{Deformation problem} & \textbf{Homotopy Lie algebra $\g$}\\
		\hline
		Structures of the desired type & Maurer--Cartan elements $\MC(\g)$\\
		\hline
		Equivalences between structures & Gauges\\
		\hline
		Infinitesimal deformations & Study of $\g\otimes\k[\epsilon]/(\epsilon^2)$\\
		\hline
		Formal deformations & Study of $\g\otimes\k[[t]]$\\
		\hline
	\end{tabular}
\end{center}
And so on. One can also identify the space of infinitesimal deformation of a fixed Maurer--Cartan element $\alpha$ with the (formal) tangent space of $\MC(\g^\alpha)$, and thus deduce e.g. that the dimension of the moduli space of the structures of the desired type at $\alpha$ is the same as the dimension of the cohomology $H^1(\g^\alpha)$.

\medskip

To conclude this section, we should mention that while the principle exposed above is just a heuristic, it has in fact been made into a formal result in the context of $\infty$-categories by Pridham \cite{pri10} and Lurie \cite{lur11}. There, a sensible $\infty$-category of deformation problems is defined, and it is shown that it is equivalent to the category of differential graded Lie algebras as an $\infty$-category through what is essentially the Maurer--Cartan functor.

\section{Other examples and applications}\label{sect:other examples and applications}

In order to demonstrate the importance of deformation theory in mathematics, we sketch various classical examples and applications.

\subsection{Lie algebras and commutative algebras}

If one is interested in deforming a fixed Lie algebra structure on a vector space $\g$, then proceeding as for associative algebras one is led to consider the Chevalley--Eilenberg complex\index{Chevalley--Eilenberg complex}
\[
\CE_n(\g)\coloneqq\hom(\Lambda^{n+1}A,A)
\]
for $n\ge-1$, where $\Lambda^nA$ denotes the $n$th exterior power of $A$, with differential given by
\begin{align*}
	d(f)(x_0,\ldots,x_{n})\coloneqq&\ \sum_{i<j}(-1)^{i+j}f([x_i,x_j],x_0,\ldots,\widehat{x}_i,\ldots,\widehat{x}_j,\ldots,x_n) +\\
	&+ \sum_{k=1}^{n+1}(-1)^{k+1}[x_k,f(x_0,\ldots,\widehat{x}_k,\ldots,x_n)]\ ,
\end{align*}
where the hat denotes omission. This cochain complex was introduced by C. Chevalley and S. Eilenberg in \cite{ce48}. One puts the following bracket on the Chevalley--Eilenberg complex. For $f\in\hom(\Lambda^nA,A)$ and $g\in\hom(\Lambda^mA,A)$, one defines
\begin{align*}
	[f,g](x_1,\ldots,x_{n+m-1})\coloneqq&\ \sum_{\sigma\in\sh(m,n-1)}(-1)^\sigma f(g(x_{\sigma(1)},\ldots,x_{\sigma(m)}),x_{\sigma(m+1)},\ldots,x_{\sigma(n+m-1)})\\
	&\ + \sum_{\sigma\in\sh(n,m-1)}(-1)^\sigma g(f(x_{\sigma(1)},\ldots,x_{\sigma(n)}),x_{\sigma(n+1)},\ldots,x_{\sigma(n+m-1)})\ .
\end{align*}
This bracket makes the Chevalley--Eilenberg complex into a differential graded Lie algebra, providing the deformation complex for Lie algebra structures on $\g$.

\medskip

For commutative algebra structures, the situation is similar. The cochain complex one considers is the Harrison complex, which was defined in \cite{har62} by D. K. Harrison. Its explicit description is more complicated, and we will not carry it out here.

\subsection{Algebras over an operad}

More generally, let $\P\coloneqq\Cobar\C$ be a cofibrant operad, where $\C$ is a reduced cooperad. Fix a chain complex $A$. Then the deformation complex of $\P$-algebra structures\index{Deformation!complex of $\P$-algebras} on $A$ is the Lie algebra
\[
\Def_\P(A)\coloneqq\hom(\C,\End_A)\ ,
\]
where the Lie bracket is given by
\[
[f,g]\coloneqq f\star g - g\star f\ ,
\]
with $\star$ the pre-Lie product described in \cref{subsect:operadic twisting morphisms}. One can in fact directly work in the pre-Lie setting to obtain powerful results, see \cite{dsv16}. The fact that this is the correct deformation complex is given by the following result, which follows from \cref{thm:Rosetta stone operads} and \cite[Thm. 3]{dsv16}.

\begin{theorem}
	The Maurer--Cartan elements of $\Def_\P(A)$ are in bijective correspondence with the possible $\P$-algebra structures on $A$. Moreover, two Maurer--Cartan elements are gauge equivalent if, and only if the respective $\P$-algebra structures are $\infty$-isotopic, i.e. if they are linked by an $\infty$-morphism whose first component is the identity.
\end{theorem}

One can recover the Hochschild, Chevalley--Eilenberg, and Harrison complex as subcomplexes of $\Def_\P(A)$ if one takes $\P=\as$, $\lie$, and $\com$ respectively.

\subsection{Complex analytic structures on a manifold}

All the examples we have given until now are algebraic in nature, but deformation theory also works in geometrical situations. One such example is the deformation of complex analytic structures on manifolds, due to K. Kodaira and D. C. Spencer \cite{ks58i}, \cite{ks58ii}. The deformation complex in this case is given in degree $n\ge0$ by the cochain complex of holomorphic vector fields tensor $(0,n)$-forms.

\medskip

A first intuition on this problem goes back to Riemann, who in 1857 famously calculated that the ``number of independent parameters'' on which the deformations of a complex structure on a closed Riemann surface of genus $g$, i.e. the dimension of the moduli space of complex structures, is $3g-3$.

\subsection{Deformation quantization}

A last famous example is due to M. Kontsevich \cite{kon03}, and is closely related to mathematical physics, more precisely to the quantization procedure.

\medskip

The physical motivation is as follows. One encodes physics via a ``space of states'' given by a manifold $M$ and ``observables'', the smooth functions on $M$. Physics is then given by how the observables interact. If one wants to describe classical mechanics, then one must ask for a Poisson structure on $C^\infty(M)$. For quantum mechanics, one must give a star product on $C^\infty(M)$, which is a deformation of the usual pointwise product of functions as a formal power series in a parameter $\hbar$. One can recover a Poisson bracket in the first order of the expansion of a star product, which is interpreted as the procedure of getting back classical physics from quantum mechanics. The question is how to go the other way: given a Poisson structure on $M$, is it possible to find a star product recovering the Poisson bracket? In other words, is it always possible to quantize classical mechanics?

\medskip

To give some details, a star product on $C^\infty(M)$ is an associative, $\mathbb{R}[[\hbar]]$-linear product on the space $C^\infty(M)[[\hbar]]$ of formal power series of smooth functions in $\hbar$ of the form
\[
f\star g = fg + \sum_{n\ge1}B_n(f,g)\hbar^n,
\]
where the $B_n$ are bidifferential operators on $C^\infty(M)$. One obtains a Poisson bracket from a star product by
\[
\{f,g\}\coloneqq B_1(f,g)-B_1(g,f)\ .
\]
Two star product are said to be equivalent if they are related by a certain type of gauge relation. Kontsevich associated a deformation complex to star products, and to Poisson brackets, and then proved that the two complexes are linked by a zig-zag of $\infty$-quasi-isomorphisms. This implies that equivalence classes of Poisson structures are in bijection with equivalence classes of star products, answering the quantization question in the positive.

%% file: RelativeInfinityMorphisms.tex
\chapter{\texorpdfstring{$\infty$}{infinity}-morphisms relative to twisting morphisms}\label{chapter:oo-morphisms relative to a twisting morphism}

In \cref{subsection:homotopy algebras and homotopy morphisms}, a notion of $\infty$-morphisms between $\P_\infty$-algebras was introduced, where $\P$ is a Koszul operad. Namely, such an $\infty$-morphism is the same thing as a morphism of conilpotent $\P^{\antishriek}$-coalgebras between the bar constructions of the algebras relative to the canonical twisting morphism
\[
\iota:\P^{\antishriek}\longrightarrow\P_\infty\ ,
\]
cf. \cref{notation:iota and pi}. There is a natural, useful generalization of this notion. Namely, one can consider morphisms of coalgebras between the bar construction with respect to \emph{any} twisting morphism
\[
\alpha:\C\longrightarrow\P\ .
\]
Dually, this gives us a notion of $\infty$-morphism between conilpotent coalgebras by considering morphisms of algebras between their cobar constructions relative to a twisting morphism. Versions of this idea have already been used more or less explicitly in the literature e.g. in \cite{mar04}, \cite{ber14transfer}, and \cite{ber14koszul}, at least for algebras. In this section, we give a precise definition of this generalized notion of $\infty$-morphisms, and then study some of their properties. In particular, we look in detail at their homotopical behavior.

\medskip

We expect the theory presented here to work without great changes for morphisms between connected weight graded (co)operads. We give here the case of reduced (co)operads for ease of presentation, and because all the cases of more immediate interest (e.g. associative, commutative, and Lie algebras, as well as their up to homotopy counterparts) are included in this framework.

\medskip

Most of the material of this chapter is extracted from the article \cite{rnw17}.

\section{Basic definitions and rectifications}

In this section, we define the notion of $\infty$-morphisms of algebras and coalgebras relative to a twisting morphism, and study some of their basic properties.

\subsection{Basic definitions}

The definition of $\infty$-morphisms of algebras and coalgebras relative to an arbitrary twisting morphism is the following one.

\begin{definition}\label{def:infinity morphisms relative to a twisting morphism}
	Let $\C$ be a cooperad, let $\P$ be an operad, and let $\alpha:\C\to\P$ be a twisting morphism.
	\begin{enumerate}
		\item An \emph{$\infty$-morphism of $\P$-algebras relative to $\alpha$}\index{$\infty$-morphism!of algebras relative to $\alpha$}, or an \emph{$\infty_\alpha$-morphism of $\P$-algebras}\index{$\infty_\alpha$-morphism!of algebras}, between two $\P$-algebras $A$ and $A'$ is a morphism $\Psi$ of $\C$-coalgebras
		\[
		\Psi:\Bar_\alpha A\longrightarrow\Bar_\alpha A'\ .
		\]
		Composition of $\infty_\alpha$-morphisms of $\P$-algebras is given by the standard composition of morphisms of $\C$-coalgebras between the bar constructions. We denote the category of $\P$-algebras with $\infty_\alpha$-morphisms by $\infty_\alpha\text{-}\Palg$.
		\item An \emph{$\infty$-morphism of conilpotent $\C$-coalgebras relative to $\alpha$}\index{$\infty$-morphism!of coalgebras relative to $\alpha$}, or an \emph{$\infty_\alpha$-morphism of conilpotent $\C$-algebras}\index{$\infty_\alpha$-morphism!of coalgebras}, between two conilpotent $\C$-algebras $D'$ and $D$ is a morphism $\Phi$ of $\P$-algebras
		\[
		\Phi:\Cobar_\alpha D'\longrightarrow\Cobar_\alpha D\ .
		\]
		Composition of $\infty_\alpha$-morphisms of $\C$-coalgebras is given by the standard composition of morphisms of $\P$-algebras between the cobar constructions. We denote the category of conilpotent $\C$-coalgebras with $\infty_\alpha$-morphisms by $\infty_\alpha\text{-}\Ccog$.
	\end{enumerate}
\end{definition}

\begin{remark}
	If $\P$ is a Koszul operad and
	\[
	\iota:\P^{\antishriek}\longrightarrow\P_\infty
	\]
	is the canonical twisting morphism, then the notion of $\infty_\iota$-morphisms of $\P_\infty$-algebras coincides with the classical one, cf. \cref{prop:extended bar to oo-morphisms}.
\end{remark}

When confronted with such a definition, it is natural to wonder what happens to the notion of $\infty_\alpha$-morphism under changes in the twisting morphism $\alpha$. Here is a first result in that direction.

\begin{lemma}\label{lemma:equality of oo-morphisms}
	Let $\C',\C$ be two cooperads and let $\P$ be an operad. Let $\alpha\in\Tw(\C,\P)$, let $f:\C'\to\C$ be a morphism of cooperads, and let $D$ be a conilpotent $\C'$-coalgebra. Then
	\[
	\Cobar_\alpha(f_*C) = \Cobar_{f^*\alpha}C\ .
	\]
	In particular, $\infty_{f^*\alpha}$-morphisms between conilpotent $\C'$-coalgebras are the same as $\infty_\alpha$-morphisms between the same coalgebras seen as $\C$-coalgebras by pushforward of the structure along $f$.
	
	\medskip
	
	Dually, let $\C$ be a cooperad and let $\P,\P'$ be two operads. Let $\alpha\in\Tw(\C,\P)$, let $g:\P\to\P'$ be a morphism of operads, and let $A$ be a $\P'$-algebra. Then
	\[
	\Bar_\alpha(g^*A) = \Bar_{g_*\alpha}A\ .
	\]
	In particular, $\infty_{g_*\alpha}$-morphisms between $\P'$-algebras are the same as $\infty_\alpha$-morphisms between the same algebras seen as $\P$-algebras by pullback of the structure along $g$.
\end{lemma}

\begin{proof}
	We only prove the first of the two facts, the proof of the second one being dual. As algebras over graded vector spaces, it is clear that we have
	\[
	\Omega_{f^*\alpha}C = \P(C) = \Omega_{\alpha}(f_*C)\ ,
	\]
	so that we only have to check that the differentials agree.	We have
	\[
	d_{\Omega_{f^*\alpha}C} = d_{\P(C)} + d^{f^*\alpha}_2
	\]
	with $d^{f^*\alpha}_2$ given by the composite
	\[
	\P(C)\xrightarrow{1_\P\circ'\Delta_C}\P\circ(C;\C'\circ C)\xrightarrow{1_\P\circ(1_C;f^*\alpha\circ1_C)}\P\circ(C;\P\circ C)\cong(\P\circ_{(1)}\P)(C)\xrightarrow{\gamma_{(1)}\circ1_C}\P(C)\ .
	\]
	The part $d_{\P(C)}$ is independent of the twisting morphism, and thus of no interest to us. For the other part, we notice that
	\begin{align*}
	(1_\P\circ(1_C;f^*\alpha\circ1_C))(1_\P\circ'\Delta_C) =&\ (1_\P\circ(1_C;\alpha f\circ1_C))(1_\P\circ'\Delta_C)\\
	=&\ (1_\P\circ(1_C;\alpha\circ1_C))(1_\P\circ'f\Delta_C)\\
	=&\ (1_\P\circ(1_C;\alpha_1\circ1_C))(1_\P\circ'\Delta_{f_*C})\ ,
	\end{align*}
	which implies the result.
\end{proof}

\subsection{Rectifications}

Suppose we have a commutative diagram
\begin{center}
	\begin{tikzpicture}
		\node (a) at (0,1.5){$\C'$};
		\node (b) at (0,0){$\C$};
		\node (c) at (2,0){$\P$};
		
		\draw[->] (a) -- node[left]{$f$} (b);
		\draw[->] (a) -- node[above right]{$f^*\alpha$} (c);
		\draw[->] (b) -- node[below]{$\alpha$} (c);
	\end{tikzpicture}
\end{center}
where $\alpha$ is a twisting morphism, and where $f$ is a morphism of cooperads. Then \cref{lemma:equality of oo-morphisms} tells us that whenever we are given two conilpotent $\C'$-coalgebras, the $\infty_{f^*\alpha}$-morphisms between them are exactly the same thing as the $\infty_\alpha$-morphisms between the same coalgebras seen as $\C$-coalgebras by pushing forward their structure along $f$. Suppose instead that we are given two conilpotent $\C$-coalgebras. Is it possible to go the other way around and understand the $\infty_\alpha$-morphisms between them in terms of $\infty_{f^*\alpha}$-morphisms? The dual question is asked for $\infty$-morphisms of algebras. The answer is not so immediate this time, and goes through what we call rectification functors, in analogy to \cref{subsection:rectification for Poo-algebras}.

\medskip

Let $\C$ and $\C'$ be two cooperads, let $\P$ be an operad, let $\alpha:\C\to\P$ be a twisting morphism, and let $f:\C'\to\C$ be a morphism of cooperads. We define the \emph{rectification functor}\index{Rectification functor!for $\C$-coalgebras}
\[
R_{\alpha,f}:\infty_\alpha\text{-}\C\text{-}\mathsf{cog}\longrightarrow\infty_\alpha\text{-}\C\text{-}\mathsf{cog}
\]
by
\[
R_{\alpha,f}(C)\coloneqq f_*\Bar_{f^*\alpha}\Cobar_\alpha C
\]
on conilpotent $\C$-coalgebras, and
\[
R_{\alpha,f}(\Phi)\coloneqq f_*\Bar_{f^*\alpha}\Phi
\]
on $\infty_\alpha$-morphisms of $\C$-coalgebras. The counit of the bar-cobar adjunction relative to $f^*\alpha$ induces a natural transformation
\[
E_C:\Cobar_\alpha R_{\alpha,f}(C)=\Cobar_{f^*\alpha}\Bar_{f^*\alpha}\Cobar_\alpha C\xrightarrow{\epsilon_{\Cobar_\alpha C}}\Cobar_\alpha C\ ,
\]
where $C$ is a conilpotent $\C$-coalgebra, and where the equality is given by \cref{lemma:equality of oo-morphisms}. Therefore, we have a natural $\infty_\alpha$-morphism of conilpotent $\C$-coalgebras
\[
E_C:R_{\alpha,f}(C)\rightsquigarrow C
\]
from the rectification functor to the identity functor.

\begin{remark}
		If $f$ is the identity of $\C$, then $E$ is a strict morphism, namely the counit of the bar-cobar adjunction.
\end{remark}

Dually, let $\C$ be a cooperad, let $\P,\P'$ be operads, let $\alpha:\C\to\P$ be a twisting morphism, and take $g:\P\to\P'$ a morphism of operads. We define the \emph{rectification functor}\index{Rectification functor!for $\P$-algebras}
\[
R^{g,\alpha}:\infty_\alpha\text{-}\Palg\longrightarrow\infty_\alpha\text{-}\Palg
\]
by
\[
R^{g,\alpha}(A)\coloneqq g^*\Cobar_{g_*\alpha}\Bar_\alpha A
\]
for a $\P$-algebra $A$, and
\[
R^{g,\alpha}(\Psi)\coloneqq g^*\Cobar_{g_*\alpha}\Psi
\]
on $\infty_\alpha$-morphisms. There is a natural transformation $N$ from $R^{g,\alpha}$ to the identity of $\infty_\alpha\text{-}\Palg$ given by
\[
N_A:\Bar_\alpha A\xrightarrow{\eta_{\Bar_\alpha A}}\Bar_{g_*\alpha}\Cobar_{g_*\alpha}\Bar_\alpha A = \Bar_\alpha R^{g,a}(A)\ ,
\]
where $\eta$ is the unit of the bar-cobar adjunction relative to $g_*\alpha$. Therefore, we have a natural $\infty_\alpha$-morphism of $\P$-algebras
\[
N_A:A\rightsquigarrow R^{g,\alpha}(A)
\]
from the identity to the rectification functor.

\begin{remark}
	If $g$ is the identity of $\P$, then $N$ is a strict morphism, namely the unit of the bar-cobar adjunction.
\end{remark}

Next, we will see that the rectification functors are homotopically well behaved.

\section{Homotopy theory of \texorpdfstring{$\infty$}{infinity}-morphisms of coalgebras}\label{sect:homotopy theory of oo-morphisms of coalgebras}

The homotopy theory for classical $\infty$-morphisms is quite well known, cf. \cref{sect:homotopy theory of homotopy algebras} and \cref{subsect:homotopy theory of algebras with oo-moprhisms}. Notice that the theory developed there passes without problems to $\infty_\alpha$-morphisms of algebras when $\alpha$ is a Koszul morphism\footnote{Here, as usual, we take Koszul morphisms between reduced (co)operads. We expect the theory presented in this chapter to work more generally for Koszul morphisms between connected weight graded (co)operads.}. In this section, we will develop the analogous results for coalgebras. For the rest of this section, we fix a Koszul morphism $\alpha:\C\to\P$.

\medskip

We begin with the following definition.

\begin{definition}
	Let $\C$ be a cooperad, let $\P$ be an operad, and let $\alpha:\C\to\P$ be a twisting morphism.
	\begin{enumerate}
		\item An $\infty_\alpha$-morphism of $\P$-algebras $\Psi:A\rightsquigarrow A'$ is an \emph{$\alpha$-weak equivalence} if the morphism
		\[
		\Psi:\Bar_\alpha A\longrightarrow \Bar_\alpha A'
		\]
		is a weak equivalence of coalgebras in the category of conilpotent $\C$-coalgebras with the Vallette model structure \cite{val14}, i.e. if
		\[
		\Cobar_\alpha\Psi:\Cobar_\alpha\Bar_\alpha A\longrightarrow\Cobar_\alpha\Bar_\alpha A'
		\]
		is a quasi-isomorphism.
		\item An $\infty_\alpha$-morphism of $\P$-algebras $\Psi:A\rightsquigarrow A'$ is an \emph{$\infty_\alpha$-quasi-isomorphism}\index{$\infty$-quasi-isomorphism!of algebras} if the chain map
		\[
		\psi_1:A\longrightarrow A'
		\]
		is a quasi-isomorphism.
		\item An $\infty_\alpha$-morphism of $\C$-coalgebras $\Phi:C'\rightsquigarrow C$ is an \emph{$\alpha$-weak equivalence}\index{$\alpha$-weak equivalences of coalgebras} if the morphism
		\[
		\Phi:\Cobar_\alpha C'\longrightarrow\Cobar_\alpha C
		\]
		is a quasi-isomorphism, i.e. if it is a weak equivalence in the classical Hinich model structure on the category of $\P$-algebras.
		\item An $\infty_\alpha$-morphism of $\C$-coalgebras $\Phi:C'\rightsquigarrow C$ is an \emph{$\infty_\alpha$-quasi-isomorphism}\index{$\infty$-quasi-isomorphism!of coalgebras} if the chain map
		\[
		\phi_1:C'\longrightarrow C
		\]
		is a quasi-isomorphism.
	\end{enumerate}
\end{definition}

We will now try to understand how these four notions are related to each other. We begin with a classical fact. It was originally stated for $\P$ a Koszul operad and the classical $\infty$-morphisms of homotopy $\P$-algebras, but the proof readily generalizes to our setting. See also \cite[Prop. 32]{leg16}.

\begin{theorem}[{\cite[Prop. 11.4.7]{LodayVallette}}]\label{thm:oo-alpha-qi of algebras iff alpha-we of algebras}
	An $\infty_\alpha$-morphism of $\P$-algebras is an $\alpha$-weak equivalence if, and only if it is an $\infty_\alpha$-quasi-isomorphism.
\end{theorem}

Thanks to this result, we see that the rectification functors for algebras are naturally $\infty_\alpha$-quasi-isomorphic to the identity functor.

\begin{lemma}
	Let $\alpha:\C\to\P$ be a Koszul morphism, and let $g:\P\to\P'$ be a quasi-isomorphism of operads. Then the natural $\infty_\alpha$-morphism
	\[
	N_A:A\rightsquigarrow R^{g,\alpha}(A)
	\]
	is an $\infty_\alpha$-quasi-isomorphism for any $\P$-algebra $A$.
\end{lemma}

\begin{proof}
	Since $\alpha$ is Koszul and $g$ is a quasi-isomorphism, it follows that $g_*\alpha$ is also Koszul by \cref{thm:fundamental thm of operadic twisting morphisms}. This is equivalent to $\eta_{\Bar_\alpha A}$ being a quasi-isomorphism by \cref{thm:alpha Koszul iff unit is qi}. Therefore,
	\[
	N_A:\Bar_\alpha A\xrightarrow{\eta_{\Bar_\alpha A}}\Bar_{g_*\alpha}\Cobar_{g_*\alpha}\Bar_\alpha A = \Bar_\alpha R^{g,a}(A)\ ,
	\]
	is an $\alpha$-weak equivalence of $\P$-algebras, and thus an $\infty_\alpha$quasi-isomorphism by \cref{thm:oo-alpha-qi of algebras iff alpha-we of algebras}.
\end{proof}

For coalgebras, we can proceed similarly to prove a slightly weaker statement.

\begin{proposition}\label{prop:rectification is oo-qi to original}
	Let $\alpha:\C\to\P$ be a Koszul morphism, and let $f:\C'\to\C$ be a quasi-isomorphism of cooperads. The $\alpha$-weak equivalence
	\[
	E_C:R_{\alpha,f}(C)\rightsquigarrow C
	\]
	is an $\infty_\alpha$-quasi-isomorphism for any conilpotent $\C$-coalgebra $C$.
\end{proposition}

\begin{proof}
	Denote $\alpha'\coloneqq\alpha f$. We have to prove that the first component
	\[
	e_1:R_{\alpha,f}(C)\longrightarrow C
	\]
	of $E_C$ is a quasi-isomorphism. We will do this by a spectral sequence argument analogous to the one of \cite[Thm. 11.3.3 and 11.4.4]{LodayVallette}. We start by noticing that
	\[
	e_1:(\C'\circ\P)(C)\longrightarrow C
	\]
	is given by the projection onto $C$. We filter the left-hand side by the number of times that $C$ appears, i.e. by
	\[
	F_p\coloneqq\bigoplus_{k\le p}(\C'\circ\P)(k)\otimes_{\S_k}C^{\otimes k}\ .
	\]
	This filtration is increasing, bounded below and exhaustive. The page $E^0$ of the associated spectral sequence equals $(\C'\circ_{\alpha'}\P)(C)$, since the only parts of the differential that preserve the weight (that is, the arity) are the internal differential of $C$ and the part coming from the twisting morphism $\alpha'$. The page $E^1$ of the spectral sequence is
	\[
	H_\bullet((\C'\circ_{\alpha'}\P)(C))\cong H_\bullet(\C'\circ_{\alpha'}\P)\circ H_\bullet(C)\cong H_\bullet(C)
	\]
	by the operadic Künneth formula, \cref{thm:operadic Kunneth}, and the fact that $\alpha'$ is a Koszul morphism. On the other side, we filter $C$ by $F_pC=C$ for $p\ge0$ and $F_pC=0$ otherwise. This filtration is also increasing, bounded below and exhaustive. The map $e_1$ is a map of spectral sequences, and so the induced map at the page $E^1$ is $H_\bullet(e_1)$, which induces an isomorphism. Therefore, the chain map $e_1$ is a quasi-isomorphism.
\end{proof}

Notice that if $f=1_\C$, then the rectification becomes the functor $\Bar_\alpha\Omega_\alpha$, and the natural $\infty_\alpha$-morphism $E_C$ is given by the counit $\varepsilon$ of the bar-cobar adjunction (seen as an $\infty_\alpha$-morphism). As a consequence of this result, we have the following.

\begin{theorem}\label{thm: alpha-we are oo-alpha-qi}
	If an $\infty_\alpha$-morphism of $\C$-coalgebras is an $\alpha$-weak equivalence, then it is an $\infty_\alpha$-quasi-isomor\-phism.
\end{theorem}

Notice that the inverse implication is not true: it is known that there are (strict) quasi-isomor\-phisms of $\C$-coalgebras that are not sent to quasi-isomorphisms under the cobar construction, cf. \cite[Prop. 2.4.3]{LodayVallette}.

\begin{proof}
	The proof is similar to the one of \cite[Prop. 11.4.7]{LodayVallette}. Suppose $\Phi:C'\rightsquigarrow C$ is an $\alpha$-weak equivalence of $\C$-coalgebras. We have the commutative diagram
	\begin{center}
		\begin{tikzpicture}
		\node (a) at (0,1.5){$\Bar_\alpha\Omega_\alpha C'$};
		\node (b) at (3,1.5){$\Bar_\alpha\Omega_\alpha C$};
		\node (c) at (0,0){$C'$};
		\node (d) at (3,0){$C$};
		
		\draw[->] (a) -- node[above]{$\Bar_\alpha\Phi$}node[below]{$\sim$} (b);
		\draw[->,line join=round,decorate,decoration={zigzag,segment length=4,amplitude=.9,post=lineto,post length=2pt}] (a) -- node[left]{$\varepsilon_{C'}$} (c);
		\draw[->,line join=round,decorate,decoration={zigzag,segment length=4,amplitude=.9,post=lineto,post length=2pt}] (b) -- node[right]{$\varepsilon_{C}$} (d);
		\draw[->,line join=round,decorate,decoration={zigzag,segment length=4,amplitude=.9,post=lineto,post length=2pt}] (c) -- node[above]{$\Phi$} (d);
		\end{tikzpicture}
	\end{center}
	where the vertical arrows are $\infty_\alpha$-quasi-isomorphisms by \cref{prop:rectification is oo-qi to original} and the top arrow is a quasi-isomorphism by \cref{prop:bar of qi is qi}. Restriction to the first component gives us the diagram
	\begin{center}
		\begin{tikzpicture}
		\node (a) at (0,1.5){$\Bar_\alpha\Omega_\alpha C'$};
		\node (b) at (3,1.5){$\Bar_\alpha\Omega_\alpha C$};
		\node (c) at (0,0){$C'$};
		\node (d) at (3,0){$C$};
		
		\draw[->] (a) -- node[above]{$\Bar_\alpha\Phi$}node[below]{$\sim$} (b);
		\draw[->] (a) -- node[left]{$\varepsilon_{C'}$} node[above,sloped]{$\sim$} (c);
		\draw[->] (b) -- node[right]{$\varepsilon_{C}$} node[below,sloped]{$\sim$} (d);
		\draw[->] (c) -- node[above]{$\Phi$} (d);
		\end{tikzpicture}
	\end{center}
	from which the statement follows.
\end{proof}

Finally, we can show that $\alpha$-weak equivalences of coalgebras are equivalent to zig-zags of weak equivalences of coalgebras, which is analogous to \cref{thm:oo-qi of algebras is zig-zag of qi}.

\begin{theorem}\label{prop:alpha-we of coalgebras is zig-zag of we}
	Let $C$ and $D$ be two $\C$-coalgebras. The following are equivalent.
	\begin{enumerate}
		\item \label{pt:1 alpha we} There is a zig-zag of weak equivalences
		\[
		C\longrightarrow\bullet\longleftarrow\bullet\longrightarrow\cdots\longleftarrow D
		\]
		of $\C$-coalgebras in the Vallette model structure.
		\item \label{pt:2 alpha we} There are two weak equivalences of $\C$-coalgebras forming a zig-zag
		\[
		C\longrightarrow\bullet\longleftarrow D\ .
		\]
		\item \label{pt:3 alpha we} There is an $\alpha$-weak equivalence
		\[
		C\rightsquigarrow D\ .
		\]
	\end{enumerate}
\end{theorem}

\begin{proof}
	The fact that (\ref{pt:2 alpha we}) implies (\ref{pt:1 alpha we}) is obvious.
	
	\medskip
	
	We prove that (\ref{pt:3 alpha we}) implies (\ref{pt:2 alpha we}). Suppose we have an $\alpha$-weak equivalence
	\[
	\Phi:C\rightsquigarrow D\ .
	\]
	Then $\Bar_\alpha\Phi$ is a weak equivalence, and thus we have the zig-zag
	\[
	C\xrightarrow{\eta_C}\Bar_\alpha\Omega_\alpha C\xrightarrow{\Bar_\alpha\Phi}\Bar_\alpha\Omega_\alpha D\xleftarrow{\eta_D}D\ ,
	\]
	where the units of the bar-cobar adjunction are weak equivalences by  \cref{corollary:unit of bar-cobar is a we in the Vallette model structure}.
	
	\medskip
	
	Finally, we show that (\ref{pt:1 alpha we}) implies (\ref{pt:3 alpha we}). Every weak equivalence of coalgebras is in particular an $\alpha$-weak equivalence. Therefore, it is enough to prove that whenever we have a weak equivalence
	\[
	C\stackrel{\phi}{\longleftarrow}D
	\]
	of coalgebras, then we have an $\alpha$-weak equivalence going the other way round. Since $\phi$ is a weak equivalence, we have that
	\[
	\Omega_\alpha C\xleftarrow{\Omega_\alpha\phi}\Omega_\alpha D
	\]
	is a quasi-isomorphism. Moreover, every $\P$-algebra is fibrant and $\Omega_\alpha$ lands in the cofibrant $\P$-algebras by \cite[Thm. 2.9(1)]{val14}. Therefore, we can apply \cite[Lemma 4.24]{ds95} to obtain a homotopy inverse
	\[
	\Omega_\alpha C\stackrel{\Psi}{\longrightarrow}\Omega_\alpha D\ ,
	\]
	which is again a quasi-isomorphism, and thus defines an $\alpha$-weak equivalence from $C$ to $D$, as we desired.
\end{proof}

%% file: NaturalHomotopyAlgebras.tex
\chapter{Convolution homotopy Lie algebras}\label{chapter:convolution homotopy algebras}

Independently in \cite{w16} and \cite{rn17tensor}, and then jointly in \cite{rnw17} and \cite{rnw18}, the author and Felix Wierstra showed that operadic twisting morphisms from a cooperad $\C$ to an operad $\P$ are equivalent to morphisms from the operad $\L_\infty$ to the convolution operad $\hom(\C,\P)$. It follows that the chain complex of linear maps from a conilpotent $\C$-coalgebra to a $\P$-algebra is an $\L_\infty$-algebra in a functorial way, by restriction of the structure. The $\L_\infty$-algebra structures obtained this way are really well behaved with respect to the tools of homotopical operadic algebra: the homotopy transfer theorem and $\infty$-morphisms.

\medskip

The results presented here were developed and generalized progressively during the last few years. We present only the current state of the art in this chapter. In the whole chapter, we work with shifted $\L_\infty$-algebras. This greatly reduces the signs appearing in the proofs, improving readability. One can pass to usual $\L_\infty$-algebras simply by desuspending everything, see \cref{sect:shifted Loo-alg}.

\section{Convolution homotopy Lie algebras}

We begin by showing how operadic twisting morphisms are equivalent to morphisms from $\SLoo$ to the convolution operad. We proceed by studying the Maurer--Cartan elements of the $\SLoo$-algebras obtained that way, and take a look at what happens in other settings, such as non-symmetric operads.

\subsection{Definitions}

We begin with the following remark.

\begin{lemma}
	Let $\P$ be an operad. Then we have
	\[
	\hom(\com^\vee,\P)\cong\P
	\]
	as operads.
\end{lemma}

\begin{proof}
	The statement is straightforwardly true at the level of $\S$-modules, the isomorphism being given by
	\[
	\phi\in\hom(\com^\vee,\P)(n)\longmapsto\phi(\mu_n^\vee)\in\P(n)\ .
	\]
	We only have to prove that the composition maps coincide. Let $\phi\in\hom(\com^\vee,\P)(k)$, and $\psi_i\in\hom(\com^\vee,\P)(n_i)$ for all $1\le i\le k$. Fix $\theta\in\sh(n_1,\ldots,n_k)$. Then we have
	\begin{align*}
		\gamma_{\hom(\com^\vee,\P)}&(\phi\circ(\psi_1,\ldots,\psi_k)^\theta)(\mu_n^\vee) =\\
		&= \sum_{\sigma\in\sh(n_1,\ldots,n_k)}\gamma_\P(\phi\circ(\psi_1,\ldots,\psi_k)^\theta)(\mu_k^\vee\circ(\mu_{n_1}^\vee,\ldots,\mu_{n_k}^\vee)^\sigma)\\
		&= \gamma_\P(\phi(\mu_k^\vee)\circ(\psi_1(\mu_{n_1}^\vee),\ldots,\psi_k(\mu_{n_k}^\vee))^\theta)\ ,
	\end{align*}
	concluding the proof.
\end{proof}

Thanks to this, we can give a clean proof of the following theorem, on which relies all of the theory of convolution $\L_\infty$-algebras.

\begin{theorem}[{\cite[Sect. 7]{w16} and \cite[Thm. 3.1]{rn17tensor}}]\label{thm:convolution Loo-algebras}
	Let $\C$ be a cooperad, and let $\P$ be an operad. There is a bijection
	\[
	\Tw(\C,\P)\cong\hom_\op(\SLoo,\hom(\C,\P))\ ,
	\]
	given by sending $\alpha\in\Tw(\C,\P)$ to the morphism $\manin_\alpha$ sending $\ell_n\coloneqq s^{-1}\mu_n^\vee\in\SLoo(n)$ to
	\[
	\manin_\alpha(s^{-1}\mu_n^\vee) = \alpha(n)\in\hom(\C(n),\P(n))\ .
	\]
	It is natural in both $\C$ and $\P$ in the following sense. Let $\alpha\in\Tw(\C,\P)$, then for any morphism $f:\C'\to\C$ of cooperads, we have
	\[
	\manin_{f^*\alpha} = (f^*)\manin_\alpha\ ,
	\]
	and for any morphism $g:\P\to\P'$ of operads, we have
	\[
	\manin_{g_*\alpha} = (g_*)\manin_\alpha\ .
	\]
\end{theorem}

\begin{proof}
	The bijection is given by
	\begin{align*}
		\Tw(\C,\P) =&\ \MC(\hom(\C,\P))\\
		\cong&\ \MC(\hom(\com^\vee,\hom(\C,\P)))\\
		=&\ \Tw(\com^\vee,\hom(\C,\P))\\
		\cong&\ \hom_\op(\Cobar\com^\vee,\hom(\C,\P))\\
		=&\ \hom_\op(\SLoo,\hom(\C,\P))\ ,
	\end{align*}
	where in the first line, $\hom(\C,\P)$ denotes the Lie algebra associated to the convolution operad. The other properties are straightforward to check, and left as an exercise to the reader.
\end{proof}

\begin{remark}
	If one wants to work with $\L_\infty$-algebras, instead of shifted ones, then one needs to desuspend everything, obtaining a bijection
	\[
	\Tw(\C,\P)\cong\hom_\op(\L_\infty,\susp^{-1}\otimes\hom(\C,\P))\ .
	\]
\end{remark}

\begin{remark}
	Some special cases of this result and ideas hinting to it have appeared in the existing literature, for example already in Ginzburg--Kapranov \cite[Prop. 3.2.18]{gk94}, and then more recently in \cite[Appendix C]{bl15}, \cite{dhr15}, \cite{dp16}.
\end{remark}

\begin{remark}\label{rem:original case convolution Loo-algebras}
	Given a morphism $\Psi:\Q\to\P$ of operads, then one gets a twisting morphism
	\[
	\psi\coloneqq\left(\Bar\Q\stackrel{\pi}{\longrightarrow}\Q\stackrel{\Psi}{\longrightarrow}\P\right),
	\]
	and thus a morphism of operads
	\[
	\manin_\psi:\SLoo\longrightarrow\hom(\Bar\Q,\P)\ .
	\]
	This is the special case originally treated in \cite{rn17tensor}.
\end{remark}

Suppose now that we have a cooperad $\C$, and operad $\P$, and a twisting morphism $\alpha\in\Tw(\C,\P)$. Then, given a conilpotent $\C$-coalgebra $C$ and a $\P$-algebra $A$, we know by \cref{thm:hom of algebras is algebras over hom} that $\hom(C,A)$ is a $\hom(\C,\P)$-algebra. We can therefore apply restriction of structure by $\manin_\alpha$ to obtain a $\SLoo$-algebra structure on $\hom(C,A)$. We denote the $\SLoo$-algebra obtained this way by $\homa(C,A)$.

\begin{definition}
	The algebra $\homa(C,A)$ is called the \emph{convolution $\SLoo$-algebra}\index{Convolution!$\L_\infty$-algebra} of $C$ and $A$.
\end{definition}

Given a morphism of $\C$-coalgebras, resp. of $\P$-algebras, then we get a morphism of $\SLoo$-algebras by pullback, resp. pushforward. Therefore, the assignment $\homa$ defines a bifunctor
\begin{equation}\label{eq:bifunctor}
\homa:\Ccog^\mathrm{op}\times\Palg\longrightarrow\SLoo\text{-}\mathsf{alg}\ .
\end{equation}
The rest of this chapter will be dedicated to the study of some of the properties of this bifunctor. Here is a first, straightforward fact. It is a direct consequence of the functoriality of $\manin$.

\begin{lemma}\label{lemma:naturality of convolution algebras wrt compositions}
	Let $\alpha\in\Tw(\C,\P)$ be a twisting morphism, and suppose $f:\C'\to\C$ is a morphism of cooperads, and that $g:\P\to\P'$ is a morphism of operads.
	\begin{enumerate}
		\item\label{pt:1 naturality convolution} Let $A$ be a $\P$-algebra, let $D$ be a conilpotent $\C'$-coalgebra. We have
		\[
		\hom^{f^*\alpha}(D,A) = \homa(f_*D,A)
		\]
		as $\SLoo$-algebras.
		\item Dually, let $D$ be a $\C$-coalgebra, let $A$ be a $\P'$-algebra. We have
		\[
		\hom^{g_*\alpha}(D,A) = \homa(D,g^*A)
		\]
		as $\SLoo$-algebras.
	\end{enumerate}
\end{lemma}

\begin{proof}
	We only prove (\ref{pt:1 naturality convolution}), the other case being dual. Let $f_1,\ldots,f_n\in\hom(D,A)$, and denote by $F\coloneqq f_1\otimes\cdots\otimes f_n$. We have
	\begin{align*}
		\gamma_{\hom^{f^*\alpha}(D,A)}(s^{-1}\mu_n^\vee\otimes_{\S_n}F) =&\ \gamma_{\hom(D,A)}(\manin_{f^*\alpha}(s^{-1}\mu_n^\vee)\otimes F)\\
		=&\ \gamma_{\hom(D,A)}((f^*)\manin_\alpha(s^{-1}\mu_n^\vee)\otimes F)\\
		=&\ \gamma_{\hom(f_*D,A)}(\manin_\alpha(s^{-1}\mu_n^\vee)\otimes F)\\
		=&\ \gamma_{\homa(f_*D,A)}(s^{-1}\mu_n^\vee\otimes_{\S_n}F)\ ,
	\end{align*}
	where $\hom(D,A)$ is seen as a $\hom(\C',\P)$-algebra and similarly for $\hom(f_*D,A)$.
\end{proof}

\subsection{Maurer--Cartan elements of convolution \texorpdfstring{$\L_\infty$}{Loo}-algebras}

Whenever one has a $\SLoo$-algebra, it is a very natural question to ask what the Maurer--Cartan elements are. We answer this question for convolution homotopy Lie algebras.

\medskip

Fix a cooperad $\C$, an operad $\P$, and a twisting morphism $\alpha\in\Tw(\C,\P)$.

\begin{lemma}\label{lemma:compatibility filtrations first case}
	Let $C$ be a conilpotent $\C$-coalgebra, and let $A$ be a $\P$-algebra. Then the ascending filtration
	\[
	\F_n\homa(C,A)\coloneqq \left\{f\in\homa(C,A)\mid \F^n_\C C\subseteq\ker(f)\right\}
	\]
	makes $\homa(C,A)$ into a proper complete $\SLoo$-algebra.
\end{lemma}

\begin{proof}
	Let $f\in\F_n\homa(C,A)$. Let $x\in\F^n_\C C$, then
	\[
	\partial(f)(x) = d_Af(x) - (-1)^ff(d_Cx) = 0
	\]
	since $d_C(\F^n_\C C)\subseteq\F^n_\C C$.
	
	\medskip
	
	Let $f_i\in\F_{n_i}\homa(C,A)$ for $i=1,\ldots,k$, and let $x\in\F^n_\C C$, where $n\coloneqq n_1+\cdots+n_k$. Then
	\[
	\ell_n(f_1,\ldots,f_k)(x) = \gamma_A(\alpha\otimes F)\Delta_C(x) = 0\ ,
	\]
	since
	\[
	\Delta_C(\F^n_\C C)\subseteq\bigoplus_{\substack{k\ge0\\n_1'+\cdots+n_k' = n}}\C(k)\otimes \F^{n_1'}_\C C\otimes\cdots\otimes\F^{n_k'}_\C C\ .
	\]
	Finally, there is a natural identification
	\[
	\frac{\homa(C,A)}{\F_n\homa(C,A)}\stackrel{\cong}{\longrightarrow}\homa(\F^n_\C C,A)\ ,
	\]
	given by taking any representative of an equivalence class of morphisms on the right hand side, and restricting it to $\F^n_\C C$. It is injective, because if a morphism from $C$ to $A$ restricts to zero on $\F^n_\C C$, then it is in $\F_n\homa(C,A)$, and thus it is zero in the quotient. To see that it is surjective, choose a complement $V$ of $\F^n_\C C$ in $C$ --- as graded vector spaces, the differential plays no role here. An inverse to the map described above is then given by sending a linear map $f:\F^n_\C C\to A$ to its extension by $0$ on $V$, and then taking the equivalence class in the quotient. Notice that this is independent of the choice of $V$. It is straightforward to check that this isomorphism holds at the level of $\SLoo$-algebras, and not only as chain complexes. Therefore, we have
	\begin{align*}
		\lim_n\frac{\homa(C,A)}{\F_n\homa(C,A)}\cong&\ \lim_n\homa(\F^n_\C C,A)\\
		\cong&\ \homa(\colim_n\F^n_\C C,A)\\
		\cong&\ \homa(C,A)\ ,
	\end{align*}
	concluding the proof.
\end{proof}

Therefore, it makes sense to speak about Maurer--Cartan elements in $\homa(C,A)$.

\begin{theorem}[{\cite[Thm. 7.1]{w16} and \cite[Thm. 6.3]{rn17tensor}}]\label{thm:MC-el of convolution Loo-algebras}
	Let $C$ be a conilpotent $\C$-coalgebra, and let $A$ be a $\P$-algebra. We have
	\[
	\MC(\homa(C,A)) = \Tw_\alpha(C,A)\ ,
	\]
	where the set $\Tw_\alpha(C,A)$ of twisting morphisms relative to $\alpha$ was given in \cref{def:relative twisting morphisms}.
\end{theorem}

\begin{proof}
	Let $\varphi\in\hom(C,A)$. Denote by $\star_\alpha^{(n)}(\varphi)$ the part of $\star_\alpha(\varphi)$ passing through $\C(n)\otimes C^{\otimes n}$. With this notation, we have
	\begin{align*}
		\frac{1}{n!}\ell_n(\varphi,\ldots,\varphi) =&\ \frac{1}{n!}\gamma_{\hom(C,A)}(\manin_\alpha(\ell_n)\otimes\varphi^{\otimes n})\\
		=&\ \gamma_A(\manin_\alpha(\ell_n)\circ\varphi)\Delta_C^n\\
		=&\ \gamma_A(\alpha\circ\varphi)\Delta_C^n\\
		=&\ \star_\alpha^{(n)}(\varphi)\ ,
	\end{align*}
	where the factor $\frac{1}{n!}$ is eliminated by the fact that we sum over all permutations when applying the $\hom(\C,\P)$-algebra structure $\gamma_{\hom(C,A)}$, cf. \cref{thm:hom of algebras is algebras over hom}. Since
	\[
	\star_\alpha(\varphi) = \sum_{n\ge2}\star_\alpha^{(n)}(\varphi)\ ,
	\]
	we have that
	\[
	\partial(\varphi) + \star_\alpha(\varphi) = \partial(\varphi) + \sum_{n\ge2}\frac{1}{n!}\ell_n(\varphi,\ldots,\varphi)\ .
	\]
	Thus, the two Maurer--Cartan equations coincide\footnote{Motivating \emph{a posteriori} the name ``Maurer--Cartan equation" for the equation $\partial(\varphi) + \star_\alpha(\varphi) = 0$.}, concluding the proof.
\end{proof}

\begin{remark}
	The result above was already known in the case where $\P$ is a Koszul operad and $\alpha$ is given by $\kappa:\P^{\antishriek}\to\P$, the twisting morphism provided by Koszul duality. In this case, $\hom^\kappa(C,A)$ is a Lie algebra, since $\kappa$ is non-zero only on the binary part of $\P^{\antishriek}$, cf. its definition in \cref{subsect:Koszul duality}. See \cite[Sect. 11.1.2]{LodayVallette}.
\end{remark}

\begin{corollary}
	Let $C$ be a conilpotent $\C$-coalgebra, and let $A$ be a $\P$-algebra.
	\begin{enumerate}
		\item We have a natural isomorphism
		\[
		\MC(\homa(C,A))\cong\hom_{\Palg}(\Cobar_\alpha C,A)\ .
		\]
		In particular, if $A = \Cobar_\alpha C'$ for some conilpotent $\C$-coalgebra $C'$, then
		\[
		\MC(\homa(C,A))\cong\hom_{\infty_\alpha\text{-}\Ccog}(C,C')\ .
		\]
		\item We have a natural isomorphism
		\[
		\MC(\homa(C,A))\cong\hom_{\Ccog}(C,\Bar_\alpha A)\ .
		\]
		In particular, if $C = \Bar_\alpha A'$ for some $\P$-algebra $A'$, then
		\[
		\MC(\homa(C,A))\cong\hom_{\infty_\alpha\text{-}\Palg}(A',A)\ .
		\]
	\end{enumerate}
\end{corollary}

\begin{proof}
	This is a direct consequence of \cref{thm:MC-el of convolution Loo-algebras} and \cref{thm:Rosetta stone algebras}.
\end{proof}

\subsection{Other settings: tensor products, non-symmetric (co)operads}\label{subsect:other settings}

There are some other contexts where it is natural to try to apply the theory developed above.

\subsubsection*{Tensor products}

First of all, if $\C$ is a cooperad which is finite dimensional in every arity, then for any operad $\P$ we have
\[
\hom(\C,\P)\cong\C^\vee\otimes\P
\]
with the obvious isomorphism, so that all results pass to tensor products. In order to have everything to pass through without having to do the proofs all over again, one should then only consider dualizable $\C^\vee$-algebras. However, if one is willing to do the effort, then one can prove that the result.

\begin{theorem}\label{thm:convolution Loo tensor products}
	Let $\C$ be a cooperad which is finite dimensional in every arity, and let $\P$ be an operad. There is a bijection
	\[
	\Tw(\C,\P)\cong\hom_\op(\SLoo,\P\otimes\C^\vee)\ ,
	\]
	given by sending $\alpha\in\Tw(\C,\P)$ to the morphism $\manin_\alpha$ sending $s^{-1}\mu_n^\vee\in\SLoo(n)$ to
	\[
	\manintensor_\alpha(s^{-1}\mu_n^\vee) = \sum_{i}\alpha(c_i)\otimes c_i^\vee\ ,
	\]
	where $\{c_i\}_i$ is a basis of $\C(n)$, and where $\{c_i^\vee\}_i$ is the dual basis. It is natural in both $\C$ and $\P$ in the following sense. Let $\alpha\in\Tw(\C,\P)$, then for any morphism $f:\C'\to\C$ of cooperads, we have
	\[
	\manintensor_{f^*\alpha} = (1\otimes f^\vee)\manintensor_\alpha\ ,
	\]
	and for any morphism $g:\P\to\P'$ of operads, we have
	\[
	\manintensor_{g_*\alpha} = (g\otimes1)\manintensor_\alpha\ .
	\]
\end{theorem}

Given a $\P$-algebra $A$ and a $\C^\vee$-algebra $D$, then one obtains an $\SLoo$-algebra $A\otimes^\alpha D$ by pulling back the natural $\P\otimes\C^\vee$-algebra structure of $A\otimes D$ by $\manintensor_\alpha$. This assignment is compatible with morphisms of algebras in both slots, and thus defines a bifunctor
\[
-\otimes^\alpha-:\Palg\times\C^\vee\text{-}\mathsf{alg}\longrightarrow\SLoo\text{-}\mathsf{alg}\ .
\]
This is compatible with $\homa(-,-)$, in the sense that if $A$ is a $\P$-algebra and $C$ is a $\C$-coalgebra, then the natural morphism
\[
A\otimes C^\vee\longrightarrow\hom(C,A)
\]
is a morphism of $\SLoo$-algebras
\[
A\otimes^\alpha C^\vee\longrightarrow\homa(C,A)\ ,
\]
which is an isomorphism if $C$ is dualizable.

\begin{theorem}\label{thm:equality MC and twisting in tensor case}
	Let $D$ be a dualizable $\C^\vee$-algebra, and let $A$ be a $\P$-algebra. We have
	\[
	\MC(A\otimes^\alpha D)\cong\Tw_\alpha(D^\vee,A)\ .
	\]
\end{theorem}

\subsubsection*{Non-symmetric (co)operads}

If we consider non-symmetric (co)operads, then everything goes through with exactly the same proofs\footnote{Removing all symmetric group actions, of course.} simply by changing the operad $\L_\infty$ with the ns operad $\A_\infty$. Moreover, since we don't need to identify invariants and coinvariants, we can work over any field, without restrictions on the characteristic.

\begin{theorem}
	Let $\C$ be an ns cooperad, and let $\P$ be an ns operad. There is a bijection
	\[
	\Tw(\C,\P)\cong\hom_\op(\SAoo,\hom(\C,\P))\ ,
	\]
	given by sending $\alpha\in\Tw(\C,\P)$ to the morphism $\manin_\alpha$ sending $s^{-1}\mu_n^\vee\in\SAoo(n)$ to
	\[
	\manin_\alpha(s^{-1}\mu_n^\vee) = \alpha(n)\in\hom(\C(n),\P(n))\ .
	\]
	It is natural in both $\C$ and $\P$ in the following sense. Let $\alpha\in\Tw(\C,\P)$, then for any morphism $f:\C'\to\C$ of cooperads, we have
	\[
	\manin_{f^*\alpha} = (f^*)\manin_\alpha\ ,
	\]
	and for any morphism $g:\P\to\P'$ of operads, we have
	\[
	\manin_{g_*\alpha} = (g_*)\manin_\alpha\ .
	\]
\end{theorem}

\begin{theorem}
	Let $C$ be a conilpotent $\C$-coalgebra, and let $A$ be a $\P$-algebra. We have
	\[
	\MC(\homa(C,A)) = \Tw_\alpha(C,A)\ .
	\]
\end{theorem}

Here, the Maurer--Cartan elements of an $\SAoo$-algebra $A$ are the elements $x\in A_0$ of degree $0$ satisfying the non-symmetric Maurer--Cartan equation
\[
dx + \sum_{n\ge2}m_n(x,\ldots,x) = 0\ .
\]
The results of the rest of the present chapter all have versions for tensor products and non-symmetric operads. We will only mention those that we will need in the rest of this work.

\subsection{The binary quadratic case and Manin products}

Here, we restrict to the binary quadratic case, and study convolution tensor products in this context. We start by recalling the notion of what we call the Manin morphisms, which are morphisms arising from maps between operads via the adjunction between the black and white Manin products. We go on to prove that convolution tensor products directly generalize Manin morphisms.

\medskip

In the category of operads given by binary quadratic data and morphisms induced by morphisms of quadratic data, one can define two operations, called the \emph{white} and \emph{black Manin products}\index{Manin!products} and denoted by $\wmanin$ and $\bmanin$ respectively, both taking two binary quadratic operads and giving back another one. These objects first appeared in the context of algebras in \cite{man87} and \cite{man89}, and then in \cite{gk94} in relation to operads. For a more conceptual treatment, see \cite{val08} or \cite[Sect. 8.8]{LodayVallette}.

\begin{proposition}
	Fix a binary quadratic operad $\Q$. Then there is a natural isomorphism
	\[
	\hom_{\mathsf{bin.\ quad.\ op.}}(\R\bmanin\Q,\P)\cong\hom_{\mathsf{bin.\ quad.\ op.}}(\R,\P\wmanin\Q^{\shriek})\ .
	\]
	That is to say, the functors $-\bmanin\Q$ and $-\wmanin\Q^{\shriek}$ are adjoint. Moreover, the operad $\lie$ is a unit for the black product.
\end{proposition}

Therefore, any morphism
\[
\Psi:\lie\bmanin\Q\cong\Q\longrightarrow\P
\]
coming from a quadratic data is equivalent to a morphism
\[
\lie\longrightarrow\P\wmanin\Q^{\shriek}\ .
\]
As explained in \cite[Sect. 3.2]{val08}, the white product is the best binary quadratic approximation of the Hadamard product, and there is a canonical morphism
\[
\P\wmanin\Q^{\shriek}\longrightarrow\P\otimes\Q^{\shriek}\ .
\]

\begin{definition}
	We call the composite
	\[
	\smallmanin_\Psi\coloneqq\big(\lie\longrightarrow\P\wmanin\Q^{\shriek}\longrightarrow\P\otimes\Q^{\shriek}\big)
	\]
	the \emph{Manin morphism}\index{Manin!morphisms} associated to $\Psi$.
\end{definition}

The Manin morphism $\smallmanin_\Psi$ has the following explicit description. Assume $\P = \P(E,R)$ and $\Q = \P(F,S)$, fix a basis $f_1,\ldots,f_k$ of $F$, and let $e_1,\ldots,e_k\in E$ be the images of the $f_i$ under $\Psi$. Then $\smallmanin_\Psi$ is the unique morphism of operads extending
\[
\smallmanin_\Psi(b)\coloneqq\sum_{i=1}^ke_i\otimes s^{-1}\susp_2^{-1}f_i^\vee,
\]
where $b\in\lie(2)$ is the Lie bracket.

\medskip

For any quadratic binary operad $\Q$, there is a canonical Manin morphism, namely the one associated to the identity of $\Q$, giving
\[
\smallmanin_\Q\coloneqq\smallmanin_{\id_\Q}:\lie\longrightarrow\Q\otimes\Q^!\ .
\]
It is easy to see that
\[
\smallmanin_\Psi = (\Psi\otimes 1)\smallmanin_\Q\ ,
\]
so that it is only necessary to know $\smallmanin_\Q$ to compute $\smallmanin_\Psi$.

\medskip

We fix two binary quadratic data $(E,R)$ and $(F,S)$ and denote by $\P=\P(E,R)$ and $\Q = \P(F,S)$ the two associated operads. Furthermore, we assume that $F$ is finite dimensional. We fix a morphism
\[
\Psi:\Q\longrightarrow\P
\]
in the category of binary quadratic operads, and we consider the associated twisting morphism
\[
\psi\coloneqq\left(\Bar\Q\stackrel{\pi}{\longrightarrow}\Q\stackrel{\Psi}{\longrightarrow}\P\right).
\]
The morphism of operads associated to this twisting morphism by \cref{thm:convolution Loo tensor products} is
\[
\manintensor_\Psi:\L_\infty\longrightarrow\P\otimes\Q^{\shriek}_\infty\ ,
\]
after a suspension. Here, we used the fact that
\[
(\susp\otimes\Bar\Q)^\vee\cong\Cobar((\susp^{-1})^c\otimes\Q^\vee)\cong\Cobar\big(\big(\Q^{\shriek}\big)^{\antishriek}\big)\cong\Q^{\shriek}_\infty\ ,
\]
because, since $F$ is finite dimensional, $\Q(n)$ is finite dimensional for all $n\ge0$. The following proposition shows that $\manintensor_\Psi$ gives a direct generalization of $\smallmanin_\Psi$.

\begin{proposition}
	The following square
	\begin{center}
		\begin{tikzpicture}
		\node (a) at (0,0){$\L_\infty$};
		\node (b) at (2.5,0){$\P\otimes\Q^{\shriek}_\infty$};
		\node (c) at (0,-1.5){$\lie$};
		\node (d) at (2.5,-1.5){$\P\otimes\Q^{\shriek}$};
		
		\draw[->] (a)--node[above]{$\manintensor_\Psi$}(b);
		\draw[->] (a)--(c);
		\draw[->] (b)--(d);
		\draw[->] (c)--node[above]{$\smallmanin_\psi$}(d);
		\end{tikzpicture}
	\end{center}
	where the vertical maps are the canonical ones coming from the resolutions, is commutative.
\end{proposition}

\begin{proof}
	The left vertical arrow sends $\ell_2$ to $b$ and $\ell_n$ to $0$ for all $n\ge3$. Therefore, the south--west composite is the map sends
	\[
	\ell_2\longmapsto \sum_{i=1}^ke_i\otimes s^{-1}\susp^{-1}_2f_i^\vee
	\]
	and all the higher $\ell_n$ to zero. On the other hand, the right vertical arrow is given by tensoring the identity of $\P$ with the canonical resolution map
	\[
	\Q^{\shriek}_\infty\longrightarrow\Q^{\shriek}\ ,
	\]
	which is defined on the generators $s^{-1}(\Q^{\shriek})^{\antishriek}\cong s^{-1}(\susp^{-1})^c\otimes\Q^\vee$ as being the identity on arity $2$ and zero on all higher arities, this because $\Q$ is quadratic. By definition, the morphism $\manin_\Psi$ sends $\ell_n$ to an element of $\P(n)\otimes s^{-1}(\Q^{\shriek})^{\antishriek}(n)$. Therefore, the north--east composite gives zero on $\ell_n$, for $n\ge3$ and sends
	\[
	\ell_2\longmapsto \sum_{i=1}^ke_i\otimes s^{-1}\susp^{-1}_2f_i^\vee
	\]
	just like the other map.
\end{proof}

\subsection{Non-conilpotent coalgebras}\label{subsect:complete cobar cosntruction}

In some applications, such as the one we will see in \cref{ch:representing the deformation oo-groupoid}, we are naturally lead to consider non-conilpotent coalgebras. This can be done, but in order to have the results that we desire we need to take proper complete algebras on the other side.

\medskip

For the rest of this section, fix a cooperad $\C$, an operad $\P$ and a twisting morphism $\alpha$. Given a $\C$-coalgebra\footnote{Remember: non-conilpotent!} $C$, one defines a $\P$-algebra by
\[
\cCobar_\alpha C\coloneqq\left(\widehat{\P}(C)\coloneqq\prod_{n\ge0}\P(n)\otimes_{\S_n}C^{\otimes n}, d_{\cCobar_\alpha C}\coloneqq d_1 + d_2\right),
\]
where $d_1$ and $d_2$ are defined exactly as for the usual cobar construction relative to $\alpha$, but this time $d_2$ passes through $\C\widehat{\circ}C$, instead of $\C\circ C$, cf. \cref{subsection:relative bar and cobar}. The filtration
\[
\F_n\coloneqq\prod_{k\ge n}\P(k)\otimes_{\S_k}C^{\otimes k}
\]
makes it into a proper complete $\P$-algebra.

\begin{definition}
	The proper complete $\P$-algebra $(\cCobar_\alpha C,\F_\bullet)$ is called the \emph{complete cobar construction relative to $\alpha$}\index{Complete!cobar construction}.
\end{definition}

\begin{lemma}\label{lemma:free complete P-alg}
	Let $V$ be a chain complex, seen as a proper complete chain complex by imposing the filtration $\F_1V = V$ and $\F_nV = 0$ for all $n\ge2$. Let $A$ be a proper complete $\P$-algebra. We have a natural bijection
	\[
	\hom_\fchain(V,A)\cong\hom_{\Phatalg}(\Phat(V),A)\ .
	\]
\end{lemma}

\begin{proof}
	Notice that the left hand side is also equal to the set of all morphisms of chain complexes from $V$ to $A$, forgetting the filtrations. Given such a chain map $\phi:V\to A$, we get a map of algebras by
	\[
	\Phat(V)\xrightarrow{\Phat(\phi)}\Phat(A)\xrightarrow{\gamma_A}A\ .
	\]
	It is straightforward to check that it is well-defined and that it respects the filtrations. The other direction is given by restricting a morphism of algebras $\Phat(V)\to A$ to $V$.
\end{proof}

\begin{remark}
	For the result above, it is important that $\P$ is reduced, else we might get infinite sums at every level of the filtration, resulting in the morphisms in the proof not being defined.
\end{remark}

\begin{proposition}\label{prop:partial Rosetta stone for non-conilpotent coalgebras}
	Let $\alpha\in\Tw(\C,\P)$ be a twisting morphism, let $C$ be a $\C$-coalgebra, and let $A$ be a proper complete $\P$-algebra. There is a natural bijection
	\[
	\hom_{\Phatalg}(\cCobar_\alpha C,A)\cong\Tw_\alpha(C,A)\ .
	\]
\end{proposition}

\begin{proof}
	The proof is analogous to the one for the similar bijection in \cref{thm:Rosetta stone algebras}. One uses \cref{lemma:free complete P-alg}.
\end{proof}

\subsection{Examples}

We give some examples of $\L_\infty$-algebra structures arising through \cref{thm:convolution Loo-algebras}. We put ourselves in the situation of \cref{rem:original case convolution Loo-algebras} and in the tensor product setting. These examples are all extracted from \cite[Sect. 8]{rn17tensor}.

\medskip

We will study the $\L_\infty$-algebra structures obtained from \cref{thm:convolution Loo tensor products} for some canonical morphisms between the three most often appearing operads: the three graces $\com$, $\lie$ and $\ass$. Namely, we will study the identities of these operads and the sequence of morphisms
\[
\lie\stackrel{a}{\longrightarrow}\ass\stackrel{u}{\longrightarrow}\com\ ,
\]
where the first morphism corresponds to the antisymmetrization of the multiplication of an associative algebra, and where the second one corresponds to forgetting that the multiplication of a commutative algebra is commutative to get an associative algebra.

\medskip

Many more examples of less common, but still very interesting operads, both in the symmetric and in the ns case, as well as various morphisms relating them, can be found in \cite[Sect. 13]{LodayVallette}.

\subsubsection*{Notations} \label{subsect:notations examples convolution Loo-algebras}

We will denote by $b\in\lie(2)$ the generating operation of $\lie$, i.e. the Lie bracket. The operad $\ass$ is the symmetric version  the non-symmetric operad $\as$ coding associative algebras. It is given by $\ass(n) = \k[\S_n]$. We denote the canonical basis of $\ass(n)$ by $\{m_\sigma\}_{\sigma\in\S_n}$. The element $m_\sigma\in\ass(n)$ corresponds to the operation
\[
(a_1,\ldots,a_n)\longmapsto a_{\sigma^{-1}(1)}\cdots a_{\sigma^{-1}(n)}
\]
at the level of associative algebras. The action of the symmetric group is of course given by $(m_\sigma)^\tau = m_{\sigma\tau}$. As before, we denote by $\mu_n\in\com(n)$ the canonical element. The morphism $a:\lie\to\ass$ is given by sending $b$ to $m_{\id} - m_{(12)}$, and corresponds to antisymmetrization at the level of algebras. The morphism $u:\ass\to\com$ is given by sending both $m_{\id}$ and $m_{(12)}$ to $\mu_2$. For the homotopy counterparts of the operads mentioned above: as before we denote by $\ell_n\in\L_\infty(n)$ the element of $\L_\infty(n)$ corresponding to the $n$-ary bracket. We have $\ass^{\antishriek}\cong(\susp^{-1})^c\otimes\ass^\vee$, thus in each arity $n\ge2$ the operad $\ass_\infty$ has $n!$ generators
\[
\overline{m}_\sigma\coloneqq s^{-1}\susp^{-1}_nm_\sigma^\vee\in\ass_\infty\ ,\qquad \sigma\in\S_n\ .
\]
The action of the symmetric group on these generators is given by $(\overline{m}_\sigma)^\tau = (-1)^\tau\overline{m}_{\sigma\tau}$. Finally, the operad $\C_\infty$ coding homotopy commutative algebras the same thing as an $\A_\infty$-algebra that vanishes on the sum of all non-trivial shuffles, see \cref{subsect:homotopy commutative algebras}.

\subsubsection*{The identity $\com\longrightarrow\com$} This is the simplest example. The identity of $\com$ induces the morphism
\[
\manintensor_\com:\L_\infty\longrightarrow\com\otimes\L_\infty
\]
which sends the element $\ell_n$ to
\[
\mu_n\otimes s^{-1}\susp_n\mu_n^\vee = \mu_n\otimes\ell_n\ .
\]
Therefore, it is the canonical isomorphism
\[
\L_\infty\cong\com\otimes\L_\infty\ .
\]
If $A$ is a commutative algebra and $C$ is an $\L_\infty$-algebra, then the operations on $A\otimes C$ are given by
\[
\ell_n(a_1\otimes c_1,\ldots,a_n\otimes c_n) = (-1)^\epsilon \mu_n(a_1,\ldots,a_n)\otimes\ell_n(c_1,\ldots,c_n)\ ,
\]
where $(-1)^\epsilon$ is the sign obtained by commuting the $a_i$'s and the $c_i$'s.

\subsubsection*{The identity $\ass\longrightarrow\ass$} Since the operad $\ass$ satisfies $\ass^{\shriek} = \ass$, the induced morphism is
\[
\manintensor_\ass:\L_\infty\longrightarrow\ass\otimes\ass_\infty\ .
\]
It sends $\ell_n$ to
\[
\sum_{\sigma\in\S_n}m_\sigma\otimes s^{-1}\susp_n m_\sigma^\vee = \sum_{\sigma\in\S_n}m_\sigma\otimes \overline{m}_\sigma = \sum_{\sigma\in\S_n}(-1)^\sigma(m_{\id}\otimes\overline{m}_{\id})\ .
\]
If $A$ is an associative algebra and $C$ is an $\ass_\infty$-algebra, then the $\L_\infty$ operations on $A\otimes C$ are given by
\[
\ell_n(a_1\otimes c_1,\ldots,a_n\otimes c_n) = \sum_{\sigma\in\S_n}(-1)^{\sigma+\epsilon} m_e(a_{\sigma^{-1}(1)},\ldots a_{\sigma^{-1}(n)})\otimes\overline{m}_e(c_{\sigma^{-1}(1)},\ldots c_{\sigma^{-1}(n)})\ ,
\]
where $\epsilon$ is the sign obtained by switching the $a_i$'s and the $c_i$'s, and correspond therefore to a kind of antisymmetrization of $\ass_\infty$.

\subsubsection*{The identity $\lie\longrightarrow\lie$} The last identity we have to look at is the identity of the operad $\lie$. It gives rise to a morphism of operads
\[
\manintensor_\lie:\L_\infty\longrightarrow\lie\otimes\C_\infty\ .
\]
It is of more complicated description, but comparing formul{\ae} we see that it is the same structure that is used in a fundamental way in the article \cite[pp.19--20]{tw15} on Hochschild--Pirashvili homology.

\subsubsection*{The forgetful morphism $u:\ass\longrightarrow\com$} This morphism is given by sending
\[
m_\sigma\longmapsto \mu_n
\]
for all $\sigma\in\S_n$. The corresponding morphism
\[
\manintensor_u:\L_\infty\longrightarrow\com\otimes\ass_\infty\cong\ass_\infty
\]
is given by
\[
\manintensor_u(\ell_n) = \sum_{\sigma\in\S_n}\mu_n\otimes s^{-1}\susp_nm_\sigma^\vee = \sum_{\sigma\in\S_n}\mu_n\otimes\overline{m}_\sigma = \mu_n\otimes\sum_{\sigma\in\S_n}(-1)^\sigma(\overline{m}_e)^\sigma.
\]
Therefore, under the canonical identification $\com\otimes\ass_\infty\cong\ass_\infty$, it is the standard antisymmetrization of an $\ass_\infty$-algebra structure giving an $\L_\infty$-algebra structure.

\subsubsection*{The antisymmetrization morphism $a:\lie\longrightarrow\ass$} The induced morphism is a morphism of dg operads
\[
\manintensor_a:\L_\infty\longrightarrow\ass\otimes\C_\infty\ .
\]
It can be interpreted as follows: a $\C_\infty$-algebra can be seen as an $\ass_\infty$-algebra vanishing on the sum of all non-trivial shuffles, that is, we have a natural morphism of operads
\[
i:\ass_\infty\longrightarrow\C_\infty\ ,
\]
which is in fact given by $\Omega((\susp^{-1})^c\otimes a^\vee)$. Now \cref{thm:convolution Loo tensor products} tells us that
\[
\manintensor_a = \manintensor_{1_\ass a} = (1\otimes i)\manintensor_{1_\ass}\ .
\]
Therefore, the $\L_\infty$-algebra structure on the tensor product of an associative and a $\C_\infty$-algebra is given by first looking at the $\C_\infty$-algebra as an $\ass_\infty$-algebra, and then antisymmetrizing the resulting $(\ass\otimes\ass_\infty)$-algebra as already done above.

\subsubsection*{The non-symmetric case} The analogues to the three graces in the non-symmetric setting are the operad $\as$ encoding associative algebras, which we already know well, the operad $\mathrm{Dend}$ of \emph{dendriform algebras}\index{Dendriform algebras} (\cite[Sect. 13.6.5]{LodayVallette}), and the operad $\mathrm{Dias}$ encoding \emph{diassociative algebras}\index{Diassociative algebras} (\cite[Sect. 13.6.7]{LodayVallette}). They fit into a sequence of morphisms
\[
\mathrm{Dias}\longrightarrow\as\longrightarrow\mathrm{Dend}\ .
\]
We get induced morphisms from $\A_\infty$ to
\[
\mathrm{Dias}\otimes\mathrm{Dend}_\infty\ ,\quad\A_\infty\ ,\quad\mathrm{Dend}\otimes\mathrm{Dias}_\infty\ ,\quad\mathrm{Dend}_\infty\,\quad\text{and}\quad\mathrm{Dend}\otimes\A_\infty\ .
\]
We leave their explicit computation to the interested reader.

\subsection{Compatibility with (co)limits of (co)algebras}

To complete this section, we prove a first compatibility result between (co)algebras and the induced convolution homotopy algebras. As before, let $\C$ be a cooperad, let $\P$ be an operad, and let $\alpha:\C\to\P$ be a twisting morphism.

\begin{proposition}\label{prop:compatibility convolution and (co)limits}
	Let $I$ be a small category.
	\begin{enumerate}
		\item Let $A:I\to\Palg$ be a functor, suppose $\lim_{i\in I}A(i)$ exists, and let $C$ be a conilpotent $\C$-coalgebra. Then
		\[
		\lim_{i\in I}\homa(C,A(i))\cong\homa\left(C,\lim_{i\in I}A(i)\right)
		\]
		as $\SLoo$-algebras through the canonical morphism.
		\item Dually, let $C:I\to\Ccog$ be a functor, suppose $\colim_{i\in I}C(i)$ exists, and let $A$ be a $\P$-algebra. Then
		\[
		\lim_{i\in I}\homa(C(i),A)\cong\homa\left(\colim_{i\in I}C(i),A\right)
		\]
		as $\SLoo$-algebras through the canonical morphism.
	\end{enumerate}
\end{proposition}

\begin{proof}
	The isomorphisms hold at the level of chain complexes, and the canonical maps are morphisms of $\SLoo$-algebras. One concludes by \cref{lemma:isomorphisms of P-algebras}, resp. \cref{lemma:isomorphisms of C-coalgebras}.
\end{proof}

We also have a ``linear maps-tensor products'' duality analogue to what happens for chain complexes.

\begin{proposition}
	Let $\alpha:\C\to\P$ be a twisting morphism, and suppose that $\C$ is finite dimensional in every arity. Let $C$ be a finite-dimensional, conilpotent $\C$-coalgebra, and let $A$ be a $\P$-algebra. Then
	\[
	\homa(C,A)\cong A\otimes^\alpha C^\vee
	\]
	as $\SLoo$-algebras via the canonical isomorphism of chain complexes
	\[
	f\in\hom(C,A)\longmapsto\sum_{j=1}^Nf(x_j)\otimes x_j^\vee\ ,
	\]
	where $\{x_j\}_{j=1}^n$ is a basis $C$.
\end{proposition}

\begin{proof}
	In this proof, we will avoid writing down any explicit signs, and work in the non-symmetric setting. The symmetric case works in exactly the same way, but one has to write down the correct permutations occurring in the formul{\ae}.
	
	\medskip
	
	Since the described morphism is an isomorphism of chain complexes, by \cref{lemma:isomorphisms of P-algebras} it is enough to show that it is a morphism of $\SLoo$-algebras. In order to do this, let $f_1,\ldots,f_n\in\hom(C,A)$, and write $F\coloneqq f_1\otimes\cdots\otimes f_n$. Fix a basis $\{c_k\}_k$ of $\C(n)$. In $\homa(C,A)$, we have
	\[
	\ell_n(F) = \gamma_A(\alpha\otimes F)\Delta_C\in\hom(C,A)\ .
	\]
	This is sent to the element of $A\otimes C^\vee$ given by
	\begin{align*}
		\sum_j(&\gamma_A(\alpha\otimes F)\Delta^n_C(x_j))\otimes x_j^\vee =\\
		=&\ \sum_{k,j,i_1,\ldots,i_n}\pm\langle c_k^\vee\otimes x_{i_1}^\vee\otimes\cdots\otimes x_{i_n}^\vee,\Delta_C(x_j)\rangle\left(\gamma_A(\alpha(c_k)\otimes f_1(x_{i_1})\otimes\cdots\otimes f_n(x_{i_n}))\right)\otimes x_j^\vee\\
		=&\ \sum_{k,i_1,\ldots,i_n}\left(\gamma_A(\alpha(c_k)\otimes f_1(x_{i_1})\otimes\cdots\otimes f_n(x_{i_n}))\right)\otimes \sum_j\pm\langle\gamma_{C^\vee}(c_k^\vee\otimes x_{i_1}^\vee\otimes\cdots\otimes x_{i_n}^\vee),x_j\rangle x_j^\vee\\
		=&\ \sum_{k,i_1,\ldots,i_n}\left(\gamma_A(\alpha(c_k)\otimes f_1(x_{i_1})\otimes\cdots\otimes f_n(x_{i_n}))\right)\otimes \sum_j\pm\gamma_{C^\vee}(c_k^\vee\otimes x_{i_1}^\vee\otimes\cdots\otimes x_{i_n}^\vee)\\
		=&\ \ell_n\left(\sum_{i_1}f_1(x_{i_1})\otimes x_{i_1}^\vee,\ldots,\sum_{i_n}f_n(x_{i_n})\otimes x_{i_n}^\vee\right),
	\end{align*}
	which concludes the proof.
\end{proof}

\section{Compatibility with \texorpdfstring{$\infty$}{infinity}-morphisms}\label{sect:compatibility of convolution algebras with oo-morphisms}

In this section, we study the compatibility between convolution homotopy Lie algebras and $\infty$-morphisms of algebras and coalgebras. More precisely, we will show that the bifunctor $\homa(-,-)$ can be extended two natural ways to take either $\infty_\alpha$-morphisms in the first slot, or in the second one. We prove that those two extension do not admit a further common extension to a bifunctor accepting $\infty_\alpha$-morphisms in both slots. We conclude the section by proving that however, this last extension is possible if one accepts to work up to homotopy.

\subsection{The fundamental theorem of convolution homotopy algebras}

Fix a cooperad $\C$, an operad $\P$, and a twisting morphism $\alpha:\C\to\P$. Let $C$ be a conilpotent $\C$-coalgebra, and let $A,A'$ be two $\P$-algebras. Given an element
\[
x\in\homa(\Bara A,A')\ ,
\]
we define an element
\[
\homa(1,x)\in\homi(\Bari\homa(C,A),\homa(C,A'))
\]
as follows. Let $f_1,\ldots,f_n\in\hom(C,A)$ and denote $F\coloneqq f_1\otimes\cdots\otimes f_n$, then $x_*(\mu_n^\vee\otimes F)$ is given by the composite
\begin{center}
	\begin{tikzpicture}
		\node (a) at (0,0){$C$};
		\node (b) at (3.5,0){$\C(n)\otimes_{\S_n}C^{\otimes n}$};
		\node (c) at (0,-2){$A'$};
		\node (d) at (3.5,-2){$\C(A)$};
		
		\draw[->] (a) -- node[above]{$\Delta^n_C$} (b);
		\draw[->,dashed] (a) -- node[left]{$\homa(1,x)(\mu_n^\vee\otimes F)$} (c);
		\draw[->] (b) -- node[right]{$F$} (d);
		\draw[->] (d) -- node[above]{$x$} (c);
	\end{tikzpicture}
\end{center}
where $F$ acts on $\C(n)\otimes_{\S_n}C^{\otimes n}$ by
\[
F(c\otimes y_1\otimes\cdots y_n)\coloneqq\sum_{\sigma\in\S_n}(-1)^\epsilon c\otimes f_{\sigma(1)}(y_1)\otimes\cdots\otimes f_{\sigma(n)}(y_n)\in\C(n)\otimes_{\S_n}A^{\otimes n}\ .
\]
Dually, let $C,C'$ be two conilpotent $\C$-coalgebras, and let $A$ be a $\P$-algebra. Given
\[
x\in\homa(C',\Cobara C)\ ,
\]
we define an element
\[
\homa(x,1)\in\homi(\Bari\homa(C,A),\homa(C',A))
\]
as follows. Let $f_1,\ldots,f_n\in\hom(C,A)$, then $x^*(\mu_n^\vee\otimes F)$ is given by the composite
\begin{center}
	\begin{tikzpicture}
	\node (a) at (0,0){$C'$};
	\node (b) at (3.5,0){$\P(C)$};
	\node (e) at (3.5,-2){$\P(n)\otimes_{\S_n}C^{\otimes n}$};
	\node (c) at (0,-4){$A'$};
	\node (d) at (3.5,-4){$\P(A)$};
	
	\draw[->] (a) -- node[above]{$x$} (b);
	\draw[->,dashed] (a) -- node[left]{$\homa(x,1)(\mu_n^\vee\otimes F)$} (c);
	\draw[->] (b) -- node[right]{$\proj_n$} (e);
	\draw[->] (e) -- node[right]{$F$} (d);
	\draw[->] (d) -- node[above]{$\gamma_A$} (c);
	\end{tikzpicture}
\end{center}
with the action of $F$ similar to the one defined above.

\begin{theorem}[{\cite[Thm. 3.1]{rnw18}}]\label{thm:fundamental thm of convolution Loo-algebras}
	Let $\C$ be a cooperad, let $\P$ be an operad, and let $\alpha:\C\to\P$ be a twisting morphism. Let $C,C'$ be two conilpotent $\C$-coalgebras, and let $A,A'$ be two $\P$-algebras.
	\begin{enumerate}
		\item The map
		\[
		\homa(1,-):\homa(\Bara A,A')\longrightarrow\homi(\Bari\homa(C,A),\homa(C,A'))
		\]
		defined above is a morphism of $\SLoo$-algebras.
		\item The map
		\[
		\homa(-,1):\homa(C',\Cobara C)\longrightarrow\homi(\Bari\homa(C,A),\homa(C',A))
		\]
		defined above is a morphism of $\SLoo$-algebras.
	\end{enumerate}
\end{theorem}

This theorem can reasonably be considered one of the fundamental results about convolution homotopy algebras --- at least as of the time of writing of this thesis. It should be considered a refinement of \cite[Prop. 4.4]{rn17tensor} and \cite[Thm. 5.1]{rnw17}, and indeed we will see these results appear as corollaries of the statement above later on.

\medskip

For the proof, we begin with two technical lemmas.

\begin{lemma}\label{lemma:technical lemma 1}
	Let $k,n_1,\ldots,n_k\ge0$, let $M$ be an $\S$-module, let $V,W$ be two chain complex. Denote $n\coloneqq k_1+\cdots+n_k$, let $f_1,\ldots,f_n\in\hom(V,W)$, and as usual denote $F\coloneqq f_1\otimes\cdots\otimes f_n$. Then, under the isomorphism
	\[
	\left(M(k)\otimes\bigotimes_{i=1}^kM(n_i)\right)\otimes V^{\otimes n}\cong M(k)\otimes\bigotimes_{i=1}^k\left(M(n_i)\otimes V^{\otimes n_i}\right),
	\]
	we have
	\[
	F = \sum_{S_1\sqcup\cdots\sqcup S_k=[n]}1_M\circ\left(F^{S_1}\otimes\cdots\otimes F^{S_k}\right),
	\]
	where $F^{S_i}\coloneqq\bigotimes_{j\in S_i}f_j$, with the elements of $S_i$ in ascending order.
\end{lemma}

\begin{proof}
	This is a straightforward computation --- if messy at the level of signs. It is left to the reader.
\end{proof}

\begin{lemma}\label{lemma:technical lemma 2}
	Let $\C$ be a cooperad, and let $C$ be a conilpotent $\C$-coalgebra. Then
	\[
	(\Delta_{(1)}\circ1_C)\Delta_C^n = \sum_{\substack{n_1+n_2=n+1\\1\le i\le n_1}}(1_\C\circ(1_C^{\otimes(i-1)}\otimes\Delta_C^{n_2}\otimes1_C^{\otimes(n-i)}))\Delta_C^{n_1}\ .
	\]
	Moreover, let $f_1,\ldots,f_n\in\hom(C,V)$ for $V$ a chain complex. Under the canonical inclusion
	\[
	\bigoplus_{\substack{n_1+n_2 = n+1\\1\le i\le n_1}}\C(n_1)\circ(C^{\otimes(i-1)}(\C(n_2)\otimes C^{\otimes n_2})\otimes C^{\otimes(n-i)})\xhookrightarrow{\qquad}(\C\circ\C)(n)\otimes C^{\otimes n}
	\]
	we have
	\[
	F(\Delta_{(1)}\circ1_C)\Delta_C^n = \sum_{\substack{S_1\sqcup S_2=[n]}}((F^{S_1}\Delta_C^{n_1})\otimes F^{S_2})\Delta_C^{n_2}\ ,
	\]
	where $n_1 = |S_1|$ and $n_2 = |S_2|+1$.
\end{lemma}

\begin{proof}
	For the first identity, one considers the equality
	\[
	(\Delta_\C\circ 1_C)\Delta_C = (1_\C\circ\Delta_C)\Delta_C
	\]
	and then projects on the subspace
	\[
	(\C\circ_{(1)}\C)(C)\cong \C\circ(C;\C(C))\ .
	\]
	We leave the details to the reader. The second statement then follows in a straightforward way.
\end{proof}

\begin{proof}[Proof of \cref{thm:fundamental thm of convolution Loo-algebras}]
	We prove the first case, the second one being dual. Let $x_1,\ldots,x_k\in\homa(\Bara A,A')$. Then we have
	\[
	\ell_k(x_1,\ldots,x_k) = \gamma_{A'}(\alpha\otimes X)(\Delta_\C^k\circ1_A)\ ,
	\]
	and thus for $f_1,\ldots,f_n\in\homa(C,A)$
	\begin{align*}
	\homa(1,\ell_k&(X))(\mu_n^\vee\otimes F) =\\
	=&\ \ell_k(X)F\Delta_C^n\\
	=&\ \gamma_{A'}(\alpha\otimes X)(\Delta_\C^k\circ1_A)F\Delta_C^n\\
	=&\ \gamma_{A'}(\alpha\otimes X)F(\Delta_\C^k\circ1_C)\Delta_C^n\\
	=&\ \sum_{n_1+\cdots+n_k=n}\gamma_{A'}(\alpha\otimes X)F(1_\C\circ(\Delta_C^{n_1}\otimes\cdots\otimes\Delta_C^{n_k}))\Delta_C^k\\
	=&\ \sum_{S_1\sqcup\cdots\sqcup S_k=[n]}(-1)^{\epsilon_1}\gamma_{A'}(\alpha\otimes X)(1_\C\circ(F^{S_1}\Delta_C^{n_1}\otimes\cdots\otimes F^{S_k}\Delta_C^{n_k}))\Delta_C^k\\
	=&\ \sum_{\substack{S_1\sqcup\cdots\sqcup S_k=[n]\\\sigma\in\S_k}}(-1)^{\epsilon_1+\epsilon_2}\gamma_{A'}(\alpha\circ1_{A'})(1_\C\circ(x_{\sigma(1)}F^{S_1}\Delta_C^{n_1}\otimes\cdots \otimes\Phi_{\sigma(k)}F^{S_k}\Delta_C^{n_k}))\Delta_C^k
	\end{align*}
	where the fourth line follows from $(\Delta_\C\circ C)\Delta_C = (1_\C\circ\Delta_C)\Delta_C$, and in the fifth line we used \cref{lemma:technical lemma 1} and denoted $n_i\coloneqq|S_i|$. The Koszul signs are
	\[
	\epsilon_1 = \sum_{i=1}^k\sum_{s\in S_i}|f_s|\sum_{j<i}\sum_{p\in S_j\text{ s.t. }p>s}|f_p|\ ,
	\]
	obtained by shuffling the $f_i$s, and $\epsilon_2$, which is similarly obtained by permuting the $x_i$s and making them jump over the $f_j$s.
	
	\medskip
	
	On the other hand, we have
	\begin{align*}
	\ell_k(\homa(1,X))&(\mu_n^\vee\otimes F) =\\
	=&\ \big(\gamma_{\homa(C,A')}(\iota\otimes\homa(1,X))(\Delta_{\com^\vee}^k\otimes1_{\hom(C,A)})\big)(\mu_n^\vee\otimes F)\\
	=&\ \big(\gamma_{\homa(C,A')}(\iota\otimes\homa(1,X))\big)\left(\sum_{\substack{S_1\sqcup\ldots\sqcup S_k = [n]}}(-1)^{\epsilon_1}\mu_k^\vee\otimes\bigotimes_{i=1}^k(\mu_{n_i}^\vee\otimes F^{S_i})\right)\\
	=&\ \gamma_{\homa(C,A')}\left(\sum_{\substack{S_1\sqcup\ldots\sqcup S_k = [n]\\\sigma\in\S_k}}(-1)^{\epsilon_1+\epsilon_2}s^{-1}\mu_k^\vee\otimes\bigotimes_{i=1}^k\homa(1,x_{\sigma(i)})(\mu_{n_i}^\vee\otimes F^{S_i})\right)\\
	=&\ \sum_{\substack{S_1\sqcup\ldots\sqcup S_k = [n]\\\sigma\in\S_k}}(-1)^{\epsilon_1+\epsilon_2}\gamma_{A'}\left(\alpha\otimes\bigotimes_{i=1}^k\homa(1,x_{\sigma(i)})(\mu_{n_i}^\vee\otimes F^{S_i})\right)\Delta_C^k\\
	=&\ \sum_{\substack{S_1\sqcup\ldots\sqcup S_k = [n]\\\sigma\in\S_k}}(-1)^{\epsilon_1+\epsilon_2}\gamma_{A'}\left(\alpha\otimes\bigotimes_{i=1}^kx_{\sigma(i)}F^{S_i}\Delta_C^{n_i}\right)\Delta_C^k\ .
	\end{align*}
	The reader might have the impression that a sum over permutations coming from $\gamma_{\homa(C,A')}$ has been forgotten in the fourth line. This is not the case, because the term
	\[
	\sum_{\substack{S_1\sqcup\ldots\sqcup S_k = [n]\\\sigma\in\S_k}}(-1)^{\epsilon_1+\epsilon_2}s^{-1}\mu_k^\vee\otimes\bigotimes_{i=1}^k\homa(1,x_{\sigma(i)})(\mu_{n_i}^\vee\otimes F^{S_i})
	\]
	in the third line naturally lives in invariants, not coinvariants.
	
	\medskip
	
	In conclusion, we have
	\[
	\homa(1,\ell_k(X)) = \ell_k(\homa(1,X))\ .
	\]
	We are left to prove that the morphism commutes with the differentials. Let $x\in\homa(\Bara A,A')$, and let $f_1,\ldots,f_n\in\homa(C,A)$. On one hand, we have
	\begin{align*}
	d(x) =&\ d_{A'}x - (-1)^{|x|}x d_{\Bara A}\\
	=&\ d_{A'}x - (-1)^{|x|}x d_{\C\circ A} - (-1)^{|x|}x (1_\C\circ(1_A;\gamma_A))((1_\C\circ_{(1)}\alpha)\circ1_A)(\Delta_{(1)}\circ1_A)\ ,
	\end{align*}
	and thus
	\begin{align*}
	\homa(1,d(x))(\mu_n^\vee\otimes F) =&\ d(x)F\Delta_C^n\\
	=&\ d_{A'}x F\Delta_C^n - (-1)^{|x|}x d_{\C\circ A}F\Delta_C^n\\
	&- (-1)^{|x|}x (1_\C\circ(1_A;\gamma_A))((1_\C\circ_{(1)}\alpha)\circ1_A)(\Delta_{(1)}\circ1_A)F\Delta_C^n\ .
	\end{align*}
	The second term in the second line equals
	\begin{align*}
	xd_{\C\circ A}F\Delta_C^n =&\ x(d_\C\circ1_A)F\Delta_C^n + x(1_\C\circ'd_A)F\Delta_C^n\\
	=&\ (-1)^{|F|}x F(d_\C\circ1_C)\Delta_C^n + x\partial(F)\Delta_C^n + (-1)^{|F|} Fx(1_\C\circ'd_C)\Delta_C^n\\
	=&\ (-1)^{|F|}x F\Delta_C^nd_C + x\partial(F)\Delta_C^n\ ,\tag{\text{T1}}\label{T1}
	\end{align*}
	while the term in the last line gives
	\begin{align*}
	x (1_\C\circ(1_A;\gamma_A))&((1_\C\circ_{(1)}\alpha)\circ1_A)(\Delta_{(1)}\circ1_A)F\Delta_C^n =\\
	=&\ x (1_\C\circ(1_A;\gamma_A))((1_\C\circ_{(1)}\alpha)\circ1_A)F(\Delta_{(1)}\circ1_C)\Delta_C^n\\
	=&\ \sum_{S_1\sqcup S_2=[n]}(-1)^\epsilon x (1_\C\circ(1_A;\gamma_A))((1_\C\circ_{(1)}\alpha)\circ1_A)((F^{S_1}\Delta_C^{n_1})\otimes F^{S_2})\Delta_C^{n_2}\ ,\tag{\text{T2}}\label{T2}
	\end{align*}
	where in the third line we used \cref{lemma:technical lemma 2}. On the other hand,
	\begin{align*}
	d(\homa(1,x)) =&\ d_{\homa(C,A')}\homa(1,x) - (-1)^{|\homa(1,x)|}\homa(1,x)d_{\Bari\homa(C,A)}\ .
	\end{align*}
	We notice that $(-1)^{|\homa(1,x)|} = (-1)^{|x|}$. We apply this to $\mu_n^\vee\otimes F$ and obtain
	\begin{align*}
	d(\homa(1,x))(\mu_n^\vee\otimes F) =&\ d_{A'}\homa(1,x) - (-1)^{|x|+|F|}\homa(1,x)(\mu_n^\vee\otimes F)d_C\\
	&-(-1)^{|x|}\homa(1,x)(d_{\Bari\homa(C,A)}(\mu_n^\vee\otimes F))\ .
	\end{align*}
	The first term equals $d_{A'}\Phi F\Delta_C^n$ and cancels with the first term of $\homa(1,d(\Phi))(\mu_n^\vee\otimes F)$, and the second term equals the first term of (\ref{T1}). For the third term, we have
	\begin{align*}
	\homa(1,\Phi)(d_{\Bari\homa(C,A)}&(\mu_n^\vee\otimes F)) =\\
	=&\ \homa(1,\Phi)(\mu_n^\vee\otimes \partial(F))\\
	&+ \homa(1,\Phi)\left(\sum_{S_1\sqcup S_2=[n]}\mu_{n_2}^\vee\otimes(\ell_{n_1}(F^{S_1})\otimes F^{S_2})\right).
	\end{align*}
	The first term of this expression cancels the remaining term of (\ref{T1}). Therefore, we are left to show that the second term equals (\ref{T2}). We have
	\begin{align*}
	\homa(1,x)&\left(\sum_{S_1\sqcup S_2=[n]}(-1)^\epsilon\mu_{n_2}^\vee\otimes(\ell_{n_1}(F^{S_1})\otimes F^{S_2})\right) =\\
	=&\ \sum_{S_1\sqcup S_2=[n]}(-1)^\epsilon x(\ell_{n_1}(F^{S_1})\otimes F^{S_2})\Delta_C^{n_2}\\
	=&\ \sum_{S_1\sqcup S_2=[n]}(-1)^\epsilon x((\gamma_A(\alpha\otimes F^{S_1})\Delta_C^{n_1})\otimes F^{S_2})\Delta_C^{n_2}\\
	=&\ \sum_{S_1\sqcup S_2=[n]}(-1)^\epsilon x(1_\C\circ(\gamma_A(\alpha\circ1_A)\otimes1_A^{\otimes (n_2-1)}))(1_\C\circ(F^{S_1}\Delta_C^{n_1}\otimes F^{S_2}))\Delta_C^{n_2}\\
	=&\ (\ref{T2})\ ,
	\end{align*}
	where $n_1 = |S_1|$, and $n_2 = |S_2|+1$. This concludes the proof.
\end{proof}

\subsection{Compatibility with \texorpdfstring{$\infty$}{infinity}-morphisms}\label{subsection:compatibility with oo-morphisms}

The first application of \cref{thm:fundamental thm of convolution Loo-algebras} is the fact that convolution $\L_\infty$-algebras are compatible with $\infty$-morphisms (relative to the twisting morphism $\alpha$ under consideration).

\medskip

Let $\C$ be a cooperad, let $\P$ be an operad, and let $\alpha:\C\to\P$ be a twisting morphism. Take two $\P$-algebras $A,A'$, as well as a conilpotent $\C$-coalgebra $C$. By \cref{thm:fundamental thm of convolution Loo-algebras}, we have a morphism of $\SLoo$-algebras
\[
\homa(1,-):\homa(\Bara A,A')\longrightarrow\homi(\Bari\homa(C,A),\homa(C,A'))\ .
\]
Let's look at what happens to Maurer--Cartan elements. By \cref{thm:MC-el of convolution Loo-algebras}, the Maurer--Cartan elements of $\homa(\Bara A,A')$ are exactly the $\infty_\alpha$-morphisms $A\rightsquigarrow A'$. Now, since $\homa(1,-)$ is a morphism of $\SLoo$-algebras, it preserves Maurer--Cartan elements. Thus, an $\infty_\alpha$-morphism $\Psi:A\rightsquigarrow A'$ is sent to a Maurer--Cartan element of $\homi(\Bari\homa(C,A),\homa(C,A'))$. But again by \cref{thm:MC-el of convolution Loo-algebras}, these are exactly the $\infty$-morphisms $\homa(C,A)\rightsquigarrow\homa(C,A')$. The same thing is true for the dual case. In conclusion, we have the following.

\begin{corollary}[{\cite[Thm. 5.1]{rnw17}}]\label{cor:preservation of oo-morphisms 1}
	We place ourselves in the setting of \cref{thm:fundamental thm of convolution Loo-algebras}.
	\begin{enumerate}
		\item Suppose $\Psi:A\rightsquigarrow A'$ is an $\infty_\alpha$-morphism of $\P$-algebras. Then
		\[
		\homa(1,\Psi):\homa(C,A)\rightsquigarrow\homa(C,A')
		\]
		is an $\infty$-morphism of $\SLoo$-algebras.
		\item Dually, suppose that $\Phi:C'\rightsquigarrow C$ is an $\infty_\alpha$-morphism of $\C$-coalgebras. Then
		\[
		\homa(\Phi,1):\homa(C,A)\rightsquigarrow\homa(C',A)
		\]
		is an $\infty$-morphism of $\SLoo$-algebras.
	\end{enumerate}
\end{corollary}

But now, in the situation above, $\homa(1,\Psi)$ also preserves Maurer--Cartan elements. Therefore, it sends morphisms $C\to\Bara A$ of $\C$-coalgebras to morphisms $C\to\Bara A'$ of $\C$-coalgebras.

\begin{lemma}\label{lemma:hom of oo-morphism is pre or postcomposition}
	We place ourselves in the situation of \cref{thm:fundamental thm of convolution Loo-algebras} and \cref{cor:preservation of oo-morphisms 1}.
	\begin{enumerate}
		\item Suppose $\Psi:A\rightsquigarrow A'$ is an $\infty_\alpha$-morphism of $\P$-algebras, and let $f:C\to\Bara A$ be a morphism of $\C$-coalgebras. Then
		\[
		\MC(\homa(1,\Psi))(f) = \Psi f = \left(C\xrightarrow{f}\Bara A\xrightarrow{\Psi}\Bara A'\right).
		\]
		\item Suppose $\Phi:C'\rightsquigarrow C$ is an $\infty_\alpha$-morphism of $\C$-coalgebras, and let $g:\Cobara C\to A$ be a morphism of $\P$-algebras. Then
		\[
		\MC(\homa(\Phi,1))(g) = g\Phi= \left(\Cobara C'\xrightarrow{\Phi}\Cobara C\xrightarrow{g}A\right).
		\]
	\end{enumerate}
\end{lemma}

\begin{proof}
	In order to give a clear proof, we will write $f\in\MC(\homa(C,A))$ for the element $f$ seen as a linear map $C\to A$, and $\widetilde{f}$ for the equivalent map of $\C$-coalgebras $C\to\Bara A$. We pass from the former to the latter by
	\[
	\widetilde{f} = (1_\C\circ f)\Delta_C\ .
	\]
	When writing $\Psi$, we will mean the map of $\C$-coalgebras $\Bara A\to\Bara A'$, and the associated Maurer--Cartan element is
	\[
	\widetilde{\Psi}\coloneqq\proj_{A'}\Psi\in\MC(\homa(\Bara A,A'))\ .
	\]
	With this notation, we have
	\begin{align*}
		\MC(\homa(1,\widetilde{\Psi}))(f) =&\ \sum_{n\ge1}\frac{1}{n!}\homa(1,\widetilde{\Psi})(\mu_n^\vee\otimes f^{\otimes n})\\
		=& \sum_{n\ge1}\frac{1}{n!}\widetilde{\Psi}f^{\otimes n}\Delta_C^n\\
		=& \sum_{n\ge1}\proj_{A'}\Psi(1_\C\circ f)\Delta_C^n\\
		=&\ \proj_{A'}\Psi(1_\C\circ f)\Delta_C\\
		=&\ \proj_{A'}\Psi\widetilde{f}\ ,
	\end{align*}
	where $\mu_1^\vee = \id$. The other case is similar, and left to the reader.
\end{proof}

\begin{remark}
	If $\Psi$ is a strict morphism of $\P$-algebras, then the $\infty$-morphism $\MC(\homa(1,\Psi))$ is actually a strict morphism, and it is just given by $\homa(1,\Psi)$, the action of the bifunctor (\ref{eq:bifunctor}) on morphisms.
\end{remark}

Another powerful corollary of \cref{thm:fundamental thm of convolution Loo-algebras} is the following.

\begin{corollary}
	We place ourselves in the setting of \cref{thm:fundamental thm of convolution Loo-algebras}.
	\begin{enumerate}
		\item Suppose $\Psi,\Psi':A\rightsquigarrow A'$ are two $\infty_\alpha$-morphisms of $\P$-algebras. If $\Psi$ and $\Psi'$ are homotopic as Maurer--Cartan elements, then so are $\homa(1,\Psi)$ and $\homa(1,\Psi')$.
		\item Dually, suppose $\Phi,\Phi':C'\rightsquigarrow C$ are two $\infty_\alpha$-morphisms of $\C$-coalgebras. If $\Phi$ and $\Phi'$ are homotopic as Maurer--Cartan elements, then so are $\homa(\Phi,1)$ and $\homa(\Phi',1)$.
	\end{enumerate}
\end{corollary}

This is especially important when combined with the following result.

\begin{theorem}
	We place ourselves in the setting of \cref{thm:fundamental thm of convolution Loo-algebras}.
	\begin{enumerate}
		\item Two $\infty_\alpha$-morphisms $A\rightsquigarrow A'$ of $\P$-algebras are homotopic if, and only if they are gauge equivalent when seen as Maurer--Cartan elements of $\homa(\Bar_\alpha A,A')$.
		\item Two $\infty_\alpha$-morphisms $C'\rightsquigarrow C$ of $\C$-coalgebras are homotopic if, and only if they are gauge equivalent when seen as Maurer--Cartan elements of $\homa(C',\Cobar_\alpha C)$.
	\end{enumerate}
\end{theorem}

\begin{proof}
	We only give a sketch of the proof. The reader can find a detailed version in \cite[Thm. 2.4]{rnw18}.
	
	\medskip
	
	The proof relies on the fact that there is a natural homotopy equivalence
	\[
	\MC_\bullet(\homa(C,A))\simeq\MC(\homa(C,A\otimes\Omega_\bullet))
	\]
	induced by the inclusion of $\homa(C,A)\otimes\Omega_\bullet$ into $\homa(C,A\otimes\Omega_\bullet)$. Assuming that this is true for a second, the statement follows by looking at the $0$th homotopy groups of the two simplicial sets --- corresponding to Maurer--Cartan elements modulo gauges --- and noticing that a good path object for the $\C$-coalgebra $\Bar_\alpha A'$ is given by $\Bar_\alpha(A'\otimes\Omega_1)$, and a good path object for the $\P$-algebra $\Cobar_\alpha C$ is given by $\Cobar_\alpha(C)\otimes\Omega_1$.
	
	\medskip
	
	To prove that we have the aforementioned homotopy equivalence, one proceeds as in \cref{sect:alternative model for space of MC} on both $\homa(C,A)\otimes\Omega_\bullet$ and $\homa(C,A\otimes\Omega_\bullet)$, showing that their Maurer--Cartan spaces are homotopically equivalent to the Maurer--Cartan spaces of $\homa(C,A)\otimes C_\bullet$ and $\hom(C,A\otimes C_\bullet)$, both with the $\SLoo$-algebra structures obtained by homotopy transfer theorem applied on the contraction induced by Dupont's contraction. These two last algebras are easily checked to be isomorphic, concluding the proof.
\end{proof}

\subsection{Two bifunctors}

The results of \cref{subsection:compatibility with oo-morphisms} give us two natural ways to extend the bifunctor
\[
\homa:\Ccog^\mathrm{op}\times\Palg\longrightarrow\SLoo\text{-}\mathsf{alg}
\]
to categories of (co)algebras with $\infty$-morphisms. Namely, we define
\[
\homa_\ell:\infty_\alpha\text{-}\Ccog^\mathrm{op}\times\Palg\longrightarrow\infty\text{-}\SLoo\text{-}\mathsf{alg}
\]
by
\[
\homa_\ell(\Phi,f)\coloneqq\MC(\homa(\Phi,1))\homa(1,f)
\]
for a morphism $f$ of $\P$-algebras and an $\infty_\alpha$-morphism $\Phi$ of $\C$-coalgebras. Dually, we define
\[
\homa_r:\Ccog^\mathrm{op}\times\infty_\alpha\text{-}\Palg\longrightarrow\infty\text{-}\SLoo\text{-}\mathsf{alg}
\]
by
\[
\homa_r(g,\Psi)\coloneqq\homa(g,1)\MC(\homa(1,\Psi))
\]
for a morphism $g$ of $\C$-coalgebras and an $\infty_\alpha$-morphism $\Psi$ of $\P$-algebras.

\begin{theorem}[{\cite[Cor. 5.4]{rnw17}}]
	The two assignments $\homa_\ell$ and $\homa_r$ defined above are bifunctors.
\end{theorem}

\begin{proof}
	This follows immediately from \cref{lemma:hom of oo-morphism is pre or postcomposition}.
\end{proof}

After such a statement, it is very natural to ask if one can extend the original bifunctor $\homa$ to take $\infty_\alpha$-morphisms in both slots, and not just one, in such a way as to agree with $\homa_\ell$ and $\homa_r$ in the obvious subcategories. In \cref{subsect:failure to be a bifunctor}, we will show that this is in fact not possible. However, it is possible to do this \emph{up to homotopy}, as will be shown in \cref{subsect:bifunctor up to homotopy}.

\subsection{Failure of the extension of the two bifunctors}\label{subsect:failure to be a bifunctor}

We will work over a field of characteristic $0$ and in the non-symmetric setting (see \cref{subsect:other settings}). If there were such a bifunctor, then we would necessarily have
\[
\hom^\alpha_\ell(\Phi,1)\hom^\alpha_r(1,\Psi) = \hom^\alpha(\Phi,\Psi) = \hom^\alpha_r(1,\Psi)\hom^\alpha_\ell(\Phi,1)
\]
for any couple of $\infty_\alpha$-morphisms. We will give an explicit example where this is not the case. For reference, notice that the two composites are given by the diagrams
\begin{center}
	\begin{tikzpicture}
	\node (a) at (0,3){$D'$};
	\node (b) at (3,3){$\P(D)$};
	\node (c) at (6,3){$\P(\C(D))$};
	\node (d) at (0,0){$A'$};
	\node (e) at (3,0){$\P(A')$};
	\node (f) at (6,0){$\P(\C(A))$};
	\node (g) at (6,1.5){$(\P\circ\C)(n)\otimes_{\S_n}D^{\otimes n}$};
	
	\draw[->] (a) -- node[above]{$\Phi$} (b);
	\draw[->] (b) -- node[above]{$\P(\Delta_D)$} (c);
	\draw[->] (g) -- node[right]{$F$} (f);
	\draw[->] (f) -- node[above]{$\P(\Psi)$} (e);
	\draw[->] (e) -- node[above]{$\gamma_{A'}$} (d);
	\draw[dashed,->] (a) -- node[left]{$\hom^\alpha_\ell(\Phi,1)\hom^\alpha_r(1,\Psi)(\mu_n^\vee\otimes F)$} (d);
	\draw[->] (c) -- node[right]{$\mathrm{proj}_n$} (g);
	\end{tikzpicture}
\end{center}
and
\begin{center}
	\begin{tikzpicture}
	\node (a) at (0,3){$D'$};
	\node (b) at (3,3){$\C(D')$};
	\node (c) at (6,3){$\C(\P(D))$};
	\node (d) at (0,0){$A'$};
	\node (e) at (3,0){$\C(A)$};
	\node (f) at (6,0){$\C(\P(A))$};
	\node (g) at (6,1.5){$(\C\circ\P)(n)\otimes_{\S_n}D^{\otimes n}$};
	
	\draw[->] (a) -- node[above]{$\Delta_{D'}$} (b);
	\draw[->] (b) -- node[above]{$\C(\Phi)$} (c);
	\draw[->] (g) -- node[right]{$F$} (f);
	\draw[->] (f) -- node[above]{$\C(\gamma_A)$} (e);
	\draw[->] (e) -- node[above]{$\Psi$} (d);
	\draw[dashed,->] (a) -- node[left]{$\hom^\alpha_r(1,\Psi)\hom^\alpha_\ell(\Phi,1)(\mu_n^\vee\otimes F)$} (d);
	\draw[->] (c) -- node[right]{$\mathrm{proj}_n$} (g);
	\end{tikzpicture}
\end{center}
respectively, when applied to $\mu_n^\vee\otimes F\in\Bar_\iota(\hom(D,A))$.

\medskip

We will work with \emph{non-symmetric} associative algebras and (suspended) coassociative coalgebras. Since $\as(n)\cong\k$ for each $n\ge1$, for any associative algebra $A$ we will implicitly identify $\as(n)\otimes A^{\otimes n}$ with $A^{\otimes n}$ in some places, and similarly for coassociative coalgebras.

\subsubsection*{The families \texorpdfstring{$A^n$}{An} and \texorpdfstring{$H^n$}{Hn}}

We define $A^n$ for $n\ge1$ as the commutative algebra
\[
A^n\coloneqq\overline{\k[x,y]}
\]
seen as an associative algebra. The overline means that we take the augmentation ideal of $\k[x,y]$, i.e. that we only consider polynomials with no constant term. The degrees are $|x|=0$ and $|y|=1$ and the differential is given by $dy = x^n$. Notice that $y^2=0$. We have
\[
d(x^a) = 0\ ,\qquad d(x^ay) = x^{a+n}\ .
\]
It follows that, as a chain complex,
\[
H^n\coloneqq H_\bullet(A^n) \cong \bigoplus_{a=1}^{n-1}\k z_a\ ,
\]
where $z_a=[x^a]$ is the class of $x^a$. We have three maps
\begin{center}
	\begin{tikzpicture}
	\node (a) at (0,0){$A_n$};
	\node (b) at (2,0){$H_n$};
	
	\draw[->] (a)++(.3,.1)--node[above]{\mbox{\tiny{$p$}}}+(1.4,0);
	\draw[<-,yshift=-1mm] (a)++(.3,-.1)--node[below]{\mbox{\tiny{$i$}}}+(1.4,0);
	\draw[->] (a) to [out=-150,in=150,looseness=4] node[left]{\mbox{\tiny{$h$}}} (a);
	\end{tikzpicture}
\end{center}
given by
\begin{enumerate}
	\item $i(z_a) = x^a$.
	\item $p(x^a) = z_a$ for $a<n$ and zero on all other monomials.
	\item $h(x^a) = x^{a-n}y$ for $a\ge n$ and zero on all other monomials.
\end{enumerate}

\begin{lemma}
	The maps described above form a contraction.
\end{lemma}

\begin{proof}
	This is a straightforward computation.
\end{proof}

Now we apply the homotopy transfer theorem to obtain an $\as_\infty$-algebra structure on $H^n$ and $\infty$-morphisms between the two algebras.

\begin{lemma}\label{lemma:An is formal}
	The algebra $A^n$ is formal, and
	\[
	H^n\cong\frac{\overline{\k[z]}}{(z^n)}
	\]
	as associative (and $\as_\infty$-) algebras.
\end{lemma}

\begin{proof}
	The arity $2$ operation in $H^n$ is given by
	\[
	m_2(z_a,z_b) = p(i(z_a)i(z_b)) = p(x^{a+b}) = \begin{cases}z_{a+b}&\text{if }a+b<n,\\0&\text{otherwise}.
	\end{cases}
	\]
	Therefore, the underlying associative algebra is indeed
	\[
	H^n\cong\frac{\overline{\k[z]}}{(z^n)}
	\]
	(keeping in mind that $d=0$ on $H^n$, so that associativity is indeed satisfied). For the higher operations, we notice that
	\[
	h(i(z_a)i(z_b)) = \begin{cases}x^{a+b-n}y&\text{if }a+b\ge n,\\0&\text{otherwise}.
	\end{cases}
	\]
	it follows that if we multiply by any element of $A_n$ and then apply either $h$ or $p$, we always get $0$. Therefore, all higher operations are $0$, concluding the proof.
\end{proof}

\begin{lemma}
	The $\infty$-quasi-isomorphism $i_\infty$ of $\as_\infty$-algebras extending $i$ is given by $i_1=i$,
	\[
	i_2(z^a,z^b) = \begin{cases}x^{a+b-n}y&\text{if }a+b\ge n,\\0&\text{otherwise},
	\end{cases}
	\]
	and $i_n=0$ for all $n\ge3$.
\end{lemma}

\begin{proof}
	This is proven with computations analogous to the ones in the proof of \cref{lemma:An is formal}.
\end{proof}

Notice that the $\infty$-morphism $i_\infty$ is in fact an $\infty_\kappa$-morphism of associative algebras.

\subsubsection{A coalgebra and an \texorpdfstring{$\infty_\kappa$}{ook}-morphism}

A structure of conilpotent dg $\as^{\antishriek}$-coalgebra, that is a shifted coassociative coalgebra, on a graded vector space $V$ is the same thing as a square zero differential $d$ on $\as(V)$ such that
\[
d(V)\subseteq V\oplus V^{\otimes 2}.
\]
Let
\[
V\coloneqq\bigoplus_{i\ge1}\k v_i
\]
with $|v_i| = i$. We define
\[
d:\as(V)\longrightarrow\as(V)
\]
of degree $-1$ by
\[
d(v_i) = \sum_{j+k = i-1}(-1)^jv_j\otimes v_k\ .
\]

\begin{lemma}
	The map $d$ squares to $0$.
\end{lemma}

\begin{proof}
	This is a straightforward routine computation.
\end{proof}

Thus, we have an $\as^{\antishriek}$-coalgebra $V$. Notice that, since $d(V)\subseteq V^{\otimes 2}$, the underlying chain complex $V$ of the $\as^{\antishriek}$-coalgebra has the zero differential. We define
\[
\Phi:\as(V)\longrightarrow\as(V)
\]
by
\[
\Phi(v_n) = \sum_{k\ge1}\sum_{i_1+\cdots+i_k = n}v_{i_1}\otimes\cdots\otimes v_{i_k}\ .
\]

\begin{lemma}
	The map $\Phi$ commutes with the differential, and therefore defines an $\infty_\kappa$-morphism $\Phi:V\rightsquigarrow V$.
\end{lemma}

\begin{proof}
	We have to show that $\Phi:V\to\Omega_\kappa V$ satisfies the Maurer--Cartan equation
	\[
	\partial(\Phi)+\star_\kappa(\Phi) = 0\ .
	\]
	We have
	\small
	\begin{align*}
	\partial(\Phi)&(v_n) = d_{\Omega_\kappa V}\left(\sum_{k\ge1}\sum_{i_1+\cdots+i_k = n}v_{i_1}\otimes\cdots\otimes v_{i_k}\right)\\
	=&\ \sum_{k\ge1}\sum_{i_1+\cdots+i_k = n}\sum_{j=0}^{k-1}(-1)^{i_1+\cdots+i_j}v_{i_1}\otimes\cdots\otimes v_{i_j}\otimes\left(\sum_{\alpha+\beta = i_{j+1}-1}(-1)^\alpha v_\alpha\otimes v_\beta\right)\otimes v_{i_{j+2}}\otimes\cdots\otimes v_{i_k}\\
	=&\ \sum_{a,b\ge1}\sum_{\substack{x_1+\cdots+x_a+\\+y_1+\cdots+y_b = n-1}}(-1)^{x_1+\cdots+x_a}v_{x_1}\otimes\cdots\otimes v_{x_a}\otimes v_{y_1}\otimes\cdots\otimes v_{y_b}
	\end{align*}
	\normalsize
	where in the first line we used the fact that $d_V = 0$, and in the last line we substituted $a=j+1,b=k-j+1$, $x_s = i_s$ for $s\le a$, $x_a = \alpha$, $y_1 = \beta$, and $y_s = i_{j+s}$ for $s\ge2$. At the same time, we have
	\begin{align*}
	\star_\kappa(\Phi)(v_n) =&\ (\gamma_\as\circ 1)(\kappa\circ\Phi)\Delta_V(v_n)\\
	=&\ -\sum_{i+j=n-1}(-1)^i\Phi(v_i)\otimes\Phi(v_j)\\
	=&\ -\sum_{a,b\ge1}\sum_{\substack{x_1+\cdots+x_a+\\+y_1+\cdots+y_b = n-1}}(-1)^{x_1+\cdots+x_a}v_{x_1}\otimes\cdots\otimes v_{x_a}\otimes v_{y_1}\otimes\cdots\otimes v_{y_b}\ .
	\end{align*}
	Notice the sign in the second line: it comes from the signs in the definition of the differential $d_2$ in the cobar construction. This concludes the proof.
\end{proof}

\subsubsection*{The counterexample}

We now prove what claimed at the beginning of the present section by considering $\C=\as^{\antishriek}$, $\P=\as$, the canonical twisting morphism
\[
\kappa:\as^{\antishriek}\longrightarrow\as\ ,
\]
the associative algebras $A=A^2$, $A'=H^2$, the $\as^{\antishriek}$-coalgebras $D'=D=V$, and the $\infty$-morphisms $\Psi=i_\infty$ and $\Phi$ described above. We take the linear maps $f_1,f_2,f_3:V\to H^2$ such that $f_i(v_1) = z$ for $i=1,2,3$, $f_1(v_2) = z$, and $f_2(v_2) = f_3(v_2) = 0$. Notice that $f_2$ and $f_3$ have degree $-1$, while $f_1$ decomposes as the sum of a degree $-1$ map and a degree $-2$ map.

\medskip

We start by computing how $\hom^\kappa_\ell(\Phi,1)\hom^\kappa_r(1,i_\infty)(\mu^\vee_3\otimes F)$ acts on $v_4\in V$. We have
\begin{align*}
\Phi(v_4) =& \id\otimes v_4 + \mu_2\otimes(v_1\otimes v_3 + v_2\otimes v_2 + v_3\otimes v_1)\\
& + \mu_3\otimes(v_1\otimes v_1\otimes v_2 + v_1\otimes v_2\otimes v_1 + v_2\otimes v_1\otimes v_1) + \mu_4\otimes v_1\otimes v_1\otimes v_1\otimes v_1
\end{align*}
Since we will project on the part with only three copies of $V$, we don't care about the last term and will omit it in the following step. Notice that the coproduct of $V$ is explicitly given by
\small
\[
\Delta_V(v_n) = \id\otimes v_n + \sum_{i_1+i_2 = n-1}(-1)^{i_1}\susp_2^{-1}\mu_2^\vee\otimes v_{i_1}\otimes v_{i_2} - \sum_{j_1+j_2+j_3 = n-2}(-1)^{j_2}\susp_3^{-1}\mu_3^\vee\otimes v_{j_1}\otimes v_{j_2}\otimes v_{j_3} + \cdots\ ,
\]
\normalsize
where the dots indicate terms with at least $4$ copies of $V$. Applying this to the above, and then using $\mathrm{proj}_3$, we get
\begin{align*}
\mathrm{proj}_3\as(\Delta_V)&\Phi(v_4) =\\
=&\ -\mu_2\otimes\Big((\id\otimes v_1)\otimes(\susp_2^{-1}\mu_2^\vee\otimes v_1\otimes v_1) + (\susp_2^{-1}\mu_2^\vee\otimes v_1\otimes v_1)\otimes(\id\otimes v_1)\Big)\\
&\ + \mu_3\otimes\Big((\id\otimes v_1)\otimes(\id\otimes v_1)\otimes(\id\otimes v_2) + (\id\otimes v_1)\otimes(\id\otimes v_2)\otimes(\id\otimes v_1)\\
&\qquad\qquad + (\id\otimes v_2)\otimes(\id\otimes v_1)\otimes(\id\otimes v_1)\Big)
\end{align*}
Applying $F$ gives
\begin{align*}
F\mathrm{proj}_3\as(\Delta_V)&\Phi(v_4) =\\
=&\ \mu_2\otimes\Big((\id\otimes z)\otimes(\susp_2^{-1}\mu_2^\vee\otimes z\otimes z) - (\susp_2^{-1}\mu_2^\vee\otimes z\otimes z)\otimes(\id\otimes z)\Big)\\
&\ - \mu_3\otimes\big((\id\otimes z)\otimes(\id\otimes z)\otimes(\id\otimes z)\big)\ ,
\end{align*}
and thus
\begin{align*}
\as(i_\infty)F\mathrm{proj}_3\as(\Delta_V)&\Phi(v_4) = \mu_2\otimes(x\otimes y - y\otimes x) - \mu_3\otimes x\otimes x\otimes x\ .
\end{align*}
Finally, we have
\begin{align*}
\hom^\kappa_\ell(\Phi,1)\hom^\kappa_r(1,i_\infty)(\mu^\vee_3\otimes F)(v_4) =&\ \gamma_{A_2}\as(i_\infty)F\mathrm{proj}_3\as(\Delta_V) = -x^3.
\end{align*}
Now we look at the action of $\hom^\kappa_r(1,i_\infty)\hom^\kappa_\ell(\Phi,1)(\mu_3^\vee\otimes F)$ on $v_4$. We have
\[
\Delta_V(v_4) = \id\otimes v_4 + \susp_2^{-1}\mu_2^\vee\otimes(-v_1\otimes v_2 + v_2\otimes v_1)\ ,
\]
and thus
\begin{align*}
\mathrm{proj}_3\as^{\antishriek}(\Phi)&\Delta_V(v_4) =\\
=&\ \id\otimes\mu_3\otimes(v_1\otimes v_1\otimes v_2 + v_1\otimes v_2\otimes v_1 + v_2\otimes v_1\otimes v_1)\\
&\ + \susp_2^{-1}\mu_2^\vee\otimes\big(-(\id\otimes v_1)\otimes(\mu_2\otimes v_1\otimes v_1) + (\mu_2\otimes v_1\otimes v_1)\otimes(\id\otimes v_1)\big)\ .
\end{align*}
Applying $F$ we obtain
\begin{align*}
F\mathrm{proj}_3&\as^{\antishriek}(\Phi)\Delta_V(v_4) =\\
=&\ -\id\otimes\mu_3\otimes z\otimes z\otimes z + \susp_2^{-1}\mu_2^\vee\otimes\big(-(\id\otimes z)\otimes(\mu_2\otimes z\otimes z) + (\mu_2\otimes z\otimes z)\otimes(\id\otimes z)\big)\ ,
\end{align*}
and thus
\begin{align*}
\as^{\antishriek}(\gamma_{H^2})F\mathrm{proj}_3\as^{\antishriek}(\Phi)&\Delta_V(v_4) = 0
\end{align*}
since $z^2 = 0$ in $H^2$ and by \cref{lemma:An is formal}. Therefore,
\[
\hom^\kappa_r(1,i_\infty)\hom^\kappa_\ell(\Phi,1)(\mu_3^\vee\otimes F)(v_4) = 0\ ,
\]
showing that
\[
\hom^\kappa_r(1,i_\infty)\hom^\kappa_\ell(\Phi,1)\neq\hom^\kappa_\ell(\Phi,1)\hom^\kappa_r(1,i_\infty)
\]
as claimed. This implies the result we wanted.

\begin{theorem}
	There is no bifunctor
	\[
	\hom^\alpha:\infty_\alpha\text{-}\Ccog^\mathrm{op}\times\infty_\alpha\text{-}\Palg\longrightarrow\infty\text{-}\SLoo.
	\]
	that restricts to the functors
	\[
	\hom^\alpha_\ell:\infty_{\alpha}\text{-}\Ccog^\mathrm{op}\times\Palg\longrightarrow\infty\text{-}\SLoo\ 
	\]
	and
	\[
	\hom^\alpha_r:\Ccog^\mathrm{op}\times\infty_\alpha\text{-}\Palg\longrightarrow\infty\text{-}\SLoo
	\]
	defined above in the respective subcategories.
\end{theorem}

\begin{remark}
	The result is true in any characteristic in the non-symmetric case by the same counterexample as above, and in the symmetric case as well, by considering the same counterexample and tensoring the operads by the regular representation of the symmetric groups.
\end{remark}

\subsection{Compatibility with filtrations and quasi-isomorphisms}

We will show that convolution homotopy algebras are compatible with filtrations and quasi-isomorphisms. This is a very useful fact, because it will make it easy to apply the Dolgushev--Rogers theorem to obtain informations about the homotopy type of the homotopy Lie algebras obtained by convolution.

\medskip

Fix a cooperad $\C$, an operad $\P$, and a twisting morphism $\alpha:\C\to\P$. We begin with proving that the convolution homotopy algebra functor $\homa$ is compatible with (co)filtrations.

\begin{proposition}
	Let $C$ be a $\C$-coalgebra, and let $A$ be a $\P$-algebra.
	\begin{enumerate}
		\item Suppose that $C$ is cofiltered with cofiltration $\F^\bullet C$, in the sense of \cref{def:cofiltered coalgebras}. Then the sequence of subspaces
		\[
		\F_n\coloneqq\left\{f\in\homa(C,A)\mid \F^nC\subseteq\ker f\right\}
		\]
		makes the convolution homotopy algebra $\homa(C,A)$ into a complete $\SLoo$-algebra, which is proper if $\F^1C = 0$.
		\item Dually, suppose that $(A,\F_\bullet A)$ is a complete $\P$-algebra. Then the sequence
		\[
		\F_n\coloneqq\homa(C,\F_nA)
		\]
		makes the convolution homotopy algebra $\homa(C,A)$ into a complete $\SLoo$-algebra, which is proper if $A$ is.
	\end{enumerate}
\end{proposition}

\begin{proof}
	The first statement was proven in \cref{lemma:compatibility filtrations first case}. The second one is similar.
\end{proof}

Next, we prove that convolution homotopy algebras behave well with respect to $\infty_\alpha$-quasi-isomorphisms of (co)algebras.

\begin{proposition}\label{prop:oo-qi of (co)algebras induce oo-qi between convolution homotopy algebras}
	Let $C,C'$ be two conilpotent $\C$-coalgebras, and let $A,A'$ be two $\P$-algebras.
	\begin{enumerate}
		\item Suppose $\Phi:C'\rightsquigarrow C$ is an $\infty_\alpha$-quasi-isomorphism of $\C$-coalgebras. Then
		\[
		\homa(\Phi,1):\homa(C,A)\rightsquigarrow\homa(C',A)
		\]
		is an $\infty$-quasi-isomorphism of $\SLoo$-algebras. If $C,C'$ are cofiltered $\C$-coalgebras, and $\Phi$ is a cofiltered $\infty_\alpha$-quasi-isomorphism, in the sense that the restriction of its linear component to every level of the filtration is a quasi-isomorphism, then $\homa(\Phi,1)$ is a filtered $\infty$-quasi-isomorphism with respect to the induced filtration on the convolution $\SLoo$-algebras.
		\item Suppose $\Psi:A\rightsquigarrow A'$ is an $\infty_\alpha$-quasi-isomorphism of $\P$-algebras. Then
		\[
		\homa(1,\Psi):\homa(C,A)\rightsquigarrow\homa(C,A')
		\]
		is an $\infty$-quasi-isomorphism of $\SLoo$-algebras. If $\Psi$ is a filtered $\infty_\alpha$-quasi-isomorphism, then $\homa(1,\Psi)$ is a filtered $\infty$-quasi-isomorphism with respect to the induced filtration on the convolution $\SLoo$-algebras.
	\end{enumerate}
\end{proposition}

\begin{proof}
	We only prove the first statement, and leave the second one to the reader. We have that the first component
	\[
	\homa(\Phi,1)_1 = \phi_1^*
	\]
	is given by the pullback by $\phi_1$. Since $\phi_1$ is a quasi-isomorphism, and we are working over a field, this is also a quasi-isomorphism, proving that $\homa(\Phi,1)$ is an $\infty$-quasi-isomorphism. The same argument applies in the case $\Phi$ is a cofiltered $\infty_\alpha$-quasi-isomorphism to obtain a filtered $\infty$-morphism of $\SLoo$-algebras.
\end{proof}

\begin{corollary}\label{cor:what happens to MC.(conv. alg.) under filtered oo-qi}
	Let $C,C'$ be two conilpotent $\C$-coalgebras, and let $A,A'$ be two $\P$-algebras.
	\begin{enumerate}
		\item Let $\Phi:C'\rightsquigarrow C$ be an $\infty_\alpha$-morphism of $\C$-coalgebras, and suppose either
		\begin{itemize}
			\item $C',C$ are cofiltered $\C$-coalgebras, we endow the convolution homotopy algebras with the induced filtrations, and $\Phi$ is a cofiltered $\infty$-quasi-isomorphism, or
			\item $A$ is a filtered $\P$-algebra, and we endow the convolution homotopy algebras with the filtrations induced by the filtration on $A$.
		\end{itemize}
		Then
		\[
		\MC_\bullet(\homa(\Phi,1)):\MC_\bullet(\homa(C,A))\longrightarrow\MC_\bullet(\homa(C',A))
		\]
		is a weak equivalence of simplicial sets.
		\item Dually, let $\Psi:A\rightsquigarrow A'$ be an $\infty_\alpha$-morphism of $\P$-algebras, and suppose either
		\begin{itemize}
			\item $A,A'$ are filtered $\P$-algebras, we endow the convolution homotopy algebras with the induced filtrations, and $\Psi$ is a filtered $\infty$-quasi-isomorphism, or
			\item $C$ is a cofiltered $\C$-coalgebra, and we endow the convolution homotopy algebras with the filtrations induced by the cofiltration on $C$.
		\end{itemize}
		Then
		\[
		\MC_\bullet(\homa(1,\Psi)):\MC_\bullet(\homa(C,A))\longrightarrow\MC_\bullet(\homa(C,A'))
		\]
		is a weak equivalence of simplicial sets.
	\end{enumerate}
\end{corollary}

\begin{proof}
	This follows immediately from \cref{prop:oo-qi of (co)algebras induce oo-qi between convolution homotopy algebras} and the Dolgushev--Rogers theorem.
\end{proof}

\subsection{Extension of the bifunctor up to homotopy}\label{subsect:bifunctor up to homotopy}

To conclude the section, we prove that there exists an extension of the convolution homotopy algebra functor accepting $\infty$-morphisms in both slots, provided we accept to work only up to homotopy.

\begin{theorem}[{\cite{rnw18}}]\label{thm:compositions are homotopic}
	Let $\alpha:\C\to\P$ be a Koszul twisting morphism, let $\Phi:C'\rightsquigarrow C$ be an $\infty_\alpha$-morphism of $\C$-coalgebras, and let $\Psi:A\rightsquigarrow A'$ be an $\infty_\alpha$-morphism of $\P$-algebras. The two compositions
	\[
	\homa(\Phi,1)\homa(1,\Psi) \sim \homa(1,\Psi)\homa(\Phi,1)
	\]
	are homotopic.
\end{theorem}

\begin{proof}
	Denote by $R(A)\coloneqq\Cobar_\alpha\Bar_\alpha A$ the bar-cobar resolution of $A$, and similarly for $A'$. Since $\alpha$ is Koszul, the counit of the adjunction
	\[
	\epsilon_A:R(A)\longrightarrow A
	\]
	is a quasi-isomorphism by \cref{thm:bar-cobar resolution for algebras}. The rectification of the $\infty_\alpha$-morphism $\Psi$ is given by the strict morphism
	\[
	R(\Psi):\Cobar_\alpha\Bar_\alpha A\xrightarrow{\Cobar_\alpha\Psi}\Cobar_\alpha\Bar_\alpha A'\ .
	\]
	The proof is outlined by the following commutative diagram.
	
	\begin{center}
		\begin{tikzpicture}
		\node (a) at (0,0){$\homa(C,A)$};
		\node (b) at (6,6){$\homa(C',A)$};
		\node (c) at (8,0){$\homa(C',R(A'))$};
		\node (d) at (4,0){$\homa(C,R(A))$};
		\node (e) at (6,3.5){$\homa(C',R(A))$};
		\node (f) at (6,-3.5){$\homa(C,R(A'))$};
		\node (g) at (6,-6){$\homa(C,A')$};
		\node (h) at (12,0){$\homa(C',A')$};
		
		\draw[->,line join=round,decorate,decoration={zigzag,segment length=4,amplitude=.9,post=lineto,post length=2pt}] (a) to [out = 90, in = -180] node[above,sloped]{$\homa(\Phi,1)$} (b);
		\draw[->,line join=round,decorate,decoration={zigzag,segment length=4,amplitude=.9,post=lineto,post length=2pt}] (a) to [out = -90, in = 180] node[below,sloped]{$\homa(1,\Psi)$} (g);
		\draw[<-] (a) -- node[above]{\scriptsize $\homa(1,\epsilon_A)$} node[below]{\scriptsize filtered qi} (d);
		\draw[<-] (b) -- node[above,sloped]{\scriptsize $\homa(1,\epsilon_A)$} node[below,sloped]{\scriptsize filtered qi} (e);
		\draw[<-] (g) -- node[above,sloped]{\scriptsize $\homa(1,\epsilon_{A'})$} node[below,sloped]{\scriptsize filtered qi} (f);
		\draw[->,line join=round,decorate,decoration={zigzag,segment length=4,amplitude=.9,post=lineto,post length=2pt}] (d) to [out = 90, in = -155] node[above,sloped]{\small $\homa(\Phi,1)$} (e);
		\draw[->,line join=round,decorate,decoration={zigzag,segment length=4,amplitude=.9,post=lineto,post length=2pt}] (f) to [out = 25, in = -90] node[below,sloped]{\small $\homa(\Phi,1)$} (c);
		\draw[->,line join=round,decorate,decoration={zigzag,segment length=4,amplitude=.9,post=lineto,post length=2pt}] (b) to [out = 0, in = 90] node[above,sloped]{$\homa(1,\Psi)$} (h);
		\draw[->,line join=round,decorate,decoration={zigzag,segment length=4,amplitude=.9,post=lineto,post length=2pt}] (g) to [out = 0, in = -90] node[below,sloped]{$\homa(\Phi,1)$} (h);
		\draw[->] (d) to [out = -90, in = 155] node[below,sloped]{\small $\homa(1,R(\Psi))$} (f);
		\draw[->] (e) to [out = -25, in = 90] node[above,sloped]{\small $\homa(1,R(\Psi))$} (c);
		\draw[<-] (h) -- node[above]{\scriptsize $\homa(1,\epsilon_{A'})$} node[below]{\scriptsize filtered qi} (c);
		\end{tikzpicture}
	\end{center}
	The innermost square is commutative since $R(\Psi)$ is a strict morphism of $\P$-algebras, and the maps passing from the outer rim to the inner one are filtered quasi-isomorphisms. Notice that all squares are commutative, except for the outer one, which fails to be commutative at $\homa(C,A)$.
	
	\medskip
	
	Now consider the morphism of $\SLoo$-algebras
	\small
	\[
	\homi(\Bar_\iota\homa(C,A),\homa(C',A'))\xrightarrow{\homi(\Bar_\iota\homa(1,\epsilon_A),1)}\homi(\Bar_\iota\homa(C,R(A)),\homa(C',A'))\ .
	\]
	\normalsize
	It is a filtered quasi-isomorphism, and it is given on Maurer--Cartan elements by precomposition with $\homa(1,\epsilon_A)$. The two compositions
	\[
	\homa(\Phi,1)\homa(1,\Psi)\quad\text{and}\quad\homa(1,\Psi)\homa(\Phi,1)
	\]
	are naturally elements of $\homi(\Bar_\iota\homa(C,A),\homa(C',A'))$ and are mapped to the same elements, and thus, by the Dolgushev--Rogers theorem, they are homotopic.
\end{proof}

\begin{remark}
	The proof above supposes that we are filtering our convolution homotopy algebras with the filtration induced by a filtration on the $\C$-coalgebras --- usually the coradical filtration. If one filters them by a filtration induced by filtrations on the $\P$-algebras, then the exact same proof goes through with the sole difference that one has to rectify the $\infty_\alpha$-morphism $\Phi$ instead of $\Psi$.
\end{remark}

\begin{remark}
	The assumption that the twisting morphism is Koszul is necessary for the proof of \cref{thm:compositions are homotopic}. If one removes this assumption, then it is possible to find a counterexample to the conclusion of the result by taking e.g. the zero twisting morphism $\as^\vee\to\as$ and some explicit (co)algebras. For one such example, the reader is invited to consult \cite[Appendix A]{rnw18}.
\end{remark}

\section{Compatibility with the homotopy transfer theorem}

Another very powerful tool of homotopical algebra is given by the homotopy transfer theorem, of which we have seen a case in \cref{subsect:HTT}. In this section, we will start by giving a more general case of this theorem, where one does homotopy transfer for algebras over an operad of the form $\Cobar\C$, and not only for minimal models of a Koszul operad. Then we will explain and prove how the convolution $\L_\infty$-algebra functor is compatible with the homotopy transfer theorem.

\subsection{A generalized homotopy transfer theorem for algebras}\label{subsect:generalized HTT}

The homotopy transfer theorem holds in a slightly more general situation than the one presented in \cref{subsect:HTT} with the same exact formul{\ae}, as was proven in \cite[Thm. 1.5]{ber14transfer}.

\medskip

Let $\P$ be an operad of the form $\P=\Cobar\C$ for some reduced cooperad $\C$. In particular, $\P$ is a cofibrant operad. Let $A$ be a $\P$-algebra, and suppose that we have a contraction of chain complexes
\begin{center}
	\begin{tikzpicture}
	\node (a) at (0,0){$A$};
	\node (b) at (2,0){$B$};
	
	\draw[->] (a)++(.3,.1)--node[above]{\mbox{\tiny{$p$}}}+(1.4,0);
	\draw[<-,yshift=-1mm] (a)++(.3,-.1)--node[below]{\mbox{\tiny{$i$}}}+(1.4,0);
	\draw[->] (a) to [out=-150,in=150,looseness=4] node[left]{\mbox{\tiny{$h$}}} (a);
	\end{tikzpicture}
\end{center}
The structure of $\P$-algebra on $A$ is equivalent to a twisting morphism $\varphi_A\in\Tw(\C,\End_A)$. Define
\[
\varphi_B = \left(\C\xrightarrow{\Delta_{\C}^\mathrm{monadic}}\T^c\big(\overline{\C}\big)\xrightarrow{\T^c(s\varphi_A)}\Bar(\End_A)\xrightarrow{\vdl_B}\End_B\right),
\]
and
\[
i_\infty\coloneqq\left(\C\circ B\xrightarrow{\Delta_{\C}^\mathrm{monadic}\circ1_B}\T^c\big(\overline{\C}\big)\circ B\xrightarrow{\T^c(s\varphi_A)\circ1_B}\Bar(\End_A)\circ B\xrightarrow{\vdl_B^i\circ1_B}\End^B_A\circ B\longrightarrow A\right)\ .
\]
The same proof as for the case presented in \cref{subsect:HTT} gives the following result.

\begin{theorem}[Homotopy transfer theorem]\label{thm:generalized HTT}\index{Homotopy!transfer theorem}
	The map $\varphi_B:\C\to\End_B$ is a twisting morphism in $\Tw(\C,\End_B)$, and therefore defines a $\P$-algebra structure on $B$. The map $i_\infty$ defines an $\infty_\iota$-quasi-iso\-mor\-phism $B\rightsquigarrow A$ of $\P$-algebras, and the map $p$ can also be extended to an $\infty_\iota$-quasi-isomorphism $p_\infty$ such that $p_\infty i_\infty = 1_B$.
\end{theorem}

Let $\C,\C'$ be two cooperads, and let $f:\C'\to\C$ be a morphism of cooperads. Let $A$ be a $\Cobar\C$-algebra, and suppose we have a contraction
\begin{center}
	\begin{tikzpicture}
	\node (a) at (0,0){$A$};
	\node (b) at (2,0){$B$};
	
	\draw[->] (a)++(.3,.1)--node[above]{\mbox{\tiny{$p$}}}+(1.4,0);
	\draw[<-,yshift=-1mm] (a)++(.3,-.1)--node[below]{\mbox{\tiny{$i$}}}+(1.4,0);
	\draw[->] (a) to [out=-150,in=150,looseness=4] node[left]{\mbox{\tiny{$h$}}} (a);
	\end{tikzpicture}
\end{center}
We have two ways of putting a $\Cobar\C'$-algebra structure on $B$:
\begin{enumerate}
	\item The homotopy transfer theorem applied to the $\Cobar\C$-algebra $A$ gives a $\Cobar\C$-algebra structure on $B$. We obtain a $\Cobar\C'$-algebra structure on $B$ by pullback by the morphism of operads $\Cobar f:\Cobar\C'\to\Cobar\C$.
	\item We obtain a $\Cobar\C'$-algebra structure on $A$ by pullback by the morphism of operads $\Cobar f:\Cobar\C'\to\Cobar\C$. Then the homotopy transfer theorem applied to this algebra gives us a $\Cobar\C'$-algebra structure on $B$.
\end{enumerate}

\begin{proposition}\label{prop:HTT and pullbacks}
	The two $\Cobar\C'$-algebra structures thus obtained on $B$ are the same.
\end{proposition}

\begin{proof}
	The proof is given by the following diagram. The two algebra structures on $B$ are the two extremal paths.
	\begin{center}
		\begin{tikzpicture}
		\node (a) at (0,0){$\C'$};
		\node (b) at (3,0){$\T^c(\C')$};
		\node (c) at (6.5,0){$\T^c(s\End_A)$};
		\node (d) at (9,0){$\End_B$};
		\node (e) at (0,-2){$\C$};
		\node (f) at (3,-2){$\T^c(\C)$};
		
		\draw[->] (a) -- node[above]{$\Delta_{\C'}^{\mathrm{monadic}}$} (b);
		\draw[->] (b) -- node[above]{$\T^c(s\varphi_{(\Cobar f)^*A})$} (c);
		\draw[->] (c) -- node[above]{$\vdl_B$} (d);
		\draw[->] (a) -- node[left]{$f$} (e);
		\draw[->] (b) -- node[left]{$\T^c(f)$} (f);
		\draw[->] (e) -- node[above]{$\Delta_\C^{\mathrm{monadic}}$} (f);
		\draw[->] (f) to [out = 0, in = -110] node[below right]{$\T^c(s\varphi_A)$} (c);
		\end{tikzpicture}
	\end{center}
	The diagram is obviously commutative, concluding the proof.
\end{proof}

\subsection{Compatibility between convolution algebras and homotopy transfer}\label{subsect:compatibility HTT algebra side}

We can now prove that taking convolution algebras is compatible with the homotopy transfer theorem.

\medskip

Let $\C$ be a cooperad, and denote by $\iota:\C\to\Cobar\C$ the canonical twisting morphism. Let $A$ be a $\Cobar\C$-algebra, let $B$ be a chain complex and suppose we have a contraction
\begin{center}
	\begin{tikzpicture}
	\node (a) at (0,0){$A$};
	\node (b) at (2,0){$B$};
	
	\draw[->] (a)++(.3,.1)--node[above]{\mbox{\tiny{$p$}}}+(1.4,0);
	\draw[<-,yshift=-1mm] (a)++(.3,-.1)--node[below]{\mbox{\tiny{$i$}}}+(1.4,0);
	\draw[->] (a) to [out=-150,in=150,looseness=4] node[left]{\mbox{\tiny{$h$}}} (a);
	\end{tikzpicture}
\end{center}
Let $C$ be a $\C$-coalgebra. Then, we have two ways to endow the chain complex $\hom(C,B)$ with a $\SLoo$-algebra structure.
\begin{enumerate}
	\item We consider the $\SLoo$-algebra $\hom^\iota(C,A)$ . The contraction above induces a contraction
	\begin{center}
		\begin{tikzpicture}
		\node (a) at (0,0){$\hom^\iota(C,A)$};
		\node (b) at (3.5,0){$\hom(C,B)$};
		
		\draw[->] (a)++(1,.1)--node[above]{\mbox{\tiny{$p_*$}}}+(1.4,0);
		\draw[<-,yshift=-1mm] (a)++(1,-.1)--node[below]{\mbox{\tiny{$i_*$}}}+(1.4,0);
		\draw[->] (a) to [out=-160,in=160,looseness=4] node[left]{\mbox{\tiny{$h_*$}}} (a);
		\end{tikzpicture}
	\end{center}
	and thus the homotopy transfer theorem gives a $\SLoo$-algebra structure on $\hom(C,B)$. We will denote this algebra by $\hom^{HTT}(C,B)$.
	\item The generalized homotopy transfer theorem gives a $\Cobar\C$-algebra structure on $B$. Then we take the convolution $\SLoo$-algebra $\hom^\iota(C,B)$.
\end{enumerate}

\begin{theorem}\label{thm:compatibility with htt}
	The two $\SLoo$-algebra structures on $\hom(C,B)$ are the same. Moreover, we have that
	\[
	(i_*)_\infty = \hom_r^\iota(1,i_\infty),\qquad\text{and}\qquad(p_*)_\infty = \hom_r^\iota(1,p_\infty)\ .
	\]
\end{theorem}

The proof of this result is postponed to later in this section.

\begin{corollary}
	Let $f:\D\to\C$ be a morphism of cooperads, let
	\[
	\alpha\coloneqq f^*\iota = \left(\D\stackrel{f}{\longrightarrow}\C\stackrel{\iota}{\longrightarrow}\Cobar\C\right),
	\]
	and now suppose that $C$ is a conilpotent $\D$-coalgebra. Then the same constructions as for \cref{thm:compatibility with htt} can be done to obtain two $\SLoo$-algebra structures on $\hom(C,B)$. Once again, the two structures are the same.
\end{corollary}

\begin{proof}
	By \cref{lemma:naturality of convolution algebras wrt compositions}(\ref{pt:1 naturality convolution}) and \cref{thm:compatibility with htt}, we have
	\[
	\hom^\alpha(C,B) = \hom^\iota(f_*C,B) = \hom^{HTT}(f_*C,B)\ .
	\]
	Then for $s^{-1}\mu_n^\vee\otimes F\in s^{-1}\com^\vee\otimes_{\S_n}\hom(C,B)^{\otimes n}$ we notice that
	\begin{align*}
	\gamma_{\hom^{HTT}(f_*D,B)}(s^{-1}\mu_n^\vee\otimes F) =&\ \vdl_{\hom(D,B)}\T(\varphi_{\hom^\iota(f_*D,A)})\Delta_{\com^\vee}^{\mathrm{monadic}}(\mu_n^\vee)\\
	=&\ \vdl_{\hom(D,B)}\T(\varphi_{\hom^\alpha(D,A)})\Delta_{\com^\vee}^{\mathrm{monadic}}(\mu_n^\vee)\\
	=&\ \gamma_{\hom^{HTT}(D,B)}(s^{-1}\mu_n^\vee\otimes F)\ ,
	\end{align*}
	where in the second line we used \cref{lemma:naturality of convolution algebras wrt compositions} once again. Therefore, we have
	\[
	\hom^\alpha(C,B) = \hom^{HTT}(C,B)\ ,
	\]
	concluding the proof.
\end{proof}

We will now prove \cref{thm:compatibility with htt}. Let's introduce some notation. Let $\C$ be a cooperad, and let $C$ be a conilpotent $\C$-coalgebra. For every reduced rooted tree $\tau\in\rRT$ we define
\[
\Delta_C^\tau:C\longrightarrow\T^c(\C)\circ C
\]
recursively as follows. If $\tau=c_n$ is the $n$-corolla, with $n\ge2$, then
\[
\Delta_C^{c_n}\coloneqq\Delta_C^n\ .
\]
Else, we have $\tau=c_k\circ(\tau_1,\ldots,\tau_k)$, where $k\ge2$, the $\tau_i$ are allowed to be the empty tree, and we define
\[
\Delta_C^\tau\coloneqq\big(1_\C\circ(\Delta_C^{\tau_1},\ldots,\Delta_C^{\tau_1})\big)\Delta_C^k\ ,
\]
where $\Delta_C^\emptyset\coloneqq1_C$. Notice that $\id\in\C(1)$ will never appear in the image of such an operator.

\begin{lemma}\label{lemma:monadic decomposition map}
	For any $n\ge2$, we have
	\[
	\big(\Delta_\C^\mathrm{mon}\circ1_C\big)\Delta_C^n = \sum_{\tau\in RT_n}\Delta_C^\tau\ ,
	\]
	where the monadic decomposition map $\Delta_\C^\mathrm{mon}$ was defined at page \pageref{page:monadic decomposition map}.
\end{lemma}

\begin{proof}
	The proof is done by induction on $n$. For $n=2$, the statement is trivial, since the cooperad is supposed to be reduced, so that we have that $\Delta_\C^\mathrm{monadic}$ is the identity on $\C(2)$. For $n>2$, we have
	\begin{align*}
	\sum_{\tau\in RT_n}\Delta_C^\tau =&\ \sum_{\substack{k\ge2,\ n_1+\cdots+n_k=n\\\tau_i\in RT_{n_i}\ \forall 1\le i\le k}}\Delta_C^{c_k\circ(\tau_1,\ldots,\tau_k)}\\
	=&\ \sum_{\substack{k\ge2,\ n_1+\cdots+n_k=n\\\tau_i\in RT_{n_i}\ \forall 1\le i\le k}}\big(1_\C\circ(\Delta_C^{\tau_1},\ldots,\Delta_C^{\tau_1})\big)\Delta_C^k\\
	=&\ \sum_{\substack{k\ge2\\n_1+\cdots+n_k=n}}\left(1_\C\circ\left(\sum_{\tau_1\in RT_{n_1}}\Delta_C^{\tau_1},\ldots,\sum_{\tau_k\in RT_{n_k}}\Delta_C^{\tau_k}\right)\right)\Delta_C^k\\
	=&\ \sum_{\substack{k\ge2\\n_1+\cdots+n_k=n}}\left(1_\C\circ\left(\big(\Delta_\C^\mathrm{monadic}\circ1_C\big)\Delta_C^{n_1},\ldots,\big(\Delta_\C^\mathrm{monadic}\circ1_C\big)\Delta_C^{n_k}\right)\right)\Delta_C^k\\
	=&\ \left(1_\C\circ\Delta_\C^\mathrm{monadic}\circ1_C\right)\sum_{\substack{k\ge2\\n_1+\cdots+n_k=n}}\big(1_\C\circ(\Delta_C^{n_1},\ldots,\Delta_C^{n_k})\big)\Delta_C^k\\
	=&\ \left(1_\C\circ\Delta_\C^\mathrm{monadic}\circ1_C\right)(\overline{\Delta}_\C\circ1_C)\Delta_C^n\\
	=&\ \big(\Delta_\C^\mathrm{monadic}\circ1_C\big)\Delta_C^n\ ,
	\end{align*}
	where we consider the empty tree as a rooted tree in order for the first equality to hold, and where in the third line we used the induction hypothesis.
\end{proof}

Recall that the Van der Laan map
\[
\vdl_B:\T(s\End_A)\longrightarrow\End_B
\]
is given by
\[
\vdl_B(\tau(f)) = p\tau^h(f)i^{\otimes n},
\]
cf. \cref{subsect:HTT}, and similarly for the contraction from $\hom(C,A)$ to $\hom(C,B)$.

\medskip

Every chain complex $V$ is an $\End_V$-algebra in a canonical way, we will denote by
\[
\widetilde{\gamma}_V:\End_V\circ V\longrightarrow V
\]
its composition map.

\begin{proof}[Proof of \cref{thm:compatibility with htt}]
	We will prove our claims by explicitly comparing the two $\SLoo$-algebra structures on $\hom(C,B)$. We will denote by
	\[
	\varphi_A\in\Tw(\C,\End_A)\ ,\qquad\rho_A:\Cobar\C\longrightarrow\End_A\ ,\qquad\text{and}\qquad\gamma_A:\Cobar\C\circ A\longrightarrow A
	\]
	the $\Cobar\C$-algebra structure of $A$, seen in three equivalent ways, and similarly for the other algebras. Notice that we have
	\[
	\varphi_A = \rho_A\iota\ ,\qquad\text{and}\qquad\gamma_A = \widetilde{\gamma}_A(\rho_A\circ1_A)\ ,
	\]
	as well as
	\[
	\gamma_A = \widetilde{\gamma}_A(\T(s\varphi_A)\circ1_A)\ ,
	\]
	where we used $\Cobar\C = \T(s^{-1}\C)$.
	
	\medskip
	
	Fix $n\ge2$ and $f_1,\ldots,f_n\in\hom(C,B)$. As usual, denote $F\coloneqq f_1\otimes\cdots\otimes f_n$.
	
	\medskip
	
	We begin by making explicit the structure obtained by taking the convolution algebra between $C$ and $B$ with the $\Cobar\C$-algebra structure given by the homotopy transfer theorem. The algebraic structure of $B$ is given by the twisting morphism
	\[
	\varphi_B = \vdl_B\T(s\varphi_A)\Delta_\C^\mathrm{mon},
	\]
	and thus
	\[
	\gamma_B = \widetilde{\gamma}_B\Big(\T(s\vdl_B\T(s\varphi_A)\Delta_\C^\mathrm{mon})\circ1_B\Big)\ .
	\]
	For the structure on $\hom(C,B)$, we have
	\begin{align*}
		\varphi_{\hom^\iota(C,B)}(\mu_n^\vee)(F) =&\ \gamma_B(\iota\otimes F)\Delta^n_C\\
		=&\ \widetilde{\gamma}_B\Big(\T(s\vdl_B\T(s\varphi_A)\Delta_\C^\mathrm{mon})\circ1_B\Big)(\iota\otimes F)\Delta^n_C\\
		=&\ \widetilde{\gamma}_B\Big(s\vdl_B\T(s\varphi_A)\Delta_\C^\mathrm{mon}\circ1_B\Big)(\iota\otimes F)\Delta^n_C\ .
	\end{align*}
	In the third line, we used the fact that the image of $\iota:\C\to\Cobar\C$ is $s^{-1}\overline{\C}$. Using that $\iota$ is essentially just given by desuspension, we have
	\[
	\Delta_\C^\mathrm{mon}s\iota = \Delta_\C^\mathrm{mon}\ ,
	\]
	and thus
	\begin{align*}
	\varphi_{\hom^\iota(C,B)}(\mu_n^\vee)(F) =&\ \widetilde{\gamma}_B\Big(\vdl_B\T(s\varphi_A)\circ1_B\Big)(1_\C\otimes F)(\Delta_\C^\mathrm{mon}\circ1_C)\Delta^n_C\\
	=&\ \sum_{\tau\in\RT_n}\widetilde{\gamma}_B\Big(\vdl_B\T(s\varphi_A)\circ1_B\Big)(1_\C\otimes F)\Delta^\tau_C\\
	=&\ \sum_{\tau\in\RT_n}\widetilde{\gamma}_B\Big(p\tau^h(\varphi_A)i^{\otimes n}\circ1_B\Big)(1_\C\otimes F)\Delta^\tau_C\\
	=&\ \sum_{\tau\in\RT_n}p\widetilde{\gamma}_A(\tau^h(\varphi_A)\otimes iF)\Delta^\tau_C\tag{A}\label{eq:A}
	\end{align*}
	where in the second line we used \cref{lemma:monadic decomposition map}.
	
	\medskip
	
	The next step is to make as explicit as possible the other $\SLoo$-algebra structure on $\hom(C,B)$, obtained by homotopy transfer theorem between the two hom spaces. We have
	\[
	\varphi_{\hom^{HTT}(C,B)} = \vdl_{\hom(C,B)}\T(s\varphi_{\hom^\iota(C,A)})\Delta_{\com^\vee}^\mathrm{mon}\ .
	\]
	Therefore, we compute
	\begin{align*}
		\varphi_{\hom^{HTT}(C,B)}(\mu_n^\vee)(F) =&\ \left(\vdl_{\hom(C,B)}\T(s\varphi_{\hom^\iota(C,A)})\Delta_{\com^\vee}^\mathrm{mon}(\mu_n^\vee)\right)(F)\\
		=&\ \left(\vdl_{\hom(C,B)}\T(s\varphi_{\hom^\iota(C,A)})\sum_{\tau\in\RT_n}\tau(v\mapsto\mu_{|v|}^\vee)\right)(F)\\
		=&\ \sum_{\tau\in\RT_n}\left(\vdl_{\hom(C,B)}\tau\left(v\mapsto s\gamma_A(\iota\otimes-)\Delta_\C^{|v|}\right)\right)(F)\\
		=&\ \sum_{\tau\in\RT_n}\left(p_*\tau^{h_*}\left(v\mapsto \gamma_A(\iota\otimes-)\Delta_\C^{|v|}\right)i_*^{\otimes n}\right)(F)\\
		=&\ \sum_{\tau\in\RT_n}p\left(\tau^{h_*}\left(v\mapsto \gamma_A(\iota\otimes-)\Delta_\C^{|v|}\right)\right)(iF)\\
		=&\ \sum_{\tau\in\RT_n}p\left(\tau^{h_*}\left(v\mapsto \widetilde{\gamma}_A(\varphi_A\otimes-)\Delta_\C^{|v|}\right)\right)(iF)\ ,\tag{B}\label{eq:B}
	\end{align*}
	where in the last line we used
	\[
	\gamma_A(\iota\circ1_A) = \widetilde{\gamma}_A(\rho_A\circ1_A)(\iota\circ1_A) = \widetilde{\gamma}_A(\varphi_A\circ1_A)\ .
	\]
	The last step is proving that $(\ref{eq:A})=(\ref{eq:B})$. We will use induction to prove that, for any $n\ge2$ and $\tau\in\RT_n$, we have
	\begin{equation}
	\widetilde{\gamma}_A(\tau^h(\varphi_A)\otimes iF)\Delta^\tau_C = \left(\tau^{h_*}\left(v\mapsto \widetilde{\gamma}_A(\varphi_A\otimes-)\Delta_\C^{|v|}\right)\right)(iF)\ ,\label{eq:comparison on trees}
	\end{equation}
	which implies the claim. If $n=2$, the only possible tree is the $2$-corolla, and one immediately sees that both sides of (\ref{eq:comparison on trees}) are equal to $p\widetilde{\gamma}_A(\varphi_A\otimes iF)\Delta_C^2$. Similarly, for all the corollas it is straightforward to see that the identity holds. If $\tau\in\RT_n$ is a composite tree, we can write it as
	\[
	\tau = c_k\circ(\tau_1,\ldots,\tau_k)\ .
	\]
	Then we have
	\begin{align*}
		\bigg(\tau^{h_*}\bigg(v\mapsto& \widetilde{\gamma}_A(\varphi_A\otimes-)\Delta_\C^{|v|}\bigg)\bigg)(iF) =\\
		&= \sum_{S_1\sqcup\cdots\sqcup S_k=[n]}\widetilde{\gamma}_A\left(\varphi_A\circ\bigotimes_{j=1}^kh\left(\tau_j^{h_*}\left(v\mapsto \widetilde{\gamma}_A(\varphi_A\otimes-)\Delta_\C^{|v|}\right)\right)(iF^{S_j})\right)\Delta_C^k\\
		&= \sum_{S_1\sqcup\cdots\sqcup S_k=[n]}\widetilde{\gamma}_A\left(\varphi_A\circ\bigotimes_{j=1}^kh\widetilde{\gamma}_A(\tau_j^h(\varphi_A)\otimes iF)\Delta^{\tau_j}_C\right)\Delta_C^k\\
		&= \widetilde{\gamma}_A(1_{\End_A}\circ\widetilde{\gamma}_A)\left(\sum_{S_1\sqcup\cdots\sqcup S_k=[n]}\varphi_A\circ\bigotimes_{j=1}^kh(\tau_j^h(\varphi_A)\otimes iF)\right)(1_\C\circ(\Delta^{\tau_1}_C),\ldots,\Delta^{\tau_k}_C)\Delta_C^k\\
		&= \widetilde{\gamma}_A(\tau^h(\varphi_A)\otimes iF)\Delta_C^\tau\ ,
	\end{align*}
	as desired.
	
	\medskip
	
	The fact that
	\[
	(i_*)_\infty = \homi_r(1,i_\infty)
	\]
	is proven in a completely analogous way, and then we have
	\[
	\homi_r(1,p_\infty)\homi_r(1,i_\infty) = \homi_r(1,p_\infty i_\infty) = 1_{\hom(C,B)}\ ,
	\]
	which shows that we can take $(p_*)_\infty = \homi_r(1,p_\infty)$.
\end{proof}

\subsection{Compatibility on the coalgebra side}

We want to do the same thing as in \cref{subsect:compatibility HTT algebra side}, but this time on the coalgebra side. While we believe that one should be able to write down a sensible version of the homotopy transfer theorem for coalgebras --- probably by working on coalgebras over \emph{operads}, see \cite[Sect. 5.2.15]{LodayVallette} --- for simplicity we will only prove the dual version of the result we think to be true, working with tensor products of algebras.

\medskip

Let $\P$ be an operad which is finite dimensional in every arity, and as always let $\pi:\Bar\P\to\P$ be the canonical twisting morphism. Denote by $\C\coloneqq\P^\vee$ the dual cooperad of $\P$. Suppose that we are given a $\P$-algebra $X$, a $\Cobar\C$-algebra $A$, and a contraction
\begin{center}
	\begin{tikzpicture}
	\node (a) at (0,0){$A$};
	\node (b) at (2,0){$B$};
	
	\draw[->] (a)++(.3,.1)--node[above]{\mbox{\tiny{$p$}}}+(1.4,0);
	\draw[<-,yshift=-1mm] (a)++(.3,-.1)--node[below]{\mbox{\tiny{$i$}}}+(1.4,0);
	\draw[->] (a) to [out=-150,in=150,looseness=4] node[left]{\mbox{\tiny{$h$}}} (a);
	\end{tikzpicture}
\end{center}
from $A$ to $B$. There are two ways to endow the chain complex $X\otimes B$ with an $\SLoo$-algebra structure.
\begin{enumerate}
	\item We consider the $\SLoo$-algebra $A\otimes^\iota X$ . The contraction above induces a contraction
	\begin{center}
		\begin{tikzpicture}
		\node (a) at (0,0){$X\otimes^\pi A$};
		\node (b) at (2.8,0){$X\otimes B$};
		
		\draw[->] (a)++(0.7,.1)--node[above]{\mbox{\tiny{$1_X\otimes p$}}}+(1.4,0);
		\draw[<-,yshift=-1mm] (a)++(0.7,-.1)--node[below]{\mbox{\tiny{$1_X\otimes i$}}}+(1.4,0);
		\draw[->] (a) to [out=-160,in=160,looseness=4] node[left]{\mbox{\tiny{$1_X\otimes h$}}} (a);
		\end{tikzpicture}
	\end{center}
	and thus the homotopy transfer theorem gives a $\SLoo$-algebra structure on $B\otimes X$.
	\item The generalized homotopy transfer theorem gives a $\Cobar\C$-algebra structure on $B$. Then we take $B\otimes^\pi X$.
\end{enumerate}

\begin{theorem}\label{thm:compatibility htt and convolution algebras on coalgebra side}
	The two $\SLoo$-algebra structures on $X\otimes B$ are equal. Moreover,
	\[
	(1\otimes i)_\infty = 1\otimes^\pi(i_\infty)\qquad\text{and}\qquad(1\otimes p)_\infty = 1\otimes^\pi(p_\infty)\ .
	\]
\end{theorem}

\begin{remark}
	This result of course looks very similar to \cref{thm:compatibility with htt}, being its version for the ``coalgebra side". One should think of the proof as a dual version of \cref{thm:compatibility with htt}. The above theorem is a direct generalization of \cite[Thm. 5.1]{rn17tensor}.
\end{remark}

\begin{proof}
	We compare the two structures explicitly. In order to do so, let
	\[
	\manintensor_\pi:\SLoo\longrightarrow\P\otimes\Cobar\C
	\]
	be the map defined in \cref{thm:convolution Loo tensor products}. Explicitly, it is given by
	\[
	\manintensor_\pi(s^{-1}\mu_n^\vee) = \sum_{i\in I(n)}(-1)^{c_i}p_i\otimes s^{-1}c_i\ ,
	\]
	where $\{p_i\}_{i\in I(n)}$ is a basis of $\P(n)$, and $c_i\coloneqq p_i^\vee$ is the dual basis. Notice that $(-1)^{c_i}s^{-1}c_i = (sp_i)^\vee$. We encode the $\P$-algebra structure of $X$ by the morphism of operads
	\[
	\rho_X:\P\longrightarrow\End_X\ ,
	\]
	and the $\Cobar\C$-algebra structures of $A$, $B$ (obtained by homotopy transfer theorem), and $X\otimes A$ by the twisting morphisms
	\[
	\varphi_A:\C\longrightarrow\End_A\ ,
	\]
	and similarly for the other algebras. Given two $\S$-modules $M$ and $N$, the map
	\[
	\Phi:\T(M\otimes N)\longrightarrow\T(M)\otimes\T(N)
	\]
	is given by sending a tree with vertices indexed by pure tensors in $M\otimes N$ to the tensor product of two copies of the underlying tree, the first one with node indexed by the respective elements of $M$, the second one by the respective elements of $N$, all multiplied by the appropriate Koszul sign.
	
	\medskip	
	
	The first $\SLoo$-algebra structure on $X\otimes B$ is given by
	\begin{align*}
		\ell_n =&\ \vdl_{X\otimes B}\T(s\varphi_{X\otimes A})\Delta_{\com^\vee}^\mathrm{mon}(\mu_n^\vee)\\
		=&\ \left(\gamma_{\End_X}\otimes\vdl_B\right)\Phi\T(s\varphi_{X\otimes A})\Delta_{\com^\vee}^\mathrm{mon}(\mu_n^\vee)\\
		=&\ \left(\gamma_{\End_X}\otimes\vdl_B\right)\Phi\T\left((\rho_X\otimes s\varphi_A)s\manintensor_\pi s^{-1}\right)\Delta_{\com^\vee}^\mathrm{mon}(\mu_n^\vee)\\
		=&\ \left(\gamma_{\End_X}\otimes\vdl_B\right)\left(\T(\rho_X)\otimes\T(s\varphi_A)\right)\Phi\T(s\manintensor_\pi s^{-1})\Delta_{\com^\vee}^\mathrm{mon}(\mu_n^\vee)\\
		=&\ \left(\rho_X\otimes\vdl_B\T(s\varphi_A)\right)\left(\gamma_\P\otimes1_{\T(\C)}\right)\Phi\T(s\manintensor_\pi s^{-1})\Delta_{\com^\vee}^\mathrm{mon}(\mu_n^\vee)
	\end{align*}
	where in the third line we see $s\manintensor_\pi s^{-1}$ as a map $\com^\vee\to\P\otimes\C$, given by
	\[
	s\manintensor_\pi s^{-1}(\mu_n^\vee) = \sum_{i\in I(n)}p_i\otimes c_i\ ,
	\]
	and in the last line we used the fact that
	\[
	\gamma_{\End_X}\T(\rho_X) = \rho_X\gamma_\P\ .
	\]
	For the second structure, we have
	\begin{align*}
		\ell_n =&\ \left(\rho_X\otimes\varphi_B\right)s\manintensor_\pi s^{-1}(\mu_n^\vee)\\
		=&\ \left(\rho_X\otimes\vdl_B\T(s\varphi_A)\Delta_\C^\mathrm{mon}\right)s\manintensor_\pi s^{-1}(\mu_n^\vee)\\
		=&\ \left(\rho_X\otimes\vdl_B\T(s\varphi_A)\right)\left(1_\P\otimes\Delta_\C^\mathrm{mon}\right)s\manintensor_\pi s^{-1}(\mu_n^\vee)\ .
	\end{align*}
	Therefore, to conclude we need to show that
	\[
	\left(\gamma_\P\otimes1_{\T(\C)}\right)\Phi\T(s\manintensor_\pi s^{-1})\Delta_{\com^\vee}^\mathrm{mon}(\mu_n^\vee) = \left(1_\P\otimes\Delta_\C^\mathrm{mon}\right)s\manintensor_\pi s^{-1}(\mu_n^\vee)\ .
	\]
	Using the notation introduced at the end of \cref{sect:rooted trees}, we have that
	\[
	\Delta_{\com^\vee}^\mathrm{mon}(\mu_n^\vee) = \sum_{\tau\in\RT_n}\tau(v\mapsto\mu_{|v|}^\vee)\ .
	\]
	Therefore, the left-hand side above is equal to
	\[
	\sum_{\tau\in\RT_n}(-1)^\epsilon\sum_{i_v\in I(|v|)\text{ for }v\in V_\tau}\gamma_\P(\tau(v\mapsto p_{i_v}))\otimes\tau(v\mapsto c_{i_v})\ ,
	\]
	where $\epsilon$ is the Koszul sign coming from $\Phi$. At the same time, the right-hand side is given by
	\[
	\sum_{i\in I(n)}p_i\otimes\Delta_\C^\mathrm{mon}(c_i)\ .
	\]
	Both are expression for the map
	\[
	\gamma_\P:\T(\P)\longrightarrow\P
	\]
	seen as an element of $\P\otimes\T(\P)^\vee$, and thus they are equal.
	
	\medskip
	
	Once again, the statement on the $\infty$-morphisms is proven in an analogous way.
\end{proof}

%% file: RepresentationDeformationGroupoid.tex
\chapter{Representation of the deformation \texorpdfstring{$\infty$}{infinity}-groupoid}\label{ch:representing the deformation oo-groupoid}

We present here a first important application of the theory developed in \cref{chapter:convolution homotopy algebras}. It was done in \cite{rn17cosimplicial}, and it was the main goal of the author when developing the material of \cite{rn17tensor}. In a sense, it should be considered the central result of the present work. Some of the results given here are newer, and are extracted from \cite{rnv18}.

\medskip

The idea is the following. A very important object in various areas of mathematics, e.g. deformation theory and rational homotopy theory, is given by the Maurer--Cartan space $\MC_\bullet(\g)$ of a homotopy Lie algebra. However, since the Sullivan algebra $\Omega_\bullet$ is infinite dimensional, this object is always ``really big". One has Getzler's $\infty$-groupoid $\gamma_\bullet(\g)$, which is much smaller, but this object is somewhat complicated --- at least in its original presentation --- as Dupont's contraction map $h_\bullet$ is. The new idea to obtain a ``nice" model for the space of Maurer--Cartan elements\footnote{Namely, we want Kan complex which is homotopically equivalent to $\MC_\bullet(\g)$ in a natural way.} is the following. Start with Dupont's contraction
\begin{center}
	\begin{tikzpicture}
	\node (a) at (0,0){$\Omega_\bullet$};
	\node (b) at (2,0){$C_\bullet$};
	
	\draw[->] (a)++(.3,.1)--node[above]{\mbox{\tiny{$p_\bullet$}}}+(1.4,0);
	\draw[<-,yshift=-1mm] (a)++(.3,-.1)--node[below]{\mbox{\tiny{$i_\bullet$}}}+(1.4,0);
	\draw[->] (a) to [out=-150,in=150,looseness=4] node[left]{\mbox{\tiny{$h_\bullet$}}} (a);
	\end{tikzpicture}
\end{center}
and transfer the simplicial commutative algebra structure on $\Omega_\bullet$ to a simplicial $\C_\infty$-algebra structure on $C_\bullet$. Since $C_\bullet$ is finite dimensional at every simplicial degree, we can take its dual, which is a cosimplicial $\Bar\lie$-coalgebra (up to a suspension). Then we take the complete cobar construction with respect to the canonical twisting morphism $\Bar\lie\to\lie$ to obtain a cosimplicial Lie algebra
\[
\mc_\bullet\coloneqq\widehat{\Cobar}_\pi(sC_\bullet^\vee)\ .
\]
Intuitively, this should be a model for Maurer--Cartan elements at cosimplicial degree $0$, for gauges at cosimplicial degree $1$, and so on. One of the main results of this chapter formalizes this in the form of a natural homotopy equivalence of simplicial sets
\[
\hom_{\dgl}(\mc_\bullet,\g)\simeq\MC_\bullet(\g)
\]
when $\g$ is a Lie algebra. The proof of this result is made in two steps.
\begin{enumerate}
	\item First, we prove that
	\[
	\MC_\bullet(\g)\simeq\MC(\g\otimes C_\bullet)\ ,
	\]
	where $\g\otimes C_\bullet$ is given a homotopy Lie algebra structure by the homotopy transfer theorem applied to the contraction induced by Dupont's contraction. The proof is similar to the demonstration of the Dolgushev--Rogers theorem, and mainly uses methods of simplicial homotopy theory.
	\item Then, we prove that
	\[
	\hom_{\dgl}(\mc_\bullet,\g)\cong\MC(\g\otimes C_\bullet)\ .
	\]
	This fact will be a straightforward consequence of the results of \cref{chapter:convolution homotopy algebras}.
\end{enumerate}
As a nice, immediate consequence, we have that taking the space of Maurer--Cartan elements commutes with limits --- only up to homotopy, if we take the functor $\MC_\bullet(-)$. We will also study some properties of $\MC(\g\otimes C_\bullet)$ and of $\mc_\bullet$. As a corollary, we will obtain an explicit way of ``rectifying" homotopy equivalences between Maurer--Cartan elements to gauges between the same elements. We will also present an extension of these results to the case where $\g$ is a $\SLoo$-algebra. This was not present in \cite{rn17cosimplicial}, but will be contained in \cite{rnv18}.

\medskip

It should be remarked that the author was not aware of Bandiera's results --- see \cref{sect:formal Kuranishi} --- at the time of publication of \cite{rn17cosimplicial}. They give another way of proving that $\MC_\bullet(\g)\simeq\MC(\g\otimes C_\bullet)$, using very different methods. We believe that the two ways of proceeding are complementary, completing each other to yield a picture of the various relations between the three known models of the space of Maurer--Cartan elements of an $\L_\infty$-algebra.

\medskip

Contrarily to what we did in most of the rest of this thesis, in this chapter we will work exclusively over cochain complexes. We will also work over shifted Lie and homotopy Lie algebras, but as usual the theory behaves in exactly the same way in the unshifted setting. In particular, Maurer--Cartan elements are in degree $0$, and gauges are in degree $-1$.

\section{Another model for the space of Maurer--Cartan elements}\label{sect:alternative model for space of MC}

We begin by giving a simplicial set which is smaller, but homotopically equivalent to the space of Maurer--Cartan elements $\MC_\bullet(\g)$. It is given by the Maurer--Cartan elements of the tensorization of the $\SLoo$-algebra into consideration with the cellular cochains of the geometric simplices.

\subsection{Statement of the main theorem}

Let $\g$ be a complete $\SLoo$-algebra. Dupont's contraction induces a contraction
\begin{center}
	\begin{tikzpicture}
	\node (a) at (0,0){$\g\otimes\Omega_\bullet$};
	\node (b) at (2.8,0){$\g\otimes C_\bullet$};
	
	\draw[->] (a)++(.6,.1)--node[above]{\mbox{\tiny{$1\otimes p_\bullet$}}}+(1.6,0);
	\draw[<-,yshift=-1mm] (a)++(.6,-.1)--node[below]{\mbox{\tiny{$1\otimes i_\bullet$}}}+(1.6,0);
	\draw[->] (a) to [out=-160,in=160,looseness=4] node[left]{\mbox{\tiny{$1\otimes h_\bullet$}}} (a);
	\end{tikzpicture}
\end{center}
of $\g\otimes\Omega_\bullet$ onto $\g\otimes C_\bullet$. Applying the homotopy transfer theorem to this contraction, we obtain a simplicial $\SLoo$-algebra structure on $\g\otimes C_\bullet$. We also know that we can extend the maps $1\otimes p_\bullet$ and $1\otimes i_\bullet$ to simplicial $\infty$-morphisms of simplicial $\SLoo$-algebras $(1\otimes p_\bullet)_\infty$ and $(1\otimes i_\bullet)_\infty$. We denote $P_\bullet$ and $I_\bullet$ the induced maps on Maurer--Cartan elements. We will also use the notation
\[
(1\otimes r_\bullet)_\infty \coloneqq (1\otimes i_\bullet)_\infty(1\otimes p_\bullet)_\infty\ ,
\]
and we dub $\rect_\bullet$ the map induced by $(1\otimes r_\bullet)_\infty$ on Maurer--Cartan elements.

\begin{theorem} \label{thm:homotopy equivalence MC. and MC(g x C)}
	Let $\g$ be a filtered $\SLoo$-algebra. The maps $P_\bullet$ and $I_\bullet$ are inverse one to the other in homotopy, and thus provide a weak equivalence
	\[
	\MC_\bullet(\g)\simeq\MC(\g\otimes C_\bullet)
	\]
	of simplicial sets which is natural in $\g$.
\end{theorem}

\subsection{Proof of the main theorem}

The rest of this section is dedicated to the proof of this result. We begin with the following lemma.

\begin{lemma} \label{lemma:PIis1}
	We have
	\[
	P_\bullet I_\bullet = \id_{\MC(\g\otimes C_\bullet)}\ .
	\]
\end{lemma}

\begin{proof}
	This is because $(1\otimes p_\bullet)_\infty(1\otimes i_\bullet)_\infty$ is the identity, see e.g. \cite[Theorem 5]{dsv16}, and the functoriality of the Maurer--Cartan functor $\MC$.
\end{proof}

Therefore, it is enough to prove that the map
\[
\rect = I_\bullet P_\bullet:\MC_\bullet(\g)\longrightarrow\MC_\bullet(\g)
\]
is a weak equivalence. The idea is to use the same methods as for the proof of the Dolgushev--Rogers theorem, cf. \cref{sect:Dolgushev-Rogers thm}. The situation is however slightly different, as the map $\rect_\bullet$ is not of the form $\Phi\otimes1_{\Omega_\bullet}$, and thus the Dolgushev--Rogers theorem itself cannot be directly applied. The first, easy step is to understand what happens at the level of the zeroth homotopy group.

\begin{lemma} \label{lemma:pi0}
	The map
	\[
	\pi_0(\rect):\pi_0\MC_\bullet(\g)\longrightarrow\pi_0\MC_\bullet(\g)
	\]
	is a bijection.
\end{lemma}

\begin{proof}
	We have $\Omega_0=C_0=\k$, and the maps $i_0$ and $p_0$ both are the identity of $\k$. Therefore, the map $R_0$ is the identity of $\MC_0(\g)$, and thus obviously induces a bijection on $\pi_0$.
\end{proof}

For the higher homotopy groups, we start with a simplified version of \cref{prop:abelian Loo-alg DR}, which gives in some sense the base for an inductive argument. If the $\SLoo$-algebra $\g$ is abelian, i.e. all of its brackets vanish, then so do the brackets at all levels of $\g\otimes\Omega_\bullet$. In this case, the Maurer--Cartan elements are exactly the cocycles of the underlying cochain complex, and therefore $\MC_\bullet(\g)$ is a simplicial vector space.

\begin{lemma} \label{lemma:abelianDGL}
	If the $\SLoo$-algebra $\g$ is abelian, then $\rect_\bullet$ is a weak equivalence of simplicial vector spaces.
\end{lemma}

\begin{proof}
	Recall that the Moore complex of a simplicial vector space $V_\bullet$ is defined by
	\[
	\mathcal{M}(V_\bullet)_n\coloneqq s^nV_n
	\]
	endowed with the differential
	\[
	\partial\coloneqq\sum_{i=0}^n(-1)^id_i\ ,
	\]
	where the maps $d_i$ are the face maps of the simplicial set $V_\bullet$. It is a standard result that
	\[
	\pi_0(V_\bullet) = H_0(\mathcal{M}(V_\bullet))\ ,\qquad \pi_i(V_\bullet,v)\cong\pi_i(V_\bullet,0) = H_i(\mathcal{M}(V_\bullet))
	\]
	for all $i\ge1$ and $v\in V_0$, and that a map of simplicial vector spaces is a weak equivalence if and only if it induces a quasi-isomorphism between the respective Moore complexes \cite[Cor. 2.5, Sect. III.2]{GoerssJardine}.
	
	In our case,
	\[
	V_\bullet\coloneqq\MC_\bullet(\g) = \mathcal{Z}^1(\g\otimes\Omega_\bullet)
	\]
	is the simplicial vector space of $1$-cocycles of $\g\otimes\Omega_\bullet$. As in the proof of \cref{prop:abelian Loo-alg DR}, it can be proven that the map
	\[
	\mathcal{M}(1\otimes p_\bullet):\mathcal{M}(\mathcal{Z}^1(\g\otimes\Omega_\bullet))\longrightarrow\mathcal{M}(\mathcal{Z}^1(\g\otimes C_\bullet))
	\]
	is a quasi-isomorphism.  But as the bracket vanishes, this is exactly $P_\bullet$. Now
	\[
	\mathcal{M}(1\otimes p_\bullet)\mathcal{M}(1\otimes i_\bullet) = 1_{\mathcal{M}(\mathcal{Z}^1(\g\otimes\Omega_\bullet))}\ ,
	\]
	which implies that $\mathcal{M}(1\otimes i_\bullet)$ also is a quasi-isomorphism. It follows that $\rect_\bullet$ is a weak equivalence, concluding the proof.
\end{proof}

Now we basically follow the structure of the proof of \cref{thm:Dolgushev-Rogers}. We define a filtration of $\g\otimes\Omega_\bullet$ by
\[
\F_k(\g\otimes\Omega_\bullet)\coloneqq (\F_k\g)\otimes\Omega_\bullet\ .
\]
We denote by
\[
(\g\otimes\Omega_\bullet)^{(k)}\coloneqq\g\otimes\Omega_\bullet/\F_k(\g\otimes\Omega_\bullet) = \g^{(k)}\otimes\Omega_\bullet\ .
\]
The composite $(1\otimes i_\bullet)(1\otimes p_\bullet)$ induces an endomorphism $(1\otimes i_\bullet)^{(k)}(1\otimes p_\bullet)^{(k)}$ of $(\g\otimes\Omega_\bullet)^{(k)}$. All the $\infty$-morphisms coming into play obviously respect this filtration, and moreover $1\otimes h_\bullet$ passes to the quotients, so that we have
\[
1_{(\g\otimes\Omega_\bullet)^{(k)}} - (1\otimes i_\bullet)^{(k)}(1\otimes p_\bullet)^{(k)} = d(1\otimes h_\bullet)^{(k)} + (1\otimes h_\bullet)^{(k)}d
\]
for all $k$, which shows that $(1\otimes r_\bullet)_\infty$ is a filtered $\infty$-quasi isomorphism.

\medskip

The next step is to reduce the study of the homotopy groups with arbitrary basepoint to the study of the homotopy groups with basepoint $0\in\MC_0(\g)$.

\begin{lemma} \label{lemma:shift to zero}
	Let $\alpha\in\MC(\g)$, and let $\g^\alpha$ be the $\L_\infty$-algebra obtained by twisting $\g$ by $\alpha$, that is the $\L_\infty$-algebra with the same underlying graded vector space, but with differential
	\[
	d^\alpha(x) \coloneqq dx + \sum_{n\ge2}\frac{1}{(n-1)!}\ell_n(\alpha,\ldots,\alpha,x)
	\]
	and brackets
	\[
	\ell^\alpha(x_1,\ldots,x_m) \coloneqq \sum_{n\ge m}\frac{1}{(n-m)!}\ell_n(\alpha,\ldots,\alpha,x_1,\ldots,x_m)\ .
	\]
	Let
	\[
	\mathrm{Shift}_\alpha:\MC_\bullet(\g^\alpha)\longrightarrow\MC_\bullet(\g)
	\]
	be the isomorphism of simplicial sets induced by the map given by
	\[
	\beta\in\g\longmapsto\alpha+\beta\in\g^\alpha\ .
	\]
	Then the following diagram commutes
	\begin{center}
		\begin{tikzpicture}
			\node (a) at (0,0){$\MC_\bullet(\g^\alpha)$};
			\node (b) at (3.5,0){$\MC_\bullet(\g)$};
			\node (c) at (0,-2){$\MC_\bullet(\g^\alpha)$};
			\node (d) at (3.5,-2){$\MC_\bullet(\g)$};
			
			\draw[->] (a) to node[above]{$\mathrm{Shift}_\alpha$} (b);
			\draw[->] (a) to node[left]{$\rect_\bullet^\alpha$} (c);
			\draw[->] (c) to node[above]{$\mathrm{Shift}_\alpha$} (d);
			\draw[->] (b) to node[right]{$\rect_\bullet$} (d);
		\end{tikzpicture}
	\end{center}
	where
	\[
	\rect^\alpha(\beta)\coloneqq \sum_{k\ge1}(1\otimes r_\bullet)^\alpha_k(\beta^{\otimes k})
	\]
	and
	\[
	(1\otimes r_\bullet)^\alpha_k(\beta_1\otimes\cdots\otimes\beta_k)\coloneqq\sum_{j\ge0}\frac{1}{j!}(1\otimes r_\bullet)_{k+j}(\alpha^{\otimes j}\otimes\beta_1\otimes\ldots\otimes\beta_k)
	\]
	is the twist of $(1\otimes r_\bullet)_\infty$ by the Maurer--Cartan element $\alpha$. Here, we identified $\alpha\in\g$ with $\alpha\otimes1\in\g\otimes\Omega_\bullet$.
\end{lemma}

\begin{proof}
	The proof of \cite[Lemma 4.3]{dr17} goes through \emph{mutatis mutandis}.
\end{proof}

\begin{remark}
	The $\L_\infty$-algebra $\g^\alpha$ in \cref{lemma:shift to zero} is endowed with the same filtration as $\g$.
\end{remark}

Now we proceed by induction to show that $\rect^{(k)}$ is a weak equivalence from $\MC_\bullet(\g^{(k)})$ to itself for all $k\ge2$. As the $\L_\infty$-algebra $(\g\otimes\Omega_\bullet)^{(2)}$ is abelian, the base step of the induction is given by Lemma \ref{lemma:abelianDGL}.

\begin{lemma} \label{lemma:inductionStep}
	Let $m\ge2$. Suppose that
	\[
	\rect^{(k)}:\MC(\g^{(k)})\longrightarrow\MC(\g^{(k)})
	\]
	is a weak equivalence for all $2\le k\le m$. Then $\rect^{(m+1)}$ is also a weak equivalence.
\end{lemma}

\begin{proof}
	The zeroth homotopy set $\pi_0$ has already been taken care of in Lemma \ref{lemma:pi0}. Thanks to Lemma \ref{lemma:shift to zero}, it is enough to prove that $\rect^{(m+1)}$ induces isomorphisms of homotopy groups $\pi_i$ based at $0$, for all $i\ge1$.
	
	\medskip
	
	Consider the following commutative diagram
	\begin{center}
		\begin{tikzpicture}
			\node (a) at (0,0){$0$};
			\node (b) at (3,0){$\frac{\F_m(\g\otimes\Omega_\bullet)}{\F_{m+1}(\g\otimes\Omega_\bullet)}$};
			\node (c) at (6.5,0){$(\g\otimes\Omega_\bullet)^{(m+1)}$};
			\node (d) at (10,0){$(\g\otimes\Omega_\bullet)^{(m)}$};
			\node (e) at (13,0){$0$};
			
			\node (A) at (0,-2.5){$0$};
			\node (B) at (3,-2.5){$\frac{\F_m(\g\otimes\Omega_\bullet)}{\F_{m+1}(\g\otimes\Omega_\bullet)}$};
			\node (C) at (6.5,-2.5){$(\g\otimes\Omega_\bullet)^{(m+1)}$};
			\node (D) at (10,-2.5){$(\g\otimes\Omega_\bullet)^{(m)}$};
			\node (E) at (13,-2.5){$0$};
			
			\draw[->] (a) to (b);
			\draw[->] (b) to (c);
			\draw[->] (c) to (d);
			\draw[->] (d) to (e);
			
			\draw[->] (A) to (B);
			\draw[->] (B) to (C);
			\draw[->] (C) to (D);
			\draw[->] (D) to (E);
			
			\draw[->] (b) to (B);
			\draw[->] (c) to node[left]{$(1\otimes r_\bullet)^{(m+1)}_\infty$} (C);
			\draw[->] (d) to node[left]{$(1\otimes r_\bullet)^{(m)}_\infty$} (D);
		\end{tikzpicture}
	\end{center}
	where the leftmost vertical arrow is given by the linear term $(1\otimes i_\bullet)(1\otimes p_\bullet)$ of $(1\otimes r_\bullet)_\infty$ since all higher terms vanish, as can be seen by the explicit formul{\ae} for the $\infty$-quasi isomorphisms induced by the homotopy transfer theorem given in \cite[Sect. 10.3.5--6]{LodayVallette}. Therefore, it is a weak equivalence as the $\SLoo$-algebras in question are abelian. The first term in each row is the fibre of the next map, which is surjective. By \cref{thm:filtered surjections induce fibrations of MC}, we know that applying the $\MC_\bullet$ functor makes the horizontal maps on the right into fibrations of simplicial sets, while the objects we obtain on the left are easily seen to be the fibres. Taking the long sequence in homotopy and using the five-lemma, we see that all we are left to do is to prove that $\rect_\bullet^{(m+1)}$ induces an isomorphism on $\pi_1$. Notice that it is necessary to prove this, as the long sequence is exact everywhere except on the level of $\pi_0$.
	
	\medskip
	
	The long exact sequence of homotopy groups (truncated on both sides) reads
	\[
	\pi_2\MC_\bullet(\g^{(m)})\stackrel{\partial}{\longrightarrow}\pi_1\MC_\bullet\left(\tfrac{F_m\g}{F_{m+1}\g}\right)\to\pi_1\MC_\bullet(\g^{(m+1)})\to\pi_1\MC_\bullet(\g^{(m)})\stackrel{\partial}{\longrightarrow}\pi_0\MC_\bullet\left(\tfrac{F_m\g}{F_{m+1}\g}\right)\ ,
	\]
	where in the higher homotopy groups we left the basepoint implicit (as it is always $0$). The map
	\[
	\partial:\pi_1\MC_\bullet(\g^{(m)})\longrightarrow\pi_0\MC_\bullet\left(\frac{F_m\g}{F_{m+1}\g}\right) = H^1\left(\frac{F_{m+1}\g}{F_m\g}\right)
	\]
	encodes the obstruction to lifting an element of $\pi_1\MC_\bullet(\g^{(m)})$ to an element of $\pi_1\MC_\bullet(\g^{(m+1)})$ (see e.g. \cref{subsect:LES of homotopy groups}).
	
	\medskip
	
	\textit{The map $\pi_1(\rect^{(m+1)})$ is surjective:} Let $y\in\pi_1\MC_\bullet(\g^{(m+1)})$ be any element, and denote by $\overline{y}$ its image in $\pi_1\MC_\bullet(\g^{(m)})$. By the induction hypothesis, there exists a unique $\overline{x}\in\pi_1\MC_\bullet(\g^{(m)})$ which is mapped to $\overline{y}$ under $\rect_\bullet^{(m)}$. As $\overline{y}$ is the image of $y$, we have $\partial(\overline{y}) = 0$, and this implies that $\partial(\overline{x})=0$, too. Therefore, there exists $x\in\pi_1\MC_\bullet(\g^{(m+1)})$ mapping to $\overline{x}$. Denote by $y'$ the image of $x$ under $\rect_\bullet^{(m+1)}$. Then $y'y^{-1}$ is in the kernel of the map
	\[
	\pi_1\MC_\bullet(\g^{(m+1)})\longrightarrow\pi_1\MC_\bullet(\g^{(m)})\ .
	\]
	By exactness of the long sequence, and the fact that $\rect_\bullet$ induces an automorphism of
	\[
	\pi_1\MC_\bullet\left(\frac{F_{m+1}\g}{F_m\g}\right)\ .
	\]
	there exists an element $z\in\pi_1(\MC_\bullet(F_{m+1}\g/F_m\g))$ mapping to $y'y^{-1}$ under the composite
	\[
	\pi_1\MC_\bullet\left(\frac{F_{m+1}\g}{F_m\g}\right)\xrightarrow{\rect_\bullet}\pi_1\MC_\bullet\left(\frac{F_{m+1}\g}{F_m\g}\right)\longrightarrow\pi_1\MC_\bullet(\g^{(m+1)})\ .
	\]
	Let $x'$ be the image of $z$ in $\pi_1\MC_\bullet(\g^{(m+1)})$, then $(x')^{-1}x$ maps to $y$ under $\rect_\bullet^{(m+1)}$. This proves the surjectivity of the map $\pi_1(\rect^{(m+1)})$.
	
	\medskip
	
	\textit{The map $\pi_1(\rect_\bullet^{(m+1)})$ is injective:} Assume $x,x'\in\pi_1\MC_\bullet(\g^{(m+1)})$ map to the same element under $\rect^{(m+1)}$. Then $x(x')^{-1}$ maps to the neutral element $0$ under $\rect^{(m+1)}$. It follows that there is a
	\[
	z\in\pi_1\MC_\bullet\left(\frac{F_{m+1}\g}{F_m\g}\right)
	\]
	mapping to $x(x')^{-1}$, which must be such that its image $w$ is itself the image of some element $\widetilde{w}\in\pi_2\MC_\bullet(\g^{(m)})$ under the map $\partial$. But by the induction hypothesis and the exactness of the long sequence, this implies that $z$ is in the kernel of the next map, and thus that $x(x')^{-1}$ is the identity element. Therefore, the map $\pi_1(\rect_\bullet^{(m+1)})$ is injective.
	
	\medskip
	
	This ends the proof of the lemma.
\end{proof}

Finally, we can conclude the proof of \cref{thm:homotopy equivalence MC. and MC(g x C)}.

\begin{proof}[Proof of \cref{thm:homotopy equivalence MC. and MC(g x C)}]
	Lemma \ref{lemma:inductionStep}, together with all we have said before, shows that $\rect_\bullet^{(m)}$ is a weak equivalence for all $m\ge2$. Therefore, we have the following commutative diagram:
	\begin{center}
		\begin{tikzpicture}
			\node (a) at (0,0){$\vdots$};
			\node (b) at (4,0){$\vdots$};
			\node (c) at (0,-2){$\MC_\bullet(\g^{(4)})$};
			\node (d) at (4,-2){$\MC_\bullet(\g^{(4)})$};
			\node (e) at (0,-4){$\MC_\bullet(\g^{(3)})$};
			\node (f) at (4,-4){$\MC_\bullet(\g^{(3)})$};
			\node (g) at (0,-6){$\MC_\bullet(\g^{(2)})$};
			\node (h) at (4,-6){$\MC_\bullet(\g^{(2)})$};
			
			\draw[->] (a) to (c);
			\draw[->] (c) to (e);
			\draw[->] (e) to (g);
			
			\draw[->] (b) to (d);
			\draw[->] (d) to (f);
			\draw[->] (f) to (h);
			
			\draw[->] (c) to node[above]{$\sim$} node[below]{$\rect_\bullet^{(4)}$} (d);
			\draw[->] (e) to node[above]{$\sim$} node[below]{$\rect_\bullet^{(3)}$} (f);
			\draw[->] (g) to node[above]{$\sim$} node[below]{$\rect_\bullet^{(2)}$} (h);
		\end{tikzpicture}
	\end{center}
	where all objects are Kan complexes, all horizontal arrows are weak equivalences, and all vertical arrows are (Kan) fibrations by \cref{thm:filtered surjections induce fibrations of MC}. It follows that the collection of horizontal arrows defines a weak equivalence between fibrant objects in the model category of tower of simplicial sets, see \cite[Sect. VI.1]{GoerssJardine}. The functor from towers of simplicial sets to simplicial sets given by taking the limit is right adjoint to the constant tower functor, which trivially preserves cofibrations and weak equivalences. Thus, the constant tower functor is a left Quillen functor, and it follows that the limit functor is a right Quillen functor. In particular, it preserves weak equivalences between fibrant objects. Applying this to the diagram above proves that $\rect_\bullet$ is a weak equivalence.
\end{proof}

\begin{remark}
	An alternative proof of the fact that $\rect_\bullet$ induces equivalences on $\pi_n$, for all $n\ge1$, was put forward by A. Berglund in a private communication. It suggests using \cite[Thm. 1.1]{ber15} and the explicit formula given for the map $B$, together with \cref{prop:rectification for deformation oo-groupoid}, and to verify that the composite
	\[
	H_{n-1}(\g)\stackrel{B}{\longrightarrow}\pi_n(\MC_\bullet(\g))\xrightarrow{\pi_n(\rect_\bullet)}\pi_n(\MC_\bullet(\g))\xrightarrow{B^{-1}}H_{n-1}(\g)
	\]
	is an isomorphism.
\end{remark}

\section{Properties and comparison with Getzler's functor}\label{sect:properties and comparison with Getzler}

\cref{thm:homotopy equivalence MC. and MC(g x C)} shows that the simplicial set $\MC(\g\otimes C_\bullet)$ is a new model for the deformation $\infty$-groupoid. This section is dedicated to the study of some properties of this object. We start by showing that it is a Kan complex, then we give some conditions on the differential forms representing its simplices. We show how we can use it to rectify cells of the deformation $\infty$-groupoid, which provides an alternative, simpler proof of \cite[Lemma B.2]{dr15}. Finally we compare it with Getzler's functor $\gamma_\bullet$, proving that our model is contained in Getzler's. Independent results by Bandiera \cite{Bandiera}, \cite{ban17} imply that the two models are actually isomorphic.

\subsection{Properties of \texorpdfstring{$\MC_\bullet(\g\otimes C_\bullet)$}{MC(g x C.)}}\label{subsect:properties of MC(g x C.)}

The following proposition is the analogue to \cref{thm:filtered surjections induce fibrations of MC} for our model.

\begin{proposition}
	Let $\g,\h$ be two complete proper $\SLoo$-algebras, and suppose that $\Phi:\g\rightsquigarrow\h$ is an $\infty$-morphism of $\L_\infty$-algebras inducing a fibration of simplicial sets under the functor $\MC_\bullet$, see for example \cref{thm:filtered surjections induce fibrations of MC} for possible sufficient conditions. Then the induced morphism
	\[
	\MC(\phi\otimes\id_{C_\bullet}):\MC(\g\otimes C_\bullet)\longrightarrow\MC(\h\otimes C_\bullet)
	\]
	is also a fibration of simplicial sets. In particular, for any complete proper $\SLoo$-algebra $\g$, the simplicial set $\MC(\g\otimes C_\bullet)$ is a Kan complex.
\end{proposition}

\begin{proof}
	By assumption, the morphism
	\[
	\MC_\bullet(\phi):\MC_\bullet(\g)\longrightarrow\MC_\bullet(\h)
	\]
	is a fibration of simplicial set, and by Lemma \ref{lemma:PIis1} the following diagram exhibits $\MC(\phi\otimes\id_{C_\bullet})$ as a retract of $\MC_\bullet(\phi)$.
	\begin{center}
		\begin{tikzpicture}
		\node (a) at (0,0) {$\MC(\g\otimes C_\bullet)$};
		\node (b) at (3,0) {$\MC_\bullet(\g)$};
		\node (c) at (6,0) {$\MC(\g\otimes C_\bullet)$};
		\node (d) at (0,-2) {$\MC(\g\otimes C_\bullet)$};
		\node (e) at (3,-2) {$\MC_\bullet(\g)$};
		\node (f) at (6,-2) {$\MC(\g\otimes C_\bullet)$};
		
		\draw[->] (a) -- node[above]{$I_\bullet$} (b);
		\draw[->] (b) -- node[above]{$P_\bullet$} (c);
		\draw[->] (d) -- node[below]{$I_\bullet$} (e);
		\draw[->] (e) -- node[below]{$P_\bullet$} (f);
		\draw[->] (a) -- node[left]{$\MC(\phi\otimes\id_{C_\bullet})$} (d);
		\draw[->] (b) -- node[left]{$\MC_\bullet(\phi)$} (e);
		\draw[->] (c) -- node[right]{$\MC(\phi\otimes\id_{C_\bullet})$} (f);
		\end{tikzpicture}
	\end{center}
	As the class of fibrations of a model category is closed under retracts, this concludes the proof.
\end{proof}

We also consider the composite $\rect_\bullet\coloneqq I_\bullet P_\bullet$, which is not the identity in general.

\begin{definition}
	We call the morphism
	\[
	\rect_\bullet:\MC_\bullet(\g)\longrightarrow\MC_\bullet(\g)
	\]
	the \emph{rectification map}.
\end{definition}

The following result is a wide generalization of \cite[Lemma B.2]{dr15}, as well as a motivation for the name ``rectification map'' for $\rect_\bullet$.

\begin{proposition} \label{prop:rectification for deformation oo-groupoid}
	We consider an element
	\[
	\alpha \coloneqq \alpha_1(t_0,\ldots,t_n) + \cdots\in\MC_n(\g)\ ,
	\]
	where the dots indicate terms in $\g^{1-k}\otimes\Omega_n^k$ with $1\le k\le n$. Then $\beta\coloneqq \rect_n(\alpha)\in\MC_n(\g)$ is of the form
	\[
	\beta = \beta_1(t_0,\ldots,t_n) + \cdots + \xi \otimes\omega_{0\ldots n}\ ,
	\]
	where the dots indicate terms in $\g^{1-k}\otimes\Omega_n^k$ with $1\le k\le n-1$, where $\xi$ is an element of $\g^{1-n}$, and where $\alpha_1$ and $\beta_1$ agree on the vertices of $\Delta^n$. In particular, if $\alpha\in\MC_1(\g)$, then $\beta = F(\alpha)\in\MC_1(\g)$ is of the form
	\[
	\beta = \beta_1(t) +\lambda dt
	\]
	for some $\lambda\in\g^0$, and satisfies
	\[
	\beta_1(0) = \alpha_1(0)\qquad\text{and}\qquad\beta_1(1) = \alpha_1(1)\ .
	\]
\end{proposition}

\begin{remark}
	As $\rect_\bullet$ is a projector, this proposition in fact gives information on the form of all the elements of $\MC(\g\otimes C_\bullet)$.
\end{remark}

\begin{proof}
	First notice that the map $\rect_\bullet$ commutes with the face maps and is the identity on $0$-simplices, thus evaluation of the part of $\beta$ in $\g^1\otimes\Omega_n^0$ at the vertices gives the same result as evaluation at the vertices of $\alpha_1$. Next, we notice that $\beta$ is in the image of $I_\bullet$. We use the explicit formula for $(1\otimes i_n)_\infty$ of \cref{subsect:HTT}: the operator acting on arity $k\ge2$ is given, up to signs, by the sum over all rooted trees with $1\otimes i_n$ put at the leaves, the brackets $\ell_n$ of the corresponding arity at all vertices, and $1\otimes h$ at the inner edges and at the root. But the $1\otimes h$ at the root lowers the degree of the part of the form in $\Omega_n$ by $1$, and thus we cannot get something in $\g^{1-n}\otimes\Omega_n^n$ from these terms. The only surviving term is therefore the one coming from $(1\otimes i_n)(P(\alpha))$, given by $\xi\otimes\omega_{0\ldots n}$ for some $\xi\in\g^{1-n}$.
\end{proof}

In particular, take $n=1$. Then a gauge between two Maurer--Cartan elements is exactly the same as an element $\beta\in\MC_1(\g)$ of the form
\[
\beta = \beta_1(t) +\lambda dt
\]
such that $\beta_1(t)$ evaluates to the two Maurer--Cartan elements at $t=0,1$, cf. \cref{subsection:homotopies and gauges}. With this in our minds, we notice that \cref{prop:rectification for deformation oo-groupoid} above immediately implies the following two facts.
\begin{enumerate}
	\item The set $\MC(\g\otimes C_1)$ is included into $\MC_1(\g)$ as the subset consisting of gauge equivalences via the map $I_1$.
	\item The rectification map $\rect_1$ gives us an explicit formula to rectify a homotopy between two Maurer--Cartan elements to a gauge between the same elements.
\end{enumerate}
Analogously, \cref{prop:rectification for deformation oo-groupoid} also tells us that the higher maps $\rect_n$ ``rectify" higher relations, making the term of lowest degree in $\g$ become constant.

\subsection{Comparison with Getzler's \texorpdfstring{$\infty$}{infinity}-groupoid}

Finally, we compare the simplicial set $\MC(\g\otimes C_\bullet)$ with Getzler's Kan complex $\gamma_\bullet(\g)$. We start with an easy result that follows directly from our approach, before using Bandiera's results --- see \cref{subsect:formal Kuranishi theorem} --- to prove that these two simplicial sets are actually isomorphic.

\begin{lemma} \label{lemma:comparisonGamma}
	We have
	\[
	I_\bullet\MC(\g\otimes C_\bullet)\subseteq\gamma_\bullet(\g)\ .
	\]
\end{lemma}

\begin{proof}
	We have $h_\bullet i_\bullet = 0$. Therefore, by the explicit formula formula for $(i_\bullet)_\infty$ given in \cref{subsect:HTT}, we have $h_\bullet(\beta) = 0$ for any $\beta\in\g\otimes\Omega_\bullet$ in the image of $I_\bullet$. Thus
	\[
	h_\bullet(\MC(\g\otimes C_\bullet)) = h_\bullet I_\bullet P_\bullet(\MC_\bullet(\g)) = 0\ ,
	\]
	which proves the claim.
\end{proof}

An immediate consequence of the formal Kuranishi theorem is the following proposition.

\begin{theorem}[{\cite[Prop. 2.5]{ban17}}]\label{thm:bandiera}
	The map
	\[
	(P_\bullet,1\otimes h_\bullet):\MC_\bullet(\g)\longrightarrow\MC(\g\otimes C_\bullet)\times\big(\mathrm{Im}(1\otimes h_\bullet)\cap (\g\otimes\Omega_\bullet)^1\big)
	\]
	is bijective. In particular, its restriction to $\gamma_\bullet(\g) = \ker(1\otimes h_\bullet)\cap\MC_\bullet(\g)$ gives a isomorphism of simplicial sets
	\[
	P_\bullet:\gamma_\bullet(\g)\longrightarrow\MC(\g\otimes C_\bullet)\ .
	\]
\end{theorem}

\begin{proof}
	The first statement is obtained by applying \cref{thm:formal Kuranishi theorem} to the contraction
	\begin{center}
		\begin{tikzpicture}
		\node (a) at (0,0){$\g\otimes\Omega_\bullet$};
		\node (b) at (2.8,0){$\g\otimes C_\bullet$};
		
		\draw[->] (a)++(.6,.1)--node[above]{\mbox{\tiny{$1\otimes p_\bullet$}}}+(1.6,0);
		\draw[<-,yshift=-1mm] (a)++(.6,-.1)--node[below]{\mbox{\tiny{$1\otimes i_\bullet$}}}+(1.6,0);
		\draw[->] (a) to [out=-160,in=160,looseness=4] node[left]{\mbox{\tiny{$1\otimes h_\bullet$}}} (a);
		\end{tikzpicture}
	\end{center}
	The second statement is a straightforward consequence of the first one, obtained by restricting the map to $\gamma_\bullet(\g) = \MC_\bullet(\g)\cap\ker(1\otimes h_\bullet)$.
\end{proof}

\begin{remark}
	Thanks to our approach, we immediately have an inverse for the map $P_\bullet$: it is of course the map $I_\bullet$.
\end{remark}

As a consequence of Bandiera's result and of \cref{prop:rectification for deformation oo-groupoid}, we can partially characterize the thin elements of $\gamma_\bullet(\g)$.

\begin{lemma}
	For each $n\ge1$, the thin elements contained in $\gamma_n(\g)$ are those with no term in $\g^{1-n}\otimes\Omega_n^n$.
\end{lemma}

\begin{proof}
	By \cref{prop:rectification for deformation oo-groupoid} and \cref{thm:bandiera}, we know that if $\alpha\in\gamma_n(\g)$, then $\alpha$ is of the form
	\[
	\alpha = \cdots + \xi\otimes\omega_{0\ldots n}
	\]
	for some $\xi\in\g^{1-n}$, where the dots indicate terms in $\g^{1-k}\otimes\Omega_n^k$ for $0\le k\le n-1$, which will give zero after integration. Integrating, we get
	\[
	\int_{\Delta^n}\alpha = \xi\otimes\int_{\Delta^n}\omega_{0\ldots n} = \xi\otimes1\ .
	\]
	Therefore, $\alpha$ is thin if, and only if $\xi=0$. 
\end{proof}

\section{A model for Maurer--Cartan elements of Lie algebras} \label{sect:cosimplicial model for Lie algebras}

Our next goal is to represent the Maurer--Cartan functor $\MC(\g\otimes C_\bullet)$ by a cosimplicial object. We begin by doing this in the case of Lie algebras, as it was originally done in \cite{rn17cosimplicial}. The case of $\SLoo$-algebras will be treated in \cref{sect:model for MC-el in Loo-alg}.

\medskip

Since we are in the shifted setting, we consider shifted Lie algebras, that is algebras over the operad $\susp\otimes\lie$.

\subsection{Representing \texorpdfstring{$\MC(\g\otimes C_\bullet)$}{MC(g x C.)}}\label{subsect:representing MC for Lie algebras}

Using the Dupont contraction, the homotopy transfer theorem gives the structure of a simplicial $\C_\infty$-algebra to $C_\bullet$. As the underlying cochain complex $C_n$ is finite dimensional for each $n$, it follows that its dual is a cosimplicial $\Bar(\susp\otimes\lie)$-coalgebra. Therefore, we can take its complete cobar construction relative to the canonical twisting morphism
\[
\pi:\Bar(\susp\otimes\lie)\longrightarrow\susp\otimes\lie
\]
to obtain a shifted Lie algebra.

\begin{definition}\label{def:cosimplicial mc for Lie}
	We denote the cosimplicial shifted Lie algebra obtained this way by $\mc_\bullet\coloneqq\widehat{\Omega}_\pi(C_\bullet^\vee)$.\index{$\mc_\bullet$}
\end{definition}

\begin{theorem}\label{thm:cosimplicialModel}
	Let $\g$ be a proper complete shifted Lie algebra. There is a canonical isomorphism
	\[
	\MC(\g\otimes C_\bullet)\cong\hom_{\mathsf{dgLie}}(\mc_\bullet,\g)\ .
	\]
	It is natural in $\g$.
\end{theorem}

\begin{proof}
	By \cref{thm:compatibility htt and convolution algebras on coalgebra side}, the $\SLoo$-algebra structure we have on $\g\otimes C_\bullet$ is the same as the structure that we obtain on the tensor product of the shifted Lie algebra $\g$ with the simplicial $\C_\infty$-algebra $C_\bullet$ by using \cref{thm:convolution Loo tensor products} the twisting morphism $\pi$. Therefore, we can apply \cref{prop:partial Rosetta stone for non-conilpotent coalgebras} and \cref{thm:equality MC and twisting in tensor case} to obtain the desired isomorphism.
\end{proof}

With this form for $\MC(\g\otimes C_\bullet)$, \cref{thm:homotopy equivalence MC. and MC(g x C)} reads as follows.

\begin{corollary} \label{cor:exMainThm}
	Let $\g$ be a proper complete dg Lie algebra. There is a weak equivalence of simplicial sets
	\[
	\MC_\bullet(\g)\simeq\hom_{\mathsf{dgLie}}(\mc_\bullet,\g)\ ,
	\]
	natural in $\g$.
\end{corollary}

\begin{remark}
	This result was proven independently and simultaneously in \cite[Thm. 0.1]{bfmt17}.
\end{remark}

We can completely characterize the first levels of the cosimplicial Lie algebra $\mc_\bullet$. Recall from the Lawrence--Sullivan algebra from \cref{subsect:LS algebra}: it is the unique free complete dg Lie algebra generated by two Maurer--Cartan elements in degree $1$ and a single element in degree $0$ such that the element in degree $0$ is a gauge between the two generating Maurer--Cartan elements.

\begin{proposition}\label{prop:first levels of mc}
	The first two levels of the cosimplicial dg Lie algebra $\mc_\bullet$ are as follows.
	\begin{enumerate}
		\item \label{n1} The dg Lie algebra $\mc_0$ is isomorphic to the free dg Lie algebra with a single Maurer--Cartan element as the only generator.
		\item \label{n2} The dg Lie algebra $\mc_1$ is isomorphic to the Lawrence--Sullivan algebra, shifted by $1$.
	\end{enumerate}
\end{proposition}

\begin{proof}
	For (\ref{n1}), we have $\Omega_0\cong\k\cong C_0$, both $p_0$ and $i_0$ are the identity, and $h_0 = 0$. It follows that, as a complete graded free Lie algebra, $\mc_0$ is given by
	\[
	\mc_0 = \widehat{\lie}(s\k)\ .
	\]
	We denote the generator by $\alpha\coloneqq s1^\vee$. It has degree $1$. Let $\g$ be any complete dg Lie algebra, then a morphism
	\[
	\phi:\mc_0\longrightarrow\g
	\]
	is equivalent to the Maurer--Cartan element
	\[
	\phi(\alpha)\otimes 1\in\MC(\g\otimes C_\bullet) \cong \MC(\g)\ .
	\]
	Conversely, through $P_0$ every Maurer--Cartan element of $\g$ induces a morphism $\mc_0\to\g$. As this is true for any dg Lie algebra $\g$, it follows that $\alpha$ is a Maurer--Cartan element.
	
	\medskip
	
	To prove (\ref{n2}), we start by noticing that
	\[
	C_1\coloneqq \k\omega_0\oplus\k\omega_1\oplus\k\omega_{01}
	\]
	with $\omega_0,\omega_1$ of degree $0$ and $\omega_{01}$ of degree $1$. Denoting by $\alpha_i\coloneqq s\omega_i^\vee$ and by $\lambda\coloneqq s\omega_{01}^\vee$, we have
	\[
	\mc_1 = \widehat{\lie}(\alpha_0,\alpha_1,\lambda)
	\]
	as a graded Lie algebra. Let $\g$ be any dg Lie algebra, then a morphism
	\[
	\phi:\mc_1\longrightarrow\g
	\]
	is equivalent to a Maurer--Cartan element
	\[
	\phi(\alpha_0)\otimes\omega_0 + \phi(\alpha_1)\otimes\omega_1 + \phi(\lambda)\otimes\omega_{01}\in\MC(\g\otimes C_1)\ ,
	\]
	see \cite[Sect. 6.3--4]{rn17tensor}. Applying $I_1$, as in the proof of \cref{prop:rectification for deformation oo-groupoid} we obtain
	\[
	I_1(\phi(\alpha_0)\otimes\omega_0 + \phi(\alpha_1)\otimes\omega_1 + \phi(\lambda)\otimes\omega_{01}) = a(t_0,t_1) + \phi(\lambda)\otimes\omega_{01}\in\MC_1(\g)
	\]
	with $a(1,0) = \phi(\alpha_0)$ and $a(0,1) = \phi(\alpha_1)$. The Maurer--Cartan equation for $a(t_0,t_1) + \phi(\lambda)\otimes\omega_{01}$ then shows that $\phi(\lambda)$ is a gauge from $\phi(\alpha_0)$ to $\phi(\alpha_1)$. Conversely, if we are given the data of two Maurer--Cartan elements of $\g$ and a gauge equivalence between them, then this data gives us a Maurer--Cartan element of $\g\otimes\Omega_1$. Applying $P_1$ then gives back a non-trivial morphism $\mc_1\to\g$. As this is true for any $\g$, it follows that $\mc_1$ is isomorphic to the Lawrence--Sullivan algebra.
\end{proof}

\begin{remark}
	Alternatively, one could write down explicitly the differentials for both $\mc_0$ (which is straightforward) and $\mc_1$ (with the help of \cite[Prop. 19]{cg08}). An explicit description of $\mc_\bullet$ is made difficult by the fact that one needs to know the whole $\C_\infty$-algebra structure on $C_\bullet$ in order to write down a formula for the differential.
\end{remark}

\subsection{Relations to rational homotopy theory}

The cosimplicial dg Lie algebra $\mc_\bullet$ has already made its appearance in the literature not long ago, in the paper \cite{bfmt15}, in the context of rational homotopy theory, where it plays the role of a Lie model for the geometric $n$-simplex. With the goal of simplifying comparison and interaction between our work and theirs, we provide here a short review and a dictionary between our vocabulary and the notations used in \cite{bfmt15}.

\medskip

\begin{center}
	\begin{tabular}{|l|l|}
		\hline
		Notation in the present work & Notation of \cite{bfmt15}\\
		\hline
		$\mc_\bullet$ & $\mathfrak{L}_\bullet$ or $\mathfrak{L}_{\Delta^\bullet}$\\
		$\Omega_\bullet$ & $A_{PL}(\Delta^\bullet)$\\
		$\Bar_\iota$ & Quillen functor $\mathcal{C}$\\
		$\hom_{\mathsf{dgLie}}(\mc_\bullet,-)$ & $\langle-\rangle$\\
		$\hom_{\mathsf{dgCom}}(-,\Omega_\bullet)$ & $\langle -\rangle_S$\\
		\hline
	\end{tabular}
\end{center}

\medskip

\cite[Def. 2.1 and Thm. 2.8]{bfmt15} give a very satisfactory uniqueness result. They define a notion of \emph{sequence of compatible models for $\Delta$} which gives conditions for a cosimplicial Lie algebra to give sensible rational models for the geometric simplices, and they show that any two such objects are isomorphic. We refer the reader to the original reference for more details.

\medskip

The following theorem has non-empty intersection with our results. We say a shifted Lie algebra is of \emph{finite type} if it is finite dimensional in every degree and if its degrees are bounded either above or below.

\begin{theorem}[{\cite[Th. 8.1]{bfmt15}}]
	Let $\g$ be a dg Lie algebra of finite type with $H^n(\g,d) = 0$ for all $n>0$. Then there is a homotopy equivalence of simplicial sets
	\[
	\hom_{\mathsf{dgLie}}(\mc_\bullet,\g)\simeq\hom_{\mathsf{dgCom}}(\Bar_\iota(s\g)^\vee,\Omega_\bullet)\ .
	\]
\end{theorem}

We can easily recover an analogous result, which works on complete proper shifted Lie algebras of finite type such that $\g^{-1}=0$, but without restrictions on the cohomology, using our main theorem and some results of \cite{rn17tensor}.

\begin{proposition}
	Let $\g$ be a complete dg Lie algebra of finite type such that $\g^{-1}=0$. Then there is a weak equivalence of simplicial sets
	\[
	\hom_{\mathsf{dgLie}}(\mc_\bullet,\g)\simeq\hom_{\mathsf{dgCom}}(\Bar_\iota(s\g)^\vee,\Omega_\bullet)\ .
	\]
\end{proposition}

\begin{proof}
	The proof is given by the sequence of equivalences
	\begin{align*}
	\hom_{\mathsf{dgCom}}(\Bar_\iota(s\g)^\vee,\Omega_\bullet)\cong&\ \hom_{\mathsf{dgCom}}\left(\widehat{\Omega}_\pi(s^{-1}\g^\vee),\Omega_\bullet\right)\\
	\cong&\ \MC(\g\otimes\Omega_\bullet)\\
	\simeq&\ \hom_{\mathsf{dgLie}}(\mc_\bullet,\g).
	\end{align*}
	In the first line we used the natural isomorphism
	\[
	\Bar_\iota(s\g)^\vee\cong\widehat{\Omega}_\pi(s^{-1}\g^\vee)\ .
	\]
	Notice that the assumptions on $\g$ make it so that $\g^\vee$ is a $\lie^\vee$-coalgebra. In the second line we used a slight generalization of \cite[Cor. 6.6]{rn17tensor} for $\Q=\P=\com$ and $\Psi$ the identity morphism of $\com$. Notice that here the assumption that $\g^{-1}=0$ makes it so that
	\[
	\hom_{\mathsf{dgCom}}\left(\widehat{\Omega}_\pi(s^{-1}\g^\vee),\Omega_\bullet\right)\cong\hom(s^{-1}\g^\vee,\Omega_\bullet)^0
	\]
	even though $\Omega_\bullet$ is not complete. Finally, in the third line we used our Corollary \ref{cor:exMainThm}.
\end{proof}

\section{A model for Maurer--Cartan elements of homotopy Lie algebras}\label{sect:model for MC-el in Loo-alg}

One would like to do the same as in \cref{sect:cosimplicial model for Lie algebras} for $\SLoo$-algebras. It is possible to do so, but one critically has to use the generalized homotopy transfer theorem --- \cref{thm:generalized HTT} --- and consider the operad $\Cobar\Bar\com$ instead of $\C_\infty$ as cofibrant resolution of the operad $\com$. Doing this, we obtain a new cosimplicial $\SLoo$-algebra $\mcoo_\bullet$ such that
\[
\MC_\bullet(\g)\simeq\hom_{\SLooalg}(\mcoo_\bullet,\g)
\]
for any complete proper $\SLoo$-algebra $\g$. There is a natural morphism
\[
\rho:\mcoo_\bullet\longrightarrow\mc_\bullet
\]
through which every morphism of $\SLoo$-algebras from $\mcoo_\bullet$ to a strict Lie algebra splits in a canonical way, showing coherence with the results of \cref{sect:cosimplicial model for Lie algebras}. We will explicit $\mcoo_1$, which is a higher analogue of the Lawrence--Sullivan algebra in the context of homotopy Lie algebras. The work presented in this section is extracted from \cite{rnv18}.

\subsection{Representing \texorpdfstring{$\MC(\g\otimes C_\bullet)$}{MC(g x C.)}}\label{subsect:representation for Loo-alg}

Let $\g$ be an $\SLoo$-algebra. If one writes down the first operations for $\g\otimes C_\bullet$ with the $\SLoo$-algebra structure obtained through homotopy transfer along the contraction
\begin{center}
	\begin{tikzpicture}
	\node (a) at (0,0){$\g\otimes\Omega_\bullet$};
	\node (b) at (2.8,0){$\g\otimes C_\bullet$};
	
	\draw[->] (a)++(.6,.1)--node[above]{\mbox{\tiny{$1\otimes p_\bullet$}}}+(1.6,0);
	\draw[<-,yshift=-1mm] (a)++(.6,-.1)--node[below]{\mbox{\tiny{$1\otimes i_\bullet$}}}+(1.6,0);
	\draw[->] (a) to [out=-160,in=160,looseness=4] node[left]{\mbox{\tiny{$1\otimes h_\bullet$}}} (a);
	\end{tikzpicture}
\end{center}
and compares it with the the structure of the operations when $\g$ is a strict (shifted) Lie algebra, one immediately realizes that there are many more trees appearing in the former case than in the latter. This hints to the fact that we need a finer algebraic structure on $C_\bullet$ than the one of a $\C_\infty$-algebra we obtained by homotopy transfer in \cref{subsect:representing MC for Lie algebras}.

\medskip

In order to obtain a good model, we consider the resolution of the operad $\com$ given by $\Cobar\Bar\com$, and denote by
\[
\pi^\infty:\Bar(\SLoo)\longrightarrow\SLoo
\]
the canonical twisting morphism. Notice that $\Bar(\SLoo) \cong (\Cobar\Bar\com)^\vee$. In this section, we reserve the letter $\pi$ for the canonical twisting morphism
\[
\pi:\Bar(\susp\otimes\lie)\longrightarrow\susp\otimes\lie\ ,
\]
which played the same role in \cref{sect:cosimplicial model for Lie algebras} as $\pi^\infty$ will in this section.

\medskip

In Dupont's contraction
\begin{center}
	\begin{tikzpicture}
	\node (a) at (0,0){$\Omega_\bullet$};
	\node (b) at (2,0){$C_\bullet$};
	
	\draw[->] (a)++(.3,.1)--node[above]{\mbox{\tiny{$p_\bullet$}}}+(1.4,0);
	\draw[<-,yshift=-1mm] (a)++(.3,-.1)--node[below]{\mbox{\tiny{$i_\bullet$}}}+(1.4,0);
	\draw[->] (a) to [out=-150,in=150,looseness=4] node[left]{\mbox{\tiny{$h_\bullet$}}} (a);
	\end{tikzpicture}
\end{center}
we see the simplicial commutative algebra $\Omega_\bullet$ as a simplicial $\Cobar\Bar\com$-algebra. Then the generalized homotopy transfer theorem gives us a $\Cobar\Bar\com$-algebra structure on $C_\bullet$. We will denote the chain complex $C_\bullet$ endowed with this $\Cobar\Bar\com$-algebra structure by $\widetilde{C}_\bullet$. Since --- as already mentioned above --- we have
\[
\Bar(\SLoo) \cong (\Cobar\Bar\com)^\vee,
\]
we can take the cobar construction of $\widetilde{C}_\bullet^\vee$ relative to the twisting morphism $\pi^\infty$ to obtain a cosimplicial $\SLoo$-algebra.

\begin{definition}
	We denote the cosimplicial shifted homotopy Lie algebra obtained this way by $\mcoo_\bullet\coloneqq\widehat{\Cobar}_{\pi^\infty}(\widetilde{C}_\bullet^\vee)$.\index{$\mcoo_\bullet$}
\end{definition}

We immediately recover similar results as in \cref{subsect:representing MC for Lie algebras}.

\begin{theorem}\label{thm:cosimplicialModel SLoo}
	Let $\g$ be a proper complete $\SLoo$-algebra. There is a canonical isomorphism of simplicial sets
	\[
	\MC(\g\otimes C_\bullet)\cong\hom_{\SLooalg}(\mcoo_\bullet,\g)\ .
	\]
	It is natural in $\g$.
\end{theorem}

\begin{proof}
	Analogous to the proof of \cref{thm:cosimplicialModel}, using $\pi^\infty$ instead of $\pi$.
\end{proof}

\begin{corollary}
	Let $\g$ be a proper complete $\SLoo$-algebra. There is a weak equivalence of simplicial sets
	\[
	\MC_\bullet(\g)\simeq\hom_{\SLooalg}(\mcoo_\bullet,\g)\ .
	\]
	It is natural in $\g$.
\end{corollary}

\subsection{Compatibility with the strict case}

Since Maurer--Cartan elements, gauges and so on for strict shifted Lie algebras are the same as the respective notions as for the same algebras seen as $\SLoo$-algebras, one expects some compatibility between the cosimplicial algebras $\mc_\bullet$ and $\mcoo_\bullet$.

\medskip

There is a canonical morphism of operads
\[
g_\kappa:\SLoo\longrightarrow\susp\otimes\lie\ ,
\]
obtained by applying \cref{thm:universal twisting morphisms} to the canonical twisting morphism
\[
\kappa:(\susp\otimes\lie)^{\antishriek}\cong\susp^c\otimes\lie^{\antishriek}\longrightarrow\susp\otimes\lie
\]
given by Koszul duality. It induces a map
\[
\rho:\SLoo\circ C_\bullet^\vee\xrightarrow{g_\kappa\circ1_{C_\bullet^\vee}}(\susp\otimes\lie)\circ C_\bullet^\vee
\]
of chain complexes.

\begin{lemma}
	The map described above is a filtered morphism
	\[
	\rho:\mcoo_\bullet\longrightarrow g_\kappa^*\mc_\bullet
	\]
	of cosimplicial $\SLoo$-algebras.
\end{lemma}

\begin{proof}
	The fact that the map commutes with the algebraic structure is trivial, as well as the fact that it preserves the filtration, and that it respects the cosimplicial structure. In order to see that it also commutes with differentials, we begin by noticing that by \cref{prop:HTT and pullbacks} we have an equality of $\C_\infty$-algebras
	\[
	C_\bullet = (\Cobar f_\kappa)^*\widetilde{C}_\bullet\ ,
	\]
	where
	\[
	f_\kappa:\com^{\antishriek}\longrightarrow\Bar\com
	\]
	is the canonical map. The dual of $f_\kappa$ is the map of operads
	\[
	g_\kappa:\SLoo\longrightarrow\susp\otimes\lie
	\]
	mentioned above. Therefore, we have the diagram
	\begin{center}
		\begin{tikzpicture}
			\node (a) at (0,0){$\widetilde{C}_\bullet^\vee$};
			\node (b) at (4,0){$\Bar(\SLoo)\circ\widetilde{C}_\bullet^\vee$};
			\node (c) at (9,0){$\SLoo\circ\widetilde{C}_\bullet^\vee$};
			\node (d) at (0,-2.5){$C_\bullet^\vee$};
()			\node (e) at (4,-2.5){$\Bar(\susp\otimes\lie)\circ C_\bullet^\vee$};
			\node (f) at (9,-2.5){$(\susp\otimes\lie)\circ C_\bullet^\vee$};
			
			\draw[double equal sign distance] (a) to (d);
			\draw[->] (a) to node[above]{$\Delta_{\widetilde{C}_\bullet^\vee}$} (b);
			\draw[->] (b) to node[above]{$\pi^\infty\circ1$} (c);
			\draw[->] (d) to node[above]{$\Delta_{C_\bullet^\vee}$} (e);
			\draw[->] (e) to node[above]{$\pi\circ1$} (f);
			\draw[->] (b) to node[right]{$\Bar g_\kappa\circ1$} (e);
			\draw[->] (c) to node[right]{$g_\kappa\circ1$} (f);
		\end{tikzpicture}
	\end{center}
	which proves that $\rho$ commutes with the differentials of $\mc_\bullet$ and $\mcoo_\bullet$ (which are completely determined by the horizontal lines of the diagram above).
\end{proof}

\begin{remark}
	One would like to show that $\rho$ is a filtered quasi-isomorphism, so that $\mcoo_\bullet$ would be a resolution of $\mc_\bullet$. We do not have a proof of this fact for the moment.
\end{remark}

\begin{proposition}
	Let $\g$ be a shifted Lie algebra. Every morphism of $\SLoo$-algebras $\mcoo_\bullet\to g_\kappa^*\g$ splits in a unique way through $\rho$. In other words, the pullback
	\[
	\rho^*:\hom_{\dgl}(\mc_\bullet,\g)\cong\hom_{\SLooalg}(g_\kappa^*\mc_\bullet,g_\kappa^*\g)\longrightarrow\hom_{\SLooalg}(\mcoo_\bullet,g_\kappa^*\g)
	\]
	is an isomorphism of simplicial sets (the first isomorphism being given by the fact that shifted Lie algebras form a full subcategory of $\SLoo$-algebras).
\end{proposition}

\begin{proof}
	First, one notice that
	\[
	\g\otimes^\pi C_\bullet = (g_\kappa^*\g)\otimes^{\pi^\infty}\widetilde{C}_\bullet
	\]
	by the compatibility with morphisms expressed in \cref{thm:convolution Loo tensor products}. The proof is then given by the following commutative diagram.
	\begin{center}
		\begin{tikzpicture}
			\node (a) at (0,0) {$\hom_{\dgl}(\mc_\bullet,\g)$};
			\node (b) at (5,0) {$\MC(\g\otimes^\pi C_\bullet)$};
			\node (c) at (0,-2.5) {$\hom_{\SLooalg}(\mcoo_\bullet,g_\kappa^*\g)$};
			\node (d) at (5,-2.5) {$\MC((g_\kappa^*\g)\otimes^{\pi^\infty}\widetilde{C}_\bullet)$};
			
			\draw[->] (a) to node[left]{$\rho^*$} (c);
			\draw[double equal sign distance] (a) to node[above]{$\sim$} (b);
			\draw[double equal sign distance] (b) to node[above, sloped]{$\sim$} (d);
			\draw[double equal sign distance] (c) to node[above]{$\sim$} (d);
		\end{tikzpicture}
	\end{center}
	The horizontal arrows are given by \cref{thm:cosimplicialModel} and \cref{thm:cosimplicialModel SLoo} respectively, while proving that the vertical arrow is given by pullback by $\rho$ is just a matter of unwinding definitions.
\end{proof}

\subsection{A higher version of the Lawrence--Sullivan algebra}

We want now to give explicit formul{\ae} for the first two levels of $\mcoo_\bullet$. The same arguments as for \cref{prop:first levels of mc} apply, and thus we immediately have the following.

\begin{lemma}
	We have
	\[
	\mcoo_0 = \widehat{\SLoo}(\k\alpha)
	\]
	with $|\alpha| = 0$ and
	\[
	d\alpha = -\sum_{n\ge2}\frac{1}{n!}\ell_n(\alpha,\ldots,\alpha)\ .
	\]
	In other words, $\mcoo_0$ is the free $\SLoo$-algebra generated by a single Maurer--Cartan element.
\end{lemma}

We also know that $\mcoo_1$ is the free $\SLoo$-algebra generated by two Maurer--Cartan elements $\alpha_0,\alpha_1$ in degree $0$, and a gauge $\lambda$ from $\alpha_0$ to $\alpha_1$ in degree $1$. In order to translate this into explicit formul{\ae} we will need to do some work. Our starting point is \cref{subsect:explicit formula for gauges}. From there, we can deduce the formula
\[
\alpha_1 = \sum_{\tau\in\PT}\frac{1}{F(\tau)}\tau(\alpha_0)
\]
using \cref{prop:formula for formal ODE} to the formal differential equation with
\[
f_{n,1}(y_1,\ldots,y_n)\coloneqq\frac{1}{n!}\ell_{n+1}(y_1,\ldots,y_n,\lambda)
\]
for $n\ge1$, $f_{0,1}\coloneqq d\lambda$, and no other operators. Notice that the only weight appearing in the trees is $1$, so that we simply work over $\PT$. Rearranging terms, we obtain the fixed-point equation
\[
d\lambda = \alpha_1 - \alpha_0 - \sum_{\tau\in\PT\backslash\{\emptyset,c_0\}}\frac{1}{F(\tau)}\tau(\alpha_0)\ ,
\]
where one should notice that $d\lambda$ appears in the right-hand side whenever a tree contains a $0$-corolla. The existence and uniqueness of a solution is thus guaranteed by \cref{thm:existence and uniqueness of solution for fixed point equations}. We can however give an explicit formula for this solution in this case.

\medskip

We will denote by $\PTt$ the following set of trees. An element of $\PTt$ is a non-empty\footnote{Meaning that their set of vertices is non-empty.} planar rooted tree $T$ such that
\begin{enumerate}
	\item every vertex of $T$ has arity at least $1$,
	\item every leaf of $T$ is labeled by either ``black" or ``white", and
	\item every vertex $v$ of $T$, say of arity $k$, is labeled by a planar tree $\tau_v$ of arity $k$ and whose vertices all have arity at least $1$.
\end{enumerate}
Here are some examples of such trees:
\begin{center}
	\begin{tikzpicture}
		\begin{scope}
			\coordinate (r1) at (0,0); 
			\coordinate (v11) at (0,1);
			\coordinate (v12) at (-1,2);
			\coordinate (v13) at (-1.5,2.5);
			\coordinate (v14) at (1,2.5);
			\coordinate (l11) at (-2,3);
			\coordinate (l12) at (-1.25,3);
			\coordinate (l13) at (-0.75,3);
			\coordinate (l14) at (0,3);
			\coordinate (l15) at (1,3);
			
			\draw (r1) to (l14);
			\draw (v11) to (l11);
			\draw (v13) to (l12);
			\draw (v12) to (l13);
			\draw (v11) to (v14);
			\draw (v14) to (l15);
			\fill (v11) circle [radius=2pt];
			\fill (v12) circle [radius=2pt];
			\fill (v13) circle [radius=2pt];
			\fill (v14) circle [radius=2pt];
			\draw (l11)+(-2pt,0pt) rectangle +(2pt,4pt);
			\fill (l12)+(-2pt,0pt) rectangle +(2pt,4pt);
			\fill (l13)+(-2pt,0pt) rectangle +(2pt,4pt);
			\draw (l14)+(-2pt,0pt) rectangle +(2pt,4pt);
			\draw (l15)+(-2pt,0pt) rectangle +(2pt,4pt);
			\draw (v11) circle [radius=12pt];
			\draw (-1.25,2.25) circle [radius=15pt];
			\draw (v14) circle [radius=8pt];
		\end{scope}
		
		\begin{scope}[shift = {(5,0)}]
			\coordinate (r2) at (0,0);
			\coordinate (v21) at (0,.75);
			\coordinate (v22) at (0,1.5);
			\coordinate (v23) at (0,2.5);
			\coordinate (l21) at (0,3);
			
			\draw (r2) to (l21);
			\fill (v21) circle [radius=2pt];
			\fill (v22) circle [radius=2pt];
			\fill (v23) circle [radius=2pt];
			\draw (l21)+(-2pt,0pt) rectangle +(2pt,4pt);
			\draw (0,1.125) circle [radius=15pt];
			\draw (v23) circle [radius=8pt];
		\end{scope}
	\end{tikzpicture}
\end{center}
and some non-examples:
\begin{center}
	\begin{tikzpicture}
		\begin{scope}
			\coordinate (r1) at (0,0);
			\coordinate (v11) at (0,1);
			\coordinate (l11) at (-1,2);
			\coordinate (l12) at (0,2);
			\coordinate (l13) at (1,2);
			
			\draw (r1) to (l12);
			\draw (v11) to (l11);
			\draw (v11) to (l13);
			\fill (v11) circle [radius=2pt];
			\draw (v11) circle [radius=12pt];
			\draw (l11)+(-2pt,0pt) rectangle +(2pt,4pt);
			\fill (l13)+(-2pt,0pt) rectangle +(2pt,4pt);
		\end{scope}
		
		\begin{scope}[shift = {(4,0)}]
			\coordinate (r2) at (0,0);
			\coordinate (v21) at (0,1);
			\coordinate (l21) at (-.5,1.5);
			\coordinate (l22) at (0,2);
			\coordinate (l23) at (1,2);
			
			\draw (r2) to (l22);
			\draw (v21) to (l21);
			\draw (v21) to (l23);
			\fill (v21) circle [radius=2pt];
			\fill (l21) circle [radius=2pt];
			\draw (-.25,1.25) circle [radius=15pt];
			\draw (l23)+(-2pt,0pt) rectangle +(2pt,4pt);
			\draw (l22)+(-2pt,0pt) rectangle +(2pt,4pt);
		\end{scope}
		
		\begin{scope}[shift = {(8,0)}]
			\coordinate (r3) at (0,0);
			\coordinate (v31) at (0,1);
			\coordinate (l31) at (-1.2,2);
			\coordinate (l32) at (-.4,2);
			\coordinate (l33) at (.4,2);
			\coordinate (l34) at (1.2,2);
			
			\draw (r3) to (v31);
			\draw (v31) to (l31);
			\draw (v31) to (l32);
			\draw (v31) to (l33);
			\draw (v31) to (l34);
			\fill (v31) circle [radius=2pt];
			\draw (l31)+(-2pt,0pt) rectangle +(2pt,4pt);
			\draw (l32)+(-2pt,0pt) rectangle +(2pt,4pt);
			\fill (l33)+(-2pt,0pt) rectangle +(2pt,4pt);
			\fill (l34)+(-2pt,0pt) rectangle +(2pt,4pt);
		\end{scope}
	\end{tikzpicture}
\end{center}
The trees given by a single black or white leaf --- which we will denote by $\blackleaf$ and $\whiteleaf$ respectively --- are not in $\PTt$, since they don't have any vertex.

\medskip

Let $T\in\PTt$, and let $v$ be a vertex of $T$ of arity $n$ with associated planar tree $\tau_v$. Then we denote by $\overline{\tau_v}$ the planar tree obtained by taking $\tau_v$ and composing a $0$-corolla at every leaf that is linked to either another vertex of $T$ or to a black leaf of $T$. For example,
\begin{center}
	\begin{tikzpicture}
		\begin{scope}
			\coordinate (r1) at (0,0); 
			\coordinate (v11) at (0,1);
			\coordinate (v12) at (-1,2);
			\coordinate (v13) at (-1.5,2.5);
			\coordinate (v14) at (1,2.5);
			\coordinate (l11) at (-2,3);
			\coordinate (l12) at (-1.25,3);
			\coordinate (l13) at (-0.75,3);
			\coordinate (l14) at (0,3);
			\coordinate (l15) at (1,3);
			
			\draw (r1) to (l14);
			\draw (v11) to (l11);
			\draw (v13) to (l12);
			\draw (v12) to (l13);
			\draw (v11) to (v14);
			\draw (v14) to (l15);
			\fill (v11) circle [radius=2pt];
			\fill (v12) circle [radius=2pt];
			\fill (v13) circle [radius=2pt];
			\fill (v14) circle [radius=2pt];
			\draw (l11)+(-2pt,0pt) rectangle +(2pt,4pt);
			\fill (l12)+(-2pt,0pt) rectangle +(2pt,4pt);
			\fill (l13)+(-2pt,0pt) rectangle +(2pt,4pt);
			\draw (l14)+(-2pt,0pt) rectangle +(2pt,4pt);
			\draw (l15)+(-2pt,0pt) rectangle +(2pt,4pt);
			\draw (v11) circle [radius=12pt];
			\draw (-1.25,2.25) circle [radius=15pt];
			\draw (v14) circle [radius=8pt];
			
			\node at (-1.75,1.75){$v_1$};
			\node at (-.455,.545){$v_2$};
			\node at (1.4,2.1){$v_3$};
		\end{scope}
		
		\begin{scope}[shift = {(5,1.75)}]
			\coordinate (r2) at (0,0);
			\coordinate (v21) at (0,.75);
			\coordinate (l21) at (-.75,1.5);
			\coordinate (l22) at (0,1.5);
			\coordinate (l23) at (.75,1.5);
			
			\draw (r2) to (l22);
			\draw (v21) to (l21);
			\draw (v21) to (l23);
			\fill (v21) circle [radius=2pt];
			\fill (l22) circle [radius=2pt];
			\fill (l23) circle [radius=2pt];
			
			\node at (-1.2,.75){$\overline{\tau_{v_1}} =$};
		\end{scope}
		
		\begin{scope}[shift = {(5,-.25)}]
		\coordinate (r2) at (0,0);
		\coordinate (v21) at (0,.75);
		\coordinate (l21) at (-.75,1.5);
		\coordinate (l22) at (0,1.5);
		\coordinate (l23) at (.75,1.5);
		
		\draw (r2) to (l22);
		\draw (v21) to (l21);
		\draw (v21) to (l23);
		\fill (v21) circle [radius=2pt];
		\fill (l21) circle [radius=2pt];
		\fill (l23) circle [radius=2pt];
		
		\node at (-1.2,.75){$\overline{\tau_{v_2}} =$};
		\end{scope}
		
		\begin{scope}[shift = {(8,1)}]
			\coordinate (r3) at (0,0);
			\coordinate (v31) at (0,.75);
			\coordinate (l31) at (0,1.5);
			
			\draw (r3) to (l31);
			\fill (v31) circle [radius=2pt];
			
			\node at (-.5,.75){$\overline{\tau_{v_3}} =$};
		\end{scope}
	\end{tikzpicture}
\end{center}
We define a function
\[
G:\PTt\longrightarrow\k
\]
by
\[
G(T)\coloneqq\prod_{v\in V_T}-\frac{1}{F(\overline{\tau_v})}\ .
\]
To a tree $T\in\PTt$ we also associate a function
\[
T:\g\times\g\longrightarrow\g
\]
recursively by
\[
T(x,y)\coloneqq\tau_r(T_1(x,y),\ldots,T_k(x,y))
\]
whenever $T = \tau_r\circ(T_1,\ldots,T_k)$, and by setting $\whiteleaf(x,y)\coloneqq x$, and $\blackleaf(x,y)\coloneqq y$, even though the trees given by a single white or black leaf are not in $\PTt$, strictly speaking.

\begin{proposition}\label{prop:structure of mcoo1}
	The $\L_\infty$-algebra $\mcoo_1$ is the free complete $\L_\infty$-algebra generated by two degree $1$ elements $\alpha_0$ and $\alpha_1$ and a single degree $0$ element $\lambda$ satisfying
	\[
	d\lambda= \alpha_1 - \alpha_0 + \sum_{T\in\PTt}G(T)\,T(\alpha_0,\alpha_1-\alpha_0)\ .
	\]
\end{proposition}

This $\SLoo$-algebra is a higher analogue of the Lawrence--Sulliven algebra \cite{ls10}, which is a Lie model for the $1$-simplex, cf. also \cref{prop:first levels of mc}.

\begin{proof}
	We have to prove that this $d\lambda$ is the unique solution to the fixed-point equation
	\[
	d\lambda = \alpha_1 - \alpha_0 - \sum_{\tau\in\PT\backslash\{\emptyset,c_0\}}\frac{1}{F(\tau)}\tau(\alpha_0)\ ,
	\]
	obtained above. Notice that $d\lambda$ appears in the right-hand side whenever there is a $0$-corolla in a tree. The fact that the solution exists and is unique is given by \cref{thm:existence and uniqueness of solution for fixed point equations}.
	
	\medskip
	
	We have
	\small
	\begin{align*}
	d\lambda =&\ \alpha_1 - \alpha_0 + \sum_{T\in\PTt}G(T)\,T(\alpha_0,\alpha_1-\alpha_0)\\
	=&\ \alpha_1 - \alpha_0 - \sum_{\substack{k\ge1,\ \tau_r\in\PT_k\text{ reduced}\\T_1,\ldots,T_k\in\PTt\cup\{\whiteleaf,\blackleaf\}}}\frac{1}{F(\overline{\tau_r})}\left(\prod_{i=1}^kG(T_i)\right)\tau_r(T_1(\alpha_0,\alpha_1-\alpha_0),\ldots,T_k(\alpha_0,\alpha_1-\alpha_0))\\
	=&\ \alpha_1 - \alpha_0 -\\
	&\ -\sum_{\substack{k\ge1,\\\tau_r\in\PT_k\text{ red.}}}\frac{1}{F(\overline{\tau_r})}\tau_r\left(\sum_{T_1\in\PTt\cup\{\whiteleaf,\blackleaf\}}G(T_1)\,T_1(\alpha_0,\alpha_1-\alpha_0),\ldots,\sum_{T_k\in\PTt\cup\{\whiteleaf,\blackleaf\}}G(T_k)\,T_k(\alpha_0,\alpha_1-\alpha_0)\right)\\
	=&\ \alpha_1 - \alpha_0 - \sum_{\substack{k\ge1,\\\tau_r\in\PT_k\text{ reduced}}}\frac{1}{F(\overline{\tau_r})}\tau_r\left(\alpha_0 + d\lambda,\ldots,\alpha_0 + d\lambda\right)\\
	=&\ \alpha_1 - \alpha_0 - \sum_{\tau\in\PT\backslash\{\emptyset,c_0\}}\frac{1}{F(\tau)}\tau\left(\alpha_0\right)
	\end{align*}
	\normalsize
	where a tree is \emph{reduced} if it is non-empty and has no arity $0$ vertices, in the second line we have $T = \tau_r\circ(T_1,\ldots,T_k)$, and $G(\whiteleaf) = G(\blackleaf)\coloneqq1$.
\end{proof}

\subsection{Homotopies between \texorpdfstring{$\infty$}{infinity}-morphisms of \texorpdfstring{$\SLoo$-}{shifted homotopy Lie }algebras}

We sketch how, in principle, one can use the formul{\ae} for $\mcoo_1$ given above to describe an explicit notion of homotopy between $\infty$-morphisms of $\SLoo$-algebras. More details will appear in \cite{rnv18}.

\medskip

Let $\g$ be a $\SLoo$-algebra, and consider the contraction
\begin{center}
	\begin{tikzpicture}
	\node (a) at (0,0){$\g\otimes\Omega_1$};
	\node (b) at (2.8,0){$\g\otimes C_1$};
	
	\draw[->] (a)++(.6,.1)--node[above]{\mbox{\tiny{$1\otimes p$}}}+(1.6,0);
	\draw[<-,yshift=-1mm] (a)++(.6,-.1)--node[below]{\mbox{\tiny{$1\otimes i$}}}+(1.6,0);
	\draw[->] (a) to [out=-160,in=160,looseness=4] node[left]{\mbox{\tiny{$1\otimes h$}}} (a);
	\end{tikzpicture}
\end{center}
induced by Dupont's contraction. Endow $\g\otimes C_1$ with the $\SLoo$-algebra structure obtained via the homotopy transfer theorem.

\begin{proposition}[{\cite[Prop. 3.3]{val14}}]
	The $\SLoo$-algebra $\g\otimes C_1$ is a cylinder for $\g$ in the category of $\SLoo$-algebras with their $\infty$-morphisms. In other words, the cocommutative coalgebra $\Bari(\g\otimes C_1)$ is a cylinder for $\Bari\g$.
\end{proposition}

But by \cref{thm:compatibility htt and convolution algebras on coalgebra side}, $\g\otimes C_1$ is the same thing as $\g\otimes^{\pi^\infty}\widetilde{C}_1$, where $\widetilde{C}_1$ is $C_1$ the $\Cobar\Bar\com$-algebra structure obtained as in \cref{subsect:representation for Loo-alg}. Therefore, if $\Phi,\Psi:\g\rightsquigarrow\h$ are two $\infty$-morphisms, then a homotopy between them is an $\infty$-morphism
\[
H:\g\otimes^{\pi^\infty}\widetilde{C}_1\rightsquigarrow\h\ ,
\]
which translates into a certain collection of morphisms $\g^{\otimes n}\to\h$, of which two extremal subcollections give back $\Phi$ and $\Psi$, and which must respect certain compatibilities encoded by the algebraic structure of $\widetilde{C}_1$. Moreover, one can give an explicit description of the algebraic structure of $\widetilde{C}_1$, which is essentially dual to the $\SLoo$-algebra structure of $\mcoo_1$, using \cref{prop:structure of mcoo1}.

%% file: RationalModels.tex
\chapter{Rational models for mapping spaces}\label{ch:rational models mapping spaces}

In this chapter, we present an application of the theory developed in \cref{chapter:convolution homotopy algebras} to rational homotopy theory, using it to generalize a result of Berglund \cite{ber15}. Most of the material presented here is extracted from \cite{rnw17}.

\medskip

Another application of the methods of \cref{chapter:convolution homotopy algebras} to rational homotopy theory, but which is outside of the scope of the present work, is given in \cite{w16} and \cite{w17}, where F. Wierstra uses them to construct a complete invariant of the real or rational homotopy classes of maps between simply connected manifolds.

\medskip

In this chapter, we work exclusively over the field $\mathbb{Q}$ of rational numbers.

\section{Homotopy Lie and homotopy cocommutative models}

In this section, we define homotopy Lie and homotopy commutative models for spaces, and define certain additional conditions we will impose on the homotopy Lie models.

\subsection{Conditions on \texorpdfstring{$\SLoo$-}{shifted homotopy Lie }algebras}

Later, we will consider complete proper $\SLoo$-algebras satisfying some additional finiteness conditions, which we introduce here.

\begin{definition}
	A proper complete $\SLoo$-algebra $(\g,\F_\bullet\g)$ is \emph{locally finite}\index{Locally finite} if all of the quotients
	\[
	\g^{(n)}\coloneqq \g/\F_n\g\ ,\qquad\text{for }n\ge2\ ,
	\]
	are finite dimensional.
\end{definition}

Being locally finite is a good generalization of being finite dimensional, as the following result shows.

\begin{lemma}\label{lemma:isomHomTensor}
	Let $(\g,\F_\bullet\g)$ be a proper complete $\L_\infty$-algebra, and let $C$ be a cocommutative coalgebra. Suppose that either
	\begin{enumerate}
		\item the cocommutative coalgebra $C$ is finite dimensional, or
		\item the proper complete $\SLoo$-algebra $(\g,\F_\bullet\g)$ is locally finite.
	\end{enumerate}
	Then we have an isomorphism
	\[
	\g\widehat{\otimes}C^\vee\cong\hom^\iota(C,\g)
	\]
	of $\SLoo$-algebras, where $\iota:\com^\vee\to\SLoo$ is the canonical twisting morphism.
\end{lemma}

Here, we use the notation
\[
\g\widehat{\otimes}C^\vee\coloneqq\lim_n\big(\g/\F_n\g\otimes C^\vee\big)\ .
\]
It can be seen as a tensor product between filtered algebras, where we have endowed $C^\vee$ with the constant filtration $\F_nC^\vee = C^\vee$.

\begin{proof}
	The first case is straightforward, so we only give some details for the second one. First begin by considering the case where $\g$ is finite dimensional. If we fix a homogeneous basis $\{x_i\}_i$ of $\g$, then we obtain an isomorphism
	\[
	\hom^\iota(C,\g)\longrightarrow\g\otimes^\iota C^\vee
	\]
	by sending
	\[
	\phi\longmapsto\sum_ix_i\otimes(x_i^\vee\phi)\ .
	\]
	It is a straightforward exercise to check that this is independent of the chosen basis, and to see that the isomorphism holds true at the level of $\SLoo$-algebras.
	
	\medskip
	
	Now if $\g$ is not necessarily finite dimensional, but only locally finite, we have
	\begin{align*}
	\hom^\iota(C,\g) \cong&\ \hom^\iota(C,\lim_n\g/\F_n\g)\\
	\cong&\ \lim_n\hom^\iota(C,\g/\F_n\g)\\
	\cong&\ \lim_n\big(\g/\F_n\g\otimes C^\vee\big)\\
	=&\ \g\widehat{\otimes}C^\vee\ ,
	\end{align*}
	where the fact that the second isomorphism holds at the level of $\SLoo$-algebras is straightforward to check, and in the third line we used the fact that $\g^{(n)}$ is finite dimensional for all $n$ in order to apply what said above.
\end{proof}

There is another condition we will impose on some of our $\SLoo$-algebras. It was first introduced in \cite{ber15}.

\begin{definition}\label{def:degree-wise nilpotent}
	Let $(\g,\F_\bullet\g)$ be a proper complete $\SLoo$-algebra. We say that $(\g,\F_\bullet\g)$ is \emph{degree-wise nilpotent}\index{Degree-wise nilpotent} if for any $n\in\mathbb{Z}$ there is a $k\ge1$ such that $(\F_k\g)_n = 0$.
\end{definition}

The functor $\MC_\bullet$ acts in a very straightforward manner on degree-wise nilpotent filtered $\L_\infty$-algebras satisfying a boundedness condition with respect to the homological degree, as the following result demonstrates.

\begin{proposition}\label{prop:MC with deg-wise nilpotent filtration}
	Let $(\g,\F_\bullet\g)$ be a degree-wise nilpotent proper complete $\SLoo$-algebra, and suppose that the degrees in which $\g$ is non-zero are bounded below. Then
	\[
	\MC_\bullet(\g,\F\g)\cong\MC(\g\otimes\Omega_\bullet)\ .
	\]
	In particular, $\MC_\bullet(\g,\F_\bullet\g)$ is independent of the filtration $\F_\bullet\g$, as long as $(\g,\F_\bullet\g)$ is degree-wise nilpotent.
\end{proposition}

\begin{proof}
	Suppose that $(\g,\F_\bullet\g)$ satisfies the assumptions above. Fix $n\ge0$, then there exists $k_0\ge1$ such that $(\F_{k_0}\g)_{m} = 0$ for all $m\le n$, since the degrees of $\g$ are bounded below. It follows that
	\[
	(\F_k\g\otimes\Omega_n)_0 = 0\qquad\text{and}\qquad(\F_k\g\otimes\Omega_n)_{-1} = 0
	\]
	for all $k\ge k_0$, as $\Omega_n$ is concentrated in degrees from $-n$ up to $0$. The projection
	\[
	\g\otimes\Omega_n\longrightarrow\g/\F_k\g\otimes\Omega_n
	\]
	has kernel $\F_k\g\otimes\Omega_n$, and is therefore an isomorphism in degrees $0$ and $-1$. Therefore, the set of Maurer--Cartan elements $\MC(\g/\F_k\g\otimes\Omega_n)$ is constant in $k$ for $k\ge k_0$, which implies the statement.
\end{proof}

For any $\SLoo$-algebra $\g$, one can consider the filtration $\F^\mathrm{can}\g$ which at level $n$ is given by the elements that can be obtained by bracketing at least $n$ elements of $\g$.

\begin{lemma}\label{lemma:sccftGivesDegWiseNil}
	Let $\g$ be a simply-connected, proper complete $\SLoo$-algebra. Then $(\g,\F^\mathrm{can}\g)$ is degree-wise nilpotent.
\end{lemma}

\begin{proof}
	Suppose we take $k$ elements and we bracket them together. We can use at most $k-1$ brackets (taking only binary brackets), and every element has degree at least $2$. It follows that the resulting element has degree at least
	\[
	(1-k) + 2k = k + 1\ .
	\]
	Thus,
	\[
	\F_k^\mathrm{can}\g\subseteq\g_{\ge k+1}\ ,
	\]
	and the statement follows.
\end{proof}

\subsection{Homotopy Lie models}

One possible generalization of the Lie rational models for spaces of \cref{subsect:Lie models for spaces} is to use homotopy Lie algebras. These have been considered e.g. in \cite{bfm11}, \cite{laz13}, and \cite{ber15}. For coherence, we work in the shifted setting.

\begin{definition}
	Let $K\in\ssets_{1}$ be a $1$-reduced simplicial set. A proper complete $\SLoo$-algebra $(\g,\F_\bullet\g)$ is a \emph{rational model}\index{Homotopy!Lie rational model} for $K$ if there is a homotopy equivalence
	\[
	\MC_\bullet(\g,\F_\bullet\g)\simeq K_\mathbb{Q}\ ,
	\]
	where $K_\mathbb{Q}$ is the rationalization of $K$, cf. \cref{thm:rationalization exists}.
\end{definition}

We will require that our $\SLoo$-models are locally finite and degree-wise nilpotent. This assumption is needed e.g. for \cref{thm:Berglund} to hold. This condition is not that strong, as we can model many spaces of interest to us through such algebras.

\begin{proposition}\label{prop:condition ok}
	Every $1$-reduced simplicial set $K$ with only finitely many non-degenerate simplices admits a locally finite, degree-wise nilpotent $\SLoo$-model.
\end{proposition}

\begin{proof}
	According to \cite[Sect. 24.(e)]{FelixHalperinThomas}, every such simplicial set $K$ admits a free Lie model --- in the sense of the Definition found at \cite[p.322]{FelixHalperinThomas} --- of the form $\g=\lie(C_\bullet(X))$, where $C_\bullet(X)$ is the complex of simplicial chains of $K$. Since we supposed that $K$ is finite, $\g$ is a finitely generated Lie algebra, and thus it satisfies the assumptions of \cite[Prop. 6.1]{ber15}. It follows that we have
	\[
	\MC_\bullet(\g\otimes\Omega_\bullet)\simeq K_\mathbb{Q}\ .
	\]
	An apparent problem is the fact that $\g$ is not complete. However, we can replace $\g$ by
	\[
	\widehat{\g}\coloneqq\widehat{\lie}(C_\bullet(X)) = \prod_{n\ge1}\lie(n)\otimes_{\S_n}C_\bullet(X)^{\otimes n}\ .
	\]
	Then \cref{prop:MC with deg-wise nilpotent filtration} and \cref{lemma:sccftGivesDegWiseNil} give
	\[
	\MC_\bullet(\g\otimes\Omega_\bullet)\cong\MC_\bullet(\widehat{\g},\F^{can}\widehat{\g})\ ,
	\]
	so that we may use $(\widehat{\g},\F^\mathrm{can}\widehat{\g})$ as a $\L_\infty$-model. This is obviously locally finite, since the operad $\lie$ is finite dimensional in every arity and $C_\bullet(X)$ is finite dimensional, and degree-wise nilpotent e.g. by \cref{lemma:sccftGivesDegWiseNil}. Now suspend $\g$ to obtain an $\SLoo$-model satisfying the desired properties.
\end{proof}

\subsection{Homotopy cocommutative models}

By a cocommutative coalgebra up to homotopy, we mean a conilpotent coalgebra over the cooperad $\Bar\Cobar\com^\vee$. Other resolutions of the cooperad $\com^\vee$ are possible, but the chosen one has the advantage to provide the canonical commuting diagram
\begin{center}
	\begin{tikzpicture}
	\node (a) at (0,1.5){$\com^\vee$};
	\node (b) at (0,0){$\Bar\Cobar\com^\vee$};
	\node (d) at (3,0){$\L_\infty$};
	
	\draw[->] (a) -- node[left]{$f_\iota$} (b);
	\draw[->] (a) edge[out=0,in=120] node[above]{$\iota$} (d);
	\draw[->] (b) -- node[above]{$\pi$} (d);
	\end{tikzpicture}
\end{center}
where the quasi-isomorphism $f_\iota$ is obtained by \cref{thm:universal twisting morphisms} --- it is in fact the unit of the bar-cobar adjunction --- and where all twisting morphisms are Koszul.

\begin{definition}
	Let $X$ be a simply-connected, pointed topological space. A homotopy cocommutative coalgebra $C$ is a \emph{rational model}\index{Homotopy!cocommutative rational model} for $X$ if there exists a zig-zag of weak equivalences of homotopy cocommutative algebras
	\[
	(f_\iota)_*s\Bar_\kappa\lambda(X)\longrightarrow\bullet\longleftarrow\cdots\longrightarrow\bullet\longleftarrow C\ ,
	\]
	or equivalently (by \cref{prop:alpha-we of coalgebras is zig-zag of we}) if there exists a $\pi$-weak equivalence $(f_\iota)_*s\Bar_\kappa\lambda(X)\rightsquigarrow C$.
\end{definition}

\section{Rational models for mapping spaces}

Given two spaces $K$ and $L$, a natural question is the following one. Suppose we are given rational models for both $K$ and $L$. Is it possible to use them to construct a rational model of the mapping space $\map(K,L)$? A possible answer to this question was given by Berglund \cite{ber15} in the case when we have a strictly commutative model for the first space, and an $\SLoo$-model for the second one.

\begin{theorem}[\cite{ber15} Theorem 6.3]\label{thm:Berglund}
	Let $K$ be a simply-connected simplicial set, let $L$ be a nilpotent space (e.g. a simply-connected space) of finite $\mathbb{Q}$-type and $L_{\mathbb{Q}}$ the rationalization of $L$. Let $A$  be a commutative model for $K$ and $(\g,\F_\bullet\g)$ a degree-wise nilpotent, locally finite $\SLoo$-model of finite type for $L$. There is a homotopy equivalence of simplicial sets
	\[
	\map(K,L_{\mathbb{Q}})\simeq\MC_\bullet(A \widehat{\otimes}\g) ,
	\]
	i.e. the $\SLoo$-algebra $A \widehat{\otimes }\g$ is a $\SLoo$-model for the mapping space.
\end{theorem}

\begin{remark}
	In \cite{ber15}, this theorem is stated in terms of the Getzler $\infty$-groupoid $\gamma_{\bullet}(\g)$. However, the $\infty$-groupoid $\gamma_{\bullet}(\g)$ is homotopy equivalent to $\MC_\bullet(\g)$ by \cref{thm:getzler gamma is we to MC.}, and thus the statement above is equivalent to the original one. Also notice that we supposed that $(\g,\F_\bullet\g)$ is locally finite, and completed the tensor product with respect to the filtration $\F_\bullet\g$ and not with the degree filtration, as in \cite{ber15}. An inspection of the original proof reveals that the result still holds in this slightly more general context.
\end{remark}

We will now improve Berglund's Theorem in two ways: we will show that we can take homotopy cocommutative coalgebra models for $K$ instead of just cocommutative ones, and that this model is natural with respect to $\infty_\pi$-quasi-isomorphisms of  $\SLoo$-algebras, respectively $\pi$-weak equivalences of homotopy cocommutative coalgebras. We will also show that, under certain restrictions on $C$ and $\g$, this model only depends on the homotopy types of $C$ and $\g$, i.e. different choices for $C$ and $\g$ will give homotopy equivalent models for the mapping space. The first result is the following one.

\begin{lemma}
	Let $K$ be a simply-connected simplicial set, let $L$ be a nilpotent space (e.g. a simply-connected space) of finite $\mathbb{Q}$-type and $L_\mathbb{Q}$ the $\mathbb{Q}$-localization of $L$. Let $C$  be a cocommutative model for $K$ and $(\g,\F_\bullet\g)$ a degree-wise nilpotent, locally finite $\SLoo$-model of finite type for $L$. There is a homotopy equivalence of simplicial sets
	\[
	Map_*(K,L_\mathbb{Q}) \simeq MC_\bullet(\hom^\pi((f_\iota)_*C,\g)) ,
	\]
	i.e. the convolution $\SLoo$-algebra $\hom^\pi((f_\iota)_*C,\g)$ is an $\L_\infty$-model for the mapping space.
\end{lemma}

\begin{proof}
	By \cref{thm:dualization of (co)commutative models} and \cref{thm:Berglund}, we know that $\g\otimes C^\vee$ is an $\L_\infty$-model for the mapping space. Further, by \cref{lemma:isomHomTensor} we know that
	\[
	\g\widehat{\otimes}C^\vee\cong\hom^\iota(C,\g) = \hom^\pi((f_\iota)_*C,\g)\ ,
	\]
	where the second equality is \cref{lemma:naturality of convolution algebras wrt compositions}.
\end{proof}

\begin{proposition}
	Let $C$ be a homotopy cocommutative coalgebra, and let
	\[
	\Psi:(\h,\F_\bullet\h)\rightsquigarrow(\g,\F_\bullet\g)\ .
	\]
	be a filtered $\infty$-morphism of $\SLoo$-algebras. Then there is a weak equivalence of simplicial sets
	\[
	\MC_\bullet(\hom^\pi(C,\h))\simeq\MC_\bullet(\hom^\pi(C,\g))\ .
	\]
\end{proposition}

\begin{proof}
	This looks very similar to \cref{cor:what happens to MC.(conv. alg.) under filtered oo-qi} --- and indeed, that result is a fundamental ingredient of the proof --- but notice that here we have $\infty$-morphisms of $\SLoo$-algebras, i.e. $\infty_\iota$-morphisms, instead of $\infty_\pi$-morphisms.
	
	\medskip
	
	The proof is schematized by the following diagram.
	\begin{center}
		\begin{tikzpicture}
		\node (a) at (0,4){$\hom^\pi(C,\h)$};
		\node (b) at (0,2){$\hom^\pi(R_{\pi,f_\iota}(C),\h)$};		\node (c) at (0,0){$\hom^\iota(\Bar_\iota\Cobar_\pi C,\h)$};
		\node (d) at (6,0){$\hom^\iota(\Bar_\iota\Cobar_\pi C,\g)$};
		\node (e) at (6,2){$\hom^\pi(R_{\pi,f_\iota}(C),\g)$};
		\node (f) at (6,4){$\hom^\pi(C,\g)$};
		
		\draw[->,line join=round,decorate,decoration={zigzag,segment length=4,amplitude=.9,post=lineto,post length=2pt}] (b) -- node[left]{$\hom^\pi(E_D,1)$} (a);
		\draw[double equal sign distance] (b) -- (c);
		\draw[->,line join=round,decorate,decoration={zigzag,segment length=4,amplitude=.9,post=lineto,post length=2pt}] (c) -- node[above]{$\hom^\iota(1,\Psi)$} (d);
		\draw[double equal sign distance] (d) -- (e);
		\draw[->,line join=round,decorate,decoration={zigzag,segment length=4,amplitude=.9,post=lineto,post length=2pt}] (e) -- node[left]{$\hom^\pi(E_D,1)$} (f);
		\end{tikzpicture}
	\end{center}
	The vertical equalities are given by \cref{lemma:naturality of convolution algebras wrt compositions} and the definition of the rectification $R_{\pi,f_\iota}$. We apply the functor $\MC_\bullet$ on the whole diagram, and all the squiggly arrows become homotopy equivalences of simplicial sets by \cref{cor:what happens to MC.(conv. alg.) under filtered oo-qi} and \cref{prop:rectification is oo-qi to original}. The result follows.
\end{proof}

Our generalization of \cref{thm:Berglund} is a direct consequence of this result.

\begin{theorem}\label{cor:generalization of Berglund}
	Let $K$ be a simply-connected simplicial set, let $L$ be a nilpotent space (e.g. a simply-connected space) of finite $\mathbb{Q}$-type and $L_\mathbb{Q}$ the $\mathbb{Q}$-localization of $L$. Let $C$  be a homotopy cocommutative model for $K$, let $(\g,\F_\bullet\g)$ be a degree-wise nilpotent, locally finite $\SLoo$-model of finite type for $L$, and let $(\h,\F_\bullet\h)$ be an $\SLoo$-algebra such that there exists a filtered $\infty$-quasi-isomorphism
	\[
	\Psi:(\h,\F_\bullet\h)\rightsquigarrow(\g,\F_\bullet\g)\ .
	\]
	There is a homotopy equivalence of simplicial sets
	\[
	\map(K,L_\mathbb{Q})\simeq\MC_\bullet(\hom^\pi(C,\h)) ,
	\]
	i.e. the convolution $\SLoo$-algebra $\hom^\pi(C,\h)$ is an $\SLoo$-model for the mapping space.
\end{theorem}

An example of an application of this theorem is an alternative proof of \cite[Thm. 3.2]{bg16}, see \cite[Cor. 11.1]{w16}.

%% file: ModelStrGMThm.tex
\chapter{A model structure for the Goldman--Millson theorem}\label{ch:model structure for GM}

The Goldman--Millson and the Dolgushev--Rogers theorem, \cref{thm:Goldman-Millson,thm:Dolgushev-Rogers}, have a distinct homotopical flavor. After all, they state that some class of morphisms, namely filtered ($\infty$-)quasi-isomorphisms, is sent to weak equivalences of simplicial sets under the Maurer--Cartan space functor. However, the proofs of these theorems are not done by purely homotopical methods, but rather by working explicitly with the algebras and performing some induction using the filtrations, as we have seen in \cref{sect:Dolgushev-Rogers thm}. There is a reason behind this: the filtered quasi-isomorphisms are not very well behaved and do not form the class of weak equivalences of any model structure on the category of differential graded Lie algebras that we know of, even after closing them by the 2-out-of-3 property.

\medskip

The goal of this chapter is to provide a fully homotopical and self-contained approach to the proof of the Goldman--Millson theorem and to the Dolgushev--Rogers theorem. The idea is to consider the Vallette model structure on conilpotent Lie coalgebras. Linear dualization is unfortunately not an equivalence of categories with Lie algebras, but we can work around this fact by using some results of \cite{leg16} to prove that the category of conilpotent Lie coalgebras is equivalent to the category of pro-objects in finite dimensional, nilpotent Lie algebras. This way, we obtain a model structure on this last category. An interplay between this model category and the limit functor allows us to see the gauge relation for Maurer--Cartan elements in a Lie algebra as a homotopy relation between certain morphisms representing the Maurer--Cartan elements. A model categorical argument then immediately gives us a version of the Goldman--Millson theorem. Further, using simplicial framings, we extend the argument to prove a version of the Dolgushev--Rogers theorem for strict morphisms. These results are weaker than the original ones, as they only work on morphisms that are obtained via linear dualization from weak equivalences of conilpotent Lie coalgebras, i.e. essentially duals of filtered quasi-isomorphisms of Lie coalgebras.

\medskip

The results presented here are those of \cite{rn18}. They have close links with recent works by other authors, which we will try to explain throughout the text. The model structure we produce on the category of pro-objects in finite dimensional, nilpotent Lie algebras starting from the Vallette model structure on conilpotent Lie coalgebras is the same as the model structure described directly on this category of pro-objects by A. Lazarev and M. Markl in \cite{lm13}. The idea of a homotopical approach to the proof of the Goldman--Millson and the Dolgushev--Rogers theorems is already present in the work of U. Buijs, Y. F\'elix, A. Murillo, and D. Tanr\'e --- more specifically in \cite{bfmt16bis}, where they obtain a result which is strictly stronger than our \cref{thm:homotopical GM,thm:homotopical DR}. However, they need to use a version of the Dolgushev--Rogers theorem in their proof, while our approach is more self-contained, if not as powerful. Moreover, the techniques we present in \cref{sect:DRThm} can be applied to the results of Buijs--F\'elix--Murillo--Tanr\'e to give a modest generalization.

\medskip

In this chapter, we will work over cochain complexes instead of chain complexes, as we did in \cref{ch:representing the deformation oo-groupoid}.

\section{Some notions of category theory} \label{sect:ModelCat}

In this section, we give a reminder of the basic categorical notions we will need later, such as some basic results on equivalences of categories, and ind- and pro-objects in a category (i.e. categories formal colimits and limits).

\subsection{Equivalences of categories}

Given two categories $\cat$ and $\mathsf{D}$, an \emph{isomorphism of categories}\index{Isomorphism of categories} between them is a functor $F:\cat\to\mathsf{D}$ such that there exists another functor $F^{-1}:\mathsf{D}\to\cat$ satisfying $F^{-1}F = 1_\cat$ and $FF^{-1} = 1_\mathsf{D}$. This notion is far too strict to be really useful. A more sensible notion to compare categories is the following one.

\begin{definition}
	An \emph{equivalence of categories}\index{Equivalence of categories} between $\cat$ and $\mathsf{D}$ is a functor $F:\cat\to\mathsf{D}$ such that there exists a functor $G:\mathsf{D}\to\cat$ and two natural isomorphisms $GF\cong1_\cat$ and $FG\cong1_\mathsf{D}$. An \emph{anti-equivalence of categories}\index{Anti-equivalence of categories} between $\cat$ and $\mathsf{D}$ is an equivalence of categories between $\cat^{op}$ and $\mathsf{D}$.
\end{definition}

There is another notion of equivalence of categories which is at first sight stronger than the previous one.

\begin{definition}
	An \emph{adjoint equivalence of categories}\index{Adjoint equivalence of categories} between $\cat$ and $\mathsf{D}$ is an adjoint pair
	\[
	\adjunction{F}{\cat}{\mathsf{D}}{G}
	\]
	such that the unit and counit maps are natural isomorphisms.
\end{definition}

In fact, the two notions of equivalence of categories are the same.

\begin{theorem} \label{thm:equivOfCats}
	Let $F:\cat\to\mathsf{D}$ be a functor. The following statements are equivalent.
	\begin{enumerate}
		\item The functor $F$ is an equivalence of categories.
		\item The functor $F$ is part of an adjoint equivalence of categories.
		\item The functor $F$ is fully faithful and essentially surjective, i.e. for every object $d\in\mathsf{D}$ there exists $c\in\cat$ such that $d\cong F(c)$.
	\end{enumerate}
\end{theorem}

For details on these notions, see for example the book \cite[pp. 92--95]{MacLane}.

\subsection{Categories of ind-objects and categories of pro-objects}

Let $\cat$ be a category. One can then consider the category of ``formal colimits'' in $\cat$ (e.g. \cite[Sect. 8]{SGA4.1}).

\begin{definition}
	A (non-empty) small category $\cat$ is \emph{filtered}\index{Filtered!category} if
	\begin{enumerate}
		\item for every two objects $x,y\in\cat$, there exists an object $z\in\cat$ and two arrows $z\to x$ and $z\to y$, and
		\item for every two arrows $f,g:b\to a$ in $\cat$, there exists an arrow $h:c\to b$ such that $fh=gh$.
	\end{enumerate}
	The category $\cat$ is \emph{cofiltered}\index{Cofiltered!category} if it satisfies the dual properties.
\end{definition}

\begin{definition}
	The category $\ind{\cat}$ of \emph{ind-objects}\index{Ind-objects} in $\cat$ is the category that has as objects all the diagrams $F:\mathcal{D}\to\cat$ with $\mathcal{D}$ a small filtered category. If $F:\mathcal{D}\to\cat$ and $G:\mathcal{E}\to\cat$ are two objects in $\ind{\cat}$, the set of morphisms between them is
	\[
	\hom_{\ind{\cat}}(F,G)\coloneqq\lim_{d\in\mathcal{D}}\colim_{e\in\mathcal{E}}\hom_\cat(F(d),G(e))\ ,
	\]
	where the limit and the colimit are taken in $\sets$.
\end{definition}

Dually, one can also consider the category of ``formal limits'' in $\cat$ (e.g. \cite[Sect. A.2]{gro60} and \cite[Sect. 8]{SGA4.1}).

\begin{definition}
	The category $\pro{\cat}$ of \emph{pro-objects}\index{Pro-objects} in $\cat$ is the category that has as objects all the diagrams $F:\mathcal{D}\to\cat$ with $\mathcal{D}$ a small cofiltered category. If $F:\mathcal{D}\to\cat$ and ,$G:\mathcal{E}\to\cat$ are two objects in $\pro{\cat}$, the set of morphisms between them is
	\[
	\hom_{\pro{\cat}}(F,G)\coloneqq\lim_{e\in\mathcal{E}}\colim_{d\in\mathcal{D}}\hom_\cat(F(d),G(e))\ ,
	\]
	where the limit and the colimit are taken in $\sets$.
\end{definition}

\begin{lemma} \label{lemma:antiEquivIndPro}
	Let $\cat'$ and $\cat''$ be two equivalent or anti-equivalent categories. Then $\ind{\cat'}$ is equivalent to $\ind{\cat''}$, respectively anti-equivalent to $\pro{\cat''}$.
\end{lemma}

\begin{proof}
	This is straightforward.
\end{proof}

We recall the definition of a compact object in a category.

\begin{definition}
	Let $\cat$ be a category that admits filtered colimits. An object $c\in\cat$ is said to be \emph{compact}\index{Compact object} if the functor
	\[
	\hom_\cat(c,-):\cat\longrightarrow\sets
	\]
	preserves filtered colimits.
\end{definition}

\begin{proposition} \label{prop:CEquivToIndC}
	Let $\cat$ be a category and let $\cat'$ be a full subcategory of $\cat$. Further, assume that
	\begin{enumerate}
		\item \label{it:CCocomplete} the category $\cat$ is cocomplete,
		\item \label{it:delta} there exists a functor
		\[
		\delta:\cat\longrightarrow\ind{\cat'}
		\]
		such that the composite $\colim{}\,\delta$ is naturally isomorphic to the identity functor of $\cat$, where $\colim{}$ is (a choice for) the functor
		\[
		\colim{}:\ind{\cat'}\longrightarrow\cat
		\]
		given by taking the colimit in $\cat$, and
		\item \label{it:compactObj} every object in $\cat'$ is compact in $\cat$.
	\end{enumerate}
	Then the functors $\delta$ and $\colim{}$ exhibit an equivalence of categories between $\cat$ and $\ind{\cat'}$.
\end{proposition}

\begin{proof}
	Assumption (\ref{it:CCocomplete}) guarantees the existence of a colimit functor. The natural isomorphism
	\[
	\colim{}\,\delta\cong\id_\cat
	\]
	is given by assumption (\ref{it:delta}). We are left to prove that there exists a natural isomorphism
	\[
	\delta\,\colim{}\cong\id_{\ind{\cat'}}\ .
	\]
	Let $F:\mathcal{D}\to\cat'$ and $G:\mathcal{E}\to\cat'$ be two objects in $\ind{\cat'}$. Then we have
	\begin{align*}
	\hom_{\ind{\cat'}}(F,G) =&\ \lim_{d\in\mathcal{D}}\colim_{e\in\mathcal{E}}\hom_{\cat'}(F(d),G(e))\\
	\cong&\ \lim_{d\in\mathcal{D}}\colim_{e\in\mathcal{E}}\hom_\cat(F(d),G(e))\\
	\cong&\ \lim_{d\in\mathcal{D}}\hom_\cat(F(d),\colim\,G)\\
	\cong&\ \hom_\cat(\colim{}\,F,\colim{}\,G)\ .
	\end{align*}
	In the second line we used the fact that $\cat'$ is a full subcategory of $\cat$, and in the third one the fact that $F(d)\in\cat'$ is always a compact object by assumption (\ref{it:compactObj}). It follows that for any $F\in\ind{\cat'}$ we have
	\begin{align*}
	\hom_{\ind{\cat'}}(\delta\,\colim{}\,F,F) \cong&\ \hom_\cat(\colim{}\,\delta\,\colim{}\,F,\colim{}\,F)\\
	\cong&\ \hom_\cat(\colim{}\,F,\colim{}\,F)\ ,
	\end{align*}
	where we used the fact that $\colim{}\,\delta\cong\id_\cat$. Then, the identity morphism of $\colim{}\ F$ provides a natural isomorphism $\delta\,\colim{}\cong\id_{\ind{\cat'}}$ as we wanted.
\end{proof}

\begin{example}
	The easiest example of application of \cref{prop:CEquivToIndC} is the following one. Let $\cat\coloneqq\sets$ be the category of (small) sets and let $\cat'\coloneqq\fsets$ be the full subcategory of finite sets. Every set is the colimit of its finite subsets, so taking the diagram of all finite subsets gives a functor
	\[
	\delta:\sets\longrightarrow\ind{\fsets}\ .
	\]
	As $\sets$ is cocomplete and finite sets are compact objects in $\sets$, \cref{prop:CEquivToIndC} tells us that $\sets$ is equivalent to $\ind{\fsets}$.
\end{example}

\begin{example}
	Another classical example is as follows. Let $\cat \coloneqq \vect$ be the category of vector spaces and let $\cat' = \cat'' \coloneqq \fdvect$ be the full subcategory of finite dimensional vector spaces. As every vector space is naturally isomorphic to the colimit of all its finite dimensional subspaces, taking the diagram of all the finite dimensional subspaces of a vector space gives a functor
	\[
	\delta:\vect\longrightarrow\ind{\fdvect}\ .
	\]
	It is well-known that $\vect$ is cocomplete and that the finite dimensional vector spaces are compact objects in $\vect$. Moreover, linear duality gives an anti-equivalence of $\fdvect$ with itself.  Therefore, by \cref{prop:CEquivToIndC} and \cref{lemma:antiEquivIndPro} we have that $\vect$ is anti-equivalent to $\pro{\fdvect}$.
\end{example}

\begin{example} \label{ex:GG}
	The following example is a result of \cite{gg99}. Let $\cat\coloneqq\mathsf{Cog}$ be the category of coassociative counital coalgebras, let $\cat'\coloneqq \mathsf{fCog}$ be the full subcategory of finite dimensional coassociative counital coalgebras, and let $\cat''\coloneqq\mathsf{fAlg}$ be the category of finite dimensional unital associative algebras. One can prove that every coassociative counital coalgebra is the colimit of its finite dimensional sub-coalgebras By the same arguments as in the previous two examples, \cref{prop:CEquivToIndC} gives an equivalence of categories between $\mathsf{Cog}$ and $\ind{\mathsf{fCog}}$. Linear duality induces an anti-equivalence of categories between $\mathsf{fCog}$ and $\mathsf{fAlg}$. Therefore, by \cref{lemma:antiEquivIndPro} we have that $\mathsf{Cog}$ is anti-equivalent to $\pro{\mathsf{fAlg}}$.
\end{example}

\section{A model structure for the Goldman--Millson theorem} \label{sect:GMThm}

In this section, we show that a version of the Goldman--Millson theorem \cite{gm88} can be obtained by a purely homotopical argument in the model category of Lie coalgebras.

\subsection{The model structure on conilpotent Lie coalgebras} \label{subsect:coLieCog}

We are interested in the study of certain dg Lie algebras, namely the ones obtained dualizing conilpotent Lie coalgebras.

\begin{definition}
	A \emph{conilpotent Lie coalgebra}\index{Lie!coalgebras} is a conilpotent coalgebra over the cooperad $\colie\coloneqq\lie^\vee$, the linear dual of the operad $\lie$ encoding Lie algebras.
\end{definition}

Notice that the cooperad $\colie$ is Koszul and that its Koszul dual operad is the operad $\susp\otimes\com$. Both are connected, and we endow them with the canonical weight grading. By what exposed in \cref{subsection:Vallette model structure on coalgebras}, we obtain a model structure on the category of conilpotent Lie coalgebras by pulling back the Hinich model structure on (suspended) commutative algebras along the Koszul twisting morphism
\[
\kappa:\colie\longrightarrow\susp\otimes\com\ .
\]
We fix the canonical weight grading on $\colie$. It is given by
\[
\colie^{(w)}\coloneqq\colie(w+1)\ .
\]
Therefore, we can change our conventions a bit and we define a \emph{filtration} on a conilpotent Lie coalgebra $C$ to be a sequence of sub-chain complexes
\[
0\subseteq F_1C\subseteq F_2C\subseteq F_3C\subseteq\cdots\subseteq C
\]
such that
\begin{enumerate}
	\item the filtration is exhaustive: $\colim_nF_nC\cong C$,
	\item it respects the coproduct, that is
	\[
	\overline{\Delta}_C(F_nC)\subseteq\bigoplus_{\substack{k\ge1\\n_1+\cdots+n_k = n}}\big(\colie(k)\otimes F_{n_1}C\otimes\cdots\otimes F_{n_k}C\big)^{\S_k},
	\]
	and
	\item it is preserved by the differential: $d_C(F_nC)\subseteq F_nC$.
\end{enumerate}
Notice that this is equivalent to the notion of cofiltration of \cref{def:cofiltered coalgebras} by shifting the indices by $1$.

\medskip

Explicitly, the model structure on the category $\mathsf{coLie}$ of conilpotent Lie coalgebras has
\begin{itemize}
	\item the closure of the class of filtered quasi-isomorphisms under the $2$-out-of-$3$ property as weak equivalences,
	\item the monomorphisms as cofibrations, and
	\item the class of morphisms with the right lifting property with respect to trivial cofibrations as fibrations.
\end{itemize}
Every conilpotent Lie coalgebra is cofibrant, and the fibrant objects are the quasi-free conilpotent Lie coalgebras.

\medskip

Dualizing a filtered conilpotent Lie coalgebra we obtain a complete Lie algebra.

\begin{lemma}
	Let $(C,F_\bullet C)$ be a filtered conilpotent Lie coalgebra. Then the filtration on $\g\coloneqq C^\vee$ defined by
	\[
	F_n\g\coloneqq(F_nC)^\perp = \{c^\vee\in C^\vee\mid c^\vee(\F_nC) = 0\}
	\]
	makes $\g$ into a complete Lie algebra. Moreover, the dual of a filtered quasi-isomorphism of conilpotent Lie coalgebras is a filtered quasi-isomorphism of Lie algebras.
\end{lemma}

\begin{remark}\label{rem:not all filtered qi are duals}
	Not all filtered quasi-isomorphisms of Lie algebras can be obtained by linear dualization. For example, consider abelian (co)Lie (co)algebras, i.e. cochain complexes. For simplicity, take $\k=\mathbb{Q}$. Let
	\[
	V\coloneqq\bigoplus_{n\in\mathbb{Z}}\mathbb{Q}
	\]
	be concentrated in degree $0$, seen as an abelian Lie coalgebra. A quasi-isomorphism from $V$ to itself is just an isomorphism, and we have
	\[
	|\operatorname{Aut}(V)|\le|\operatorname{End}(V)|=(\dim V)^{|V|} = |\mathbb{N}|^{|\mathbb{N}|}.
	\]
	At the same time, we have $\dim(V^\vee) = |\mathbb{R}|$. Notice that if we fix a basis of $V^\vee$, then bijections from the basis to itself are automorphisms. Therefore, we have
	\[
	|\operatorname{Aut}(V^\vee)|\ge|\operatorname{Sym}(\mathbb{R})|=|\mathbb{R}|^{|\mathbb{R}|}>|\mathbb{N}|^{|\mathbb{N}|}.
	\]
	It follows that there must exist automorphisms of $V^\vee$ that are not given as the dual of an automorphism of $V$.
\end{remark}

There is another twisting morphism we can consider to construct a model structure on conilpotent Lie coalgebras, namely the canonical twisting morphism
\begin{equation} \label{eq:iota}
	\iota:\colie\cong\susp^c\otimes\com^{\antishriek}\longrightarrow\Omega(\susp^c\otimes\com^{\antishriek})\cong\susp\otimes\C_\infty\ .
\end{equation}

\begin{lemma} \label{lemma:modelStrKappaIota}
	The model structure obtained on conilpotent Lie coalgebras by pulling back the Hinich model structure on suspended $\C_\infty$-algebras along $\Omega_\iota$ is equal to the one obtained by pulling back along $\Omega_\kappa$.
\end{lemma}

\begin{proof}
	This is an immediate corollary of \cite[Prop. 32]{leg16}.
\end{proof}

This fact will be of fundamental importance later on.

\subsection{Dualization of conilpotent Lie coalgebras}

We start by giving the correct framework in which to dualize conilpotent Lie coalgebras.

\medskip

In the article \cite{leg16}, the following result is proven, generalizing what can be found in \cite[Sect. 1]{gg99} for coassociative counital coalgebras. See also \cref{ex:GG}.

\begin{proposition}[{\cite[Lemma 4 and 5 and Prop. 12]{leg16}}]
	Let $\C$ be a cooperad. The category $\Ccog$ of conilpotent $\C$-coalgebras is cocomplete, every conilpotent $\C$-coalgebra is the colimit of the diagram of all its finite-dimensional sub-coalgebras with the relative inclusions, and all finite-dimensional conilpotent $\C$-coalgebras are compact objects in $\Ccog$. In particular, the category $\Ccog$ is presentable.
\end{proposition}

Denote by $\mathsf{fd\ coLie}$ the full subcategory of $\mathsf{coLie}$ of finite dimensional conilpotent Lie coalgebras. As a direct consequence of these facts and of \cref{prop:CEquivToIndC} in the case where $\cat = \mathsf{coLie}$, we have the following.

\begin{corollary} \label{cor:equivCats}
	The categories $\mathsf{ coLie}$ and $\ind{\mathsf{fd\ coLie}}$ are equivalent. The equivalence is given by the functors
	\[
	\colim{}:\ind{\mathsf{fd\ coLie}}\longrightarrow\mathsf{coLie}
	\]
	given by taking a functorial choice of colimit of diagrams and
	\[
	\delta:\mathsf{coLie}\longrightarrow\ind{\mathsf{fd\ coLie}}
	\]
	given by associating to a conilpotent Lie coalgebra the diagram of all its finite-dimensional sub-coalgebras with the relative inclusions.
\end{corollary}

It is always true that the linear dual of a conilpotent Lie coalgebra is a Lie algebra, while the converse holds in the finite dimensional case. Denote by $\mathsf{fd\ nil.\ Lie}$ the category of finite dimensional nilpotent Lie algebras. As a direct consequence of this fact and of \cref{lemma:antiEquivIndPro}, we have the following.

\begin{lemma}
	Linear duality between finite dimensional Lie coalgebras and finite dimensional Lie algebras induces an anti-equivalence of categories between $\ind{\mathsf{fd\ conil.\ coLie}}$ and $\pro{\mathsf{fd\ nil.\ Lie}}$.
\end{lemma}

In summary, we have the following commutative diagram, where $\mathsf{pronil.\ Lie}$ is the subcategory of Lie algebras which is the image of linear dualization from conilpotent Lie coalgebras, called the category of \emph{pronilpotent Lie algebras}\index{Pronilpotent Lie algebras}.
\begin{center}
	\begin{tikzpicture}
	\node (a) at (0,0) {$\ind{\mathsf{fd\ coLie}}$};
	\node (b) at (5,0) {$\pro{\mathsf{fd\ nil.\ Lie}}$};
	\node (c) at (0,-2) {$\mathsf{coLie}$};
	\node (d) at (5,-2) {$\mathsf{pronil.\ Lie}$};
	
	\draw[->] (c) -- (d);
	\draw[->] (a) -- node[right]{$\colim{}$} (c);
	\draw[->] (b) -- node[right]{$\lim$} (d);
	\draw[->] (b) -- (a);
	\end{tikzpicture}
\end{center}
where the horizontal arrows are given by linear dualization and the upper horizontal and left vertical arrows are an anti-equivalence of categories and an equivalence of categories respectively. Notice that all pronilpotent Lie algebras are complete: one can always write them as the dual of a conilpotent Lie coalgebras an take the orthogonal of the coradical filtration.

\begin{remark}
	What stated above works in greater generality for conilpotent coalgebras over a cooperad. In a sense, this is the correct framework in which to dualize a conilpotent coalgebra, obtaining not an algebra, but a diagram of finite dimensional algebras.
\end{remark}

We endow the category $\pro{\mathsf{fd\ nil.\ Lie}}$ with the model structure obtained by transporting the Vallette model structure on Lie coalgebras along the two (anti-)equivalences of categories. The two model structures are Quillen equivalent. Notice that fibrations and cofibrations change roles as we pass through an anti-equivalence.

\begin{remark}
	One can check that the model structure thus obtained is exactly the one of \cite[Def. 9.9]{lm13}. The main difference between our approach and theirs is the fact that we work directly with Lie coalgebras, while they only ever work with pronilpotent Lie algebras. We hope that the point of view we used in the present paper will appear more natural to some readers, as it avoids dualizing every coalgebra in order to only work with algebras.
\end{remark}

\subsection{The Goldman--Millson theorem}

The idea is now to use the model structure given above to obtain statements about pronilpotent Lie algebras. By what we have seen in \cref{sect:properties and comparison with Getzler},  we have that the set of Maurer--Cartan elements of a Lie algebra $\g$ is in natural bijection with the set of morphisms
\[
\mc_0\longrightarrow\g
\]
of dg Lie algebras, where $\mc_0$ is the free Lie algebra on a Maurer--Cartan element. Similarly, gauge equivalences between Maurer--Cartan elements are coded by morphisms
\[
\mc_1\longrightarrow\g\ .
\]
Intuitively, $\mc_1$ behaves like a cylinder object for $\mc_0$, so that the gauge relation looks very similar to a left homotopy between Maurer--Cartan elements seen as maps. However, we do not have a good model structure on complete Lie algebras, and that's why we have to lift everything to the category of diagrams.

\medskip

In order to push down the results we will obtain for diagrams to actual algebras, we will need the following technical lemma.

\begin{lemma} \label{lemma:isomOfHoms}
	Let $\g$ be a Lie algebra which is complete with respect to its canonical filtration, and such that $\g^{(n)}\coloneqq\g/F^\mathrm{can}_n\g$ is finite dimensional for all $n\ge2$. Let $\widetilde{\g}\in\pro{\mathsf{fd\ nil.\ Lie}}$ be the diagram
	\[
	\widetilde{\g} = \cdots\longrightarrow \g^{(4)}\longrightarrow \g^{(3)}\longrightarrow \g^{(2)}\longrightarrow 0\ .
	\]
	Let $H:\mathsf{D}\to\mathsf{fd\ nil.\ Lie}$ be any object of $\pro{\mathsf{fd\ nil.\ Lie}}$. Then we have a natural isomorphism
	\[
	\hom_{\pro{\mathsf{fd\ nil.\ Lie}}}(\widetilde{\g},H) \cong \hom_{\mathsf{dgLie}}(\g,\lim H)\ .
	\]
\end{lemma}

\begin{proof}
	We have
	\[
	\hom_{\mathsf{dgLie}}(\g,\lim H) = \lim_{d\in\mathsf{D}}\hom_{\mathsf{dgLie}}(\g,H(d))
	\]
	with all $H(d)$ finite dimensional. Fix $d\in\mathsf{D}$, then we have the linear diagram
	\[
	\cdots\longleftarrow\hom_{\mathsf{dgLie}}(\g^{(4)},H(d))\longleftarrow\hom_{\mathsf{dgLie}}(\g^{(3)},H(d))\longleftarrow\hom_{\mathsf{dgLie}}(\g^{(2)},H(d))\longleftarrow0\ .
	\]
	Since all the maps in the diagram $\widetilde{\g}$ were surjective, all the maps in this diagram are injective. Let $N$ be the nilpotency degree of $H(d)$. Then every map $\g\to H(d)$ splits through $\g^{(n)}$ for all $n\ge N$, so that we have
	\[
	\hom_{\mathsf{dgLie}}(\g,H(d))\cong\hom_{\mathsf{dgLie}}(\g^{(n)},H(d))
	\]
	for all $n\ge N$. It follows that
	\[
	\hom_{\mathsf{dgLie}}(\g,H(d))\cong\colim_n\hom_{\mathsf{dgLie}}(\g^{(n)},H(d))\ ,
	\]
	which concludes the proof.
\end{proof}

In particular, this is true for any $\g$ of the form $\g\coloneqq\widehat{\lie}(V)$, for $V$ a finite dimensional graded vector space, with any differential. The examples of main interest to us will be the pronilpotent Lie algebras
\[
\mc_n\coloneqq \widehat{\Omega}_\pi(sC_n^\vee)
\]
introduced in \cref{def:cosimplicial mc for Lie}.

\begin{lemma} \label{lemma:mcnCofibrant}
	Let $n\ge0$. We have
	\[
	\colim\ \widetilde{\mc}_n^\vee = \Bar_\iota(s^{-1}C_n)\ .
	\]
	In particular, $\widetilde{\mc}_n$ is cofibrant.
\end{lemma}

\begin{proof}
	The dual of the quotient
	\[
	\mc_n^{(k)} = \mc_n/F^\mathrm{can}_k\mc_n
	\]
	of $\mc_n$ by the $k$th space of the canonical filtration is exactly the $k$th space of the coradical filtration of $\Bar_\iota(s^{-1}C_n)$, and the projections
	\[
	\mc_n^{(k+1)}\longrightarrow\mc_n^{(k)}
	\]
	become the inclusions of the various spaces of the coradical filtration of $\Bar_\iota(s^{-1}C_n)$. From this, the result follows.
\end{proof}

\begin{lemma} \label{lemma:mc1Cylinder}
	The diagram $\widetilde{\mc}_1$ is a cylinder object for $\widetilde{\mc}_0$. After passing to the limit, the splitting of the codiagonal map becomes
	\[
	\mc_0\sqcup\mc_0\stackrel{i}{\longrightarrow}\mc_1\stackrel{t}{\longrightarrow}\mc_0
	\]
	where the two maps are as follows. The first Lie algebra is the free complete Lie algebra generated by two Maurer--Cartan elements $\alpha_0$ and $\alpha_1$, while the second one is the free complete Lie algebra generated by two Maurer--Cartan elements $\beta_1,\beta_2$ and a gauge $\lambda$ from $\beta_1$ to $\beta_2$, and the last algebra $\mc_0$ is the free Lie algebra on a single Maurer--Cartan element $\alpha$. The first map is determined by $i(\alpha_j) = \beta_j$ for $j=1,2$, the second map is determined by $t(\beta_j) = \alpha$ for $j=1,2$ and $t(\lambda) = 0$.
\end{lemma}

\begin{proof}
	By \cref{lemma:mcnCofibrant}, we know that the diagram $\widetilde{\mc}_n$ corresponds to the conilpotent Lie coalgebra $\Bar_\iota(s^{-1}C_n)$. Since $\Bar_\iota$ is right adjoint, we have
	\[
	\Bar_\iota(s^{-1}C_0)\times\Bar_\iota(s^{-1}C_0)\cong\Bar_\iota(s^{-1}C_0\times s^{-1}C_0)
	\]
	with the obvious differential in the category of conilpotent Lie coalgebras. The diagonal map
	\[
	C_0\stackrel{\Delta}{\longrightarrow}C_0\times C_0
	\]
	in the category of $\C_\infty$-algebras splits into
	\[
	C_0\stackrel{T}{\longrightarrow}C_1\stackrel{I}{\longrightarrow}C_0\times C_0
	\]
	where $T(1) = \omega_0 + \omega_1$, $I(\omega_0) = (1,0)$, $I(\omega_1) = (0,1)$, and $I(\omega_{01}) = (0,0)$. The map $T$ is a quasi-isomorphism as it is induced by the homotopy equivalence $\Delta^1\to\Delta^0$. Therefore, by a result analogous to \cref{lemma:BarOfQIisFQI}, it is sent to a filtered quasi-isomorphism $\Bar_\iota(T)$ by the bar construction. The map $I$ is clearly surjective, and thus $\Bar_\iota(I)$ is a fibration by \cite[Thm. 2.9(2)]{val14}. Notice that the mentioned theorem is stated for the twisting morphism $\kappa$, but it also holds for $\iota$ by repeating the same proof and using \cref{lemma:modelStrKappaIota}. It follows that $\Bar_\iota(s^{-1}C_1)$ is a path object for $\Bar_\iota(s^{-1}C_0)$, which is equivalent to the statement we wanted to prove. Recovering the exact structure of the maps after passing to the limit is straightforward.
\end{proof}

We are now set for the proof of the Goldman--Millson theorem for pronilpotent Lie algebras.

\begin{theorem}\label{thm:homotopical GM}
	Let $\g,\h$ be two pronilpotent dg Lie algebras, and let $\phi:\g\to\h$ be the dual of a weak equivalence of conilpotent Lie coalgebras. Then $\phi$ induces a bijection
	\[
	\MCbar(\g)\cong\MCbar(\h)\ .
	\]
\end{theorem}

\begin{proof}
	Let $G,H\in\pro{\mathsf{fd\ nil.\ Lie}}$ and
	\[
	\Phi\in\hom_{\pro{\mathsf{fd\ nil.\ Lie}}}(G,H)
	\]
	be a weak equivalence such that $\lim G = \g$, $\lim H = \h$, and such that $\Phi$ corresponds to $\phi$ after passing to the limits. By definition, the moduli space of Maurer--Cartan elements of a Lie algebra $\g$ is the quotient
	\[
	\MCbar(\g) = \MC(\g)/\sim_{\mathrm{gauge}}
	\]
	of the space of Maurer--Cartan elements of $\g$ by the gauge relation. As explained above, we can see Maurer--Cartan elements in $\g$ as morphisms $\mc_0\to\g$ and gauge equivalences are coded by morphisms $\mc_1\to\g$. That is to say, two Maurer--Cartan elements $x_0,x_1\in\MC(\g)$ are gauge equivalent if, and only if there exists a morphism $\mc_1\to\g$ making the following diagram commute.
	\begin{center}
		\begin{tikzpicture}
		\node (a) at (0,0){$\mc_0$};
		\node (b) at (1,-1.5){$\mc_1$};
		\node (c) at (3,-1.5){$\g$};
		\node (d) at (0,-3){$\mc_0$};
		
		\draw[->] (a) -- node[right]{$i_0$} (b);
		\draw[->] (d) -- node[right]{$i_1$} (b);
		\draw[->] (b) -- (c);
		\draw[->] (a) to[out=0,in=140] node[above right]{$x_0$} (c);
		\draw[->] (d) to[out=0,in=-140] node[below right]{$x_1$} (c);
		\end{tikzpicture}
	\end{center}
	But by \cref{lemma:isomOfHoms} and \cref{lemma:mc1Cylinder}, this is exactly the left homotopy relation on morphisms $\widetilde{\mc}_0\to G$. Therefore,
	\[
	\MCbar(\g)\cong\hom_{\pro{\mathsf{fd\ nil.\ Lie}}}(\widetilde{\mc}_0,G)/\sim_\ell = \hom_{\mathsf{Ho}(\pro{\mathsf{fd\ nil.\ Lie}})}(\widetilde{\mc}_0,G)\ ,
	\]
	where the last equality is given by the fact that all elements of $\pro{\mathsf{fd\ nil.\ Lie}}$ are fibrant, as well as the fact that $\widetilde{\mc}_0$ is cofibrant by \cref{lemma:mcnCofibrant}. The same thing is of course true for $\h$, and since $\phi:\g\to\h$ comes from a weak equivalence, it naturally induces a bijection
	\[
	\MCbar(\g)\cong\MCbar(\h)
	\]
	by \cref{lemma:we induce bijections of hom sets}, as desired.
\end{proof}

Although this result is slightly weaker than the full Goldman--Millson theorem (it works only on some algebras, and we don't have all the morphisms we would like, cf. \cref{rem:not all filtered qi are duals}), it has the advantage of having a fully homotopical proof, which is good considering the homotopical flavor of the statement.

\section{Framings and the Dolgushev--Rogers theorem} \label{sect:DRThm}

The last section can be seen as the $0$th level of the Dolgushev--Rogers theorem --- it is in fact a statement at the level of the $0$th homotopy group of the Maurer--Cartan spaces. Here, we will apply the simplicial framings introduced in \cref{sect:framings} to our model category, showing how the whole cosimplicial Lie algebra $\mc_\bullet$ appears naturally in this context. We are then able to recover the Dolgushev--Rogers theorem in this context by arguments similar to the ones used for the Goldman--Millson theorem in the last section. We conclude by comparing the results of this chapter with the recent literature.

\subsection{Using framings to recover the Dolgushev--Rogers theorem}

We take for $\cat$ the category of conilpotent Lie coalgebras and we consider the Lie coalgebra $\Bar_\iota(s^{-1}C_0)$.

\begin{lemma}
	The simplicial conilpotent Lie coalgebra $\Bar_\iota(s^{-1}C_\bullet)$ is a simplicial frame on $\Bar_\iota(s^{-1}C_0)$.
\end{lemma}

\begin{proof}
	At level $0$, both maps are the identity. For $n\ge0$ the maps are induced by the maps of $\C_\infty$-algebras
	\[
	C_0\stackrel{w}{\longrightarrow}C_n\stackrel{p}{\longrightarrow}C_0\times\cdots\times C_0
	\]
	explicitly given by
	\[
	w(1) = \omega_0+\cdots+\omega_n
	\]
	and
	\[
	p(\omega_i) = (0,\ldots,0,\underbrace{1}_{i},0,\ldots,0),\qquad\mathrm{and}\qquad p(\omega_I) = 0\quad\forall |I|\ge2\ .
	\]
	The map $w$ is the pullback by the unique map $\Delta^n\to\Delta^0$, which is a homotopy equivalence. Thus, $w$ is a quasi-isomorphism and we obtain a weak equivalence when we apply the bar construction to it by an analogue to \cref{lemma:BarOfQIisFQI}. Similarly, the map $p$ is surjective, and thus is sent to a fibration by the bar construction, again by \cite[Thm. 2.9(2)]{val14} as in the proof of \cref{lemma:mc1Cylinder}.
\end{proof}

This simple fact is enough to recover all the most important results about $\hom_{\dgl}(\mc_\bullet,\g)$.

\begin{theorem}\label{thm:homotopical DR}
	The functor
	\[
	\hom_{\dgl}(\mc_\bullet,-):\mathsf{pronil.\ Lie}\longrightarrow\ssets
	\]
	sends duals of injections of conilpotent Lie coalgebras to fibrations and duals of weak equivalences of conilpotent Lie coalgebras to weak equivalences. In particular, it has image in the full subcategory of Kan complexes.
\end{theorem}

\begin{proof}
	Let $\g$ be a pronilpotent dg Lie algebra. Then $\g = \lim G$ for some $G\in\pro{\mathsf{fd\ nil.\ Lie}}$. Let $C\coloneqq \colim{}\,G^\vee$, so that $C^\vee = \g$. Then
	\begin{align*}
	\hom_{\dgl}(\mc_\bullet,\g)\cong&\ \hom_{\pro{\mathsf{fd\ nil.\ Lie}}}(\widetilde{\mc}_\bullet,G)\\
	\cong&\ \hom_{\mathsf{conil.\ coLie}}(C,\Bar_\iota(s^{-1}C_\bullet))\ ,
	\end{align*}
	where the first isomorphism is given by \cref{lemma:mcnCofibrant}, and the second one by the equivalence of categories of \cref{cor:equivCats}. By \cref{thm:Vallette model structure}, we know that $C$ is always cofibrant (and thus fibrant in the opposite category), and that $\Bar_\iota(s^{-1}C_0)$ is fibrant, since it is quasi-free. \cref{prop:framings} concludes the proof.
\end{proof}

The statement that the functor sends duals of injections to fibrations should be seen analogous to \cref{thm:Getzler}, while the assertion that duals of weak equivalences of coalgebras are sent to weak equivalences is a weaker version of \cref{thm:Dolgushev-Rogers}, in the same way that \cref{thm:homotopical GM} is a weaker version of \cref{thm:Goldman-Millson}.

\subsection{Relations with the BFMT model structure on Lie algebras}\label{subsect:relations with BFMT}

In the article \cite{bfmt16bis}, Buijs--F\'elix--Murillo--Tanr\'e introduced a model structure on a slightly larger category: the category of all complete Lie algebras and filtered morphisms. They have the declared goal of developing a way to do rational homotopy theory for all spaces using Lie algebra models. We summarize their results and compare them with the present work.

\medskip

Let $\g$ be a complete Lie algebra, and let $x\in\MC(\g)$. Then one can use $x$ to twist the differential to get
\[
d^x\coloneqq d + \ad_x\ .
\]
This new operator $d^x$ is such that $(\g,d^x)$ is once again a Lie algebra.

\begin{definition}
	The \emph{component of $\g$ at $x$} is the Lie algebra obtained by truncating $(\g,d^x)$ in positive degree and taking only $\ker d^x$ in degree $0$, i.e. it is the cochain complex
	\[
	\cdots\xrightarrow{d^x}\g^{-2}\xrightarrow{d^x}\g^{-1}\xrightarrow{d^x}\ker d^x\longrightarrow0
	\]
	endowed with the Lie algebra structure inherited from $\g$.
\end{definition}

One defines the following three classes of maps:
\begin{itemize}
	\item A filtered morphism is a \emph{fibration}\index{Fibrations!in the BFMT model structure} if it is surjective in non-negative degrees.
	\item A filtered morphism $\phi:\g\to\h$ is a \emph{weak equivalence}\index{Weak equivalences!in the BFMT model structure} if $\MCbar(\phi)$ is a bijection and $\phi:\g^x\to\h^{\phi(x)}$ is a quasi-isomorphism for all $x\in\MC(\g)$.
	\item A filtered morphism is a \emph{cofibration}\index{Cofibrations!in the BFMT model structure} if it has the left lifting property with respect to all trivial cofibrations.
\end{itemize}

\begin{theorem}[{\cite[Thm. 3.1]{bfmt16bis}}]
	These three classes of maps define a model structure on the category of complete Lie algebras and filtered morphisms. We call this model structure the \emph{BFMT model structure}\index{BFMT model structure} on complete Lie algebras.
\end{theorem}

The weak equivalences and fibrations of the model structure considered in the present paper are contained in the weak equivalences and fibrations of the BFMT model structure, as is proven in \cite[Sect. 6]{bfmt16bis}.

\medskip

One also has a Quillen pair
\[
\adjunction{\mathfrak{L}}{\ssets}{\cdgl}{\hom_{\dgl}(\mc_\bullet,-)}\ ,
\]
where the functor $\mathfrak{L}$ is obtained by sending the standard $n$-simplex $\Delta^n$ to $\mc_n$ and then applying the Yoneda lemma to extend it to all simplicial sets. This is proven in \cite[Cor. 3.6]{bfmt16bis}. In particular, both functors preserve weak equivalences. Then one proves:

\begin{proposition}[{\cite[Prop. 3.8]{bfmt16bis}}]\label{prop:filterd qi are we in BFMT}
	The class of filtered quasi-isomorphisms is contained in the class of weak equivalences.
\end{proposition}

Thus, we can give the following alternative proof of the Dolgushev--Rogers theorem.

\begin{proof}[Proof of \cref{thm:Dolgushev-Rogers}]
	This is a direct consequence of the two results above and the fact that if $\g$ is a complete Lie algebra, then there is a canonical homotopy equivalence of simplicial sets
	\[
	\MC_\bullet(\g)\simeq\hom_{\dgl}(\mc_\bullet,\g)
	\]
	as was proven in \cite[Cor. 5.3]{rn17cosimplicial} (see also \cite{bfmt17}).
\end{proof}

In fact, our argument using framings holds in this model category, too. Indeed, the natural sequence of maps
\[
\mc_0\sqcup\ldots\sqcup\mc_0\longrightarrow\mc_n\longrightarrow\mc_0\ ,
\]
which is dual to the one showing that $\Bar_\iota(s^{-1}C_\bullet)$ is a simplicial frame on $\Bar_\iota(s^{-1}C_0)$, exhibits $\mc_\bullet$ as a cosimplicial frame on $\mc_0$ in the category of complete Lie algebras with the BFMT model structure: the second map is a weak equivalence, since it comes from a weak equivalence in the Vallette model structure, while the first one is a cofibration by \cite[Thm. 4.2]{bfmt16bis}. This gives yet another alternative proof of \cref{thm:Dolgushev-Rogers}. However, one should remark that the proof of \cref{prop:filterd qi are we in BFMT} in \emph{op. cit.} relies on the Dolgushev--Rogers theorem itself, and thus their proof of \cref{thm:Dolgushev-Rogers} is --- in a sense --- not self-contained. On the other hand, \cref{thm:homotopical GM} and \cref{thm:homotopical DR} are  proved using only our homotopical approach, but they do not recover the full strength of the Dolgushev--Rogers theorem.

%% file: Trees.tex
\chapter{Trees}\label{chapter:appendix on trees}

It is very natural to use trees to express operadic concepts, and those objects appear --- implicitly and explicitly --- throughout the whole text of the present work. Here, we give the basic definitions and the notations that we need.

\section{Rooted trees}\label{sect:rooted trees}

We begin with the\footnote{Better: one. There are many different definitions of the notion of a tree in the literature, most of which are equivalent.} basic definition of a rooted tree.

\begin{definition}
	A \emph{rooted tree}\index{Rooted tree} $\tau$ is a tuple $(V,E,n,N)$ where
	\begin{enumerate}
		\item $V$ is a (possibly empty) finite set, called the \emph{set of vertices}\index{Vertex of a tree} of $\tau$,
		\item $E$ is a partial order such that
		\begin{itemize}
			\item for every $v_1,v_2\in V$, there exists a $w\in V$ such that $w\le v_1$ and $w\le v_2$, and
			\item every $v\in V$ has at most one parent, i.e. an element $w\in V$ such that $w<v$ and there exists no $w'\in V$ such that $w<w'<v$,
		\end{itemize}
	\item $n$ is an integer, and
	\item $N$ is a function
	\[
	N:\{1,\ldots,n\}\longrightarrow V
	\]
	from the set $\{1,\ldots,n\}$ to the set of vertices.
	\end{enumerate}
	These conditions imply that there is a unique minimal element of $V$, which is called the \emph{root}\index{Root of a tree} of $\tau$. The set of \emph{inner edges}\index{Inner!edges of a tree} of $\tau$ --- which equivalent to $E$, and thus will also be denoted by $E$ --- is the set of couples $(w,v)\in V\times V$ such that $w$ is a parent of $v$, and the set of \emph{leaves}\index{Leaves of a tree} of $\tau$ is the set
	\[
	\left\{(v,i)\mid v\in V_\mathrm{max}\text{ and }i\in N^{-1}(v)\right\}.
	\]
	For $v\in V$, the \emph{arity}\index{Arity!of a vertex} $|v|$ of $v$ is the number of elements of $V$ of which $v$ is the parent plus $|N^{-1}(v)|$. The natural number $n$ is called the \emph{arity}\index{Arity!of a tree} of the tree $\tau$.
\end{definition}

The graphical interpretation we have in mind when working with this definition is illustrated by the following example. Let $V\coloneqq\{a,b,c\}$, take the partial order generated by $a<b$ and $a<c$, let $n=6$, and take $N:\{1,\ldots,6\}\to\{b,c\}$ be given by sending $\{1,3,4\}$ to $a$, $5$ to $c$, and the rest to $b$. Then we draw the associated tree as
\begin{center}
	\begin{tikzpicture}
		\node (a) at (0,0){$a$};
		\node (b) at (-1,1){$b$};
		\node (c) at (2,1){$c$};
		\node (l1) at (-2,1){$1$};
		\node (l2) at (-1.5,2){$2$};
		\node (l3) at (0,1){$3$};
		\node (l4) at (1,1){$4$};
		\node (l5) at (2,2){$5$};
		\node (l6) at (-.5,2){$6$};
		
		\draw (0,-1) to (a);
		\draw (a) to (l1);
		\draw (a) to (b);
		\draw (a) to (l3);
		\draw (a) to (l4);
		\draw (a) to (c);
		\draw (b) to (l2);
		\draw (b) to (l6);
		\draw (c) to (l5);
	\end{tikzpicture}
\end{center}
In this example, the arity of the vertex $a$ is $5$, and for the other vertices we have $|b| = 2$ and $|c| = 1$.

\begin{definition}
	The \emph{$n$-corolla}\index{Corolla} $c_n$ is the tree given by
	\[
	c_n\coloneqq(*,\emptyset,n,N:k\mapsto*)\ .
	\]
\end{definition}

Graphically, the corollas are given by
\begin{center}
	\begin{tikzpicture}
		\begin{scope}
			\coordinate (r0) at (0,-.375);
			\coordinate (p0) at (0,.375);
			
			\fill (p0) circle [radius = 2pt];
			\draw (r0) to (p0);
			
			\node at (-.5,0){$c_0 =$};
		\end{scope}
		
		\begin{scope}[shift = {(3,0)}]
			\coordinate (r1) at (0,-.75);
			\coordinate (p1) at (0,0);
			\coordinate (l11) at (0,.75);
			
			\fill (p1) circle [radius = 2pt];
			\draw (r1) to (p1);
			\draw (p1) to (l11);
			
			\node at (-.7,0){$c_1 =$};
		\end{scope}
		
		\begin{scope}[shift = {(6,0)}]
			\coordinate (r2) at (0,-.75);
			\coordinate (p2) at (0,0);
			\coordinate (l21) at (-.5,.75);
			\coordinate (l22) at (.5,.75);
			
			\fill (p2) circle [radius = 2pt];
			\draw (r2) to (p2);
			\draw (p2) to (l21);
			\draw (p2) to (l22);
			
			\node at (-.7,0){$c_2 =$};
		\end{scope}
		
		\begin{scope}[shift = {(9,0)}]
			\coordinate (rn) at (0,-.75);
			\coordinate (pn) at (0,0);
			\coordinate (ln1) at (-1,.75);
			\coordinate (ln2) at (-.5,.75);
			\node at (0,.75){$\cdots$};
			\coordinate (ln3) at (.5,.75);
			\coordinate (ln4) at (1,.75);
			
			\node at (0,1.25){$n$ leaves};
			\draw [decorate,decoration={brace,amplitude=5pt}]
			(-1.2,.85) to (1.2,.85);
			
			\fill (pn) circle [radius = 2pt];
			\draw (rn) to (pn);
			\draw (pn) to (ln1);
			\draw (pn) to (ln2);
			\draw (pn) to (ln3);
			\draw (pn) to (ln4);
			
			\node at (-.5,0){$c_n =$};
		\end{scope}
	\end{tikzpicture}
\end{center}
Let $v_1,v_2\in V$, then the \emph{first common ancestor} $\FCA(v_1,v_2)$ of $v_1$ and $v_2$ is defined by
\[
\FCA(v_1,v_2)\coloneqq\max\{v\in V\mid v\le v_1\text{ and }v\le v_2\}\ .
\]

\begin{definition}
	Let $\tau_1=(V_1,E_1,n_1,N_1)$ and $\tau_2=(V_2,E_2,n_2,N_2)$ be two rooted trees. A \emph{morphism of rooted trees} $\phi:\tau_1\to\tau_2$ is the data of two maps
	\[
	\phi_V:V_1\longrightarrow V_2\quad\text{and}\quad \phi_N:\{1,\ldots,n_1\}\longrightarrow\{1,\ldots,n_2\}
	\]
	such that
	\begin{enumerate}
		\item $\phi_V$ is a map of posets, and $\phi_N$ is increasing,
		\item the diagram
		\begin{center}
			\begin{tikzpicture}
				\node (a) at (0,0) {$V_1$};
				\node (b) at (3.5,0) {$V_2$};
				\node (c) at (0,1.5) {$\{1,\ldots,n_1\}$};
				\node (d) at (3.5,1.5) {$\{1,\ldots,n_2\}$};
				
				\draw[->] (a) -- node[above]{$\phi_V$} (b);
				\draw[->] (c) -- node[left]{$N_1$} (a);
				\draw[->] (d) -- node[right]{$N_2$} (b);
				\draw[->] (c) -- node[above]{$\phi_N$} (d);
			\end{tikzpicture}
		\end{center}
	is commutative,
	\item for all $v'\in\phi_V(V_1)$, if $w'\in V_2$ is such that $w'<v'$, then $w'\in\phi_V(V_1)$, and
	\item\label{item:condition for morphisms of trees} if $v_1,v_2\in V$ are such that $\phi_V(v_1)=v'=\phi_V(v_2)$, then $\phi_V(\FCA(v_1,v_2)) = v'$.
	\end{enumerate}
\end{definition}

\begin{remark}
	Equivalently to condition (\ref{item:condition for morphisms of trees}), one can ask that the preimage of any vertex under the morphism is again a rooted tree.
\end{remark}

\begin{definition}
	A tree is \emph{reduced}\index{Reduced!tree} if all of its vertices have arity at least $2$. It is \emph{binary}\index{Binary!tree} if all of its vertices have arity exactly $2$.
\end{definition}

\begin{definition}
	We denote by $\RT$ the set of isomorphism classes of rooted trees, and by $\rRT$ the set of isomorphism classes of reduced rooted trees. For $n\ge0$, we denote by $\RT_n$ the arity $n$ elements of $\RT$, and similarly for $\rRT_n$.
\end{definition}

Let $M$ be an $\S$-module, let $\tau=(V,E,n,N)$ be a rooted tree, and let $f:V\to M$ be a function such that $f(v)\in M(|v|)$. Then we denote by $\tau(f)\in\T(M)$ the element with underlying tree $\tau$ and the vertices labeled by their image in $M$ under $f$.

\section{Planar trees}

A planar tree is just a rooted tree with a chosen way of drawing it in the plane. This can be encoded by a total order on the set of children of a vertex, for every vertex of the tree.

\begin{definition}
	Let $\tau\in\RT$ be a rooted tree, and let $v$ be a vertex of $\tau$. The set $\mathrm{ch}(v)$ of \emph{children of $v$} is
	\[
	\mathrm{ch}(v)\coloneqq N^{-1}(v)\sqcup\{w\in V\mid v\text{ is the parent of }w\}\ .
	\]
\end{definition}

Notice that $|v| = |\mathrm{ch}(v)|$.

\begin{definition}
	A \emph{planar tree}\index{Planar tree} $t$ is a rooted tree $\tau$ together with a total order on $\mathrm{ch}(v)$ for every vertex $v$ of $\tau$.
\end{definition}

%% file: Filtered.tex
\chapter{Filtered operads and filtered algebras}\label{appendix:filtered stuff}

In order to make sense of certain naturally occurring infinite sums, such as in the Maurer--Cartan equation for $\L_\infty$-algebras, cf. \cref{subsect:MC el of Loo algebras}, one has to consider additional data on chain complexes and algebras. One convenient way to do this is to look at filtered chain complexes, i.e. chain complexes equipped with a descending filtration. These ideas go back at least to Lazard's PhD thesis \cite{laz54}. We base ourselves on \cite[Sect. 2]{dsv18} for a clean, modern treatment in the context of operad theory.

\section{Filtered chain complexes}

We begin by defining the category of filtered chain complexes and its monoidal structure.

\subsection{Definitions}

\begin{definition}
	A \emph{filtered chain complex}\index{Filtered!chain complex} is a couple $(V,\F_\bullet V)$ where $V$ is a chain complex, and
	\[
	V=\F_0V\supseteq\F_1V\supseteq\F_2V\supseteq\cdots
	\]
	is a descending chain of sub-chain complexes. We will often omit the filtration from the notation when there is no risk of confusion, and speak of the filtered chain complex $V$. Moreover, we say that $V$ is \emph{proper}\index{Proper!filtered chain complex} if $\F_1V=V$.
\end{definition}

Any filtered chain complex $V$ automatically carries a first-countable topology of which a basis is given by
\[
\{v+\F_kV\mid v\in V\text{ and }k\ge0\}\ .
\]
With this topology, the chain complex $V$ becomes a Fr\'echet-Urysohn space, i.e. topological closure and sequential closure in $V$ are the same. Notice that scalar multiplication and sum are continuous with respect to this topology.

\begin{definition}
	Let $V$ and $W$ be filtered chain complexes. A \emph{filtered morphism}\index{Filtered!morphism of chain complexes} $f:V\to W$ is a morphism of chain complexes from $V$ to $W$, i.e. a chain map, such that for all $n\ge0$ we have
	\[
	f(\F_nV)\subseteq\F_nW\ .
	\]
	We denote by $\fchain$ the category of filtered chain complexes and filtered morphisms.
\end{definition}

\begin{lemma}
	Every filtered morphism is continuous with respect to topologies induced by the filtrations.
\end{lemma}

\begin{proof}
	Let $f:V\to W$ be a filtered morphism, and let $n\ge0$. Then, for any $v\in V$ such that $f(v)\in\F_nW$, we have
	\[
	f(v+\F_nV) = f(v)+f(\F_nV)\subseteq\F_nW\ .
	\]
	Therefore, we have
	\[
	f^{-1}(\F_nW) = \bigcup_{v\in f^{-1}(\F_nW)} v+\F_nV\ ,
	\]
	which implies the claim. We denote by $\cchain$ the full subcategory of $\fchain$ spanned by complete filtered morphisms.
\end{proof}

\subsection{Limits and colimits of filtered chain complexes}

We can describe all limits and colimits in the category of filtered chain complexes.

\begin{proposition}
	The category $\fchain$ of filtered chain complexes is complete and cocomplete, with limits and colimits given as follows.
	\begin{enumerate}
		\item The product of a collection $\{(V^i,\F_\bullet V^i)\}_{i\in I}$ of filtered chain complexes is the chain complex $\prod_{i\in I}V^i$ together with the filtration
		\[
		\F_n\left(\prod_{i\in I}V^i\right)\coloneqq\prod_{i\in I}\F_nV^i.
		\]
		\item The coproduct of a collection $\{(V^i,\F_\bullet V^i)\}_{i\in I}$ of filtered chain complexes is the chain complex $\bigoplus_{i\in I}V^i$ together with the filtration
		\[
		\F_n\left(\bigoplus_{i\in I}V^i\right)\coloneqq\bigoplus_{i\in I}\F_nV^i.
		\]
		\item If $f:(V,\F_\bullet V)\to(W,\F_\bullet W)$ is a filtered chain map, then the kernel of $f$ is the chain complex $\ker f\subseteq V$ endowed with the filtration
		\[
		\F_n\ker f\coloneqq\ker f\cap\F_n V\ .
		\]
		Given two filtered chain maps $f,g$ with the same domain and range, their equalizer is $\ker(f-g)$.
		\item If $f:(V,\F_\bullet V)\to(W,\F_\bullet W)$ is a filtered chain map, then the cokernel of $f$ is the chain complex $W/f(V)$ endowed with the filtration
		\[
		\F_n(W/f(V))\coloneqq \F_nW/(f(V)\cap\F_nW)\ .
		\]
		Given two filtered chain maps $f,g$ with the same domain and range, their coequalizer is the cokernel of the difference $f-g$.
	\end{enumerate}
	In particular, limits and colimits of proper filtered chain complexes are again proper.
\end{proposition}

\subsection{The closed symmetric monoidal structure}

Let $V,W$ be two filtered chain complexes. On the chain complex $\hom(V,W)$ we put the filtration
\[
\F_n\hom(V,W)\coloneqq\{f:V\to W\mid f(\F_kV)\subseteq\F_{k+n}W\text{ for all }k\ge0\}\ .
\]
This makes $\hom(V,W)$ into a filtered chain complex. Similarly, we endow the chain complex $V\otimes W$ with the filtration
\[
\F_n(V\otimes W)\coloneqq\sum_{m+k=n}\F_mV\otimes\F_kW\ .
\]
This makes $V\otimes W$ into a filtered chain complex.

\begin{theorem}
	The category $\fchain$ with the tensor product and the inner hom defined as above is a closed symmetric monoidal category.
\end{theorem}

\subsection{Complete chain complexes}

If one desires to consider infinite sums in a chain complex, asking for a filtration is not enough. Therefore, one defines a complete chain complex as follows.

\begin{definition}
	A \emph{complete chain complex}\index{Complete!chain complex} $V$ is a filtered chain complex such that the canonical map
	\[
	V\longrightarrow\lim_{n}V/\F_nV
	\]
	is an isomorphism.
\end{definition}

In a complete chain complex, an infinite sum $\sum_{k\ge0}v_k$ of terms $v_k\in V$ makes sense as long as for all $n\ge0$ there is a $k_0$ such that $v_k\in\F_nV$ for all $k\ge k_0$.

\begin{lemma}\label{lemma:complete then intersection of filtration is zero}
	If a filtered chain complex $V$ is complete, then
	\[
	\bigcap_{n\ge1}\F_nV = \{0\}\ .
	\]
\end{lemma}

\begin{proof}
	The kernel of the canonical map
	\[
	V\longrightarrow\lim_{n}V/\F_nV
	\]
	is the intersection $\bigcap_{n\ge1}\F_nV$.
\end{proof}

Given a filtered chain complex $V$, one defines its \emph{completion} as the filtered chain complex
\[
\widehat{V}\coloneqq\lim_nV/\F_nV\ ,\qquad\text{and}\qquad\F_n\widehat{V}\coloneqq\lim_k\F_nV/\F_kV = \widehat{\F_nV}\ .
\]

\begin{lemma}
	Let $V$ be a filtered chain complex. The completion $\widehat{V}$ is complete.
\end{lemma}

\begin{proof}
	The statement follows immediately from the fact that
	\[
	\widehat{V}/\F_n\widehat{V}\cong V/\F_nV\ .
	\]
	Fix $n\ge1$. An element of $\widehat{V}$ can be given as a sequence $(v_k)_{k=1}^\infty$ with $v_k\in V/\F_kV$, and $v_{k+1}$ mapping to $v_k$ under the canonical projection $\widehat{V}/\F_{k+1}\widehat{V}\to\widehat{V}/\F_k\widehat{V}$. Fix such an element, and choose representatives in $V$ for the $v_k$. We will abuse of notation and again write $v_k$ for the chosen representative of $v_k$. We define a sequence $(w_k)_{k=1}^\infty$ by
	\[
	w_k\coloneqq\begin{cases}
	0&\text{if }k\le n\ ,\\
	v_k-v_n&\text{if }k\ge n+1\ ,
	\end{cases}
	\]
	where we see $w_k$ as an element of $V/\F_kV$. It is straightforward to check that $(w_k)_{k=1}^\infty$ defines an element of $\F_n\widehat{V}$. Then
	\[
	(v_k)_{k=1}^\infty - (w_k)_{k=1}^\infty = (v_1,v_2,\ldots,v_n,v_n,v_n,\ldots)\ .
	\]
	Thus, we can map the class of $(v_k)_{k=1}^\infty$ in $\widehat{V}/\F_n\widehat{V}$ to $v_n\in V/\F_nV$. This is well defined, and gives us a map
	\[
	\widehat{V}/\F_n\widehat{V}\longrightarrow V/\F_nV\ .
	\]
	This map is an isomorphism, as it has an inverse which is given by sending $v\in V/\F_nV$ to $(v,v,v,v,\ldots)\in\widehat{V}/\F_n\widehat{V}$ (where once again we confound equivalence classes and representatives).
\end{proof}

Morphisms between filtered chain complexes induce morphisms between their completions. With this action, completion defines a functor
\[
\widehat{\ }:\fchain\longrightarrow\cchain
\]
from filtered chain complexes to complete chain complexes.

\begin{lemma}
	The completion functor is left adjoint to the forgetful functor
	\[
	(-)^\#:\cchain\longrightarrow\fchain\ .
	\]
\end{lemma}

\begin{proposition}
	The category of complete chain complexes is complete and cocomplete. Its limits are the same as the limits taken in the category of filtered chain complexes, and its colimits are the completion of the colimits taken in the category of filtered chain complexes. The subcategory of proper complete chain complexes is also complete and cocomplete, with the same limits and colimits.
\end{proposition}

Finally, one defines a tensor product in the category of complete chain complexes as the completion of the tensor product in the category of filtered chain complexes.

\section{Proper complete algebras over an operad}\label{sect:proper complete algebras}

We are interested in the notion of filtered, proper, and complete algebras over an operad. One can also define a natural notion of filtered operad, but it is not extremely interesting for the applications we have in mind in the rest of the present work.

\subsection{Filtered algebras over operads and related notions}\label{subsect:filtered algebras etc}

Let $V$ be a filtered chain complex. We define the \emph{filtered endomorphism operad}\index{Filtered!endomorphism operad} by
\[
\End_V(n)\coloneqq(\hom(V^{\otimes n},V),\F_\bullet\hom(V^{\otimes n},V))\ .
\]
Notice that the composition map is a filtered morphism. Also, if $V$ is a complete chain complex, it is indifferent if we take the completed tensor product or the usual tensor product of filtered chain complexes in the definition.

\medskip

Let $\P$ be an operad, then a \emph{filtered}\index{Filtered!$\P$-algebra}, respectively \emph{complete}\index{Complete!$\P$-algebra} (and possibly \emph{proper}\index{Proper!$\P$-algebra}) \emph{$\P$-algebra} is a filtered, resp. complete (proper) chain complex $A$ together with a morphism of operads
\[
\rho_A:\P\longrightarrow\End_A\ .
\]
In other words, it is a map
\[
\gamma_A:\P(A)\longrightarrow A
\]
satisfying the usual axioms for a $\P$-algebra and which is such that
\[
\gamma_A(\P(k)\otimes\F_{n_1}A\otimes\cdots\otimes\F_{n_k}A)\subseteq\F_{n_1+\cdots+n_k}A\ .
\]
We will sometimes denote the category of proper complete $\P$-algebras with filtered morphisms of $\P$-algebras by $\Phatalg$.

\subsection{Filtered \texorpdfstring{$\infty$}{infinity}-morphisms}

The notion of a filtered morphism of algebra over an operad is immediate. Filtered $\infty$-mor\-phi\-sms are not much more complicated.

\medskip

Let $\alpha:\C\to\P$ be a twisting morphism, and let $A,B$ be two filtered $\P$-algebras.

\begin{definition}
	A \emph{filtered $\infty_\alpha$-morphism}\index{Filtered!$\infty_\alpha$-morphism} is an $\infty_\alpha$-morphism\footnote{See \cref{chapter:oo-morphisms relative to a twisting morphism}.} $\Phi:A\rightsquigarrow B$ such that
	\[
	\phi_n(\C(n)\otimes\F_{n_1}A\otimes\cdots\otimes\F_{n_k}A)\subseteq\F_{n_1+\cdots+n_k}B\ .
	\]
	It is a \emph{filtered $\infty_\alpha$-quasi-isomorphism}\index{Filtered!$\infty_\alpha$-quasi-isomorphism} if for every $n\ge1$, the chain map
	\[
	\phi_1:\F_nA\longrightarrow\F_nB
	\]
	is a quasi-isomorphism.
\end{definition}

\subsection{Filtrations and cofiltrations}

In \cref{subsection:Vallette model structure on coalgebras}, we mentioned te notion of a cofiltered coalgebra over a cooperad, see \cref{def:cofiltered coalgebras}. The relation between filtered algebras and cofiltered coalgebras is strightforward.

\begin{lemma}
	Let $C$ be a cofiltered $\C$-coalgebra. Then $C^\vee$ is a complete $\C^\vee$-algebra when endowed with the orthogonal filtration
	\[
	\F_nC^\vee\coloneqq(\F^nC)^\perp.
	\]
\end{lemma}

\begin{proof}
	Straightforward.
\end{proof}

Notice that in the definition of cofiltered coalgebra we require $\colim_n\F^nC\cong C$, which corresponds to completeness on the algebra side. The condition of being proper corresponds to $\F^1C = 0$.

%% file: FormulaVS.tex
\chapter{Formal fixed-point equations and differential equations}\label{appendix:formal fixed-pt eq and diff eq}

We present some results about fixed-point equations and differential equations in the setting of complete, possibly graded, vector spaces --- and in particular, complete chain complexes. We prove that under reasonable assumptions, both fixed-point equations and differential equations admit unique solutions in this context.

\section{Formal fixed-point equations}

We begin by defining formal fixed-point equations in complete graded vector spaces, i.e. equations of the form $x=f(x)$, with some conditions on $f$. The conditions we impose imply that any such equation admits a unique solution.

\subsection{Definitions}

Let $V$ be a complete graded vector space.

\begin{definition}
	A \emph{homogeneous polynomial operator} $P$ of (polynomial) degree $k$ on $V$ is a map $P:V\to V$ of the form
	\[
	P(v) = F(v^{\otimes k})
	\]
	for some linear map
	\[
	F:V^{\otimes k}\longrightarrow V\ .
	\]
	We call $F$ the linear map associated to $P$. A \emph{polynomial operator}\index{Polynomial!operator} on $V$ is a finite sum of homogeneous polynomial operators.
\end{definition}

\begin{definition}
	A \emph{formal fixed-point equation}\index{Formal!fixed-point equation} on $V$ is an equation of the form
	\[
	x = P_0 + \sum_{n\ge1}P_n(x)\ ,
	\]
	in the variable $x$, where $P_0\in V$ and for $n\ge2$, $P_n$ is a polynomial operator on $V$ such that the linear map associated to any homogeneous polynomial composing $P_n$ increases the filtration by $n$.
\end{definition}

\subsection{Existence and uniqueness of solutions}

We show that every formal fixed-point equation admits a unique solution.

\begin{theorem}\label{thm:existence and uniqueness of solution for fixed point equations}
	Let
	\begin{equation}\label{eq:formal fixed-pt equation}
		x = P_0 + \sum_{n\ge1}P_n(x)\ ,
	\end{equation}
	be a formal fixed-point equation on $V$. Then there exists a unique element of $V$ solving the equation.
\end{theorem}

\begin{proof}
	We begin by constructing a solution recursively. At step $i\ge0$ we want to have an element $v_i\in\F_iV$ such that $\widetilde{v}_i\coloneqq v0+v_1+\cdots+v_i$ solves the equation in $V/\F_{i+1}V$.
	
	\medskip
	
	For steps $0$ and $1$, in $V/\F_2V$ the equation becomes
	\[
	x = P_0\in V=\F_1V\ ,
	\]
	as all other terms live at least in $\F_2V$. Therefore, we fix $v_0\coloneqq P_0$ and $v_1 = 0$.
	
	\medskip
	
	Suppose we have our procedure successfully up to step $i\ge1$. We take equation (\ref{eq:formal fixed-pt equation}), substitute the variable $x$ by $\widetilde{v}_i+y$ and project in $V/\F_{i+2}V$. Under the additional assumption that $y\in\F_{i+1}V$, the resulting equation reads
	\begin{align*}
		\widetilde{v}_i+y =&\ P_0 + \sum_{n\ge1}P_n(\widetilde{v}_i+y)\\
		=&\ P_0 + \sum_{n\ge1}P_n(\widetilde{v}_i)\ ,
	\end{align*}
	since by the definition of a formal fixed-point equation we have
	\[
	P_n(\widetilde{v}_i+y) = P_n(\widetilde{v}_i) + \text{terms in }\F_{i+2}V\ .
	\]
	Therefore, we define
	\[
	v_{i+1}\coloneqq P_0 + \sum_{n\ge1}P_n(\widetilde{v}_i) - \widetilde{v}_i\in V\ .
	\]
	It is in fact an element of $\F_{i+1}V$ because its projection to $V/\F_{i+1}V$ gives back the equation that we solved at step $i$.
	
	\medskip
	
	Passing to the limit for $i$ going to infinity is allowed because we are in a complete graded vector space. Therefore, the process above gives us a solution of the fixed point equation.
	
	\medskip
	
	For uniqueness, suppose that $v,w\in V$ both solve the fixed-point equation. We show by induction that $v-w\in\F_iV$ for all $i\ge0$. The cases $i=0$ and $1$ are trivially satisfied. For $i=2$, in $V/\F_2V$ we have
	\[
	v = P_0 = w\ .
	\]
	Therefore, we have $v-w\in\F_2V$.
	
	\medskip
	
	Suppose the statement is true up to $i\ge2$. Then the fixed-point equation (\ref{eq:formal fixed-pt equation}) in $V/\F_{i+1}V$ reads
	\begin{align*}
		v =&\ P_0 + \sum_{n\ge1}P_n(v)\\
		=&\ P_0 + \sum_{n\ge1}P_n(w)\\
		=&\ w\ ,
	\end{align*}
	where in the second line we used the fact that for all $n\ge1$ we have
	\begin{align*}
		P_n(v) =&\ P_n(w + (v-w))\\
		=&\ P_n(w) + \text{terms in }\F_{i+1}V
	\end{align*}
	by induction hypothesis and the definition of a formal fixed-point equation. Therefore, we have $v-w\in\F_{i+1}V$.
	
	\medskip
	
	We conclude by noticing that this implies that
	\[
	v-w\in\bigcap_{i\ge1}\F_iV = \{0\}\ ,
	\]
	where we used \cref{lemma:complete then intersection of filtration is zero}.
\end{proof}

\section{Formal differential equations on vector spaces}

We define formal differential equations, and prove an existence and uniqueness result for their solution.

\subsection{Definitions}

Let $\k[[t]]$ be the formal power series in the variable $t$. If $V$ is a (potentially graded) vector space, we denote
\[
V[[t]]\coloneqq V\otimes\k[[t]]\ ,
\]
the formal power series in $t$ with coefficients in $V$. We have the ``differentiation in $t$" operator
\[
\frac{d}{dt}:V[[t]]\longrightarrow V[[t]]
\]
acting by
\[
\frac{d}{dt}\left(\sum_{n\ge0}v_nt^n\right)\coloneqq \sum_{n\ge1}nv_nt^{n-1}\ .
\]
Given a linear map
\[
f:V^{\otimes n}\longrightarrow V\ ,
\]
there is a canonical extension
\[
f:V[[t]]^{\otimes k}\longrightarrow V[[t]]\ .
\]
It is given by
\[
f\left(\sum_{n_1\ge0}v_{n_1,1}t^{n_1},\ldots,\sum_{n_k\ge0}v_{n_k,k}t^{n_k}\right)\coloneqq\sum_{n\ge0}\sum_{n_1+\cdots+n_k = n}f(v_{n_1,1},\ldots,v_{n_k,k})t^n\ .
\]
Differentiation in $t$ acts as one expects on such maps.

\begin{definition}
	Let $V$ be a vector space. For each $k,n\ge0$ let
	\[
	f_{n,k}:V^{\otimes n}\longrightarrow V
	\]
	be linear maps, and suppose that for any $k$, there are finitely many non-zero $f_{n,k}$. A \emph{formal differential equation}\index{Formal!differential equation} in $V$ is an equation of the form
	\begin{equation}\label{eq:formal differential equation}
		\frac{d}{dt}v(t) = \sum_{n,k\ge0}t^kf_{n,k}(\underbrace{v(t),\ldots,v(t)}_{n\text{\ times}})
	\end{equation}
	in $V[[t]]$.
\end{definition}

\begin{remark}
	The condition that for any $k$ there are finitely many non-zero $f_{n,k}$ is necessary in order for equation (\ref{eq:formal differential equation}) to make sense. However, it can sometimes be relaxed. For example, if $V$ itself is a proper complete chain complex --- see \cref{appendix:filtered stuff} --- and the $f_{k,n}$ are filtered maps, then we don't have to make any finiteness assumption.
\end{remark}

Notice that for any $v(t)\in V[[t]]$, the value $v(0)\in V$ is well-defined. This in not true for evaluation at any other $t\in\k$ if one does not make any further assumptions.

\begin{lemma}
	A solution of an equation on the form (\ref{eq:formal differential equation}) exists and is unique for any fixed initial value at $t=0$.
\end{lemma}

\begin{proof}
	This is a straightforward induction: one proves that the $n$th coefficient is completely determined by the coefficients up to the $(n-1)$th, together with the fact that the initial value fixes the $0$th coefficient.
\end{proof}

\begin{remark}
	One can also prove this lemma by imitating formally the proof of the classical Cauchy--Lipschitz theorem\footnote{Or Picard--Lindel\"of theorem, depending on who was the reader's lecturer in his or her first course in real analysis.} to obtain a formal fixed-point equation
	\[
	v(t) = v_0 + \int_0^t\sum_{n,k\ge0}s^kf_{n,k}(v(s),\ldots,v(s))\,ds
	\]
	in the complete graded vector space $V[[t]]$, where the filtration is given by
	\[
	\F_nV[[t]]\coloneqq\{\text{power series starting at the term }t^n\}\ ,
	\]
	and then conclude by \cref{thm:existence and uniqueness of solution for fixed point equations}.
\end{remark}

\subsection{Solving formal differential equations}

We will now exhibit an explicit formula for the solution of a formal differential equation.

\medskip

Denote by $\wPT$ the set of rooted weighted planar trees, that is planar trees with the assignment of an integer weight greater or equal to $1$ to each vertex. We include the empty tree $\emptyset$ and the arity $0$ and $1$ corollas in $wPT$. Grafting the empty tree to any tree gives back the original tree. We denote by $c_n^{(w)}$ the $n$-corolla of weight $w$.

\medskip

Denoting $\tau = c_n^{(w)}\circ(\tau_1,\ldots,\tau_n)\in\wPT$ the grafting of the trees $\tau_1,\ldots,\tau_n\in\wPT$ to $c_n^{(w)}$, we construct some functions on $\wPT$ recursively. First we have the total weight function
\[
W(\emptyset) = 0,\qquad W(c_n^{(w)}) = w,\qquad W(\tau) = w + \sum_{i=1}^nW(\tau_i)\ ,
\]
which gives the sum of all the weights of the tree. Then we have the following coefficient
\[
F(\emptyset) = 1,\qquad F(c_n^{(w)}) = w,\qquad F(\tau) = W(\tau)\prod_{i=1}^nF(\tau_i)\ .
\]
We associate a function
\[
\tau:V[[t]]^{\otimes n}\longrightarrow V[[t]]
\]
to each tree in $\wPT$ by
\[
\emptyset\coloneqq \id,\qquad c_n^{(w)}\coloneqq t^wf_{n,w-1},\qquad \tau\coloneqq t^wf_{n,w-1}\circ(\tau_1,\ldots,\tau_n)\ .
\]
Finally, fix $v_0\in V$. We assign an element of $V[[t]]$ to each tree in $\wPT$ by
\[
\tau(v_0)\coloneqq\tau(v_0,\ldots,v_0)\ .
\]

\begin{proposition}\label{prop:formula for formal ODE}
	The element of $V[[t]]$ given by
	\[
	v(t) = \sum_{\tau\in wPT}\frac{1}{F(\tau)}\tau(v_0)
	\]
	is the unique solution for the formal differential equation (\ref{eq:formal differential equation}) with initial value $v(0) = v_0$.
\end{proposition}

\begin{proof}
	Let $\tau\in\wPT$, then notice that the total exponent of the formal variable $t$ in $\tau(v_0)$ is given by the weight $W(\tau)$. Therefore, we have
	\begin{align*}
	\frac{d}{dt}v(t) =&\ \frac{d}{dt}\left(\sum_{\tau\in wPT}\frac{1}{F(\tau)}\tau(v_0)\right)\\
	=&\ \frac{d}{dt}\left(v_0 + \sum_{\substack{n\ge0,w\ge1\\\tau_1,\ldots,\tau_n\in wPT}}\frac{1}{F(\tau)}\tau(v_0)\right)\\
	=&\ \sum_{\substack{n\ge0,w\ge1\\\tau_1,\ldots,\tau_n\in wPT}}\frac{1}{F(\tau)}\frac{d}{dt}\tau(v_0)\\
	=&\ \sum_{\substack{n\ge0,w\ge1\\\tau_1,\ldots,\tau_n\in wPT}}\frac{W(\tau)}{F(\tau)}t^{w-1}f_{n,w-1}(\tau_1(v_0),\ldots,\tau_n(v_0))\\
	=&\ \sum_{\substack{n\ge0,w\ge1\\\tau_1,\ldots,\tau_n\in wPT}}\frac{1}{\prod_{i=1}^nF(\tau_i)}t^{w-1}f_{n,w-1}(\tau_1(v_0),\ldots,\tau_n(v_0))\\
	=&\ \sum_{n,k\ge0}t^kf_{n,k}\left(\sum_{\tau_1\in wPT}\frac{1}{F(\tau_1)}\tau_1(v_0),\ldots,\sum_{\tau_n\in wPT}\frac{1}{F(\tau_n)}\tau_n(v_0)\right)\\
	=&\ \sum_{n,k\ge0}t^kf_{n,k}(v(t),\ldots,v(t))\ .
	\end{align*}
	In the second line, we wrote the sum over all weighted planar trees as the sum over all possible combinations of a corolla at the root with subtrees grafted at its leaves. Implicitly, we denoted $\tau = c_n^{(w)}\circ(\tau_1,\ldots,\tau_n)$. In the fourth and fifth line, we used the definition of $\tau(v_0)$ and of $F(\tau)$. In the sixth line, we renamed $k=w-1$ and used the fact that $f_{n,k}$ is multilinear. The fact that we have the correct initial value is obvious from the fact that the only term which doesn't involve the formal variable $t$ is $\emptyset(v_0) = v_0$.
\end{proof}

\begin{remark}
	The material presented here has close relations with the B-series studied in numerical analysis, cf. e.g. with \cite[Ch. 38]{Butcher}.
\end{remark}

%% file: Thesis.bbl
\begin{thebibliography}{BFMT17}

\bibitem[AGV72]{SGA4.1}
M.~Artin, A.~Grothendieck, and J.L. Verdier.
\newblock {\em S{\'e}minaire de {G\'e}om{\'e}trie {A}lg{\'e}brique du {B}ois
  {M}arie - 1963--64 - Th{\'e}orie des topos et cohomologie {\'e}tale des
  sch{\'e}mas - ({SGA} 4)}, volume~1.
\newblock Springer, 1972.

\bibitem[Ban14]{Bandiera}
R.~Bandiera.
\newblock {\em Higher Deligne groupoids, derived brackets and de\-for\-ma\-tion
  pro\-blems in ho\-lo\-morphic Poisson geometry}.
\newblock PhD thesis, Universit{\`a} degli studi di {R}oma {``La
  Sa\-pien\-za"}, 2014.

\bibitem[Ban17]{ban17}
R.~Bandiera.
\newblock Descent of {D}eligne--{G}etzler $\infty$-groupoids.
\newblock 2017.
\newblock \href{https://arxiv.org/abs/1705.02880}{arXiv:1705.02880}.

\bibitem[Ber14a]{ber14transfer}
A.~Berglund.
\newblock Homological perturbation theory for algebras over operads.
\newblock {\em Algebraic \& {G}eometric {T}opology}, 14(5):2511--2548, 2014.
\newblock \href{https://arxiv.org/abs/0909.3485}{arXiv:0909.3485}.

\bibitem[Ber14b]{ber14koszul}
A.~Berglund.
\newblock Koszul spaces.
\newblock {\em Transactions of the {American Mathematical Society}},
  366(9):4551--4569, 2014.
\newblock \href{https://arxiv.org/abs/1107.0685}{arXiv:1107.0685}.

\bibitem[Ber15]{ber15}
A.~Berglund.
\newblock Rational homotopy theory of mapping spaces via {L}ie the\-o\-ry for
  {$L_\infty$}-algebras.
\newblock {\em Homology, Homotopy and Applications}, 17(2):343--369, 2015.
\newblock \href{https://arxiv.org/abs/1110.6145}{arXiv:1110.6145}.

\bibitem[BF12]{BonfiglioliFulci}
A.~Bonfiglioli and R.~Fulci.
\newblock {\em Topics in {N}oncommutative {A}lgebra}, volume 2034 of {\em
  Lecture Notes in Mathematics}.
\newblock Springer, 2012.

\bibitem[BFM11]{bfm11}
U.~Buijs, Y.~F{\'e}lix, and A.~Murillo.
\newblock {$L_\infty$} models of based mapping spaces.
\newblock {\em Journal of the Mathematical Society of Japan}, 63(2):503--524,
  2011.

\bibitem[BFMT15]{bfmt15}
U.~Buijs, Y.~F\'elix, A.~Murillo, and D.~Tanr\'e.
\newblock Lie models of simplicial sets and representability of the {Q}uillen
  functor.
\newblock 2015.
\newblock \href{https://arxiv.org/abs/1508.01442}{arXiv:1508.01442}.

\bibitem[BFMT16]{bfmt16bis}
U.~Buijs, Y.~F\'elix, A.~Murillo, and D.~Tanr\'e.
\newblock Homotopy theory of complete {L}ie algebras and {L}ie models of
  simplicial sets.
\newblock 2016.
\newblock \href{https://arxiv.org/abs/1601.05331}{arXiv:1601.05331}.

\bibitem[BFMT17]{bfmt17}
U.~Buijs, Y.~F\'elix, A.~Murillo, and D.~Tanr\'e.
\newblock The infinity {Q}uillen functor, {M}aurer--{C}artan elements and {DGL}
  realizations.
\newblock 2017.
\newblock \href{https://arxiv.org/abs/1702.04397}{arXiv:1702.04397}.

\bibitem[BFT17]{bft17}
R.~Blumenhagen, M.~Fuchs, and M.~Traube.
\newblock $\mathcal{W}$ algebras are {$L_\infty$} algebras.
\newblock 2017.
\newblock \href{https://arxiv.org/abs/1705.00736}{arXiv:1705.00736}.

\bibitem[BG16]{bg16}
U.~Buijs and J.~J. Guti{\'e}rrez.
\newblock Homotopy transfer and rational models for mapping spaces.
\newblock {\em Journal of Homotopy and Related Structures}, 11(2):309--332,
  2016.
\newblock \href{https://arxiv.org/abs/1210.4664}{arXiv:1210.4664}.

\bibitem[BGNT15]{bgnt15}
P.~Bressler, A.~Gorokhovsky, R.~Nest, and B.~Tsygan.
\newblock Deligne groupoid revisited.
\newblock {\em Theory and Applications of Categories}, 30(29):1001--1016, 2015.
\newblock \href{https://arxiv.org/abs/1211.6603}{arXiv:1211.6603}.

\bibitem[BL15]{bl15}
C.~Brown and A.~Lazarev.
\newblock Unimodular homotopy algebras and {C}hern--{S}imons theory.
\newblock {\em Journal of Pure and Applied Algebra}, 219(11):5158--5194, 2015.
\newblock \href{https://arxiv.org/abs/1508.01442}{arXiv:1309.3219v3}.

\bibitem[BM03]{bm03}
C.~Berger and I.~Moerdijk.
\newblock Axiomatic homotopy theory for operads.
\newblock {\em Commentarii Mathematici Helvetici}, 78(4):805--831, 2003.
\newblock \href{https://arxiv.org/abs/math/0206094}{arXiv:math/0206094}.

\bibitem[But08]{Butcher}
J.~C. Butcher.
\newblock {\em Numerical {M}ethods for {O}rdinary {D}ifferential {E}quations}.
\newblock Wiley, 2nd edition, 2008.

\bibitem[BV73]{BoardmanVogt}
J.~M. Boardman and R.~M. Vogt.
\newblock {\em Homotopy invariant algebraic structures on topological spaces},
  volume 347 of {\em Lecture Notes in Mathematics}.
\newblock Springer, 1973.

\bibitem[CE48]{ce48}
C.~Chevalley and S.~Eilenberg.
\newblock Cohomology theory of {L}ie groups and {L}ie algebras.
\newblock {\em Transactions of the AMS}, 63:85--124, 1948.

\bibitem[CG08]{cg08}
X.~Z. Cheng and E.~Getzler.
\newblock Transferring homotopy commutative algebraic structures.
\newblock {\em Journal of Pure and Applied Algebra}, 212:2535--2542, 2008.
\newblock \href{https://arxiv.org/abs/math/0610912}{arXiv:math/0610912}.

\bibitem[DCH16]{dch16}
G.~C. Drummond-Cole and J.~Hirsh.
\newblock Model structures for coalgebras.
\newblock {\em Proceedings of the AMS}, 144(4):1467--1481, 2016.
\newblock \href{https://arxiv.org/abs/1411.5526}{arXiv:1411.5526}.

\bibitem[DCV13]{dcv13}
G.~C. Drummond-Cole and B.~Vallette.
\newblock The minimal model for the {B}atalin--{V}il\-ko\-vi\-sky operad.
\newblock {\em Selecta Mathematica, New Series}, pages 1--47, 2013.
\newblock \href{https://arxiv.org/abs/1105.2008}{arXiv:1105.2008}.

\bibitem[Del87]{letterDeligneMillson}
P.~Deligne.
\newblock Correspondence with {J}. {J}. {M}illson, 1987.
\newblock Can be found at
  \href{https://publications.ias.edu/sites/default/files/millson.pdf}{publications.ias.edu/sites/default/files/millson.pdf}
  as of May 2018.

\bibitem[Del94]{letterDeligneBreen}
P.~Deligne.
\newblock Correspondence with {L}. {B}reen, February 1994.
\newblock Can be found at
  \href{http://www.math.northwestern.edu/~getzler/Del1gne.pdf}{www.math.northwestern.edu/{\raise.17ex\hbox{$\scriptstyle\sim$}}getzler/Del1gne.pdf}
  as of April 2018.

\bibitem[DHR15]{dhr15}
V.~A. Dolgushev, A.~E. Hoffnung, and C.~L. Rogers.
\newblock What do homotopy algebras form?
\newblock {\em Advances in Mathematics}, 274:562--605, 2015.
\newblock \href{https://arxiv.org/abs/1406.1751}{arXiv:1406.1751}.

\bibitem[Dou66]{Douady}
A.~Douady.
\newblock {\em Le problème des modules pour les sous-espaces analytiques
  compacts d'un espace analytique donné}.
\newblock PhD thesis, Université de Paris, 1966.

\bibitem[DP16]{dp16}
V.~Dotsenko and N.~Poncin.
\newblock A tale of three homotopies.
\newblock {\em Applied Categorical Structures}, 24(6):845--873, 2016.
\newblock \href{https://arxiv.org/abs/1208.4695}{arXiv:1208.4695}.

\bibitem[DR15]{dr15}
V.~A. Dolgushev and C.~L. Rogers.
\newblock A version of the {G}oldman--{M}illson theorem for filtered
  {$\L_\infty$}-algebras.
\newblock {\em Journal of Algebra}, 430:260--302, 2015.
\newblock \href{https://arxiv.org/abs/1407.6735}{arXiv:1407.6735}.

\bibitem[DR17]{dr17}
V.~A. Dolgushev and C.~L. Rogers.
\newblock On an enhancement of the category of shifted {$\L_\infty$}-algebras.
\newblock {\em Applied Categorical Structures}, 25(4):489--503, 2017.
\newblock \href{https://arxiv.org/abs/1406.1744}{arXiv:1406.1744}.

\bibitem[DS95]{ds95}
W.~G. Dwyer and J.~Spali\'nski.
\newblock Homotopy theories and model categories.
\newblock In {\em Handbook of algebraic topology}, pages 73--126.
  North-Holland, Amsterdam, 1995.

\bibitem[DSV15]{dsv15}
V.~Dotsenko, S.~Shadrin, and B.~Vallette.
\newblock De {R}ham cohomology and homotopy {F}robenius manifolds.
\newblock {\em Journal of the European Mathematical Society}, 17:535--547,
  2015.
\newblock \href{https://arxiv.org/abs/1203.5077}{arXiv:1203.5077}.

\bibitem[DSV16]{dsv16}
V.~Dotsenko, S.~Shadrin, and B.~Vallette.
\newblock Pre-{L}ie deformation theory.
\newblock {\em Moskow {M}athematical {J}ournal}, 16(3):505--543, 2016.
\newblock \href{https://arxiv.org/abs/1502.03280}{arXiv:1502.03280}.

\bibitem[DSV18]{dsv18}
V.~Dotsenko, S.~Shadrin, and B.~Vallette.
\newblock The twisting procedure.
\newblock 2018.
\newblock Draft.

\bibitem[Dup76]{dup76}
J.~L. Dupont.
\newblock Simplicial de {R}ham cohomology and characteristic classes of flat
  bundles.
\newblock {\em Topology}, 15:233--245, 1976.

\bibitem[FHT01]{FelixHalperinThomas}
Y.~F{\'e}lix, S.~Halperin, and J.-C. Thomas.
\newblock {\em Rational homotopy theory}, volume 205 of {\em Graduate Texts in
  Mathematics}.
\newblock Springer Verlag, 2001.

\bibitem[FOOO09]{FukayaOhOhtaOno}
K.~Fukaya, Y.-G. Oh, H.~Ohta, and K.~Ono.
\newblock {\em Lagrangian intersection {F}loer theory}.
\newblock {AMS}, 2009.

\bibitem[Fre04]{fre04}
B.~Fresse.
\newblock Koszul duality of operads and homology of partition posets.
\newblock {\em Contemporary Mathematics}, 346:115--215, 2004.
\newblock \href{https://arxiv.org/abs/math/0301365}{arXiv:math/0301365}.

\bibitem[Fre09a]{Fresse}
B.~Fresse.
\newblock {\em Modules over operads and functors}.
\newblock Lecture {N}otes in {M}athematics. Springer, 2009.

\bibitem[Fre09b]{fre09}
B.~Fresse.
\newblock Operadic cobar constructions, cylinder objects and homotopy morphisms
  of algebras over operads.
\newblock {\em Contemporary {M}athematics}, 504:125, 2009.
\newblock \href{https://arxiv.org/abs/0902.0177}{arXiv:0902.0177}.

\bibitem[GCTV12]{gctv12}
I.~Galvez-Carrillo, A.~Tonks, and B.~Vallette.
\newblock Homotopy {B}atalin--{V}ilkovisky algebras.
\newblock {\em Journal of Noncommutative Geometry}, 6(3):539--602, 2012.
\newblock \href{https://arxiv.org/abs/0907.2246}{arXiv:0907.2246}.

\bibitem[Ger63]{ger63}
M.~Gerstenhaber.
\newblock The cohomology structure of an associative ring.
\newblock {\em Annals of Mathematics}, 78(2):267--288, 1963.

\bibitem[Get09]{get09}
E.~Getzler.
\newblock Lie theory for nilpotent {$\L_\infty$}-algebras.
\newblock {\em Annals of Mathematics}, 170(1):271--301, 2009.
\newblock \href{https://arxiv.org/abs/math/0404003}{arXiv:math/0404003}.

\bibitem[Get18]{get18}
E.~Getzler.
\newblock Maurer--{C}artan elements and homotopical perturbation theory.
\newblock 2018.
\newblock \href{https://arxiv.org/abs/1802.06736}{arXiv:1802.06736}.

\bibitem[GG99]{gg99}
E.~Getzler and P.~Goerss.
\newblock A model category structure for differential graded coalgebras.
\newblock 1999.
\newblock Available at
  \href{https://ncatlab.org/nlab/files/GetzlerGoerss99.pdf}{ncatlab.org/nlab/files/GetzlerGoerss99.pdf}.

\bibitem[GJ94]{gj94}
E.~Getzler and J.~D.~S. Jones.
\newblock Operads, homotopy algebra and iterated integrals for double loop
  spaces.
\newblock 1994.
\newblock \href{https://arxiv.org/abs/hep-th/9403055}{arXiv:hep-th/9403055}.

\bibitem[GJ09]{GoerssJardine}
P.~G. Goerss and J.~F. Jardine.
\newblock {\em Simplicial {H}omotopy {T}heory}.
\newblock Birkh{\"a}user, 2009.

\bibitem[GK94]{gk94}
V.~Ginzburg and M.~Kapranov.
\newblock Koszul duality for operads.
\newblock {\em Duke {M}athematical {J}ournal}, 76(1):203--272, 1994.
\newblock \href{https://arxiv.org/abs/0709.1228}{arXiv:{ }0709.1228}.

\bibitem[GM88]{gm88}
W.~M. Goldman and J.~J. Millson.
\newblock The deformation theory of representations of fundamental groups of
  compact {K\"a}hler manifolds.
\newblock {\em Publications Math{\'e}matiques de l'{IHES}}, 67(1):43--96, 1988.

\bibitem[Gro60]{gro60}
A.~Grothendieck.
\newblock Technique de descente et th{\'e}or{\`e}mes d'existence en
  g{\'e}om{\'e}trie alg{\'e}briques {II}: le th{\'e}oreme d'existence en
  th{\'e}orie formelle des modules.
\newblock {\em Report of the {S}{\'e}minaire {N}icolas {B}ourbaki},
  1958--1960(195), 1960.

\bibitem[Hal15]{Hall}
B.~C. Hall.
\newblock {\em Lie {G}roups, {L}ie {A}lgebras, and {R}epresentations}, volume
  222 of {\em Graduate Texts in Mathematics}.
\newblock Springer, 2nd edition, 2015.

\bibitem[Har62]{har62}
D.~K. Harrison.
\newblock Commutative algebras and cohomology.
\newblock {\em Transactions of the AMS}, 104(2):191--204, 1962.

\bibitem[Hin97a]{hin97descent}
V.~Hinich.
\newblock Descent of {D}eligne groupoids.
\newblock {\em International Mathematics Research Notices}, 1997(5):223--239,
  1997.
\newblock
  \href{https://arxiv.org/abs/alg-geom/9606010}{arXiv:alg-geom/9606010}.

\bibitem[Hin97b]{hin97homological}
V.~Hinich.
\newblock Homological algebra of homotopy algebras.
\newblock {\em Communications in Algebra}, 25(10):3291--3323, 1997.
\newblock \href{https://arxiv.org/abs/q-alg/9702015}{arXiv:q-alg/9702015}.

\bibitem[Hoc45]{hoc45}
G.~Hochschild.
\newblock On the cohomology groups of an associative algebra.
\newblock {\em Annals of Mathematics}, 46(1):58--67, 1945.

\bibitem[Hov99]{Hovey}
M.~Hovey.
\newblock {\em Model {C}ategories}, volume~63 of {\em Mathematical Surveys and
  Monographs}.
\newblock AMS, 1999.

\bibitem[Kad80]{kad80}
T.~V. Kadeishvili.
\newblock On the homology theory of fibre spaces.
\newblock {\em Russian Mathematical Surveys}, 35(3):231--238, 1980.
\newblock \href{https://arxiv.org/abs/math/0504437}{arXiv:math/0504437}.

\bibitem[Kad88]{kad88}
T.~Kadeishvili.
\newblock The {$A(\infty)$}-algebra structure and {H}ochschild and {H}arrison
  cohomology.
\newblock {\em Trudy Tbiliss. Mat. Inst. Razmadze Akad. Nauk Gruzin. SSR},
  91:19--27, 1988.
\newblock \href{https://arxiv.org/abs/math/0210331}{arXiv:math/0210331}.

\bibitem[Kon94]{notesKontsevich}
M.~Kontsevich.
\newblock Topics in algebra --- deformation theory, 1994.
\newblock Lecture notes. Can be found at
  \href{http://www1.mat.uniroma1.it/people/manetti/DT2011/Kontsevich.pdf}{www1.mat.uniroma1.it/people/manetti/DT2011/Kontsevich.pdf}
  as of May 2018.

\bibitem[Kon95]{Kontsevich}
M.~Kontsevich.
\newblock Homological algebra of mirror symmetry.
\newblock In {\em Proceedings of the {I}nternational {C}ongress of
  {M}athematicians, {V}ol.\ 1, 2 ({Z\"u}rich, 1994)}, pages 120--139, Basel,
  1995. Birkh{\"a}user.

\bibitem[Kon03]{kon03}
M.~Kontsevich.
\newblock Deformation quantization of {P}oisson manifolds.
\newblock {\em Letters in Mathematical Physics}, 66(3):157--216, 2003.
\newblock \href{https://arxiv.org/abs/q-alg/9709040}{arXiv:q-alg/9709040}.

\bibitem[KS58a]{ks58i}
K.~Kodaira and D.~C. Spencer.
\newblock On deformations of complex analytic structures, {I}.
\newblock {\em Annals of Mathematics}, 67(2):328--401, 1958.

\bibitem[KS58b]{ks58ii}
K.~Kodaira and D.~C. Spencer.
\newblock On deformations of complex analytic structures, {II}.
\newblock {\em Annals of Mathematics}, 67(3):403--466, 1958.

\bibitem[KS00]{ks00}
M.~Kontsevich and Y.~Soibelman.
\newblock Homological mirror symmetry and torus fibrations.
\newblock 2000.
\newblock \href{https://arxiv.org/abs/math/0011041}{arXiv:math/0011041}.

\bibitem[Kur62]{kur62}
M.~Kuranishi.
\newblock On the locally complete families of complex analytic structures.
\newblock {\em Annals of Mathematics}, 75(3):536--577, 1962.

\bibitem[Laz54]{laz54}
M.~Lazard.
\newblock Sur les groupes nilpotents et les anneaux de {L}ie.
\newblock {\em Annales Scientifiques de l'ENS}, 71(3):101--190, 1954.

\bibitem[Laz13]{laz13}
A.~Lazarev.
\newblock {M}aurer--{C}artan moduli and models for function spaces.
\newblock {\em Advances in Mathematics}, 235:296--320, 2013.
\newblock \href{https://arxiv.org/abs/1109.3715}{arXiv:1109.3715}.

\bibitem[LG16]{leg16}
B.~Le~Grignou.
\newblock Homotopy theory of unital algebras.
\newblock 2016.
\newblock \href{https://arxiv.org/abs/1612.02254}{arXiv:1612.02254}.

\bibitem[LG17]{leg17}
B.~Le~Grignou.
\newblock Algebraic operads up to homotopy.
\newblock 2017.
\newblock \href{https://arxiv.org/abs/1707.03465}{arXiv:1707.03465}.

\bibitem[LGL18]{lgl18}
B.~Le~Grignou and D.~Lejay.
\newblock Homotopy theory of linear cogebras.
\newblock 2018.
\newblock \href{https://arxiv.org/abs/1803.01376}{arXiv:1803.01376}.

\bibitem[LM15]{lm13}
A.~Lazarev and M.~Markl.
\newblock Disconnected rational homotopy theory.
\newblock {\em Advances in {M}athematics}, 238:303--361, 2015.
\newblock \href{https://arxiv.org/abs/1305.1037}{arXiv:1305.1037}.

\bibitem[LS14]{ls10}
R.~Lawrence and D.~Sullivan.
\newblock A formula for topology/deformations and its significance.
\newblock {\em Fundamenta Mathematic{\ae}}, 225:229--242, 2014.
\newblock \href{https://arxiv.org/abs/math/0610949}{arXiv:math/0610949}.

\bibitem[Lur09]{Lurie}
J.~Lurie.
\newblock {\em Higher Topos Theory}.
\newblock Princeton University Press, 2009.

\bibitem[Lur11]{lur11}
J.~Lurie.
\newblock Derived algebraic geometry {X}: formal moduli problems.
\newblock 2011.
\newblock Draft, available at
  \href{http://www.math.harvard.edu/~lurie/papers/DAG-X.pdf}{www.math.harvard.edu/{\raise.17ex\hbox{$\scriptstyle\sim$}}lurie/papers/DAG-X.pdf}
  as of May 2018.

\bibitem[LV12]{LodayVallette}
J.~L. Loday and B.~Vallette.
\newblock {\em Algebraic Operads}, volume 346 of {\em Grundlehren der
  {M}athematischen {W}issenschaften}.
\newblock Springer, 2012.

\bibitem[Maj00]{Majewski}
M.~Majewski.
\newblock {\em Rational homotopical models and uniqueness}, volume 682.
\newblock {AMS}, 2000.

\bibitem[Man87]{man87}
Y.~I. Manin.
\newblock Some remarks on {K}oszul algebras and quantum groups.
\newblock {\em Universit\'e de Grenoble. Annales de l'Institut Fourier},
  37(4):191--205, 1987.

\bibitem[Man88]{man89}
Y.~I. Manin.
\newblock Quantum groups and noncommutative geometry.
\newblock {\em Universit\'e de {M}ontr\'eal, {C}entre de {R}echerches
  {M}ath\'ematiques, {M}ontreal, {QC}}, 1988.

\bibitem[Mar04]{mar04}
M.~Markl.
\newblock Homotopy algebras are homotopy algebras.
\newblock {\em Forum Math.}, 16(1):129--160, 2004.
\newblock \href{https://arxiv.org/abs/math/9907138}{arXiv:math/9907138}.

\bibitem[Mas58]{mas58}
W.S. Massey.
\newblock Some higher order cohomology operations.
\newblock {\em Int. Symp. Alg. Top. Mexico}, pages 145--154, 1958.

\bibitem[May69]{may69}
J.~P. May.
\newblock Matric {M}assey products.
\newblock {\em Journal of Algebra}, 12(4):533--568, 1969.

\bibitem[May06]{May}
J.~P. May.
\newblock {\em The geometry of iterated loop spaces}, volume 271.
\newblock Springer, 2006.

\bibitem[MF17]{mf17}
J.~M. Moreno-Fern{\'a}ndez.
\newblock {$A_\infty$} structures and {M}assey products.
\newblock 2017.
\newblock \href{https://arxiv.org/abs/1705.06897}{arXiv:1705.06897}.

\bibitem[ML70]{MacLane}
S.~Mac~Lane.
\newblock {\em Categories for the {W}orking {M}atehmatician}.
\newblock Springer, 2nd edition, 1970.

\bibitem[Moo55]{moo54}
J.~Moore.
\newblock Homotopie des complexes mono{\"i}deaux, i.
\newblock {\em S{\'e}minaire Henri Cartan}, 7(2):1--8, 1954--1955.

\bibitem[MV09]{mv09}
S.~Merkulov and B.~Vallette.
\newblock Deformation theory of representations of prop(erad)s {II}.
\newblock {\em Journal f{\"u}r die reine und angewandte Mathematik (Crelles
  Journal)}, 636:123--174, 2009.
\newblock \href{https://arxiv.org/abs/0707.0889}{arXiv:0707.0889}.

\bibitem[Pos11]{pos11}
L.~Positselski.
\newblock Two kinds of derived categories, {K}oszul duality, and
  co\-mo\-dule-con\-tra\-module correspondence.
\newblock {\em Memoirs of the AMS}, 212(996):vi+133, 2011.
\newblock \href{https://arxiv.org/abs/0905.2621}{arXiv:0905.2621}.

\bibitem[Pri10]{pri10}
J.~P. Pridham.
\newblock Unifying derived deformation theories.
\newblock {\em Advances in Mathematics}, 224(3):772--826, 2010.
\newblock \href{https://arxiv.org/abs/0705.0344}{arXiv:0705.0344}.

\bibitem[Qui67]{QuillenMC}
D.~G. Quillen.
\newblock {\em Homotopical algebra}, volume~43 of {\em Lecture {N}otes in
  {M}athematics}.
\newblock Springer, 1967.

\bibitem[Qui69]{Quillen}
D.~Quillen.
\newblock Rational homotopy theory.
\newblock {\em Annals of Mathematics}, 90:205--295, 1969.

\bibitem[RN17]{rn17cosimplicial}
D.~Robert-Nicoud.
\newblock Representing the {D}eligne--{H}inich--{G}etzler $\infty$-groupoid.
\newblock 2017.
\newblock \href{https://arxiv.org/abs/1702.02529}{arXiv:1702.02529}.

\bibitem[RN18a]{rn17tensor}
D.~Robert-Nicoud.
\newblock Deformation theory with homotopy algebra structures on tensor
  products.
\newblock {\em Documenta Mathematica}, (23):189--240, 2018.
\newblock \href{https://arxiv.org/abs/1702.02194}{arXiv:1702.02194}.

\bibitem[RN18b]{rn18}
D.~Robert-Nicoud.
\newblock A model structure for the {G}oldman--{M}illson theorem.
\newblock 2018.
\newblock \href{https://arxiv.org/abs/1803.03144}{arXiv:1803.03144}.

\bibitem[RNV]{rnv18}
D.~Robert-Nicoud and B.~Vallette.
\newblock Homotopy {L}ie theory.
\newblock Draft.

\bibitem[RNW17]{rnw17}
D.~Robert-Nicoud and F.~Wierstra.
\newblock Homotopy morphisms between convolution homotopy {L}ie algebras.
\newblock 2017.
\newblock \href{https://arxiv.org/abs/1712.00794}{arXiv:1712.00794}.

\bibitem[RNW18]{rnw18}
D.~Robert-Nicoud and F.~Wierstra.
\newblock Convolution algebras and te deformation theory of infinity-morphisms.
\newblock 2018.
\newblock \href{https://arxiv.org/abs/1806.03371}{arXiv:1806.03371}.

\bibitem[Rog16]{rog16}
C.~L. Rogers.
\newblock Homotopical properties of the simplicial {M}aurer--{C}artan functor.
\newblock 2016.
\newblock \href{https://arxiv.org/abs/1612.07868}{arXiv:1612.07868}.

\bibitem[Sch90]{sch90}
F.~Schur.
\newblock Neue {B}egr{\"u}ndung der {T}heorie der endlichen
  {T}ransformationsgruppen.
\newblock {\em Mathematische Annalen}, 35:161--197, 1890.

\bibitem[Sei08]{Seidel}
P.~Seidel.
\newblock {\em Fukaya categories and {P}icard-{L}efschetz theory}.
\newblock Zurich Lectures in Advanced Mathematics. European Mathematical
  Society, {Z\"u}rich, 2008.

\bibitem[Ser53]{ser53}
J.-P. Serre.
\newblock Groupes d'homotopie et classes de groupes ab{\'e}liens.
\newblock {\em Annals of Mathematics}, pages 258--294, 1953.

\bibitem[SS12]{ss12}
M.~Schlessinger and J.~Stasheff.
\newblock Deformation theory and rational homotopy type.
\newblock 2012.
\newblock \href{https://arxiv.org/abs/1211.1647}{arXiv:1211.1647}.

\bibitem[Sta10]{sta10}
J.~Stasheff.
\newblock A twisted tale of cochains and connections.
\newblock {\em Georgian Mathematical Journal}, 17(1):203--215, 2010.
\newblock \href{https://arxiv.org/abs/0902.4396}{arXiv:0902.4396}.

\bibitem[Sul77]{sul77}
D.~Sullivan.
\newblock Infinitesimal computations in topology.
\newblock {\em Publications math{\'e}matiques de l'{IHES}}, 47(1):269--331,
  1977.

\bibitem[Sze99]{sze99}
B.~Szendr{\H{o}}i.
\newblock The unbearable lightness of deformation theory.
\newblock 1999.
\newblock Draft, available at
  \href{https://people.maths.ox.ac.uk/szendroi/defth.pdf}{people.maths.ox.ac.uk/szendroi/defth.pdf}
  as of May 2018.

\bibitem[TW15]{tw15}
V.~Turchin and T.~Willwacher.
\newblock Hochshild--{P}irashvili homology on suspensions and representations
  of {$Out(F_n)$}.
\newblock {\em Annales scientifiques de l'{ENS} (to appear)}, 2015.
\newblock \href{https://arxiv.org/abs/1507.08483}{arXiv:1507.08483}.

\bibitem[Val07]{val07}
B.~Vallette.
\newblock A {K}oszul duality for props.
\newblock {\em Transactions of the AMS}, 359(10):4865--4943, 2007.
\newblock \href{https://arxiv.org/abs/math/0411542}{arXiv:math/0411542}.

\bibitem[Val08]{val08}
B.~Vallette.
\newblock Manin products, {K}oszul duality, {L}oday algebras and {D}eligne
  conjecture.
\newblock {\em Journal f{\"u}r die reine und angewandte {M}athematik ({C}relles
  {J}ournal)}, 620:105--164, 2008.
\newblock \href{https://arxiv.org/abs/0609002}{arXiv:0609002}.

\bibitem[Val14]{val14}
B.~Vallette.
\newblock Homotopy theory of homotopy algebras.
\newblock 2014.
\newblock \href{https://arxiv.org/abs/1411.5533}{arXiv:1411.5533}.

\bibitem[VdL03]{vdL03}
P.~Van~der Laan.
\newblock Coloured {K}oszul duality and strongly homotopy operads.
\newblock 2003.
\newblock \href{https://arxiv.org/abs/math/0312147}{arXiv:math/0312147}.

\bibitem[Wie16]{w16}
F.~Wierstra.
\newblock Algebraic {H}opf invariants and rational models for mapping spaces.
\newblock 2016.
\newblock \href{https://arxiv.org/abs/1612.07762}{arXiv:1612.07762}.

\bibitem[Wie17]{w17}
F.~Wierstra.
\newblock Hopf invariants and differential forms.
\newblock 2017.
\newblock \href{https://arxiv.org/abs/1711.04565}{arXiv:1711.04565}.

\bibitem[Yal16]{yal16}
S.~Yalin.
\newblock {M}aurer--{C}artan spaces of filtered {$L_\infty$}-algebras.
\newblock {\em Journal of Homotopy and Related Structures}, 11(3):375--407,
  2016.
\newblock \href{https://arxiv.org/abs/1409.4741}{arXiv:1409.4741}.

\bibitem[Yek12]{yek12}
A.~Yekutieli.
\newblock {MC} elements in pronilpotent {DG L}ie algebras.
\newblock {\em Journal of Pure and Applied Algebra}, 216(11):2338--2360, 2012.
\newblock \href{https://arxiv.org/abs/1103.1035}{arXiv:1103.1035}.

\bibitem[Zwi93]{zwi93}
B.~Zwiebach.
\newblock Closed string field theory: {Q}uantum action and the
  {B}atalin--{V}ilkovisky master equation.
\newblock {\em Nuclear Physics B}, 390(1):33--152, 1993.
\newblock \href{https://arxiv.org/abs/hep-th/9206084}{arXiv:hep-th/9206084}.

\end{thebibliography}
